\documentclass[reqno,11pt]{amsart}
\usepackage[margin=1in]{geometry}

\usepackage[utf8]{inputenc}
\usepackage{amsmath,amssymb,amsthm, comment, mathtools}
\usepackage{amsthm}
\usepackage{hyperref}
\usepackage{mathrsfs}
\usepackage{slashed}
\usepackage{units}

\numberwithin{equation}{section}

\newtheorem{theorem}{Theorem}[section]

\newtheorem{prop}[theorem]{Proposition}

\newtheorem{lemma}[theorem]{Lemma}

\theoremstyle{remark}
\newtheorem*{remark}{Remark}

\theoremstyle{definition}

\newcommand{\tC}{\widetilde{\chi}}

\newcommand{\onabla}{\overline{\nabla}}
\newcommand{\sonabla}{\slashed{\onabla}}
\newcommand{\oH}{G}

\newcommand{\R}{\mathbb{R}}

\newcommand{\W}{\mathcal{W}}

\newcommand{\T}{{\mathcal{T}^{}}}
\renewcommand{\S}{\mathcal{S}}

\DeclareMathOperator{\tr}{tr}
\newcommand{\D}{\mathcal{D}}
\newcommand{\tD}{\widetilde{\D}}
\newcommand{\hD}{D}

\newcommand{\M}{\mathcal{M}}
\newcommand{\pa}{\partial}

\newcommand{\ve}{\varepsilon}
\newcommand{\fd}{\langle \pa_{\theta{}_{\!}} \rangle}
\newcommand{\fdh}{\fd^{\!\nicefrac{1}{2}}}
\newcommand{\fdm}{\fd_\mu}
\newcommand{\fdhm}{\fd^{\!\nicefrac{1}{2}}_\mu}

\newcommand{\Dve}{\widetilde{\Delta}}

\newcommand{\pave}{\widetilde{\pa}}
\newcommand{\xve}{\widetilde{x}}

\newcommand{\wGamma}{\widetilde{\Gamma}}
\newcommand{\wg}{\widetilde{g}}
\newcommand{\sm}{S_\ve}

\newcommand{\vertiii}[1]{{\vert\kern-0.25ex\vert\kern-0.25ex\vert #1
    \vert\kern-0.25ex\vert\kern-0.25ex\vert}}

\let\div\relax
\DeclareMathOperator{\curl}{curl}
\DeclareMathOperator{\div}{div}

\DeclareMathOperator{\tsum}{{\textstyle{\sum}}}

\newcommand{\shortminus}{\scalebox{0.75}[1.0]{\( - \)}}

\newcommand*{\bigtwo}[1]{\vcenter{\hbox{\scalebox{1.4}{\ensuremath#1}}}}
\newcommand*{\bigthree}[1]{\vcenter{\hbox{\scalebox{1.6}{\ensuremath#1}}}}
\newcommand*{\bigfour}[1]{\vcenter{\hbox{\scalebox{1.8}{\ensuremath#1}}}}

\newcommand*{\bigletter}[1]{{\hbox{\scalebox{1.6}{\ensuremath#1}}}}
\newcommand*{\subbigletter}[1]{{\hbox{\scalebox{1.1}{\ensuremath#1}}}}
\newcommand{\Tau}{\bigletter{\tau}}
\newcommand{\wTau}{\widetilde{\Tau}}
\newcommand{\subTau}{\subbigletter{\tau}}
\newcommand{\wsubTau}{\widetilde{\subTau}}

\newcommand{\nquad}{\!\!\!\!\!\!}

\newcommand{\na}{{\nabla\!}}
\newcommand{\nave}{\widetilde{\pa}}

\newcommand{\wcurl}{\widetilde{\curl}}

\newcommand{\sV}{\overline{V}}
\newcommand{\su}{\overline{u}}

\newcommand{\sg}{\overline{g}}

\newcommand{\tH}{\widetilde{H}}
\newcommand{\Dt}{\hD_s}
\newcommand{\N}{\mathcal{N}}
\newcommand{\n}{n}

\date{\today}

\title{On the local well-posedness for the relativistic Euler equations for
a liquid body}

\author{Daniel Ginsberg}
\address[D.G.]{Program in Applied and Computational Mathematics, Princeton University, Princeton, NJ 08544}
\email{ dg42@princeton.edu}
\author{Hans Lindblad} \address[H.L.]{Johns Hopkins University, Department of Mathematics, 3400 N.\@ Charles St., Baltimore, MD 21218, USA}
\email{lindblad@math.jhu.edu}

\begin{document}

\mathtoolsset{showonlyrefs=true}
\nocite{*}

\begin{abstract}
  We prove a local existence theorem for
  the free boundary problem for a relativistic fluid in a fixed
   spacetime. Our proof involves an a priori
  estimate which only requires control of derivatives tangential to the boundary, which holds also in the Newtonian compressible case.
\end{abstract}
\maketitle
\tableofcontents
\section{Introduction}

Fix a Lorentz metric $g$ and a four-dimensional globally hyperbolic spacetime
$(\M,g)$.
In units where the speed of light is one,
the motion of a perfect fluid in the spacetime $(\M, g)$ is described by
Einstein's equations
\begin{equation}
  R_{\mu\nu} - \frac{1}{2}g_{\mu\nu} R  =  T_{\mu\nu},
  \label{eineul}
\end{equation}
where $R_{\mu\nu}$ is the Ricci curvature of $g$, $R = g^{\mu\nu} R_{\mu\nu}$
is the scalar curvature and $T$ is the energy-momentum tensor of a perfect fluid,
\begin{equation}
    T_{\mu\nu} = (\rho + p) u_\mu u_\nu + p g_{\mu\nu}.
    \label{enmom}
\end{equation}
Here, $u = u^\mu \pa_\mu$ is the fluid velocity, by assumption a unit timelike future-directed
vector,
\begin{equation}
    g(u,u)=-1,\qquad \text{and} \qquad
    g(u, \Tau) <0,
\end{equation}
where $\Tau$ is the future-directed timelike vector defining the time axis in $(\M, g)$.
The quantity $\rho\! \geq\! 0$ is the energy density of matter and $p\! \geq \! 0$
is the pressure. In \eqref{enmom},  $u_\mu\!= \!g_{\mu\nu} u^\nu$
are the components of the one-form associated to $u$.
By the Bianchi identity, Einstein's equations \eqref{eineul} imply
\begin{equation}
    \nabla^\mu T_{\mu\nu} = 0,
    \label{conservation}
\end{equation}
where $\nabla$ denotes the Levi-Civita connection with respect to the metric $g$.

We assume that mass is conserved, so that if $n$ denotes
the mass density,
 \begin{equation}
  \na_\mu (u^\mu n) = 0.
  \label{particle}
 \end{equation}
 For an isentropic fluid, the laws of thermodynamics give
 the following relation between $p, \rho, n$,
 \begin{equation}
 \frac{d\rho}{dn} = \frac{p + \rho}{n}.
 \label{thermo}
 \end{equation}

 We will consider here a barotropic fluid, meaning that
 the energy density and pressure are determined from the mass density
 alone by prescribed equations of state,
 \begin{equation}
  \rho = E(n), \qquad p = P(n),
  \label{baro}
 \end{equation}
 where $P$ and $E$
 are assumed to be invertible smooth positive functions of $n\geq 0$.
   We can therefore think of any one
 of $p, \rho, n$ as the fundamental thermodynamical variable. In fact it
 is more conveninent to work in terms of the enthalpy $\sigma$ defined by
 \begin{equation}
 \sigma = \frac{p + \rho}{n}.
  \label{sigmadef}
 \end{equation}
 Introducing the rescaled fluid velocity $v_\mu = \sqrt{\sigma} u_\mu$,
 combining the equations \eqref{conservation}-\eqref{particle} with
 \eqref{thermo} we find the system (see \cite{Christo1})
\begin{alignat}{2}
 v^\nu\nabla_\nu v^\mu + \frac{1}{2} \nabla^\mu \sigma &=
 0,
 &&\qquad \text{ in } \D_t, \label{rescaledreleul}\\
 v^\nu\nabla_\nu e(\sigma) + \nabla_\mu v^\mu &= 0,
 &&\qquad
 \text{ in } \D_t,
 \label{rescaledrelcont}
\end{alignat}
with $e(\sigma) = \log (n(\sigma)/\sqrt{\sigma})$, where
$n(\sigma)$ is obtained by inverting the relation \eqref{sigmadef} after
expressing $p = P(n), \rho = \rho(n)$.
We define the sound speed by
\begin{equation}
 \eta^2 = \frac{d}{d \rho} P(\rho).
 \label{soundspeed}
\end{equation}
In our units the speed of light is one and so a
 basic physical requirement on $\eta$ is
\begin{equation}
 \eta^2 \leq 1.
 \label{soundspeedbd}
\end{equation}
In this case the quantity $e'(\sigma) \!\geq \!0$.
The case $\eta \!\equiv\! 1$ corresponding to $e'(\sigma) \!= \!0$
 is the relativistic analogue
of an incompressible fluid for which the continuity equation
\eqref{rescaledrelcont} takes the form $\nabla_{\!{}_{\!}\mu} v^\mu \!=\! 0$.
We consider here an equation of state with sufficiently ``large'' sound speed,
\begin{equation}
 1- \delta \leq \eta^2,
 \label{largesoundspeed}
\end{equation}
for $\delta$ sufficiently small.

Let $t$ denote the time function
associated to $(\mathcal{M}, g)$. We are interested in the system
\eqref{rescaledreleul}-\eqref{rescaledrelcont} when $(v, \sigma)$ describe a
fluid body surrounded by a pressureless dust and where
the boundary moves with the velocity
of the fluid. If at time $t$ the fluid body occupies a region $\D_t$,
 the boundary conditions are
\begin{alignat}{2}
  p &=0, \qquad \text{ on } \pa \D_t,\label{vac}\\
g(  \mathcal{N}, v)  &= 0, \qquad \text{ on } \Lambda = \cup_{0 \leq t \leq T}\,
  \pa \D_t , \label{freebdy}
\end{alignat}
where $\mathcal{N}$ is a normal vector field to $\Lambda$.
These conditions ensure that the integral form
of the conservation laws \eqref{conservation}-\eqref{particle} hold across
the surface $\Lambda$ and they
imply energy conservation \eqref{basicidentrel}.
From \eqref{vac}, \eqref{baro} we get
$\rho \!=\! \rho_0$ on $\pa \D_t\!$ for a constant $\rho_0$. We will consider
equations of state with
\begin{equation}
  \rho_0 > 0,
  \label{liquid}
\end{equation}
in which case the fluid is caled a ``liquid''.
We will also assume that the mass and energy densitities $\rho, n$ are
strictly bounded below in the fluid domain,
\begin{align}
 \rho \geq \rho_1 > 0, \qquad
 n \geq n_1 > 0, \qquad \text{ in } \D_t.
 \label{densitypositivity}
\end{align}
In this case the physical energy \eqref{basicidentrel} gives uniform
control over all components of $u$ up to the boundary since even though
$p = 0$ at the boundary, we have
$g(u,u)=u_\subTau^2+g(\su,\su)=-1$ where $u_{\subTau} = g(u, \tau)$.
In order to get bounds for higher-order energies we require that the Taylor sign
condition holds
\begin{equation}
 \pa_\mathcal{N} p \leq -c < 0, \qquad \text{ on } \Lambda.
 \label{tsc}
\end{equation}
In the non-relativistic setting it was shown in \cite{E1} that the
corresponding free-boundary problem for Euler's equations is ill-posed
in Sobolev spaces unless \eqref{tsc} holds.

The problem \eqref{rescaledreleul}-\eqref{rescaledrelcont} with liquid
boundary condition \eqref{liquid}
was considered by \cite{Christo1} as a model
for the gravitational collapse of a star. See also \cite{FS19}.

Here we consider the system \eqref{conservation}-\eqref{particle} with
$(\M, g)$ a fixed globally hyperbolic spacetime, with initial
data satisfying the conditions \eqref{densitypositivity} and the  sign
condition \eqref{tsc}. Our main result is that for sufficiently
smooth initial data satisfying
compatibility conditions (which are given in section
\ref{existencerel}), and for a sufficiently smooth background metric
$g$, the problem \eqref{rescaledreleul}-\eqref{rescaledrelcont} is locally well-posed.
\begin{theorem}
  \label{mainrelthm}
  Fix $r \geq 10$, a globally hyperbolic spacetime
  $( M\times [0,T],g)$, a global coordinate system $\{x^1, x^2, x^3\}\times \{t\}$
  on $M\times [0, T]$, and invertible functions
  $P, E \in C^\infty(\R_{\geq 0};\R_{\geq 0})$ so that the
  sound speed \eqref{soundspeed} satisfies
  \eqref{soundspeedbd} and \eqref{largesoundspeed} for $\delta$ sufficiently
  small.
   Suppose that expressed in this coordinate system the
  components of the metric $g_{\mu\nu}$ satisfy
  $\pa_t^k g_{\mu\nu}(t,\cdot) \in C^{r-k+2}(M)$ for $k = 0,..., r$, where
  $C^j(M)$ denotes the usual H\"{o}lder space on $M$.

  Let $\D_0 \subset M\times \{t = 0\}$
  be diffeomorphic to the unit ball and fix initial data $\mathring{u},
  \mathring{\rho}$ with
  \begin{equation}
     {\sum}_{\mu = 0}^3
   \|\mathring{u}^\mu\|_{H^{r}(\D_0)}
   + \|\mathring{\rho}\|_{H^{r}(\D_0)} < \infty, \quad
   S \in \mathcal{S},
   \qquad
   \mathring{\rho} \geq \rho_1 > 0,
  \end{equation}
  for a constant $\rho_1$, and moreover
  which satisfies the compatibility conditions
  \ref{relcompat} to order $r$.

  Then the problem \eqref{rescaledreleul}-\eqref{rescaledrelcont} with boundary conditions
  \eqref{vac}-\eqref{freebdy} has a unique solution
  $u^\mu(t)\! \in\! H^r(\D_t)$, $0 \!\leq \!\mu\! \leq \!3$, $\rho(t)\! \in \! H^r(\D_t)$
  with $\rho\! =\! E(n)$, $p \!=\! P(n)$ for $t\! \leq \!T_0 $ for some
  $0\! < \!T_0\! \leq\! T$,
  with initial data $u|_{t= 0} = \mathring{u}, \rho|_{t = 0} = \mathring{\rho}$.
  The Taylor sign condition \eqref{tsc} holds on $[0, T_0]$
  with $c$ replaced by $c/2$.
\end{theorem}
Apriori bounds for this system were previously proven in \cite{O17}, \cite{G19}.
Existence for this problem was first proven in
\cite{O19}, by solving an evolution equation for the boundary condition for the velocity and using a Galerkin method.
In \cite{Miao20}, the authors gave a simpler proof using the same idea
in the special case that $g$ is the Minkowski metric and the fluid is irrotational and divergence free.
Our approach is different, for existence we instead solve a Dirichlet problem for the enthalpy. We also give a simplification and an improvement of our previous proof for the related compressible case \cite{GLL19}. Our norms use only one derivative normal to the boundary and apart from that only tangential regularity, and this is new also in the compressible case. We expect this to be important for the nonlinear coupled problem where the metric satisfies Einstein's equations since these hold also outside the domain and we expect that the metric will have limited normal regularity over the boundary, as was the case for the Newtonian gravity potential in \cite{GLL19}. Moreover we get additional regularity of the Lagrangian coordinates and hence of the boundary.

In the remainder of this section we give an outline of the main ideas involved
in the proof.

\subsection{The energy estimate}
There is a physical energy associated to the conservation law \eqref{conservation}.
Multiplying \eqref{conservation}
by the generator of the time axis $\Tau$ and integrating over the region bounded by two time slices
$\D_0, \D_t$ and the lateral boundary $\Lambda$, after using the
boundary conditions \eqref{vac}-\eqref{freebdy},
\begin{equation}
 \mathcal{E}_0(t) = \mathcal{E}_0(0) +
 \int_0^t \int_{\D_t} T_{\mu\nu} \mathcal{L}_\subTau g^{\mu\nu}\, dx dt,
 \qquad \text{ where } \mathcal{E}_0(t) = \int_{\D_t} \rho u_\subTau^2 + p g(\su, \su)\, dx
 \label{basicidentrel}
\end{equation}
Here $\mathcal{L}_{\subTau}g$ denotes the Lie derivative of $g$ with
respect to $\Tau$. The last term vanishes if $g$ is stationary with
respect to $\Tau\!$, e.g.  when $g$ is the Minkowski metric and $t$ is
the standard time coordinate.

In order to prove a higher-order version of the energy identity \eqref{basicidentrel},
we introduce Lagrangian coordinates which fix the boundary.
Let $\Omega \subset \mathcal{M}\cap\{t = 0\}$
denote the unit ball. The Lagrangian coordinates $x^\mu = x^\mu(s, y)$ are
maps $x^\mu(s,\cdot): \Omega \to \mathcal{M}$
given by solving
\begin{equation}
 \frac{d}{ds} x^\mu(s, y) = v^\mu(x(s,y)),
 \qquad x^0(0, y) = 0, x^i(0, y) = y^i.
\end{equation}
We fix a family of vector fields in the $y$-coordinates $T = T^a(y)\pa_{y^a}$
which are tangent to the boundary $\pa \Omega$ at the boundary. Then
$T$ commutes with the material derivative
\begin{equation}
D_s=v^\mu \pa_\mu,
\end{equation}
 but the commutator $[T, \pa_\mu]$ involves
$x$ to highest order,
\begin{equation}
 [T, \pa_\mu] = -(\pa_\mu T x^\nu)\pa_\nu.
 \label{com}
\end{equation}
Let $T^I$ denote a collection of the vector fields $T$.
Applying $T^I\!$ to  \eqref{rescaledreleul}  using
\eqref{com} we find that
\begin{equation}
 v^\nu \na_\nu T^I \!v^\mu\! - \frac{1}{2} \nabla^\mu T^I\! x^\nu\,\, \na_\nu \sigma
 + \frac{1}{2} \nabla^\mu T^I \!\sigma = F^{\mu I}\!\!\!,
 \qquad
 e'({}_{\!}\sigma{}_{\!}) v^\mu\nabla_\mu T^I\!\sigma
 +\na_\mu T^I\! v^\mu\! - \na_\mu T^I \!x^\nu\,\,  \na_\nu v^\mu\!= G^I\!\!,
 \label{TIeqn}
\end{equation}
for lower-order terms $F^{\mu I}, G^I$.
If we define
\begin{equation}
 v_I^\mu = T^I v^\mu - T^I x^\nu \nabla_\nu v^\mu,
 \qquad
 \sigma_I = T^I \sigma - T^I x^\nu \nabla_\nu \sigma,
\end{equation}
then \eqref{TIeqn} take the form
\begin{equation}
 v^\nu \nabla_\nu v^{\mu}_I + \frac{1}{2} \nabla^\mu \sigma_I = F^{\mu}_I,
 \qquad
 e'(\sigma) v^\mu\nabla_\mu \sigma_I +
\nabla_\mu v^{\mu}_I = G_I,
 \label{diffmom}
\end{equation}
where $F^{\mu}_I, G_I$ are lower-order. The variables
$v_I, \sigma_I$ are related
to Alinhac's good unknowns and also to covariant differentiation in the Lagragian coordinates used in \cite{CL00}, see
\eqref{alinhacdiscussion}.

Multiplying both sides of the first equation in \eqref{diffmom} by $g_{\mu\nu} v^\nu_I$
we get
\begin{equation}
 \frac{1}{2} \nabla_\nu \left( v^\nu g(v_I, v_I) + v^{ \nu}_I \sigma_I
 + v^\nu e'(\sigma) (\sigma_I)^2\right)
 = g(F_I, v_I) + \frac{1}{2}\sigma G_I
 - \nabla_\nu(v^\nu e'(\sigma)) (\sigma_I)^2.
 \label{divformdiffmom}
\end{equation}
We note that since $\sigma = -g(v,v)$, to highest order we have
$\sigma^I\! = -2 g(v_I, v)$,
and we get
\begin{equation}
 \frac{1}{2}\nabla_\nu \left( v^\nu g(v_I, v_I) -2 v^{\nu}_I g(v_I, v)
 +
 v^\nu e'(\sigma) (\sigma_I)^2\right)
 = H_I,
\end{equation}
where $H_I$ is lower-order.
Introducing the higher-order energy-momentum tensor $Q[v_I]$,
\begin{equation}
 Q[v_I](X, Y) = 2 g(v_I, X) g(v_I, Y)-g(X, Y) g(v_I, v_I),
\end{equation}
and taking $e' = 0$ for the moment for the sake of simplicity,
integrating the expression \eqref{divformdiffmom} over the region
$\mathcal{R}$ bounded between two spacelike surfaces
$\Sigma_1, \Sigma_0$ and the timelike surface $\Lambda$
and using the divergence theorem leads to the identity
\begin{equation}
 \int_{\Sigma_1} Q[v_I](v, n_{\Sigma_1}) -
 \int_{\Sigma_0} Q[v_I](v, n_{\Sigma_0}) + \int_{\Lambda} Q[v_I](v, \mathcal{N})\, dS
 = \int_{\mathcal{R}}H_I.
 \label{integrated}
\end{equation}

We claim that the integrands over the spacelike surfaces $\Sigma_1, \Sigma_0$
are positive-definite. This is the usual positivity of the
energy-momentum tensor $Q_I$ evaluated at the timelike future-directed
vector fields $v, n_{\Sigma}$. This positivity can be seen by
recalling that $g(n_{\Sigma}, n_{\Sigma}) = -1$ and writing
\begin{equation}
 v_I = -g(v_I, n_{\Sigma}) n_{\Sigma} + \overline{v}_I,
\end{equation}
where $\overline{v}_I$ is orthogonal to $n_{\Sigma}$ and thus spacelike.
A simple calculation (see Lemma \ref{emtensorpositivity}) shows that
\begin{equation}
 Q[v_I](v, n_{\Sigma}) \geq
\left(g(v_I, n_{\Sigma})^2 +
 g(\overline{v}_I,\overline{v}_I)\right)\alpha,
  \qquad \text{ where }\quad
  \alpha =
  \frac{-g (v,v)}{g(\overline{v}, \overline{v})^{1/2}
  - g(v, n_{\Sigma})} > 0,
  \label{alphanotation}
\end{equation}
where the statement $\alpha > 0$ follows from the fact that $v$ is timelike $g(v,v) < 0$ and
future-directed, so $g(v, n_{\Sigma}) < 0$ if $n_{\Sigma}$ is
future-directed.

Then \eqref{integrated} implies that
\begin{equation}
 E_I[\Sigma_1] + \int_{\Lambda} Q[v_I](v, \mathcal{N})\, dS
 \lesssim E_I[\Sigma_0] + \int_{\mathcal{R}} |H_I|,
 \qquad \text{ where }\quad E_I[\Sigma] =
 \int_{\Sigma} \left(g(v_I, n_{\Sigma})^2  + g(\overline{v}_I, \overline{v}_I)\right)\,
 \alpha.
 \label{integrated2}
\end{equation}

As for the integral over $\Lambda$, the observation is that if
the Taylor sign condition \eqref{tsc} holds then this contributes a positive term
to the energy.
Recalling that $2g(v_I, v) = -\sigma_I$ to highest order,
 using \eqref{freebdy}, we find
\begin{equation}
 Q[v_I](v, \mathcal{N}) = 2g(v_I, v) g(v_I, \mathcal{N})
 =- \sigma_I g(v_I, \mathcal{N}) + \text{ lower-order terms}.
\end{equation}
Now we note that at the boundary,
\begin{equation}
  \sigma_I = T^I\sigma - T^Ix^\nu \nabla_\nu \sigma =
  g(T^I x, \mathcal{N})  \nabla_{\mathcal{N}}\sigma,
\end{equation}
where we used the Taylor sign condition \eqref{tsc} to write
$\nabla_\nu \sigma = - \mathcal{N}_\nu (\nabla_{\mathcal{N}}\sigma)$.
Therefore, since the difference $v_I - T^I v$ is lower-order, we
find that to highest order,
\begin{equation}
 Q[v_I](v, \mathcal{N}) = - g(T^Ix, \mathcal{N}) g(T^I v, \mathcal{N})\nabla_{\mathcal{N}}\sigma
 = \frac{1}{2} \frac{d}{ds} \left(g(T^Ix, \mathcal{N})^2 \right)\nabla_{\mathcal{N}}\sigma.
\end{equation}
Therefore if we set $\Lambda_{q} = \Lambda \cap \Sigma_q$ we find that
\begin{equation}
 E_I[\Sigma_1] + B_I[\Lambda_1]
 \lesssim
 E_I[\Sigma_0] + B_I[\Lambda_0] +
 \!\int_{\mathcal{R}} \!\!|H_I|
 +\! \int_{\Lambda} \!\!|R_I|,
 \quad \text{where}\quad B_I[\Lambda_q]
 =\! \int_{\Lambda_q}\!\!\! g(T^I\!x, \mathcal{N})^2 \nabla_{\mathcal{N}}\sigma,
 \label{higherorderenergy}
\end{equation}
where $R_I$ collects the error terms we generated on the boundary.

To deal with the case $e'(\sigma)\not=0$, we argue just as above
but note that since $v^\nu\mathcal{N}_\nu = 0$ there is no contribution
from the term $v^\nu e'(\sigma)  (\sigma_I)^2$ at the boundary.
We therefore get
\eqref{higherorderenergy} but where the energies $E_I[\Sigma]$ on
the time slices are replaced by
\begin{equation}
 E_I[\Sigma] = \int_{\Sigma}
 \left(g(v_I, n_{\Sigma})^2  + g(\overline{v}_I, \overline{v}_I)\right)\,
 \alpha - e'(\sigma) |\sigma_I|^2 g(v, n_{\Sigma}).
\end{equation}

\subsubsection{The $L^2$ norms}
In order to control the remainder terms in the right hand side of \eqref{higherorderenergy} it is not quite enough to only control $r=|I|$ tangential derivatives only but we have to control the full gradient
of $r-1$ tangential derivatives. However, any derivative can be controlled in terms of these and tangential derivatives by the point wise estimate:
\begin{equation}
|{\pa} T^J V|
\lesssim |{\div}\,T^J V|+|{\curl}\,T^J V|
+\!\! {\sum}_{T\in{\mathcal T}}|S T^J V|,
\end{equation}
where here the divergence and curl stands for the space time divergence and curl and $\mathcal{T}$ are the space time tangential vector fields.
This together with good equations for the divergence and for the curl of the velocity, gives us control of the energies
 \begin{equation}
  E_r(t) =
  {\sum}_{|I| \leq r}
  \int_{\D_t} |\pa T^J v|^2
  +  e'(\sigma) | T^I \sigma|^2\, dx
   + {\sum}_{|I| \leq r }\int_{\pa \D_t} \left( (T^I x^\mu) \mathcal{N}_\mu
  \right)^2 (-\nabla_{\mathcal{N}} p)^{-1}\, dS,
 \end{equation}
 where $e(\sigma) = \log(n(\sigma)/\sqrt{\sigma})$ is determined by the
 equation of state, $x^\mu$ is defined by $v^\nu \pa_\nu  x^\mu = v^\mu$,
 and $\mathcal{N}$ denotes the spacetime normal vector field to the
 boundary $\pa \D$. Energies of this type with an interior term and a boundary term was first introduced in \cite{CL00} in the Eulerian coordinates where the boundary term was interpreted as norms of the second fundamental form of the free boundary, assuming the physical condition that $-\nabla_{\mathcal{N}} p\geq c>0$ on the boundary.

\subsubsection{The curl estimate and the divergence estimates} Taking the space-time curl of the first equation in \eqref{diffmom} we see that
\begin{equation}
|D_s {\curl}\,T^J V|\lesssim |\pa \,T^J V|+\text{ lower-order terms},
\end{equation}
is bounded by the energy for $|J|=r-1$.
On the other hand, the second equation in \eqref{diffmom} gives an equation for the space time divergence
\begin{equation}
|\div \,T^J V|\lesssim e^\prime(\sigma)|D_s \sigma_J|+\text{ lower-order terms},
\end{equation}
which is bounded by the energy for $|J|=r-1$.

\subsubsection{The $L^\infty$ norms}

There are similar evolution equations for the $L^\infty$ norms assuming bounds for tangential derivatives
which allow us to control the quantity
\begin{equation}
 M_r(t)= {\sum}_{|K| \leq r} \|\pa T^K x\|_{L^\infty}+
 \|\pa T^L V\|_{L^\infty}
 +\|\pa T^L {V}\|_{L^\infty}
 +\|\pa T^L \pa\sigma\|_{L^\infty}.
 \label{capitalMdef}
\end{equation}
Let $(r/2)$ denote $r/2$ when $r$ is even and $(r-1)/2$ when $r$ is odd.
Then our energies are bounded provided we have a bound for $M_{(r/2)}$,
\begin{equation}
 M_{(r/2)}(t) \leq M < \infty,
 \label{lownorms}
\end{equation}
and moreover
we can control $M_{(r/2)}$ provided we control the energies, see the next
section.

\subsubsection{Control of the energies under a priori assumptions}
It turns out that we can prove energy estimates assuming only tangential
regularity of the background metric $g$ to top order. We will prove bounds
provided we have control over the following quantities.
We will assume that in the fluid domain $\D_t$ we have
the bounds
\begin{multline}
  \sum_{|I| \leq r}
 \sum_{\mu,\nu,\gamma = 0}^3\int_{D_t}
|\pa T^I \Gamma_{\mu\nu}^\gamma|^2
+ |\pa T^I g_{\mu\nu}|^2
+|T^I \Gamma_{\mu\nu}^\gamma|^2
+|T^I g_{\mu\nu}|^2\\
+  \sum_{|K| \leq r/2+1}
\sum_{\mu,\nu,\gamma = 0}^3
\|\pa T^K \Gamma_{\mu\nu}^\gamma(t)\|_{L^\infty}
+ \|\pa T^K g_{\mu\nu}(t)\|_{L^\infty}\\
+
\sum_{|K| \leq r/2+1}
\sum_{\mu,\nu,\gamma = 0}^3\|T^K \Gamma_{\mu\nu}^\gamma(t)\|_{L^\infty}
+\|T^K g_{\mu\nu}(t)\|_{L^\infty}
+ \|g^{\mu\nu}(t)\|_{L^\infty}
 \leq G_r.
 \label{metricbdsintro0}
\end{multline}
for $0 \leq t \leq T$. Then we have the following a priori
estimate, proven in Section \ref{eulconclusionsec}.
\begin{theorem}
  There are continuous functions $C_r$ so that any smooth solution
  of \eqref{rescaledreleul}-\eqref{rescaledrelcont} with sound speed
  $\eta$ as in \eqref{soundspeedbd}-\eqref{largesoundspeed} for
  $\delta$ sufficiently small,
  which satisfies the Taylor sign condition
  \eqref{tsc}, the a priori assumption \eqref{lownorms}, the condition
   $\rho \geq \rho_1 > 0$ in $\D$ and for which the bounds for the metric
 \eqref{metricbdsintro0} hold for $0\leq t \leq T$, satisfies the energy estimate
 \begin{equation}
  E_r(t) \leq C_r(t, M, G_{r-1}, 1/c, \delta, E_{r-1}(0))  E_r(0), \qquad 0\leq t \leq T.
  \label{highnormbd}
 \end{equation}
 Moreover, there are a continuous functions $\mathcal{T}_r
 = \mathcal{T}_r(G_{r-1}, 1/c, \delta, E_{r}(0))$ so that for
 $k\leq r/2$,
  \begin{equation}
  M_k(t) \leq 2 M_k(0), \qquad 0\leq t \leq
  \mathcal{T}_r.
  \label{lownormbd}
 \end{equation}
\end{theorem}
Using the elliptic estimates from Lemma \ref{pwHlemma}, these energies also control normal derivatives;
\begin{equation}
 \int_{\D_t}
 \sum_{|I| \leq r-1}
 |\pa T^I v|^2 + \sum_{|J| \leq r-2}
 |\pa^2 T^J \sigma|^2
 \lesssim E_r(t).
\end{equation}

\subsubsection{The wave equation estimate for the enthalphy}
 Subtracting \eqref{div1} from $\hD_s\! = \!V^\nu\pa_\nu$ applied to \eqref{rescaledrelcont2} we find
\begin{equation}
 e'(\sigma) \hD_s^2 \sigma - \frac{1}{2} \nabla_\nu
 (g^{\mu\nu}\nabla_\mu
 \sigma)
 = R,
 \label{sigmawavesetup1}
\end{equation}
where
\begin{equation}
 R =\nabla_\mu V^\nu \nabla_\nu V^\mu +  R_{\mu\nu\alpha}^\mu V^\nu V^\alpha
 - e''(\sigma) (\hD_s\sigma)^2.
 \label{Rdef}
\end{equation}
and corresponding equations for higher derivatives
\begin{equation}
 e'(\sigma) \hD_s^2 T^J \sigma - \frac{1}{2} \nabla_\nu
 (g^{\mu\nu}\nabla_\mu
T^J \sigma)
 = R_J.
 \label{sigmawavesetup1intro}
\end{equation}
When $e'(\sigma) \equiv 0$ this is just a wave equation
with respect to the metric $g$ and when $e'(\sigma) >0$
the first term in \eqref{sigmawavesetup1intro} contributes
an additional positive term to the energy.
Define the higher-order energy-momentum tensor for $\sigma$
\begin{equation}
 Q[\sigma^\prime_J]_{\alpha\beta}=\na_\alpha \sigma^\prime_J \, \na_\beta \sigma^\prime_J- \frac{1}{2}g_{\alpha\beta} g^{\mu\nu} \na_\mu \sigma^\prime_J \na_\nu \sigma^\prime_J , \qquad \sigma^\prime_{J\alpha}=\nabla_\alpha T^J\sigma .
\end{equation}
Then with $\sigma^\prime_{Js}= D_s T^J \sigma$, after multiplying \eqref{sigmawavesetup1intro}
by $\hD_s T^J \sigma = V^\alpha \pa_\alpha T^J\sigma$ we find the identity
\begin{equation}
  \big( \nabla^\alpha \na_\alpha T^J\!\sigma - 2e^\prime(\sigma) D_s^2 T^J\!\sigma\big)V^\beta \na_\beta  T^J\!\sigma\!
=\!
\nabla^\alpha \big(Q[\sigma^\prime_J]_{\alpha\beta}V^\beta + 2e^\prime(\sigma) V_\alpha {\sigma^\prime_{\!Js}\!\!}^2 \big)
+\big(\nabla^\alpha( e^\prime(\sigma) V_\alpha)\big) {\sigma^\prime_{\!Js}\!\!}^2+ K_J,
\label{identintro}
\end{equation}
with
\begin{equation}
 K_J = -Q[\sigma^\prime_J]_{\alpha\beta}\nabla^\alpha V^\beta - 2\hD_s e'(\sigma)
 (\sigma^\prime_{Js}).
\end{equation}

Taking $X = V$
and integrating the identity \eqref{identintro} over the region $\mathcal{R}$ bounded by two spacelike
surfaces $\Sigma_0, \Sigma_1$
and the timelike surface $\Lambda$, with $\Sigma_1$ lying to the future of $\Sigma_0$
gives
\begin{equation}
 \int_{\Sigma_1} Q[\sigma^\prime_J](V, n^{\Sigma_1}) -
 \int_{\Sigma_0} Q[\sigma^\prime_J](V, n^{\Sigma_0})
 +\int_{\Lambda} Q[\sigma^\prime_J](V, \mathcal{N}) =
 \int_{\mathcal{R}}K_J + R_J \hD_s T^J \sigma.
\end{equation}
The term on $\Lambda$ vanishes
since $\sigma$ is constant on the boundary and
$g(V, \mathcal{N}) = 0$ so $V$ is tangent to the boundary.
As for the terms on the spacelike surfaces, we have,
with $\overline{X}$ the part of $X$ parallel to $\Sigma$
and notation as in \eqref{alphanotation},
\begin{equation}
 Q[\sigma^\prime_J](V, n^\Sigma)
\geq
 \frac{1}{2}\left( (n^{\Sigma} \cdot \nabla T^J \sigma)^2
+ |\overline{\nabla}T^J \sigma|^2\right) \alpha + \frac{1}{2} e'(\sigma) (\hD_sT^J\sigma)^2.
\end{equation}
Therefore with
\begin{equation}
 W_J[\Sigma] = \int_{\Sigma} \left( (n^{\Sigma} \sigma)^2
 + |\overline{\nabla}\sigma|^2\right) \alpha + \int_{\Sigma} e'(\sigma) (\hD_sT^J\sigma)^2,
\end{equation}
we find the energy identity
\begin{equation}
 W[\Sigma_1] \lesssim W[\Sigma_0] + \int_{\mathcal{R}} |K_J|
 + |R_J| |\hD_s T^J \sigma|,
\end{equation}
where $|K_J|$, $|R_J|$ consist of lower-order terms, which give
control along the spacelike surface
$\Sigma_1$.

\subsubsection{The elliptic estimate for the enthalpy}
As it turns out in the proof we also need an improved elliptic estimate for
the enthalpy to get better control of spatial derivatives.
With $n_\alpha=\pa_\alpha s$ the conormal to the surfaces $s=const$, we can write
$\pa_\alpha=n_\alpha\,\pa_s +\overline{\pa}_\alpha $,
where $\overline{\pa}_\alpha$ differentiates along the surfaces $s=const$.
Since ${V}^\alpha n_\alpha ={V}^\alpha\pa_\alpha s=1$ we have $\overline{\pa}_\alpha =\gamma_\alpha^{\alpha^\prime}\pa_{\alpha^\prime}$,
where $\gamma_\alpha^{\alpha^\prime}=\delta_\alpha^{\alpha^\prime}-n_\alpha \, {V}^{\alpha^\prime}$. With $\xi_s={V}^\alpha \xi_\alpha$ and $\overline{\xi}_\alpha=\gamma_\alpha^{\alpha^\prime}\xi_{\alpha^\prime}$
the symbol for the wave operator can hence be decomposed
\begin{equation}
  \label{coordwavedecomp1}
g^{\alpha\beta}\xi_\alpha \xi_\beta =
g^{\alpha\beta}n_{\alpha} \, n_{\beta}\,\xi_s \xi_s +
2 g^{\alpha\beta}n_{\alpha}\, \xi_s\overline{\xi}_{\beta}
+g^{\alpha\beta}\overline{\xi}_{\alpha}
\overline{\xi}_{\beta}.
\end{equation}
The principal part that only differentiates along the surface $s=const$ is
\begin{equation}
  \label{coordwavedecomp2}
g^{\alpha\beta}\overline{\xi}_{\alpha}
\overline{\xi}_{\beta}=G_1^{\alpha\beta}{\xi}_{\alpha}
{\xi}_{\beta},\qquad\text{where}\qquad
G_1^{\alpha\beta}=g^{\alpha'\beta'}\gamma_{\alpha'}^{\alpha}\gamma_{\beta'}^{\beta} .
\end{equation}
This gives an elliptic operator restricted to the surfaces $s\!=\!const$.
i.e. $g^{\alpha\beta}\overline{\xi}_{\alpha}
\overline{\xi}_{\beta}\!>\!c\delta^{\alpha\beta}\overline{\xi}_{\alpha}
\overline{\xi}_{\beta}$, for some $c\!>\!0$.
In fact,
 $\overline{\xi}^\alpha\!\!=\!g^{\alpha\beta}\overline{\xi}_\beta$ is in the orthogonal complement of ${{V}}^\beta\!\!$, since $g_{\alpha\beta}\overline{\xi}^\alpha{{V}}^\beta\!
\!=\!\overline{\xi}_\beta{{V}}^\beta\!\!=\!0$, since ${V}^\alpha n_\alpha\! =\!1$.  Since
${V}\!$ is timelike
$g_{\alpha\beta}{{V}}^\alpha{{V}}^\beta\!<\!0$ it follows that
$\overline{\xi}$ is spacelike $g_{\alpha\beta}\overline{\xi}^\alpha\overline{\xi}^\beta\!>\!0$.

\subsubsection{Comparison with the Newtonian case}
In \eqref{rescaledreleul}-\eqref{rescaledrelcont}
and the following, we use the convention that Greek indices run over 0,1,2,3.
For a scalar $\nabla^\nu q = g^{\mu\nu}\pa_\mu q$ and for a
vector field $T = T^\mu \pa_{x^\mu}$,
$\nabla_\nu T^\mu = \pa_\nu T^\mu + \Gamma_{\nu\alpha}^\mu T^\alpha$ where
$\Gamma_{\nu\alpha}^\mu$ are the Christoffel
symbols of the metric $g$
\begin{equation}
  \Gamma_{\nu \alpha}^\mu = \frac{1}{2}g^{\mu \beta} \big( \pa_\nu g_{\alpha \beta}
  + \pa_{\alpha} g_{\nu \beta} - \pa_\beta g_{\nu \alpha}\big),
\end{equation}
so \eqref{rescaledreleul}-\eqref{rescaledrelcont}
can be written as
\begin{align}
 v^\nu\pa_\nu v^\mu + \frac{1}{2}g^{\mu\nu} \pa_\mu \sigma =
 \Gamma_{\alpha\nu}^\mu v^\alpha v^\nu,\qquad
 v^\nu\pa_\nu e(\sigma) + \pa_\mu v^\mu =
 -\Gamma_{\mu \nu}^\mu v^\nu,
 \qquad
 \text{ in } \D_t,
  \label{rescaledreleulcoord}
\end{align}

These equations are very similar in structure to the non-relativistic (Newtonian)
compressible
Euler equations with nonzero right-hand side,
\begin{align}
 (\pa_t + v^k\pa_k) v^i + \delta^{ij}\pa_j h &= f^i, \qquad
 (\pa_t + v^k\pa_k)e(h) + \pa_i v^i = g,\qquad
 \text{ in } \D_t,
 \label{nonrelcont}
\end{align}
where $i = 1,2,3$,
for given functions $f, g$.
 Here, $h$ denotes the Newtonian enthalpy,
defined through the equation of state $p = p(\rho)$ where
$\rho$ now denotes the mass density, by $\rho h'(\rho) = p'(\rho)$
and $e(h) =\log \rho(h)$. With the sound speed
defined as in \eqref{soundspeed},
in the Newtonian setting an incompressible fluid formally corresponds to the
case $\eta \to \infty$.

 In order to simplify notation and to focus on the ideas, in the first part of this paper
we consider the problem
\eqref{nonrelcont} with $f, g = 0$ and
 with boundary conditions
\begin{align}
 p &=0, \qquad \text{ on } \pa \D_t,\\
 n_t + v^in_i &=0, \qquad \text{ on } \pa \D_t,
 \label{}
\end{align}
where $n_t$ denotes the velocity of the boundary and $n_i$ denotes
the conormal to the boundary. In this setting the Taylor sign condition is
\begin{equation}
 \nabla_\n p \leq c < 0.
 \label{tscnonrel}
\end{equation}
In Section \ref{rescaledreleulex} we then show how the argument works in the
relativistic case.

The well-posedness result in the Newtonian case is
\begin{theorem}
  \label{mainthmnonrel}
 Fix $r \geq 10$. Let $v_0, h_0$ be initial data
 satisfying the compatibility conditions from Section \ref{compatsec}
 to order $r$, satisfying $E_0 =  \|v_0\|_{H^{r}(\mathcal{D}_0)}
 +\|h_0\|_{H^{r}(\mathcal{D}_0)} < \infty$, and for which the Taylor sign
 condition \eqref{tscnonrel} holds. Suppose also that
 the sound speed $c_s = \sqrt{P'(\rho)}$ is sufficiently large.
 Then there is a time $T' = T'(E_0,c, c_s) > 0$ so that
 the problem \eqref{nonrelcont}-\eqref{nonrelcont}
 has a solution $v(t), h(t)$
 for $0 \leq t \leq T'$ which satisfies
 $ \sup_{0 \leq t \leq T' }\| v(t)\|_{H^{r}(\mathcal{D}_t)}
 + \sum_{k \leq r}
 \|D_t^{k+1} h(t)\|_{H^{r-k}(\mathcal{D}_t)} + \|D_t^k \pa h(t)\|_{H^{r-k}(\mathcal{D}_t)} \leq
 2 E_0$ and so that the Taylor sign condition \eqref{tsc} holds
 for $0 \leq t \leq T'$ with $c$ replaced with $c/2$.
\end{theorem}

The system \eqref{nonrelcont}-\eqref{nonrelcont} with boundary conditions
\eqref{vac}-\eqref{freebdy} has been considered by many authors
and there are now many methods to prove existence. For the irrotational
incompressible case, see Wu\cite{W99}. Existence in the case of nonzero
vorticity was first shown in the incompressible case in \cite{L05a}
and then in the compressible case in
\cite{Lindblad2005}, using a Nash-Moser iteration. In later works
( \cite{CS07},\cite{CHS13},
\cite{Nordgren2008}, \cite{GLL19}) the authors used instead tangential
smoothing estimates and estimates in fractional Sobolev spaces. See also
\cite{BW} for the irrotational case with self-gravity.

In the case that $\rho|_{\pa\D_t} = 0$, the fluid is called a ``gas''.
In the Newtonian case, a priori estimates were proven in \cite{CLS10}.
For existence, see \cite{JM09}, \cite{IT20}.
A priori estimates for the relativistic problem were proven in \cite{JLM16}
and \cite{MSS19}. Local well-posedness was proven in \cite{DIT20}

We present here a new proof, which is a considerably
simplified version of the proof appearing in \cite{GLL19}. The differences between
the present proof and the one in \cite{GLL19} will be explained in the
upcoming sections.

 In section
\ref{sec:nonrel} we reformulate the problem \eqref{nonrelcont} in
Lagrangian coordinates and introduce a tangentially-smoothed
 version of this problem which is based on the method introduced by
Coutand-Shkoller in \cite{CS07}. The main result of
section \ref{sec:nonrel} is a uniform apriori bound for both the
smoothed and non-smoothed problem.
In section \ref{sec:smreldef} we introduce the tangentially-smoothed version
of the relativistic problem \eqref{rescaledreleulcoord} with
boundary conditions \eqref{vac}-\eqref{freebdy} and just as in the previous
section prove a priori bounds for this system.
In section \ref{nonlinexistence} we prove the well-posedness results. In both the
Newtonian and relativistic case the strategy is the same.
The smoothed equations are ODEs in an appropriate function space and in
the appendix we prove existence for these smoothed problems, but we are
only able to prove existence on a time interval which degenerates as the
smoothing is taken away. Since we also have a priori bounds which hold on
an interval independent of the smoothing, a standard compactness argument
then gives existence for the non-smoothed problem.

\section{Uniform energy estimates for the smoothed problem in
the Newtonian case}
\label{sec:nonrel}

In this section we consider the equations of motion of a compressible
barotropic fluid,
\begin{equation}
 \rho (\pa_t + v^j \pa_j )v_i + \pa_i p = 0,
 \qquad
 (\pa_t + v^j\pa_j) \rho + \rho \div v = 0,
 \qquad \text{ in } \D_t,
 \label{nonreleul}
\end{equation}
where $p = P(\rho)$ for a given function $P$ with $P(0) > 0$,
subject to the boundary conditions
\begin{equation}
 p =0, \qquad
 n_t + v^j n_j =0,
 \qquad \text{ on } \pa \D_t.
 \label{newtbc}
\end{equation}
It is convenient to reformulate \eqref{nonreleul} in terms of the
enthalpy $h$ defined by $h'(\rho) = P'(\rho)/\rho$,
\begin{equation}
 (\pa_t + v^j\pa_j) v_i + \pa_i h =0,
 \qquad
 (\pa_t + v^j\pa_j)e(h) + \div v =0,
 \label{}
\end{equation}
where $e(h) = \log \rho(h)$. In order to fix the position of the
boundary, in the next section we reformulate the above equations in
Lagrangian coordinates. Let us note at this point that if the pressure
satisfies the Taylor sign condition \eqref{tsc} then since
$\frac{dp}{dh}$ is assumed to be positive, we have
\begin{equation}
 \pa_\mathcal{N} h \leq -c' < 0, \qquad \text{ on } \pa \D_t.
\end{equation}

\subsection{Lagrangian coordinates}
\label{setupnonrel}
We fix $\Omega$ to be the unit ball in $\R^3$
and fix a diffeomorphism $x_0 : \Omega \to \D_0$.
We introduce Lagrangian coordinates, which fix the position of the boundary,
\begin{equation}\label{eq:lagrange}
\frac{d x(t,y)}{dt}=v(t,x),\qquad x(0,y)=x_0(y),\qquad y\in \Omega.
\end{equation}
We express Euler's equations in these coordinates, $V(t,y)=v\big(t,x(t,y)\big)$, $h=h(t,y)$
\begin{equation}\label{eq:eulerlagrangiancoordintro}
D_t V^i\!=-\delta^{ij}\pa_j h,\quad\text{in}\quad [0,t_1]\!\times\!\Omega,\quad \text{where}\quad
D_t=\partial_t\Big|_{y =const}\!\!\!\!\!=\pa_t+\!v^k\pa_k,\qquad \pa_i \!=\frac{\pa y^a}{\pa x^i} \frac{\pa}{\pa y^a},
\end{equation}
and
 the continuity equation becomes
$$
 D_t e(h)=-\div V.
 \label{eq:contlagrangiancoordintro}
 $$
 Note that the second boundary condition in \eqref{newtbc} implies that
 the operator $D_t$ is tangent to the boundary.
 Taking the material derivative $D_t$ of the continuity equation
 and the divergence of Euler's equations  we get
\begin{equation}\label{eq:waveequationintro}
 D_t^2 e(h) - \Delta h = (\pa_i V^j)(\pa_j V^i),\quad \textrm{ in }
 [0, t_1]\! \times\! \Omega, \quad\text{with}\quad
 h|_{[0,t_1]\times\pa \Omega}=0, \quad\text{where}\quad {\Delta}\!=\delta^{ij}\pa_i\pa_j.
\end{equation}
To reduce the number of lower order terms to deal with we will assume that
\begin{equation}\label{eq:e1cond}
e^\prime(h)=e_1 > 0,
\end{equation}
is constant. In general we would get more lower order terms containing $D_t$ derivatives.

Our main result in the Newtonian setting, Theorem \ref{mainthmnonrel},
is a consequence of the following existence result for the system
\eqref{eq:lagrange}-
\eqref{eq:eulerlagrangiancoordintro} in Lagrangian coordinates.
\begin{theorem}
  \label{lagthmnonrel}
   Fix $r \geq 10$. Let $V_0, h_0$ be initial data
 satisfying the compatibility conditions \eqref{eq:compatibilityconditiondef}
 to order $r$ and $E_0 = \|V_0\|_{H^{r}(\Omega)}^2
 +\| h_0\|_{H^{r}(\Omega)}^2 < \infty$, and for which the Taylor sign
 condition \eqref{tsc}, $\pa_{N}h_0 > c$ holds. Suppose also that
 the sound speed is sufficiently large.

 Then there is a time $T'\!=\! T'(E_0, c)\! >\! 0$ so that
 the problem \eqref{eq:lagrange}-\eqref{eq:eulerlagrangiancoordintro}
 has a solution $V\!:\![0,T']\!\times \!\Omega\! \to \!\R^3$, $h\!:\![0,T']\!\times\! \Omega\!
 \to \!\R$ with
 $\|V(t)\|_{H^r(\Omega)}\! + \sum_{k \leq r}
 \|D_t^{k+1} h(t)\|_{H^{r-k}(\Omega)}\! + \|D_t^k \pa_x h(t)\|_{H^{r-k}(\Omega)}\! \leq\!
 2 E_0$.
\end{theorem}
We are going to prove this result by first solving a tangentially-smoothed
version of the problem \eqref{eq:lagrange}-\eqref{eq:eulerlagrangiancoordintro}
which is introduced in the next section.

\subsection{The smoothed problem}
It is possible to obtain  apriori energy bounds for the system
\eqref{eq:eulerlagrangiancoordintro}-\eqref{eq:waveequationintro}
but it is difficult to come up with an iteration
scheme that doesn't lose regularity. We will therefore
smooth out the equations, using a tangential regularization
that was first introduced in the incompressible case in \cite{CS07}.
Let $\!S_\varepsilon^* S_\varepsilon$ be a regularization in directions tangential to the boundary that is self adjoint, see Section
\ref{smoothingsec}.
Given a velocity vector field $V\!$, we define the tangentially regularized velocity
and the regularized coordinates by
\begin{equation}\label{eq:eulerlagrangiancoordxsmoothed}
\widetilde{V}=S_\varepsilon^* S_\varepsilon V,\qquad
\frac{d \widetilde{x}}{dt}=\widetilde{V}(t,y),\qquad \xve(0,y)=x_0(y),\qquad y\in \Omega.
\end{equation}
Using these regularized coordinates we define the smoothed equations by
\begin{equation}\label{eq:eulerlagrangiancoordsmoothed}
D_t V^i=-\delta^{ij}\widetilde{\pa}_j h,\quad\text{in}\quad [0,t_1]\times\Omega,\quad \text{where}\quad
D_t=\partial_t\big|_{y =const},\qquad \widetilde{\pa}_i =\frac{\pa y^a}{\pa \widetilde{x}^i} \frac{\pa}{\pa y^a},
\end{equation}
where $h$ is given by
\begin{equation}\label{eq:waveequationsmoothed}
  D_t\big( e_1 D_t h\big) - \widetilde{\Delta} h = \widetilde{\pa}_i \widetilde{V}^j\, \widetilde{\pa}_j V^i, \quad \textrm{ in }
 [0, t_1]\! \times\!\Omega, \quad\text{with}\quad
 h\big|_{ [0, t_1] \times\pa \Omega}=0,\quad\text{where}\quad \widetilde{\Delta}\!=\delta^{ij}\nave_i\nave_j,\!
\end{equation}
Taking the divergence of \eqref{eq:eulerlagrangiancoordsmoothed} and adding it to \eqref{eq:waveequationsmoothed} gives $D_t\big( e_1 D_t h+\widetilde{\div} V\big)=0$, which shows that the continuity equation is preserved,
\begin{equation}
 e_1 D_t h=-\widetilde{\div} V.
\end{equation}

\subsection{A priori bounds for the smoothed problem}
We are going to
 prove uniform apriori energy bounds for the smoothed system \eqref{eq:eulerlagrangiancoordxsmoothed}-\eqref{eq:waveequationsmoothed} up to a time $t_1\!>\!0$, independent of $\varepsilon$.
In section \ref{nonrelsmexist}
we will also show that we have existence for the smoothed problem as long as the apriori bounds hold.
 Passing to the limit as $\varepsilon \to 0$ will then give us a solution to Euler's equations \eqref{eq:eulerlagrangiancoordintro}-\eqref{eq:waveequationintro}.

We will prove $\varepsilon$ dependent bounds for the iteration scheme:
(i) Given
$V\!$ and $x$ satisfying \eqref{eq:lagrange}, define smoothed $\widetilde{V}$ and $\widetilde{x}$ by  \eqref{eq:eulerlagrangiancoordxsmoothed},
(ii) given smoothed $\widetilde{V}\!$ and $\widetilde{x}$ solve the linear system  \eqref{eq:eulerlagrangiancoordsmoothed}-\eqref{eq:waveequationsmoothed}
for $h$ and new $V\!$ and $x$.  This leads to existence for the smoothed problem
up to a time $T(\varepsilon)\!>\!0$, depending on $\ve$. However, the local existence will also allow us to continue the solution for as long as we have energy bounds,
i.e. up to the time $t_1$ independent of $\varepsilon$. Existence for the linear system follows e.g. from the Galerkin method. (If $e_1$  is not constant we
evaluate it at the previous iterate of $h$ to get a linear system.)

\subsubsection{The lowest-order energy estimate} Let $\mathcal{E}$ be the energy for Euler's equations.
With  ${\kappa}=|\det{(\pa {x}/\pa y)}|,$
\begin{equation}\label{eq:TheEnergy}
\mathcal{E}(t)= \int_{\D_t}\!\!\!\big( |V|^2\! +Q(\rho)\big)\, \rho \, dx
=\int_{\Omega}\!\big( |V|^2\! +Q(\rho) \big)\,  \rho
\kappa dy,\quad\text{where}\quad Q(\rho)\!=\!2\!\int\! p(\rho) \rho^{-2} d\rho,\quad D_t (\kappa\rho)\!=\!0.
\end{equation}
If we take the time derivative of the integral expressed in the fixed Lagrangian coordinates we get $D_t$ applied to the integrand. We then use Euler's equation $D_t V=-\rho^{-1} \pa p$ and integrate by parts:
\begin{equation}
\frac{ d\mathcal{E}\!}{dt}
=\!\int_{{}_{\!}\D_{t}}\!\!\! 2V^i (-\pa_i p
)\!+ Q^{\,\prime}(_{\!}\rho_{\!})\rho D_{t} \rho \,  dx
=\!\int_{\!\D_{{}_{\!}t}}\!\! 2\div\! V \, p +Q^{\,\prime}(_{\!}\rho_{\!})\rho D_{{}_{\!}t} \rho  \, dx
+\!\int_{\!\pa\D_{{}_{\!}t}} \!\!\!\! 2  v_i  p\, \mathcal{N}^i dS=0,
\end{equation}
using the continuity equation $D_t \rho\!=\!-\rho\div V\!\!$ and the boundary condition $p\!=\!0$.
This energy for the smoothed problem with $\!\D_t,dx,\kappa$ replaced by $\!\tD_t,d\widetilde{x},
\widetilde{\kappa}=|\det{(\pa \widetilde{x}/\pa y)}|$ is almost conserved apart from that the measure changes a bit, $D_t(\rho\widetilde{\kappa})\!=\!
\rho\widetilde{\kappa}(\widetilde{\div}\widetilde{V}\!\!-\widetilde{\div}V\!)$.
We will obtain a priori bounds for higher-order derivatives of the solution to
the smoothed problem which will contain a boundary term where the symmetry of the smoothing matters, see section \ref{higherorderenergies}.

\subsubsection{Higher order estimates}
In Section \ref{def T and FD} we construct a set of vector fields $S\in {\mathcal{S}}$ that are tangential at the boundary of $\Omega$ and span the tangent space at the boundary. In addition we will also use the space time tangential vector fields $\mathcal{T}\!=\!\S\!\cup\! D_t$.
In section \ref{higherorderenergies} we derive higher order energies for any combination of tangential vector fields $T^I$ applied to the solution. These together with separate
estimates for the curl and the divergence gives an estimate for the full gradient of tangential vector fields applied to the solution.
Since $h=0$ and hence $D_t h=0$ on the boundary one can also get an energy estimate for the gradient of the enthalpy from the wave equation. Since $T^I h$ also vanishes at the boundary one get higher order energy estimates as well.

However, the higher order energies for the velocity contain a boundary term
(see \eqref{newtendef}
and \eqref{eq:energydef}) with the norm of the normal component of tangential derivatives of the coordinate at the boundary (or equivalently the second fundamental form at the boundary, see Christodoulou-Lindblad \cite{CL00}). It is critical that this boundary term is positive for the apriori energy bounds to hold, which is where the sign condition $\pa_\mathcal{N} p\leq -c<0$ is used. For the proof of existence, because of some lower order terms one needs to have more control at the boundary and this requires control of an extra half tangential derivative $\fdh$ in the interior of the coordinate.
In the remainder of this subsection we outline the proof of
the a priori bounds.
The energies we control are defined in section \ref{higherorderenergies} and
the uniform bounds are proven in section \ref{controlL2}.

To simplify notation, in what follows we let $c_J, c_r$ denote constants
depending on pointwise norms of lower-order terms,
\begin{equation}
 c_J = c_J\bigfour( {\sum}_{|K| \leq |J|/2} |\pave T^K \xve|+
 |\pave T^L V| + |\pave T^L \widetilde{V}|
 +|\pave T^L \pave h|\bigfour),
 \qquad
 c_r = {\sum}_{|J| \leq r} c_J,
 \label{lowercasecdef}
\end{equation}
and similarly $C_J$ depends on $L^\infty$ norms of lower-order
terms,
\begin{equation}
 C_J = C_J\bigfour( {\sum}_{|K| \leq |J|/2} \|\pave T^K \xve\|_{L^\infty}+
 \|\pave T^L V\|_{L^\infty}
 +\|\pave T^L \widetilde{V}\|_{L^\infty}
 +\|\pave T^L \pave h\|_{L^\infty}\bigfour),
 \qquad
 C_r = {\sum}_{|J| \leq r} C_J.
 \label{capitalCdef}
\end{equation}

\subsubsection{Control of the $L^2$ norms of the velocity and enthalpy}
We expect to control the norms
\begin{equation}\label{eq:normsV}
{\sum}_{|J|\leq r-1}\|\pa T^J V\|_{L^2(\Omega)}^2\!
\lesssim \!{\sum}_{|J|\leq r- 1}\!\Big({\sum}_{S\in\mathcal{S}}\| S T^J V\|_{L^2(\Omega)}^2\!
+\!\|\curl T^J V\|_{L^2(\Omega)}^2\!
+\! \|\div T^J V\|_{L^2(\Omega)}^2\!\Big),
\end{equation}
by Lemma \ref{app:pwdiff}.
We will show that we control the norms of $V$ on the right-hand side and it follows that we have control of the coordinate $\sum_{|J|\leq r-1} \|\pa T^J x\|_{L^2(\Omega)}^2$ just by taking the time derivative of this quantity.
The first term on the right will be controlled by the Euler energy, the second by a pointwise evolution equation for the curl and the last by the continuity equation and the energy for the wave equation.
From higher order wave equations we will get control of $\sum_{|J|\leq r-1} \| T^J\pave h\|_{L^2(\Omega)}^2$.

\subsubsection{Control of the $L^\infty$ norms of the velocity and enthalpy} When estimating the $L^2$ norms we will need to control commutators using $L^\infty$ bounds for a low number of tangential vector fields.
From control of the $L^2$ norms we will also derive control of lower order $L^\infty$ norms. In fact
from the pointwise estimate \eqref{app:pwdiff},
\begin{equation}\label{eq:partialV}
|\widetilde{\pa} T^J V|
\lesssim |\widetilde{\div}\,T^J V|+|\widetilde{\curl}\,T^J V|
+\!\! {\sum}_{S\in{\mathcal S}}|S T^J V|.
\end{equation}
Here the last term is controlled in $L^\infty$ by Sobolev's lemma from \eqref{eq:normsV} for $|J|\leq r-3$,
 the second term is controlled by a point wise evolution equation for the curl and
the first from the continuity equation and control of Sobolev norms from the wave equation for $h$. In addition we have a point wise evolution equation for the coordinate $\widetilde{\pa}T^J \widetilde{x}$ since $D_t x=V$.

\subsubsection{The additional norm control of the smoothed coordinate $S_\varepsilon x$} The higher order energies also give control of an additional norm of the smoothed coordinate on the boundary, $\sum_{|I|\leq r} \| \mathcal{N}\!\!\cdot\! T^I\! S_\varepsilon x\|_{L^2(\pa\Omega)}^2$.
 When $\ve > 0$, controlling this term gives rise to error terms that have to be controlled through the elliptic estimates in Section \ref{sec:divcurlL2}:
 \begin{multline}
 \sum_{|J|\leq r-1}\! \!\!\!||\pa T^J\! \fdh\sm x||_{L^2(\Omega)}^2\!
 +\!\!\sum_{|I|\leq r} \!\| T^I\! S_\varepsilon x\|_{L^2(\pa\Omega)}^2\\
 \leq c_{{}_{\!}1\!}\!\sum_{|I|\leq r} \!\| \mathcal{N}\!\cdot T^I\! S_\varepsilon x\|_{L^2(\pa\Omega)}^2
 +c_{{}_{\!}1\!}\!\!\!\!\sum_{|J|\leq r-1}\! \!\!||\div T^J\! \fdh\sm x||_{L^2(\Omega)}^2\!
 + ||\curl T^J\! \fdh\sm x||_{L^2(\Omega)}^2
 +\|\pa T^J \sm x\|_{L^2(\Omega)}^2.
\end{multline}

\subsection{Higher order equations for the velocity vector field}
\label{higherorderequations}
Before deriving the energy estimates we find a higher-order version of
our equations.
\subsubsection{Higher order Euler's equations}
If $T=T^a(y)\pa_a $ is tangential then
\begin{equation*}
D_t T V_i-\widetilde{\pa}_j h \,\widetilde{\pa}_i T\widetilde{x}^j +\widetilde{\pa}_i Th=0.
\end{equation*}
Similarly applying a product of tangential vector fields $T^I=T_{i_1}\cdots T_{i_r}$ where $I=(i_1,\dots,i_r)$ is a multiindiex of length $r=|I|$, we get
 \begin{equation}\label{eq:TIV}
D_t  T^I  V_i
-\widetilde{\pa}_j h \,\widetilde{\pa}_i T^{I\!} \widetilde{x}^j - \widetilde{\pa}_i T^{I\!}h
=F_i^I ,
\end{equation}
where $F_i^I$ is a sum of terms of the form
$\widetilde{\pa}_i  T^{I_1} \widetilde{x}\cdots \widetilde{\pa} T^{I_{k-1}} \widetilde{x}\cdot T^{I_{k}} \widetilde{\pa} h$, for $I_1+\dots+I_k=I$ and
$|I_i|<|I|$ and hence is lower order
\begin{equation}
|F^I|\lesssim c_I{\sum}_{|J|\leq |I|-1} |\widetilde{\pa} T^J \widetilde{x}|
+|D_t T^J V|,
\end{equation}
 and $c_I$ stand for a constant that depends on
$|\widetilde{\pa} T^L \widetilde{x}|$ and  $|D_t T^L V|$, for
$|L|\leq |I|/2 $.

 We now want to rewrite this in a way which to highest order is a symmetric operator for which it is easier to obtain energy conservation:
\begin{equation}\label{eq:symmetricequation}
D_t  T^I  V_i
-\widetilde{\pa}_i\big(\widetilde{\pa}_j h \, T^{I\!} \widetilde{x}^j -  T^{I\!}h\big)=F_i^{\prime I}\!,
\end{equation}
where $F_i^{\prime I}=F_i^{I}-\widetilde{\pa}_i\widetilde{\pa}_j h \,  \,T^{I\!} \widetilde{x}^j $ is lower order.

\subsubsection{Higher order continuity equations}
Similarly one can get a higher order version of the continuity equation.
From
$e_1 D_t T^I h=-T^I\widetilde{\div} V$ we have
\begin{equation}\label{eq:GIequation}
 e_1 D_t T^I h+\widetilde{\div} (T^I V)-\widetilde{\pa}_i T^{I\!} \widetilde{x}^k \,\,\widetilde{\pa}_k V^i =G^I,
\end{equation}
where $G^I$ is a sum of terms of the form $\widetilde{\pa}  T^{I_1} \widetilde{x}\cdots \widetilde{\pa} T^{I_{k-1}} \widetilde{x}\cdot \widetilde{\pa} T^{I_{k}}V$,
for $I_1+\dots+I_k=I$ and $|I_i|<|I|$, and hence is lower order
\begin{equation}\label{eq:GIestimate}
 |G^I|\lesssim c_I {\sum}_{|J|\leq |I|-1}
 |\widetilde{\pa} T^J V|+|\widetilde{\pa} T^J\! \xve|,
 \end{equation}
 and $c_I$ stands for a constant that depends on
$|\widetilde{\pa} T^L \widetilde{x}|$, $|\widetilde{\pa} T^L V|$,
 for
$|L|\leq |I|/2 $.
 Hence
\begin{equation}\label{eq:highercontinuity}
 e_1 D_t T^I h+\widetilde{\pa}_i \bigtwo(T^I V^i-T^{I\!} \widetilde{x}^k \,\,\widetilde{\pa}_k V^i\bigtwo) =G^{{}^{\,}\prime I},
\end{equation}
where $G^{{}^{\,}\prime I}=G^I-T^{I\!} \widetilde{x}^k \,\,\widetilde{\pa}_k \, \widetilde{\div} V$ is lower order.

\subsubsection{New unknowns}
Given the form of \eqref{eq:symmetricequation} and \eqref{eq:highercontinuity} it is natural to introduce
\begin{equation}\label{eq:eulerianvariables}
V^{I i} =T^I V^i- \widetilde{\pa}_k V^i\,T^{I\!} \widetilde{x}^k,\qquad
\text{and}\qquad
h^I=  T^{I\!}h-\widetilde{\pa}_j h \,\, T^{I\!} \widetilde{x}^j .
\end{equation}
In terms of these quantities \eqref{eq:symmetricequation} and \eqref{eq:highercontinuity} takes the form
\begin{align}\label{eq:highereulermodified}
D_t V^{Ii}+\widetilde{\pa}_i h^I&=F_i^{\prime\prime I},\\
\label{eq:highercontinuitymodified}
e_1 D_{t} h^{I}+\widetilde{\pa}_i V^{I i}&=G^{\prime\prime I},
\end{align}
where $F_i^{\prime\prime I}=F_i^{\prime I}-D_t\big(\widetilde{\pa}_k V^i\,T^{I\!} \widetilde{x}^k\big)$ and $G^{\prime\prime I}=G^{\prime I}-D_t\big(\widetilde{\pa}_j h \,\, T^{I\!} \widetilde{x}^j\big)$ are lower order.

\begin{remark}
 We remark that \eqref{eq:eulerianvariables} are related to Alinhac's 'good unknowns', see \cite{XW19}, well as to covariant derivatives. These quantities also indirectly showed up in this context in Christodoulou-Lindblad \cite{CL00} where energy estimates $v$ and $h$ were in terms of the original Eulerian coordinates $(t,\widetilde{x})$ instead of the Lagrangian coordinates.
 If $\widetilde{v}(t,\widetilde{x})$ denote a functions in terms of the original Eulerian coordinates that were controlled in \cite{CL00} then in the Lagrangian coordinates  $v(t,y)=\widetilde{v}(t,\widetilde{x})$, where $\widetilde{x}=\widetilde{x}(t,y)$. We have
 \begin{equation}
 \pa_y^\bold{a} \widetilde{v}(t,\widetilde{x})
 ={\sum}_{|\beta|=|\bold{a}|=r}\big(\widetilde{\pa}_x^\beta \widetilde{v}\big)(t,\widetilde{x}) \frac{\pa \widetilde{x}^{\beta_1}}{\pa y^{a_1}}\cdots \frac{\pa \widetilde{x}^{\beta_r}}{\pa y^{a_r}}
 +\big(\widetilde{\pa}_{k} \widetilde{v}\big)(t,\widetilde{x})\pa_y^\bold{a} \widetilde{x}^k+ M^\bold{a}\!,
 \label{alinhacdiscussion}
 \end{equation}
 where $M^\alpha$ is lower order. Going back to the Lagrangian coordinates it therefore follows that quantity
 $\pa_y^{\bold{a} } v\!-\!\widetilde{\pa}_k v\, \pa_y^{\bold{a} } \widetilde{x}^k\! $ can
 modulo lower order terms be controlled by $\widetilde{\pa}{}^\beta v$, for $|\beta|\!=\!|{\bold{a} }|$, which was controlled in \cite{CL00}. To leading order this is of course nothing but the covariant derivative ${\na}_a=\na_{\pa/\pa y^a}$ corresponding to the partial derivatives $\widetilde{\pa}_x$ expressed in the $y$ coordinates:
  \begin{equation}
 \big( \na_{{a_1}} \cdots  \na_{{a_r}} \widetilde{v}\big)(t,\widetilde{x})
 ={\sum}_{|\beta|=|\bold{a}|=r}\big(\widetilde{\pa}_x^\beta \widetilde{v}\big)(t,\widetilde{x}) \frac{\pa \widetilde{x}^{\beta_1}}{\pa y^{a_1}}\cdots \frac{\pa \widetilde{x}^{\beta_r}}{\pa y^{a_r}}.
 \end{equation}
 \end{remark}
 We note at this point that by working in terms of these new variables, we are able to
 prove a priori estimates for the non-smoothed problem $\ve = 0$ without working
 in fractional Sobolev spaces which were needed in \cite{CS07} and \cite{GLL19}.

 \subsection{Higher order energies for the velocity vector field}
 \label{higherorderenergies}
Multiplying the left hand side of \eqref{eq:highereulermodified} by $V^{I i}$ and integrating we get
\begin{equation}
\int_\Omega \! V^{I i} D_t  V^{I }_i  \,  \widetilde{\kappa} dy
+\int_\Omega V^{I i}\, \widetilde{\pa}_i h^I\,  \widetilde{\kappa}dy=\int_\Omega V^{I i} F_i^{\prime\prime I} \widetilde{\kappa}\, dy .
\end{equation}
If we integrate the second term by parts using that $\widetilde{\pa}_i$ is symmetric with respect to  $d\widetilde{x}=\widetilde{\kappa} dy$:
\begin{equation}
\int_\Omega  V^{I i}\, \widetilde{\pa}_i h^I\,  \widetilde{\kappa}dy
=\!\int_{\pa\Omega}\!\!\!{\widetilde{\mathcal{N}}_i }  V^{I i} h^I  \widetilde{\nu} d S
-\!\int_\Omega\!\widetilde{\pa}_i V^{I i}\,  h^I \widetilde{\kappa} dy
=\!\int_{\pa\Omega}\!\!\!\!{\widetilde{\mathcal{N}}_i }  V^{I i} h^I  \widetilde{\nu} d S
+\!\int_\Omega\!\!   e_1 h^I D_t  h^I\, \widetilde{\kappa} dy
-\!\int_\Omega\!\!G^{\prime\prime I}  h^I \widetilde{\kappa} dy,
\end{equation}
where $\widetilde{\nu} d S$ is the measure on $\pa \Omega$ induced by the measure
$\kappa dy$ on $\Omega$.

Hence
\begin{equation}
\frac{1}{2}\frac{d}{dt}\int_\Omega |V^{I }|^2\!+ e_1 (h^I)^2 \,\,  \widetilde{\kappa}dy
+\int_{\pa\Omega}\!\!\!\!{\widetilde{\mathcal{N}}_i }  V^{I i} h^I  \widetilde{\nu} d S
=\int_\Omega \!|V^{I }|^2\!+ e_1 (h^I)^2\,D_t \widetilde{\kappa}dy
+\!\!\int_\Omega\! V^{I i} F_i^{\prime\prime I}\!\!+G^{\prime\prime I}\!  h^I \, \widetilde{\kappa}\, dy.
\label{dtEnewt}
\end{equation}

\subsubsection{The boundary term}
\label{sec:nonrelbdy}
It remains to deal with the boundary term.
If we  use that  $T^I\! h\!=\!0$ on $\pa \Omega$ and  $\widetilde{\pa}_j h\!=\! \widetilde{\mathcal{N}}_j \pa_{\!\widetilde{\mathcal{N}}}h\!=\!-\widetilde{\mathcal{N}}_j |\widetilde{\pa} h|$
there since $h\!=\!0$  and by assumption $\pa_{\widetilde{\mathcal{N}}} h\!<\!0$ we see
that
\begin{equation}
{\widetilde{\mathcal{N}}_i }  V^{I i}\!=\widetilde{\mathcal{N}}_i T^I V^{i}\!-\widetilde{\mathcal{N}}_i\widetilde{\pa}_k {V}^i\,  T^I \widetilde{x}^{k}\!,\qquad\text{and}\qquad
h^I\!=\widetilde{\mathcal{N}}_{\!j} T^{I} \widetilde{x}^j\, |\widetilde{\pa} h| ,\qquad\text{on}\quad \pa\Omega.
\end{equation}
On the other hand, since $D_t \widetilde{\mathcal{N}}_i=-\widetilde{\pa}_i \widetilde{V}_k\, \widetilde{\mathcal{N}}^k
+\eta\widetilde{\mathcal{N}}_i$, where
 $ \eta = \widetilde{\pa}_j \widetilde{V}_k\, \widetilde{\mathcal{N}}^k\widetilde{\mathcal{N}}^j$ we have
\begin{equation}
D_t \big( \widetilde{\mathcal{N}}_i \,  T^I {x}^{i} \big)
=\widetilde{\mathcal{N}}_i T^I V^{i}-\widetilde{\mathcal{N}}_i\widetilde{\pa}_k \widetilde{V}^i\,  T^I {x}^{k}+\eta\,\widetilde{\mathcal{N}}_i \,  T^I {x}^{i}  ,\qquad\text{on}\quad \pa\Omega,
\end{equation}
and hence
\begin{equation}\label{eq:badboundaryterms}
{\widetilde{\mathcal{N}}_i }  V^{I i}\!=D_t \big( \widetilde{\mathcal{N}}_i \,  T^I {x}^{i} \big)
-\eta\,\widetilde{\mathcal{N}}_i \,  T^I {x}^{i}+\widetilde{\mathcal{N}}_i\widetilde{\pa}_k \widetilde{V}^i\,  T^I {x}^{k}
-\widetilde{\mathcal{N}}_i\widetilde{\pa}_k {V}^i\,  T^I \widetilde{x}^{k} .
\end{equation}
We hence have
\begin{multline}\label{eq:boundaryproblems}
\int_{\pa\Omega}\!\!\!\!\!{\widetilde{\mathcal{N}}_{{}_{\!}i} }  V^{I i} h^I  \widetilde{\nu} d S
=\!\!\int_{\pa\Omega}\!\!\!\!\!\widetilde{\mathcal{N}}_{{}_{\!}j} T^{I\!} \widetilde{x}^j D_t \big( \widetilde{\mathcal{N}}_i   T^{I\!} {x}^{i} \big)  |\widetilde{\pa} h| \widetilde{\nu} d S
-\!\int_{\pa\Omega}\!\!\!\!\!
\widetilde{\mathcal{N}}_{{}_{\!}j}   T^{I\!}  \widetilde{x}^{\,j}\,\widetilde{\mathcal{N}}_{{}_{\!}i} T^{I\!}  {x}^{i}\, \eta |\widetilde{\pa} h| \widetilde{\nu} d S\\
+\!\int_{\pa\Omega}\!\!\!\!\! \widetilde{\mathcal{N}}_{{}_{\!}j}  T^{I\!}  \widetilde{x}^{\,j}\, \bigtwo(\widetilde{\mathcal{N}}_i\widetilde{\pa}_k \widetilde{V}^i\,  T^I {x}^{k}
-\widetilde{\mathcal{N}}_i\widetilde{\pa}_k {V}^i\,  T^{I\!} \widetilde{x}^{k}\bigtwo)  |\widetilde{\pa} h| \widetilde{\nu} d S.
\end{multline}

\subsubsection{The apriori energy bounds for Euler's equation}
For a solution of Euler's equations $\varepsilon=0$ we have $\widetilde{x}=x$ and
the last integral in \eqref{eq:boundaryproblems} vanishes. It follows that
\begin{equation}
  \label{newtendef}
\mathcal{E}^{I\!}(t)=
  \int_{\Omega} |V^I|^2 \kappa dy + \int_{\Omega} e_1 (h^I)^2 \kappa dy
  +\int_\Omega |T^I x|^2\, \kappa dy
  + \int_{\pa\Omega} (\mathcal{N}_i T^I\! x^i)^2|\pa h| \nu dS,
\end{equation}
satisfies
\begin{equation}\label{eq:energyestiamteeuler}
\frac{d}{dt}\mathcal{E}^{I\!}(t)\lesssim c_0\, \mathcal{E}^{I\!}(t)+\int_{\Omega}  |F^{\prime\prime I}|^2+(G^{\prime\prime I})^2\, dy,\qquad\text{if}\quad \varepsilon=0.
\end{equation}
Here
the integral in the right can be bounded by \eqref{eq:normsV} times lower order $L^\infty$ norms that we can bound by \eqref{eq:partialV} that we expect to control at this order.
Moreover
\begin{equation}
\int_\Omega |T^I V|^2 +
e_1(T^I h)^2+ |T^I x|^2\,  dy \lesssim c_0 \, \mathcal{E}^I(t).
\end{equation}

\subsubsection{The apriori energy bounds for the smoothed Euler's equation}\label{smoothennonrel}
For the smoothed problems it is more work to close the energy bounds.
In particular the term $B^I\!$ in \eqref{eq:badboundaryterms}-\eqref{eq:boundaryproblems}
contains all components of $T^I\! x$ and not only the normal one,
and for that we need more elliptic estimates.

We will now modify the definition of the unknowns in \eqref{eq:eulerianvariables} slightly to make it more symmetric
by replacing $T^I \widetilde{x}^i=T^I \sm^2 x^i$ with
$\sm T^I \sm x^i=\sm T^I x_\varepsilon^i$, where $ x_\varepsilon^i=\sm x^i$:
\begin{equation}
V^{I i} =T^I V^i- \widetilde{\pa}_k V^i\,\sm T^{I\!} {x}_\varepsilon^k,\qquad
\text{and}\qquad
h^I=  T^{I\!}h-\widetilde{\pa}_j h \,\,\sm T^{I\!}  {x}_\varepsilon^j ,\qquad\text{where}\quad {x}_\varepsilon^j=\sm x^j.
\end{equation}
In terms of these quantities we have
\begin{equation}
D_t V^{Ii}+\widetilde{\pa}_i h^I=F_{ i}^{\prime\prime I}+C_{\varepsilon\, i}^I,
\end{equation}
\begin{equation}
e_1 D_{t} h^{I}+\widetilde{\pa}_i V^{I i}=G^{\prime\prime I}+C_\varepsilon^I,
\end{equation}
where the smoothing errors $C_{\varepsilon\, i}^I$, $C_\varepsilon^I$ are bounded by lower norms times $[\sm, T^I] x_\varepsilon^k$, or $\widetilde{\pa}_i\big( [\sm, T^I] x_\varepsilon^k \big)$, or $D_t \big( [\sm, T^I] x_\varepsilon^k \big)$,
which are lower order, 
\begin{equation}
\|C^I_{\varepsilon}\|_{L^2(\Omega)}\lesssim c_0 {\sum}_{|J|\leq |I|-1}
\|T^J x_\varepsilon\|_{L^2(\Omega)}+\|\partial T^J x_\varepsilon\|_{L^2(\Omega)}
+\|T^J D_t\, x_\varepsilon\|_{L^2(\Omega)},
\end{equation}
 by Lemma \ref{lem:mutiplicationsmoothingcommute}
since $T^I$ is tangential. These particular smoothing commutators are just a matter of which coordinates we choose to parameterize the domain and define the smoothing operators and vector fields and they would vanish in flat coordinates.
For these new variables \eqref{eq:boundaryproblems} become
\begin{multline}\label{eq:boundaryproblemssmoothed}
\int_{\pa\Omega}\!\!\!\!\!{\widetilde{\mathcal{N}}_i }  V^{I i} h^I  \widetilde{\nu} d S
=\!\!\int_{\pa\Omega}\!\!\!\!\!\widetilde{\mathcal{N}}_j \sm T^{I\!} {x}_\varepsilon^j D_t \big( \widetilde{\mathcal{N}}_i   T^{I\!} {x}^{i} \big) |\widetilde{\pa} h| \widetilde{\nu} d S
-\!\int_{\pa\Omega}\!\!\!\!\!
\widetilde{\mathcal{N}}_j \sm  T^{I\!} {x}_\varepsilon^{\,j}\,\widetilde{\mathcal{N}}_i  T^{I\!}  {x}^{i}\, \eta|\widetilde{\pa} h| \widetilde{\nu} d S\\
+\!\int_{\pa\Omega}\!\!\!\!\! \widetilde{\mathcal{N}}_j  \sm T^{I\!}  {x}_\varepsilon^{\,j}\,
\bigtwo(\widetilde{\mathcal{N}}_i\widetilde{\pa}_k \widetilde{V}^i\,  T^{I\!} {x}^{k}\!
-\widetilde{\mathcal{N}}_i\widetilde{\pa}_k {V}^i\, \sm T^{I\!} {x}_\varepsilon^{k} \bigtwo) |\widetilde{\pa} h| \widetilde{\nu} d S.
\end{multline}
We want to use that the smoothing $\sm$, as constructed in Section
\ref{smoothingsec}, is symmetric on $L^2(\pa \Omega)$ to move one smoothing $\sm$ from the first  factor $\sm T^I x_\varepsilon $ of the boundary integrals
to the other factors, and then commute it through first to $ T^I x$ and then to $x$. For the first term we have
\begin{equation*}
\sm\bigtwo( {\widetilde{\mathcal{N}}_j } |\widetilde{\pa} h|\widetilde{\nu}\,\,D_t\big({\widetilde{\mathcal{N}}_i }  T^{I\!} x^i\bigtwo)\bigtwo)=  \widetilde{\mathcal{N}}_j |\widetilde{\pa} h|\widetilde{\nu}\,\, D_t \big(\widetilde{\mathcal{N}}_i T^{I\!}\sm  x^i\big) +C_{\varepsilon j}^{I\!,1},
\end{equation*}
where $C_{\varepsilon j}^{I\!,1}$ is lower order by Lemma \ref{lem:mutiplicationsmoothingcommute}
since $T^I$ is tangential:
\begin{equation}\label{eq:smoothingcommutator}
\|C^{I\!,1}_{\varepsilon}\|_{L^2(\pa\Omega)}\lesssim C_0
\!\!\!\!\! \sum_{|J|\leq |I|-1}\!\!\!\!\!
\|T^J x\|_{L^2(\pa\Omega)}
+\|T^J D_t\, x\|_{L^2(\pa\Omega)}
\lesssim C_0 \!\!\!\!\!\sum_{|J|\leq |I|-1}\!\!\!\!\!
\|\widetilde{\pa}T^J x\|_{L^2(\Omega)}
+\|\widetilde{\pa}T^J D_t\, x\|_{L^2(\Omega)}.
\end{equation}

 Similarly we can move $\sm$ from the first factor in the other two boundary integrals to obtain
\begin{multline}\label{eq:boundaryproblemssmoothedsymmetric}
\int_{\pa\Omega}\!\!\!\!\!{\widetilde{\mathcal{N}}_i }  V^{I i} h^I  \widetilde{\nu} d S
=\!\!\int_{\pa\Omega}\!\!\!\!\!\widetilde{\mathcal{N}}_j  T^{I\!} {x}_\varepsilon^j D_t \big( \widetilde{\mathcal{N}}_i   T^{I\!} {x}_\varepsilon^{i} \big) |\widetilde{\pa} h| \widetilde{\nu} d S
-\!\int_{\pa\Omega}\!\!\!\!\!
\widetilde{\mathcal{N}}_j   T^{I\!} {x}_\varepsilon^{\,j}\,\widetilde{\mathcal{N}}_i  T^{I\!}  {x}_\varepsilon^{i}\, \eta|\widetilde{\pa} h| \widetilde{\nu} d S\\
+\!\int_{\pa\Omega}\!\!\!\!\! \widetilde{\mathcal{N}}_j   T^{I\!}  {x}_\varepsilon^{\,j}\,
\bigtwo(\widetilde{\mathcal{N}}_i\widetilde{\pa}_k \widetilde{V}^i\,  T^{I\!} {x}_\varepsilon^{k}\!
-\widetilde{\mathcal{N}}_i\widetilde{\pa}_k {V}^i\, \sm^2 T^{I\!} {x}_\varepsilon^{k} \bigtwo) |\widetilde{\pa} h| \widetilde{\nu} d S
+\int_{\pa\Omega}\!\!\!\!\!   T^{I\!} {x}_\varepsilon^{\,j} C_{\varepsilon j}^{\prime I} |\widetilde{\pa} h| \widetilde{\nu} d S ,
\end{multline}
where $C_{\varepsilon j}^{\prime I}$ satisfy \eqref{eq:smoothingcommutator}.
Here the terms on the first row are as before but the terms on the second row can only be controlled by all the components of $T^I x_\varepsilon$ which are not directly
controlled by the energy. With
\begin{equation}\label{eq:energydef}
\mathcal{E}^{I\!}(t)
=\!\!
 \int_{\Omega}\!\! |V^I|^2 \kappa dy +\!\! \int_{\Omega}\!\! e_1 (h^I)^2 \widetilde{\kappa} dy+
\!\! \int_\Omega\!\! |T^I\! x|^2\, \widetilde{\kappa} dy
 + \!\!\int_{\pa\Omega}\!\!\!\! (\mathcal{N}_i T^I \! x_{\varepsilon }^i)^2|\pa h| \nu dS,
 \end{equation}
 and
 \begin{equation}\label{eq:energyboundarydef}
\mathcal{B}^{I\!}(t)
=\!\!\int_{\pa\Omega}\!\!\!\!\! | T^{I\!} x_{\varepsilon }|^2  dS,\qquad
\mathcal{B}_{\widetilde{\mathcal{N}}}^{\,I}(t)
=\!\!\int_{\pa\Omega}\!\!\!\!\! |\, \widetilde{\mathcal{N}}\!
\bigtwo\cdot T^{I\!} x_{\varepsilon }|^2  dS,
\end{equation}
we therefore only have
\begin{equation}\label{eq:energyestimate}
\frac{d}{dt}\mathcal{E}^{I\!}(t)\lesssim C_0\,\mathcal{E}^{I\!}(t)
+C_0\mathcal{B}^{I\!}(t)
+C_I \!{\sum}_{|J|\leq |I|-1}\int_\Omega \!\!
|\widetilde{\pa}T^{J\!} x|^2\!
+|\widetilde{\pa}T^J V|^2\, d y,
\end{equation}
while the energy only bounds the normal component of $T^I x_\ve$ at the boundary,
\begin{equation}\label{eq:byenergyestimated}
\|T^I V(t,\cdot)\|_{L^2(\Omega)}^2+\mathcal{B}_{\mathcal{N}}^{\,I}(t)\lesssim \mathcal{E}^{I\!}(t).
\end{equation}
As we shall see, this together with elliptic estimates will give us control of
another half derivative of $T^I x_{\varepsilon }$ in the interior and at the same time bounds for all components of $T^I  x_{\varepsilon }$ at the boundary.

\subsubsection{The apriori energy bounds for the smoothed linear system}\label{sec:apriorilinear}
We will solve the smoothed problem by an iteration. Given $U=V_{(k)}$ define $z=x_{(k)}$ such that $dz/dt=U(t,z)$ define
$\widetilde{V}=S_\varepsilon^* S_\varepsilon U$ and
$\widetilde{x}=S_\varepsilon^* S_\varepsilon z$.
In Section \ref{existence} we prove that the linear system \eqref{eq:eulerlagrangiancoordsmoothed}
-\eqref{eq:waveequationsmoothed} is well-posed in the energy space,
and given $\widetilde{V}$ and $\widetilde{x}$ tangentially smooth define the new $V_{(k+1)}=V$ by solving the linear system \eqref{eq:eulerlagrangiancoordsmoothed}- \eqref{eq:waveequationsmoothed}, and $x_{(k+1)}=x$ by $dx/dt=V(t,x)$.
 The argument from the previous section
gives apriori bounds for the iterates $V_{(k+1)}$
and the only term that has to be estimated differently is the boundary term where we used that $V=D_t x$, where $x$ was related to $\widetilde{x}$ by $\widetilde{x}=S_\varepsilon^* S_\varepsilon x$, because now $\widetilde{x}=S_\varepsilon^* S_\varepsilon z$ is related to the previous iterate.
More precisely we can no longer estimate the boundary term in
\eqref{eq:boundaryproblems}
\begin{equation}\label{eq:boundaryproblems2}
\!\!\int_{\pa\Omega}\!\!\!\!\!\widetilde{\mathcal{N}}_j T^{I\!} \widetilde{x}^j D_t \big( \widetilde{\mathcal{N}}_i   T^{I\!} {x}^{i} \big)  |\widetilde{\pa} h| \widetilde{\nu} d S
=\!\!\int_{\pa\Omega}\!\!\!\!\!\widetilde{\mathcal{N}}_j T^{I\!}  S_\varepsilon z^j D_t \big( \widetilde{\mathcal{N}}_i   T^{I\!} S_\varepsilon  {x}^{i} \big)  |\widetilde{\pa} h| \widetilde{\nu} d S+\text{Lower order} ,
\end{equation}
by moving $S_\varepsilon$ to the other factor
since when $z\!\neq \! x$ the integrand can no longer be written as a time derivative plus lower order.
However as long as at least one of the vector fields in $T^I$ is a space tangential vector field we can use the smoothing to trade a tangential derivative for a power of $1/\varepsilon$. With one less derivative it can be estimated from the interior norm of $x$ using the restriction theorem.
On the other hand if one of the vector fields in  $T^I\!$ is a time derivative
then we can estimate it by one less derivative of $V\!\!$ on the boundary and hence in the interior.  For the iterates
\begin{equation}
\mathcal{E}_k^{I\!}(t)=
 \int_{\Omega} |V_{(k)}^I|^2 \kappa dy + \int_{\Omega} e_1 (h_{(k)}^I)^2 \kappa dy+\int_\Omega |T^I\! x_{(k)}|^2\, \kappa dy,
 \label{eq:eniteratedef}
\end{equation}
we therefore only have
\begin{multline}
\frac{d}{dt}\mathcal{E}_{k+1}^{I\!}(t)\lesssim \frac{c_0}{\varepsilon}\,\mathcal{E}_{k+1}^{I\!}(t)
+\frac{C_J^{(k)}}{\varepsilon}  \!\!\!\!\!  \!\!\!\!\!  \sum_{\,\,\,\,\,\,\,\,\,|J|\leq |I|-1}\!\!\!\!\!\!\!\!\!
\|\widetilde{\pa}T^J x_{(k+1)}\|_{L^2(\Omega)}^2
+\|\widetilde{\pa}T^J V_{(k+1)}\|_{L^2(\Omega)}^2\\
+\frac{C_J^{(k+1)}}{\ve}\sum_{|J| \leq |I|-1}
\|\widetilde{\pa}T^J x_{(k)}\|_{L^2(\Omega)}^2
+\|\widetilde{\pa}T^J V_{(k)}\|_{L^2(\Omega)}^2,
\label{dtEmainnonreliter}
\end{multline}
where $C_J^{(\ell)}$ denotes a constant as in
\eqref{capitalCdef} but with $V, \xve$ replaced with
$V_{(\ell)}$ and $x_{(\ell)}$.
This only gives a uniform energy bound up to a time $t\leq T=O(\varepsilon)$.

\subsubsection{Estimates for time derivatives of the velocity}
We are also going to need estimates or time derivatives but that is easier. The apriori estimate above for a solution of Euler's equations works if the vector fields $T^I$ are any combination of space derivatives $T=T^a(y)\pa/\pa y^a$ and time derivatives $T=D_t$. However for the smoothing estimates one needs at least one space derivative $T^I=T^J T$, where $T=T^a(y)\pa/\pa y^a$.
On the other hand if $T^I=T^J D_t$, where $|J|=r-1$ then
 \begin{equation}
D_t  T^J  V_i
=-T^J \widetilde{\pa}_i  h,
\end{equation}
which  we shall see  is controlled by the energy for the wave equation.
We remark that the additional boundary estimate is only needed for all space tangential derivatives since $D_t x_\varepsilon\!=\!S_\varepsilon V$.

\subsection{Higher order wave and elliptic estimates for the enthalpy}\label{sec:enthalpywave}
We have
\begin{equation}
   e_1  D_t^2 h  - \widetilde{\Delta} h = \widetilde{\pa}_i \widetilde{V}^j\,\widetilde{\pa}_j V^i.
\end{equation}
Hence
\begin{equation}\label{eq:thehigherwaveequationerror}
   e_1  D_t^2 T^{J\!} h
   -\! \widetilde{\pa}_i \big( T^J{\widetilde{\pa}}^i  h\big)
   = P^J\!+Q^J,
   \end{equation}
   where $ P^J\!=\big[\widetilde{\pa}_i,T^J]\widetilde{\pa}^i h$ and
   $ Q^J\!\!=T^J\big( \widetilde{\pa}_i \widetilde{V}^j\,\widetilde{\pa}_j V^i\big)$.
   Here
   \begin{equation}\label{eq:PJeq}
   P^J\!= {\sum}_{J_1+\dots+J_k=J,\,|J_k|<J|}\,\, p_{J_1\dots J_k}^J \widetilde{\pa}_i T^{J_1} \widetilde{x}\cdots  \! \widetilde{\pa}T^{J_{k-1}} \widetilde{x}\cdot \widetilde{\pa} T^{J_k} {\widetilde{\pa}}^{i}  h,
   \end{equation}
    \begin{equation}\label{eq:QJeq}
   Q^J\!=\!{\sum}_{J_1+\dots+J_k=J,\, 1\leq \ell\leq k}\, \, q_{J_1\dots J_k}^{J\ell}\widetilde{\pa}_i T^{J_1} \widetilde{x}\cdots \! \widetilde{\pa}T^{J_{\ell-1}} \widetilde{x}\cdot \widetilde{\pa} T^{J_\ell} \widetilde{V}\!\cdot\widetilde{\pa} T^{J_{\ell+1}} \widetilde{x}\cdots  \widetilde{\pa}T^{J_{k-1}} \widetilde{x}\cdot\widetilde{\pa} T^{J_k}V^i\!\! ,
   \end{equation}
   for some constants $p_{J_1\dots J_k}^J$ and $q_{J_1\dots J_k}^{J\ell}$.
    $P^J$ and $Q^J$ are hence are lower order:
 \begin{equation}\label{eq:PJ}
 |P^J|\lesssim c_J {\sum}_{|K|\leq |J|}
 |\widetilde{\pa} T^K \widetilde{x}|
 +c_J {\sum}_{|K|\leq |J|-1}| \widetilde{\pa} T^{K} {\widetilde{\pa}}  h|,
 \end{equation}
 \begin{equation}\label{eq:QJ}
 |Q^J|\lesssim c_J {\sum}_{|K|\leq |J|}
 |\widetilde{\pa} T^K V|+|\widetilde{\pa} T^K \widetilde{V}|+
 |\widetilde{\pa} T^K \widetilde{x}|,
 \end{equation}
 and $c_J$ stands for a constant that depends on
$|\widetilde{\pa} T^L \widetilde{x}|$, $|\widetilde{\pa} T^L V|$,
$|\widetilde{\pa} T^L \widetilde{V}|$
and $|\widetilde{\pa} T^{L} {\widetilde{\pa}} h|$, for
$|L|\leq |J|/2 $.

\subsubsection{Higher order elliptic equations for the enthalpy}
\label{sec:elliptcenthalpy}
 To deal with the lower order terms $\widetilde{\pa} T^K {\widetilde{\pa}}  h$ on the right of \eqref{eq:thehigherwaveequationerror} we need the pointwise elliptic estimate in terms
of the divergence and the curl and tangential components of $T^K \widetilde{\pa} h$:
\begin{equation}
|\widetilde{\pa} T^K \widetilde{\pa} h|
\lesssim |_{\,}\widetilde{\div}\,T^K \widetilde{\pa} h|+|_{\,}\widetilde{\curl}\,T^K \widetilde{\pa} h|
+ {\sum}_{S\in{\mathcal S}}|S T^K \widetilde{\pa} h|,
\end{equation}
by \eqref{pwnonrel}.
At a lower order we can think of \eqref{eq:thehigherwaveequationerror}
as an elliptic equation
\begin{equation}\label{eq:thehigherwaveequationerrorelliptic}
  \widetilde{\div}(T^K \pave h)
   =e_1  D_t^2 T^{K\!} h- P^K-Q^K,
   \end{equation}
   where
   $P^K$ and $Q^K$ satisfy \eqref{eq:PJ}-\eqref{eq:QJ}.
  Moreover, the antisymmetric part satisfy
  \begin{equation}\label{eq:thecurltangentialgradhA}
\widetilde{\pa}_i T^K \widetilde{\pa}_j h
-\widetilde{\pa}_j T^K \widetilde{\pa}_i h
=A_{ij}^K,
\end{equation}
where $A_{ij}^K=[\widetilde{\pa}_i, T^K ]\widetilde{\pa}_j h
-[\widetilde{\pa}_j , T^K ]\widetilde{\pa}_i h$. We have
\begin{equation}\label{eq:AKdef}
A^K_{ij}={\sum}_{K_1+\dots+K_k=K,\, |K_k|<|K|}\,\, a^K_{K_1\dots K_k}
\widetilde{\pa}_i T^{K_1} \widetilde{x}\cdots  \! \widetilde{\pa}T^{K_{k-1}} \widetilde{x}\cdot \!\widetilde{\pa} T^{K_k} {\widetilde{\pa}}_{j}  h,
\end{equation}
for some constants $a^K_{K_1\dots K_k}$.
This is hence is lower order:
\begin{equation}\label{eq:anitsymmetricestimate}
|A^K|\lesssim c_0^\prime |\widetilde{\pa} T^K \widetilde{x}|+ c_K{\sum}_{|L|\leq |K|-1} |\widetilde{\pa} T^{L} {\widetilde{\pa}} h|+ |\widetilde{\pa}T^{L}\widetilde{x}|,
\end{equation}
where $c_0^\prime=|\widetilde{\pa} \widetilde{\pa} h|$ and $c_K$ is as in \eqref{lowercasecdef}.

Using the equations
\eqref{eq:thehigherwaveequationerrorelliptic} and \eqref{eq:thecurltangentialgradhA}
for the divergence and the curl of $ T^K \widetilde{\pa} h$ we therefore have
\begin{equation}
| \widetilde{\pa} T^{K} {\widetilde{\pa}}  h|
\lesssim |D_t^2 T^K h|+ \sum_{S\in \S} |S T^K \widetilde{\pa}h|
+c_K\!\!\!\!\!\sum_{|L|\leq |K|}\!\!\!
 |\widetilde{\pa} T^L V|+|\widetilde{\pa} T^L \widetilde{V}|+|\widetilde{\pa} T^L \widetilde{x}|
 +c_K \!\!\!\!\!\!\!\!\sum_{|L|\leq |K|-1}
 \!\!\!\!\! \!\! | \widetilde{\pa} T^{L} {\widetilde{\pa}}  h|,
\end{equation}
Repeated use of this gives
\begin{equation}\label{eq:paTpah2}
| \widetilde{\pa} T^{K} {\widetilde{\pa}}  h|
\lesssim c_r{\sum}_{|K'|\leq |K|}\Big(|D_t^2 T^{K'} \! h|+ \!{\sum}_{S\in\mathcal{S}} |S T^{K'\!} \widetilde{\pa}h|
+\!|\widetilde{\pa} T^{K'\!} V|+|\widetilde{\pa} T^{K'\!} \widetilde{V}|+|\widetilde{\pa} T^{K'\!} \widetilde{x}|\Big).
\end{equation}

\subsubsection{Higher order wave equation estimates for the enthalpy}
\label{enthalpyestimates}
Multiplying \eqref{eq:thehigherwaveequationerror} by $D_t T^J h$ and integrating we get
\begin{equation}
\int_\Omega e_1 D_t T^{J{}_{\!}}  h \, \,  D_t^2 T^{J{}_{\!}} h\,
  -  D_t T^{J{}_{\!}}  h \, \widetilde{\pa}_i \big( T^{J{}_{\!}}  \,{\widetilde{\pa}}^i  h\big) \,  d\widetilde{x}
  =\int_\Omega D_t T^{J{}_{\!}}  h \,\,  (P^{J}\! +Q^J)\, d\widetilde{x} .
\end{equation}
Integrating by parts and commuting we get
\begin{equation}
\int_\Omega e_1\frac{1}{2}D_t \big( D_t T^{J{}_{\!}}  h \big)^2 \,
   +\widetilde{\pa}_i D_t T^{J{}_{\!}}  h \,  \,  T^{J{}_{\!}}  \,{\widetilde{\pa}}^i  h \,  d\widetilde{x}
  =\int_\Omega D_t T^{J{}_{\!}}  h \,\,  (P^{J}\!+Q^J) \, d\widetilde{x}.
\end{equation}
Here \begin{equation}
 \widetilde{\pa}_i D_t T^{J{}_{\!}}  h
=  D_t T^{J}  \widetilde{\pa}_i h
+{R}_i^{J,1},
\end{equation}
where $R^{J,s}_i=R_i^{I}$, with $I=\{J,D_t^s\}$, where
\begin{equation}
R_i^{I}={\sum}_{I_1+\dots +I_k=I,\, |I_k|< |I|}\,\, r^I_{I_1\dots I_k} \, \widetilde{\pa} T^{I_1}\widetilde{x}\cdots \widetilde{\pa} T^{I_{k-1}} \widetilde{x}\cdot T^{I_{k}}\widetilde{\pa} h,
\end{equation}
for some constants $ r^I_{I_1\dots I_k}$.
This is lower order:
\begin{equation}\label{eq:Best}
|R^{J,1}|\lesssim c_J {\sum}_{|J^\prime|\leq |J|}\bigtwo(
|\widetilde{\pa}  T^{J^\prime}\widetilde{V}|
+|\widetilde{\pa}  T^{J^\prime}\widetilde{x}|
+|T^{J^\prime} \widetilde{\pa} h|\bigtwo),
\end{equation}
where $c_J$ depends on the above quantities for $|K|\leq |J|/2$.
We get
\begin{equation}
\frac{1}{2}\int_\Omega D_t \Big( e_1 \big( D_t T^{J{}_{\!}}  h \big)^2 \,
   + | T^{J{}_{\!}}  \,\widetilde{\pa}^i  h |^2 \Big) \,  d\widetilde{x}
  =\int_\Omega D_t T^{J{}_{\!}}  h \,\,  (P^{J}\!+Q^J)+ {R}_i^{J,1}T^{J}  \widetilde{\pa}^i h \, d\widetilde{x}.
\end{equation}
Hence with
\begin{equation}
\mathcal{W}^{J\!}(t)=\int_\Omega e_1 \big( D_t T^{J{}_{\!}}  h \big)^2 \,
   + | T^{J{}_{\!}}  \,{\widetilde{\pa}}  h |^2  \,  d\widetilde{x},
\end{equation}
it follows from \eqref{eq:PJ}-\eqref{eq:QJ}, \eqref{eq:paTpah2} and \eqref{eq:Best} that
\begin{equation}\label{eq:waveequationenergy}
\frac{d}{dt}\mathcal{W}^{J\!}(t)\lesssim c_0\,\mathcal{W}^{J\!}(t)
+c_J {\sum}_{|{J^{\prime}}|\leq |J|}\int_{\Omega}
 |\widetilde{\pa} T^{J{}^{\prime}} \!V|^2\!
 +|\widetilde{\pa} T^{J^{\prime}\!} \widetilde{V}|^2\!
 +|\widetilde{\pa} T
 ^{J^{\prime}\!} \widetilde{x}|^2\!+|D_t T^{J^{\prime}}\! h|^2\!+| T^{J^{\prime}} \widetilde{\pa} h|^2\, d y.
\end{equation}

\begin{remark} The estimate for the enthalpy above from the wave equation at order $|J|\leq |I|-1$, could also be obtained from the estimate for Euler's equation for the velocity at order $|I|$ since $T^J \pa h=T^J D_t V$.
\end{remark}

 To close the apriori energy bounds for Euler's equations we only need estimates for the wave equation with $|J|\leq |I|-1$ tangential derivatives. However, one can obtain estimates for the wave equation with $|I|$ derivatives at the same time, and this is needed for the additional  bound for the smoothed coordinate.

\subsubsection{Wave equation estimates for the enthalpy with an additional time derivative}\label{sec:extradt}

We will in fact estimate $\|D_t^{3} T^K h\|_{L^2(\Omega)}$, $\|  D_t^2 T^{K\!} \widetilde{\pa}h\|_{L^2(\Omega)}$ and
$\|\widetilde{\pa}  D_t T^K\widetilde{\pa} h\|_{L^2(\Omega)}$,
for $|K|\leq |I|-2$, and as a consequence also $\|D_t^2 T^J h\|_{L^2(\Omega)}$
and $\|D_t T^J \widetilde{\pa} h\|_{L^2(\Omega)}$, for $|J|\leq |I|-1$.
Let $|K|\leq r-2$, where $r=|I|$, and let $s\leq 2$. Then
\begin{equation}\label{eq:wavestime}
   e_1   D_t^2 D_t^s T^K\! h
   -\! \widetilde{\pa}_i \big( D_t^s T^K {\widetilde{\pa}}^i  h\big)
   = P^{K,s}\!+Q^{K,s},
   \end{equation}
   where $P^{K,s}=[\widetilde{\pa}_i ,D_t^s T^K]\widetilde{\pa}^i h$ and
   $Q^{K,s}=D_t^s T^K\big( \widetilde{\pa}_i \widetilde{V}^j\,\widetilde{\pa}_j V^i\big)$ are given by \eqref{eq:PJeq} respectively \eqref{eq:QJeq} with $J=D_t^s K$:
   We have
   \begin{equation}
   Q^{K,2}\!=\! \widetilde{\pa}_i T^{K}\! D_t^2 \widetilde{V}^j\,\, \widetilde{\pa}_j V^i\!
   +\widetilde{\pa}_i \widetilde{V}^j\,\, \widetilde{\pa}_j  T^{K}\!D_t^2 V^i\!
   +Q^{\,\prime K,2}\!\!,
   \end{equation}
   where $Q^{\,\prime K,2}$ consist of the terms of the form \eqref{eq:QJeq} with $J\!=\!D_t^2 K$, with $|J_i|\leq |I|$ and   $|J_\ell|<|I|$, $|J_k|\!<\!|I|$. The terms in $Q^{\,\prime K,2}$ are already in the lower order energy estimate \eqref{eq:waveequationenergy}, since
    $\widetilde{\pa} T^K\! D_t^2\widetilde{x}=
    \widetilde{\pa} T^K\! D_t\widetilde{V}$.
 Similarly
   \begin{equation}
   P^{K,2}
   = {\sum}_{D_t+T^{K'}\!\!+T=D_t^2+T^K}\widetilde{\pa}_i  T\widetilde{x}^k \,\,
   \widetilde{\pa}_k T^{K^\prime\!}\! D_t \widetilde{\pa}^{i{}_{\!}} h
   +P^{\, \prime K,2},
   \end{equation}
   where $P^{\, \prime K,2}$ consist of terms in \eqref{eq:PJeq} with $|J_i|\leq |I|$ and $|J_k|\leq |I|-2$,
   that  are already in the lower order energy estimate \eqref{eq:waveequationenergy}.
Since $D_t V =-\widetilde{\pa} h $ we see that to estimate $Q^{K,2}$ and $P^{K,2}$ it only remains to estimate
$\widetilde{\pa}_k T^{K^\prime\!}\! D_t \widetilde{\pa}^{i{}_{\!}} h$ for $|K^\prime|\leq |K|\leq r-2$. By \eqref{eq:thewholecase}
\begin{multline}
\|\widetilde{\pa}  T^K\! D_t \widetilde{\pa}h\|_{L^2(\Omega)}\lesssim
C_0{\sum}_{S\in\mathcal{S}}\| \widetilde{\pa} {S} T^{K}\! D_t \widetilde{x}\|_{L^2(\Omega)}
+c_0\| \widetilde{\div} \big( T^{K} D_t \widetilde{\pa}h\big)\|_{L^2(\Omega)}\\
 +C_K{\sum}_{|K^\prime|\leq |K|}\| \widetilde{\div} \big( T^{K^\prime} \widetilde{\pa} h\big)\|_{L^2(\Omega)}
+C_K{\sum}_{|J^\prime|\leq |K|+1}\| \widetilde{\pa} T^{J'} \widetilde{x}\|_{L^2(\Omega)}.
\end{multline}
Using \eqref{eq:wavestime} for $s=1$ to substitute the divergence gives
\begin{multline}\label{eq:paTkDtpa}
\|\widetilde{\pa}  T^K\! D_t \widetilde{\pa}h\|_{L^2(\Omega)}\lesssim
C_0\|  T^{K\!} D_t^3 h\|_{L^2(\Omega)}\\
 +C_K{\sum}_{|K^\prime|\leq |K|}\| \widetilde{\pa}  T^{K^\prime} \widetilde{\pa} h\|_{L^2(\Omega)}
+C_K{\sum}_{|J^\prime|\leq |K|+1}\| \widetilde{\pa} T^{J'}\widetilde{V}\|_{L^2(\Omega)}
+\| \widetilde{\pa} T^{J'} \widetilde{x}\|_{L^2(\Omega)},
\end{multline}
where the terms on the second row are already controlled by the terms in the lower order energy estimate \eqref{eq:waveequationenergy}, using
\eqref{eq:paTpah2}, and $\|  T^{K\!} D_t^3 h\|_{L^2(\Omega)}$ will be controlled by the higher order energy.

Multiplying \eqref{eq:wavestime} with $s=2$ by $D_t^3 T^K h$ and integrating by parts as in the previous section we see that we must estimate
\begin{equation}
R_i^{K,3}={\sum}_{T+D_t^2+T^{K^\prime}\!=T^K\!+D_t^3} \widetilde{\pa}_i T \widetilde{x}^k\,\,T^{K^\prime\!}\!D_t^2\widetilde{\pa}_{k} h
+\widetilde{\pa}_i T^K\! D_t^3 \widetilde{x}^k\,\,\widetilde{\pa}_{k} h+R_i^{\prime K,3},
\end{equation}
where $R_i^{\prime K,3}$ contain terms that are  controlled by the terms in the lower order energy estimate \eqref{eq:waveequationenergy}.
Here the sum is bounded by $c_0$ times $\|T^{K^\prime\!}\!D_t^2\widetilde{\pa}^{k{}_{\!}} h\|_{L^2(\Omega)}$,
for $|K^\prime|\leq |K|$ which will be part of the new energy and the second term is bounded \eqref{eq:paTkDtpa} which is also bounded by the new energy.

Summing up, with
\begin{equation}\label{eq:timederivativeswavenergydef}
\mathcal{W}^{K,s{}_{\!}}(t)=\int_\Omega e_1 \big( D_t^{1+s} T^{K{}_{\!}}  h \big)^2
   + | D_t^s T^{K{}_{\!}} {\widetilde{\pa}}  h |^2  \,  d\widetilde{x},
\end{equation}
we have
\begin{equation}\label{eq:waveequationenergytime}
\frac{d}{dt}\mathcal{W}^{K,2{}_{\!}}(t)\lesssim C_0\nquad\,\sum_{|K^\prime|\leq |K|}\nquad\mathcal{W}^{K^\prime\!{}_{\!},2{}_{\!}}(t)
+C_K\nquad\!\! \sum_{|{J^{\prime}}|\leq |K|+1}\int_{\Omega}
 |\widetilde{\pa} T^{J{}^{\prime}\!} V|^2\!
 +|\widetilde{\pa} T^{J^{\prime}\!} \widetilde{V}|^2\!
 +|\widetilde{\pa} T
 ^{J^{\prime}}\! \widetilde{x}|^2\!+|D_t T^{J^{\prime}}\! h|^2\!+| T^{J^{\prime}} \widetilde{\pa} h|^2\, d y.
\end{equation}
Moreover, because of \eqref{eq:paTkDtpa} we also have for $|J|=r-1$
\begin{equation}\label{eq:waveequationenergytimeJ}
\mathcal{W}^{J,1{}_{\!}}(t)\lesssim C_0\nquad\! \sum_{|K^\prime|\leq |J|-1}\nquad\! \mathcal{W}^{K^\prime\!{}_{\!},2{}_{\!}}(t)
+C_K\nquad\,\,\sum_{|{J^{\prime}}|\leq |J|}\int_{\Omega}
 |\widetilde{\pa} T^{J{}^{\prime}\!} V|^2\!
 +|\widetilde{\pa} T^{J^{\prime}\!} \widetilde{V}|^2\!
 +|\widetilde{\pa} T
 ^{J^{\prime}}\! \widetilde{x}|^2\!+|D_t T^{J^{\prime}}\! h|^2\!+| T^{J^{\prime}} \widetilde{\pa} h|^2\, d y.
\end{equation}

\subsubsection{Elliptic estimates for the enthalpy with a half derivative additional tangential regularity}\label{sec:extrahalf} Applying $\fdh T^K$ to the equation
\begin{equation}
   \widetilde{\Delta} h
   =  e_1  D_t^2 h -\widetilde{\pa}_i \widetilde{V}^j\,\widetilde{\pa}_j V^i,
\end{equation}
gives
\begin{equation}
  \fdh T^K\! \widetilde{\Delta} h
   =  e_1  D_t^2\fdh T^K\! h -\fdh Q^{K},
\end{equation}
where $Q^K$ satisfy \eqref{eq:QJeq}. At this point there is a lot of room in estimating
$\fdh Q^K$ so we just crudely estimate
$\|\fdh Q^K\|_{L^2(\Omega)}\lesssim \|\mathcal{S} Q^K\|_{L^2(\Omega)}$
with notation as in \eqref{eq:simplifiedtangentialnotation} and hence by \eqref{eq:QJ}
 \begin{equation}
 \|\fdh Q^K\|_{L^2(\Omega)}\lesssim {\sum}_{|J|\leq |K|+1} \|Q^J\|_{L^2(\Omega)}\lesssim C_K {\sum}_{|J|\leq |K|+1}
 \|\widetilde{\pa} T^{J} V\|_{L^2(\Omega)}+\|\widetilde{\pa} T^{J} \widetilde{V}\|_{L^2(\Omega)}+
 \|\widetilde{\pa} T^{J} \widetilde{x}\|_{L^2(\Omega)},
 \end{equation}
 where $C_K$ stand for a constant that depends on
$\|\widetilde{\pa} T^N \widetilde{x}\|_{L^\infty}$, $\|\widetilde{\pa} T^N V\|_{L^\infty}$,
$\|\widetilde{\pa} T^N \widetilde{V}\|_{L^\infty}$
and $\|\widetilde{\pa} T^{N} {\widetilde{\pa}}  h\|_{L^\infty}$ for
$|N|\leq |K|/2 $, $L^\infty = L^\infty(\Omega)$.
By Proposition \ref{prop:dirichlet} we have
\begin{equation}\label{eq:thehalfcasedirichlet}
\|\widetilde{\pa} \fdh  T^K \widetilde{\pa} h\|_{L^2(\Omega)}\lesssim C_K\!{\sum}_{|K'|\leq |K|,\, k=0,1}\bigtwo( \| \fd^{\!\nicefrac{k}{2}} T^{K'\!}\widetilde{\triangle} h \|_{L^2(\Omega)}
+\| \widetilde{\pa}\fd^{\!\nicefrac{k}{2}} \mathcal{S}^1 T^{K'\!} \widetilde{x}\|_{L^2(\Omega)} \bigtwo),
\end{equation}
where $C_K$ depends on $\| \widetilde{\pa} T^{N}\widetilde{x}\|_{L^\infty}$ and
$\|T^N \widetilde{\pa} h\|_{L^\infty}$ for $|N|\leq |K|/2+3$.
It follows that
\begin{multline}
\|\widetilde{\pa} \fdh  T^K \widetilde{\pa} h\|_{L^2(\Omega)}
\lesssim C_K \| D_t^2\fdh T^K h\|_{L^2(\Omega)}
+ C_K{\sum}_{|J|\leq |K|+1}\|\widetilde{\pa}\fdh  T^{J} \widetilde{x}\|_{L^2(\Omega)}\\
+ C_K {\sum}_{|J|\leq |K|+1}
 \|\widetilde{\pa} T^{J} V\|_{L^2(\Omega)}+\|\widetilde{\pa} T^{J} \widetilde{V}\|_{L^2(\Omega)}+
 \|\widetilde{\pa} T^{J} \widetilde{x}\|_{L^2(\Omega)},
\end{multline}
and hence
\begin{multline}\label{eq:pahalfTKpa}
\|\widetilde{\pa} \fdh  T^K \widetilde{\pa} h\|_{L^2(\Omega)}
\lesssim C_K  {\sum}_{|J|\leq |K|+1}\mathcal{W}^{J,1{}_{\!}}(t)
+ C_K{\sum}_{|J|\leq |K|+1}\|\widetilde{\pa}\fdh  T^{J} \widetilde{x}\|_{L^2(\Omega)}\\
+ C_K {\sum}_{|J|\leq |K|+1}
 \|\widetilde{\pa} T^{J} V\|_{L^2(\Omega)}+\|\widetilde{\pa} T^{J} \widetilde{V}\|_{L^2(\Omega)}+
 \|\widetilde{\pa} T^{J} \widetilde{x}\|_{L^2(\Omega)},
\end{multline}

\subsection{The divergence estimates for the velocity and coordinates}

\subsubsection{The divergence estimates used to estimate $V$}
 By \eqref{eq:GIequation}
\begin{equation}\label{eq:divergence}
D^J=\widetilde{\div} (T^J V)+e_1 D_t T^J h-\widetilde{\pa}_i T^{J\!} \widetilde{x}^k \,\,\widetilde{\pa}_k V^i=G^J,
\end{equation}
where by \eqref{eq:GIestimate}  $G^J$
is lower order:
\begin{equation}
|G^J| \lesssim c_J {\sum}_{|K|\leq |J|-1}
 |\widetilde{\pa} T^K V|+|\widetilde{\pa} T^K x|.
 \end{equation}
\subsubsection{The improved half derivative divergence estimates used to estimate the coordinates}
\label{sec:improvediv} We only need to prove an additional estimate for all space tangential derivatives of the coordinate since if we have one time derivative it follows from the estimates for $V$.
 We have
\begin{equation}
 D_t \bigtwo( e_1  T^J h+  \widetilde{\div} (T^J \!x)\bigtwo)
 =\widetilde{\pa}_i T^{J\!} \widetilde{x}^k \,\,\widetilde{\pa}_k V^i -\widetilde{\pa}_i T^{J\!} {x}^k \,\,\widetilde{\pa}_k  \widetilde{V}^i+G^J,
 \label{divxeqn}
\end{equation}
where $G^J$ is lower order. We need to commute this with $\fdh\sm$. Note first that
\begin{equation}
\| \fdh\sm G^J \|_{L^2(\Omega)}
\lesssim{\sum}_{|N|\leq 1} \| S^N G^J \|_{L^2(\Omega)}
\lesssim C_J {\sum}_{|J'|\leq |J|}
 \|\widetilde{\pa} T^{J'} V\|_{L^2(\Omega)}+\|\widetilde{\pa} T^{J'}\! x\|_{L^2(\Omega)}.
\end{equation}
By Lemma \ref{lem:halfderleibnitzandsmoothing} and  Lemma \ref{lem:gradientfractionalsmoothingandfractionalderivativecommute} we have
\begin{equation}
\bigtwo\|\fdh \sm\big(\widetilde{\pa}_i T^{J\!} {x}^k \,\,\widetilde{\pa}_k  \widetilde{V}^i\big)
-\widetilde{\pa}_i T^{J\!}\fdh \sm {x}^k \,\,\widetilde{\pa}_k  \widetilde{V}^i\bigtwo\|_{L^2(\Omega)}
\lesssim C_2 \| \widetilde{\pa} T^{J\!} {x}\|_{L^2(\Omega)},
\end{equation}
and the same inequality holds with $x$ replaced by $\widetilde{x}$ and $\widetilde{V}$ replaced by $V$. Hence
\begin{equation}\label{eq:divhalf}
D^{J,\nicefrac{1}{2}}_\varepsilon= e_1\fdh \sm T^J h+ \fdh \sm\,  \widetilde{\div} (T^J\! x),
\end{equation}
satisfy
\begin{equation}\label{eq:Divhalfder}
\|D_t D^{J,\nicefrac{1}{2}}_\varepsilon\|_{L^2(\Omega)}
\lesssim C_2 \| \widetilde{\pa} T^{J\!}\fdh \sm {x}\|_{L^2(\Omega)}
+C_J {\sum}_{|J'|\leq |J|}
 \|\widetilde{\pa} T^{J'\!} V\|_{L^2(\Omega)}+\|\widetilde{\pa} T^{J'}\! x\|_{L^2(\Omega)}.
\end{equation}
We have
\begin{equation}
\| \fdh \sm\, T^J\! h\|_{L^2(\Omega)}
\lesssim {\sum}_{|N|\leq 1} \| S^N T^J\! h\|_{L^2(\Omega)}
\lesssim C_J {\sum}_{|J'|\leq |J|} \| T^J\widetilde{\pa} h\|_{L^2(\Omega)} +\|\widetilde{\pa}  T^J\! x\|_{L^2(\Omega)}.
\end{equation}
By
Lemma \ref{lem:mutiplicationsmoothingcommute}, Lemma \ref{lem:gradientsmoothingcommute} and Lemma \ref{lem:halfderleibnitz}
\begin{equation}
\bigtwo\| \widetilde{\div} (T^J\fdh \sm x)-\fdh \sm \widetilde{\div} (T^J x)\bigtwo\|_{L^2(\Omega)}
\lesssim  C_0\| \widetilde{\pa} T^{J}V\|_{L^2(\Omega)}.
\end{equation}
Hence
\begin{equation}\label{eq:divhalfest}
\big\| \widetilde{\div} (T^J\fdh \sm x)-D_\varepsilon^{J,\nicefrac{1}{2}}\big\|_{L^2(\Omega)}
\lesssim  C_J {\sum}_{|J'|\leq |J|} \| T^{J'}\widetilde{\pa} h\|_{L^2(\Omega)} +\|\widetilde{\pa}  T^{J'}\! x\|_{L^2(\Omega)}
+\|\widetilde{\pa} T^{J'\!} V\|_{L^2(\Omega)}.
\end{equation}

\subsection{The curl estimates for the velocity and coordinates}
\subsubsection{The curl estimates used to estimate $V$}
\label{sec:curlestim}
By \eqref{eq:symmetricequation}
\begin{equation}
D_{t\,} T^J  V_i= - T^J \widetilde{\pa}_i h,
\end{equation}
and hence
\begin{equation}
 \widetilde{\curl}\,  D_t  T^J  V_{ij}= - A_{ij}^J,
\end{equation}
where $ A_{ij}^J$ is given by \eqref{eq:thecurltangentialgradhA}.
We note that
\begin{equation*}
D_t \bigtwo( \widetilde{\pa}_i D_t-[D_t,\widetilde{\pa}_i]\bigtwo)
=\widetilde{\pa}_i D_t^2 -\big[D_t,[D_t,\widetilde{\pa}_i]\big],
\end{equation*}
where $[D_t,\widetilde{\pa}_i]=-\widetilde{\pa}_i \widetilde{V}^k\, \widetilde{\pa}_k $ and
$\big[D_t,[D_t,\widetilde{\pa}_i]\big]= \big[\widetilde{\pa}_i
D_t\widetilde{V}^{k\!}\!
-2\widetilde{\pa}_i \widetilde{V}^n
\,\widetilde{\pa}_n v^k\big]\widetilde{\pa}_k\,$. Applying this to $T^J\! x_j$ gives
\begin{equation*}
D_t \bigtwo( \widetilde{\pa}_i T^J V_j-[D_t,\widetilde{\pa}_i]T^J\! x_j\bigtwo)
=\widetilde{\pa}_i D_t T^J V_j -\big[D_t,[D_t,\widetilde{\pa}_i]\big]T^J\! x_j,
\end{equation*}
Hence, there are linear forms $L_{ij}^1[\widetilde{\pa}T^J\!  x]$ and $L_{ij}^2[\widetilde{\pa} T^J\! x]$ such that with
\begin{equation}\label{eq:modifiedcurl}
K_{ij}^J=\widetilde{\curl}\,  T^J V_{ij}+L^1_{ij}[\widetilde{\pa}T^J\!  x],
\end{equation}
we have
\begin{equation}\label{eq:curlequation}
D_t K_{ij}^J
=L^2_{ij}[\widetilde{\pa} T^J\! x]- A_{ij}^J ,
\end{equation}
where $ A_{ij}^J$, the antisymmetric part of $\widetilde{\pa}_i T^J\widetilde{\pa}_j h$, is lower order
by \eqref{eq:anitsymmetricestimate} and \eqref{eq:paTpah2}:
\begin{equation}\label{eq:anitsymmetricestimate2}
|A^J|
\lesssim c_0 |\widetilde{\pa} T^{J\!} \widetilde{x}|
+c_{r}{\sum}_{|K|\leq |J|-1}\Big(|D_t^2 T^K\! h|+ \!{\sum}_{S\in \S} |S T^K \widetilde{\pa}h|
+\! |\widetilde{\pa} T^K V|+|\widetilde{\pa} T^K \widetilde{V}|+|\widetilde{\pa} T^K \widetilde{x}|\Big).
\end{equation}
We further note that there is a linear form $L^3_{ij}[\widetilde{\pa}T^J x]$ such that
\begin{equation}
D_{t\,} \widetilde{\curl} \big(T^J \!x\big)_{ij}
\!=K_{ij}^J+L^3_{ij}[\widetilde{\pa}T^J\! x].
\end{equation}

\subsubsection{The improved half derivative curl estimates used to estimate the coordinates}
\label{sec:improvedcurl}
We need to commute \eqref{eq:curlequation} with $\sm$ and with $\fdh$. We have
\begin{equation}
D_{t\,}\fdh \sm  T^J  V_i=-\fdh \sm  T^J \widetilde{\pa}_i h,
\end{equation}
and hence
\begin{equation}
 \widetilde{\curl}\bigtwo(\!  D_t \fdh \sm  T^J  V{}_{\!}\bigtwo)_{ij}= -A_{ij,\ve}^{J, \nicefrac{1}{2}},
\end{equation}
where
\begin{equation}
A_{ij,\ve}^{J, \nicefrac{1}{2}}
=\widetilde{\pa}_i \fdh \sm  T^J \widetilde{\pa}_j h
-\widetilde{\pa}_j \fdh \sm  T^J \widetilde{\pa}_i h.
\end{equation}
With
\begin{equation}\label{eq:curlhalf}
K_{ij,\ve}^{J,\nicefrac{1}{2}}=\widetilde{\curl}\big(T^J \fdh \sm \, V\big)_{\!ij}\!+L_{ij}^1\big[\widetilde{\pa}\fdh T^J\!\sm   x\big]
\end{equation}
we have
\begin{equation}\label{eq:curlhalfest}
D_t K_{ij,\ve}^{J,\nicefrac{1}{2}}
=L_{ij}^2\big[\widetilde{\pa} \fdh T^J\! \sm x\big]-A_{ij,\varepsilon}^{J,\nicefrac{1}{2}}.
\end{equation}
Here
$ A_{ij,\varepsilon}^{J,\nicefrac{1}{2}}$ is lower order. We have
\begin{equation}
 A_{ij,\varepsilon}^{J,\nicefrac{1}{2}}=\fdh \sm  A_{ij}^{J}+\big[\widetilde{\pa}_i,\fdh \sm\big]T^J\widetilde{\pa}_j h-\big[\widetilde{\pa}_j,\fdh \sm\big]T^J\widetilde{\pa}_i h .
\end{equation}

 We may assume that at least one of the vector fields in $T^J$ is space tangential since if one is a time derivative $D_t$ we already have stronger estimate at a lower order using that $D_t\widetilde{x}=\widetilde{V}$.
Here using that $T^J=ST^K$, where $S$ is space tangential we can use Lemma \ref{lem:gradientfractionalsmoothingandfractionalderivativecommute} to estimate
\begin{equation}
\|\big[\widetilde{\pa},\fdh \sm\big]T^J\widetilde{\pa} h \|_{L^2(\Omega)}
\lesssim \|\widetilde{\pa}\fdh T^K\widetilde{\pa} h \|_{L^2(\Omega)},
\end{equation}
which is under control by \eqref{eq:pahalfTKpa}. To estimate
$\fdh \sm  A_{ij}^{J}$ we apply $\fdh \sm$ to \eqref{eq:AKdef} using
Lemma \ref{lem:halfderleibnitzandsmoothing} and Lemma \ref{lem:gradientfractionalsmoothingandfractionalderivativecommute}
\begin{equation*}
\|\fdh \sm A^J\|_{L^2(\Omega)}\lesssim C_0 \|\widetilde{\pa}\fdh \sm T^J \widetilde{x}\|_{L^2(\Omega)}+ C_J\!\!\!\!\!\sum_{|K|\leq |J|-1}\!\!\!\!\!\|\widetilde{\pa}\fdh T^{K} {\widetilde{\pa}} h\|_{L^2(\Omega)}+ \|\widetilde{\pa}\fdh \sm T^K\widetilde{x}\|_{L^2(\Omega)}.
\end{equation*}
We conclude that the same is true for $ \! A_{ij,\varepsilon}^{J,\nicefrac{1}{2}}\!$
as long as there is a space tangential derivative in $T^J$:
\begin{equation}\label{eq:anitsymmetricestimatehalf}
\|A_{ij,\varepsilon}^{J,\nicefrac{1}{2}}\|_{L^2(\Omega)}\lesssim C_0 \|\widetilde{\pa}\fdh \sm T^J \widetilde{x}\|_{L^2(\Omega)}+ C_{\!J}\!{\sum}_{|K|\leq |J|\shortminus 1} \|\widetilde{\pa}\fdh T^{K}\! {\widetilde{\pa}} h\|_{L^2(\Omega)}+ \|\widetilde{\pa}\fdh \sm T^K\!\widetilde{x}\|_{L^2(\Omega)}.
\end{equation}

Moreover
\begin{equation}\label{eq:curlequation2}
D_t{}_{\,}\widetilde{\curl}\big( {}_{\!} \fdh T^J \!\sm x\big)_{\!ij}\!
=K_{ij,\ve}^{J,\nicefrac{1}{2}}+L^3_{ij}[\widetilde{\pa}\fdh T^J \!\sm x].
\end{equation}

Note also that by
Lemma \ref{lem:mutiplicationsmoothingcommute}, Lemma \ref{lem:gradientsmoothingcommute} and Lemma \ref{lem:halfderleibnitz}
\begin{equation}
\bigtwo\| \widetilde{\curl}\big(  T^J \!\fdh\sm V\big)-\fdh\sm{}_{\,}\widetilde{\curl}\big(  T^J  V\big)\bigtwo\|_{L^2(\Omega)}
\lesssim C_0\| \widetilde{\pa} T^{J}V\|_{L^2(\Omega)}.
\end{equation}

\subsection{The elliptic estimates}\label{sec:thenormevolution}

\subsubsection{The elliptic estimate for the velocity}
Using Lemma \ref{app:pwdiff} we have
\begin{equation}
|\widetilde{\pa} T^J V|
\lesssim |_{\,}\widetilde{\div}\,T^J V|+|_{\,}\widetilde{\curl}\,T^J V|
+\!\! {\sum}_{S\in{\mathcal S}}|S T^J V|.
\end{equation}
and hence with $D^J$ as in \eqref{eq:divergence} and $K^J$ as in \eqref{eq:modifiedcurl} we have
\begin{equation}\label{eq:ellipticV}
|\widetilde{\pa} T^J V|
\lesssim |D^J|+|K^J|
+ c_0( |\widetilde{\pa}T^J\widetilde{x}|+ |\widetilde{\pa}T^J\!{x}|)
+|D_t T^{J\!} h|+ \!\!{\sum}_{S\in{\mathcal S}}|S T^J V|,
\end{equation}
where $c_0=c_0( |\widetilde{\pa}\widetilde{V}|,|\widetilde{\pa}{V}|)$.
Here $D^J$ is lower order:
\begin{equation}
|D^J|\lesssim c_J {\sum}_{|K|\leq |J|-1} |\widetilde{\pa} T^K \widetilde{x}| +  |\widetilde{\pa} T^K V|,
\end{equation}
where $c_J$ stands for a constant that depends on
$|\widetilde{\pa} T^L \widetilde{x}|$ and $|\widetilde{\pa} T^L V|$ for
$|L|\leq |J|/2 $. Hence
\begin{equation}\label{eq:ellipticV0}
|\widetilde{\pa} T^J V|
\lesssim  |K^J|
+c_0\big(|\widetilde{\pa}T^J\widetilde{x}|+ |\widetilde{\pa}T^J{x}|\big)
+|D_t T^J h|+\!{\sum}_{S\in{\mathcal S}}|S T^J V|
+c_J {\sum}_{|K|\leq |J|-1}|\widetilde{\pa} T^K \widetilde{x}| +  |\widetilde{\pa} T^K V|.
\end{equation}
Hence by repeated use of this we get, with a constant $c_r$ depending on
$\sum_{|J| \leq r} c_J$,
\begin{equation}\label{eq:ellipticV1}
{\sum}_{|J|\leq r} |\widetilde{\pa} T^J V|
\lesssim c_r{\sum}_{|J|\leq r} \Big(|K^J|
+|\widetilde{\pa}T^J\widetilde{x}|+ |\widetilde{\pa}T^{J\!}{x}|
+|D_t T^{J\!} h|+\!{\sum}_{S\in{\mathcal S}}|S T^J V|+|T^J V|\Big).
\end{equation}

\subsubsection{The elliptic estimate for the enthalpy}
 To deal with lower order terms with $\!\widetilde{\pa} T^{\!J} {\widetilde{\pa}}  h\!$ we have
\begin{equation}\label{eq:paTpah3}
{\sum}_{|K|\leq r}| \widetilde{\pa} T^{K} {\widetilde{\pa}}  h|
\lesssim c_r{\sum}_{|K|\leq r}\Big(|D_t^2 T^{K\!} h|+ \!{\sum}_{S\in \S} |S T^K \widetilde{\pa}h|
+\! |\widetilde{\pa} T^K V|+|\widetilde{\pa} T^K \widetilde{V}|+|\widetilde{\pa} T^K \widetilde{x}|\Big),
\end{equation}
 from \eqref{eq:paTpah2}.
Note that \eqref{eq:paTpah3} can be seen as a special case of
\eqref{eq:ellipticV} with $T^J\!\!=T^K\! D_t$ and $K^J$ replaced by $\curl{D_t T^K V}$.

We also note that
\begin{equation}\label{eq:normequivalence}
|\widetilde{\pa}T^J h|
\lesssim c_J {\sum}_{|K|\leq |J|}\big( |T^K\widetilde{\pa} h|+|\widetilde{\pa} T^K \widetilde{x}|\big),
\end{equation}
where $c_J$ depends on $|T^L\widetilde{\pa} h|$ and $|\widetilde{\pa} T^L \widetilde{x}|$ for $|L|\leq |J|/2$.

\subsubsection{The additional elliptic estimate for the smoothed coordinate $S_\varepsilon x$}
This one will be control from the boundary term with normal components only
 using the estimates in Section \ref{sec:divcurlL2}:
 \begin{multline}\label{eq:coordinatehalfest}
 {\sum}_{|J|\leq r-1}||\widetilde{\pa} T^J\! \fdh\sm x||_{L^2(\Omega)}^2\!
 +\!\!{\sum}_{|I|\leq r}\| T^I\! S_\varepsilon x\|_{L^2(\pa\Omega)}^2
 \leq C_{{}_{\!}1\!}{\sum}_{|I|\leq r}\| \mathcal{N}\!\cdot\! T^I\! S_\varepsilon x\|_{L^2(\pa\Omega)}^2\\
 +C_{{}_{\!}1}{\sum}_{|J|\leq r-1}\!||\widetilde{\div} T^J\! \fdh\sm x||_{L^2(\Omega)}^2\!
 + ||\widetilde{\curl}_{\,} T^J\! \fdh\sm x||_{L^2(\Omega)}^2
 +\|\widetilde{\pa}  T^J \sm x\|_{L^2(\Omega)}^2.
\end{multline}

\subsection{The combined div-curl evolution system}
\label{divcurlsection}
We now want to control in particular $|\widetilde{\pa} T^J V|$.
Although we do not have evolution equation for $|\widetilde{\pa} T^J V|$,
 it is by \eqref{eq:ellipticV} bounded by quantities for which we have evolution equations, plus lower order terms that can be bounded recursively.
 For the first term in \eqref{eq:ellipticV}, $K^J$, we have \eqref{eq:curlequation}, for $\widetilde{\pa} T^J x$ and $\widetilde{\pa} T^J \widetilde{x}$, we get an evolution equation from $D_t x=V$, see below, for the next two terms we have the energies for the wave equation and for Euler's equations, and the last two terms are lower order.

\subsubsection{The lowest order curl-divergence system}
For the lowest $r$ we have
\begin{equation}\label{eq:lowestcurldivcoord}
|D_t \widetilde{\curl}_{\,} V|\lesssim |\widetilde{\pa}V|\, |\widetilde{\pa} \widetilde{V}| ,\qquad |D_t \pa_y x|\lesssim |\widetilde{\pa}  V|, \qquad |\widetilde{\div} V|\lesssim |D_t h|,
\end{equation}
together with
\begin{equation}\label{eq:lowestellipticV}
|\widetilde{\pa}  V|
\lesssim |\widetilde{\curl}_{\,} V|+|\widetilde{\div} V|+ {\sum}_{S\in{\mathcal S}}|SV|.
\end{equation}
Since $\widetilde{\div} \widetilde{\pa} h=\widetilde{\triangle} h=e_1 D_t^2 h +\widetilde{\pa}_i\widetilde{V}^j\widetilde{\pa}_j V^i$ and $\widetilde{\curl}\widetilde{\pa} h=0$
\begin{equation}\label{eq:lowestelliptich}
|\widetilde{\pa}^2 h|
\lesssim | D_t^2 h| + {\sum}_{S\in{\mathcal S}}|S\widetilde{\pa} h|
+ |\widetilde{\pa}  V|\, |\widetilde{\pa}  \widetilde{V}|.
\end{equation}
These equations together with the energy estimates for tangential
vector fields $T$ applied to $V$ and to $(D_t h,\widetilde{\pa} h)$ form a closed system. $L^2$ estimates of higher order versions of the above equations for tangential vector fields applied to these quantities together with the energy estimates for tangential vector fields applied to $V$ and to  $(D_t h,\widetilde{\pa} h)$ gives a closed system in $L^2$
assuming that we have bounds in $L^\infty$ for fewer tangential derivatives of these quantities. On the other hand $L^2$ control of tangential derivatives of
\eqref{eq:lowestellipticV} and \eqref{eq:lowestelliptich} gives $L^\infty$ of fewer
tangential
vector fields applied to $V$ and to $(D_t h,\widetilde{\pa} h)$ and given this control one can use \eqref{eq:lowestellipticV} and \eqref{eq:lowestelliptich} to estimate also the $L^\infty$ norm of these quantities and then together with
higher order version of \eqref{eq:lowestcurldivcoord} they form a closed system also in $L^\infty$.

\subsubsection{The point wise evolution equation for the coordinate}
\label{pwcoordeqn}
Note that $\widetilde{\pa} T^J\! x$ is equivalent to
\begin{equation}
X^{1,J}_{ai}=\pa_{y^a}T^J x_i,\qquad \widetilde{X}^{1,J}_{ai}=\pa_{y^a}T^J \widetilde{x}_i,\qquad
\end{equation}
Moreover we also express $\partial_{y^a}$ in spherical coordinates then it commutes with
the smoothing in the tangential directions and so in these coordinates for any function $\pa \sm f= \sm\pa f$ and so
 $\|\pa \sm f\|_{L^p(\mathbf{S}^2)}\lesssim \| \pa f\|_{L^p(\mathbf{S}^2)}$,
 for $p=2,\infty$.  Moreover
 $\|\pa [\sm,T^J] f \|_{L^p(\mathbf{S}^2)}\lesssim \sum_{|K|\leq |J|-1} \|\pa Z^K f\|_{L^p(\mathbf{S}^2)}$,
 for $p=2,\infty$.
 It follows that
 \begin{equation}
 \| \widetilde{X}^{1,J}\|_{L^p(\mathbf{S}^2)}
 \lesssim {\sum}_{|J'|\leq |J|}\|X^{1,J'}\|_{L^p(\mathbf{S}^2)},\qquad p=2,\infty.
 \end{equation}
 We have the simple evolution equation
 \begin{equation}
|D_t X^{1,J}|\lesssim |\widetilde{\pa} T^J V|.
\end{equation}
Hence using \eqref{eq:ellipticV}
we have the simple evolution equation
\begin{equation}\label{eq:coordinateevolution}
|D_t X^{1,J}|
\lesssim c_{J}{\sum}_{|J'|\leq |J|} \Big(|K^{J'\!}|
+|\widetilde{\pa}T^{J'\!}\widetilde{x}|+ |\widetilde{\pa}T^{J'}\!{x}|
+|D_t T^{J'} \! h|+\!{\sum}_{S\in{\mathcal S}}|S T^{J'\!} V|+|T^{J'\!} V|\Big),
\end{equation}
where $K^J$ given by \eqref{eq:modifiedcurl} is a lower order modification of $\widetilde{\curl}\, T^J V$.

\subsubsection{The point wise evolution equation for the curl}
By \eqref{eq:curlequation} and \eqref{eq:anitsymmetricestimate2}
\begin{equation}
 |D_t K^J|
\lesssim c_0\big( |\widetilde{\pa} T^J \widetilde{x}|+|\widetilde{\pa} T^J{x}|\big)
+c_{J}\!\!\!\!\!\!\!\!\sum_{|K|\leq |J|-\!1}\!\!\!\!\!\bigtwo(|D_t^2 T^K h|+ \!\sum_{S\in \S} |S T^K \widetilde{\pa}h|
+\! |\widetilde{\pa} T^K V|+|\widetilde{\pa} T^K \widetilde{V}|+|\widetilde{\pa} T^K \widetilde{x}|\bigtwo).
\end{equation}

\subsubsection{The combined curl-divergence system}
Let us introduce some notation:
\begin{equation}
V^{1,r}\!=\!{\sum}_{|I|\leq r}|\widetilde{\pa} T^{I\!} V|,\qquad
X^{1,r}\!=\!{\sum}_{|I|\leq r}|\widetilde{\pa} T^I \!X|,\qquad
 K^r\!=\!{\sum}_{|I|\leq r}|K^{I\!}|,
\end{equation}
and
\begin{equation}
 V^{r}\!=\!{\sum}_{|I|\leq r}|T^{I\!} V|,\qquad
W^{r}\!=\!{\sum}_{|I|\leq r}|T^I\widetilde{\pa}h |+|D_t T^I h|,\qquad H^r\!=\!{\sum}_{|I|\leq r}|\widetilde{\pa} T^I\widetilde{\pa} h|.
\end{equation}
By \eqref{eq:curlequation} and substituting \eqref{eq:ellipticV} in the right of \eqref{eq:coordinateevolution}
\begin{equation}
|D_t K^r|\lesssim c_r\big( X^{1,r}+\widetilde{X}^{1,r}+ V^{1,r-1}+\widetilde{V}^{1,r-1}+W^r\big),
\end{equation}
\begin{equation}
|D_t X^{1,r}|\lesssim c_r\big(K^r+ X^{1,r}+\widetilde{X}^{1,r}+V^{1+r}+W^{r}\big),
\end{equation}
and
\begin{equation}
V^{1,r}\lesssim c_r\big(K^r+ X^{1,r}+\widetilde{X}^{1,r}+V^{1+r}+W^{r}\big),
\end{equation}
where $c_r$ depends on bounds these quantities with $r$ replaced by $r/2$.
Moreover
\begin{equation}
H^{r-1}\lesssim c_r\big(\widetilde{X}^{1,r-1}+V^{1,r-1}+\widetilde{V}^{1,r-1}+W^{r}\big).
\end{equation}

Let
\begin{align}
V^{1,r}_p(t)&=\|V^{1,r}(t,\cdot)\|_{L^p},\quad
&K^r_p(t)=\|K^r(t,\cdot)\|_{L^p},\qquad X^{1,r}_p(t)&=\|X^{1,r}(t,\cdot)\|_{L^p},\label{normsdef1}\\
V^r_p(t)&=\|V^r(t,\cdot)\|_{L^p},\quad
&W^r_p(t)=\|W^r(t,\cdot)\|_{L^p},\qquad
H^r_p(t)&=\|H^r(t,\cdot)\|_{L^p}.
\label{normsdef2}
\end{align}
where $L^p=L^p(\Omega)$, $p=2,\infty$.  Then
\begin{align}
| K^{r\, \prime}_p(t)|&\lesssim C_r\big( X^{1,r}_p(t)+ V_p^{1,r-1}(t)+W^r_p(t)\big),
\label{eq:evolutionsystem.a}\\
|{X}^{1,r\, \prime}_p(t)|&\lesssim C_r\big(K^r_p(t)+ X^{1,r}_p(t)+V^{1+r}_p(t)+W^{r}_p(t)\big),\label{eq:evolutionsystem.b}
\end{align}
and
\begin{align}
V^{1,r}_p(t)&\lesssim C_r\big(K^r_p(t)+ X^{1,r}_p(t)+V^{1+r}_p(t)+W^{r}_p(t)\big),\label{eq:ellipticsystem.a}\\
H^{r-1}_p(t)&\lesssim C_r\big({X}^{1,r-1}_p(t)+{V}_p^{1,r-1}(t)+W^{r}_p(t)\big),
\label{eq:ellipticsystem.b}
\end{align}
where $C_r$ depends on bounds for $X^{1,s}_\infty$, $V^{1,s}_\infty$, $H^s_\infty$,
for $s\leq r/2$. To close this system we also need bounds for
${V}_p^{1,r}(t)$ and for $W^{r+1}_p(t)$.
The above curl-divergence evolution system will be used both for $p=2$ for large $r$ and for $p=\infty$ for small $r$. However, we also need the estimates for tangential derivatives of $V$ and $(D_t h,\widetilde{\pa} h)$. For $p=2$ these are given by the energy estimates and for $p=\infty$ these are obtained from using Sobolev's Lemma and the $L^2$ estimates of $H^r_2$ and $V^{1,r}_2$ above.

\subsubsection{The additional control of half a derivative of the coordinate}
Let
\begin{equation}
X^{J,\nicefrac{1}{2}}_\varepsilon=\fdh T^J\!\sm   x, \quad\text{and}\quad
V^{J,\nicefrac{1}{2}}_\varepsilon=\fdh T^J\!\sm  {}_{\!} V,
\end{equation}
and let
\begin{equation}
K^{r,\nicefrac{1}{2}}_{\varepsilon,2}(t)
={\sum}_{|J|\leq r}\|K^{J,\nicefrac{1}{2}}_\varepsilon(t,\cdot)\|_{L^2(\Omega)},\quad\text{and}\quad
D^{r,\nicefrac{1}{2}}_{\varepsilon,2}(t)={\sum}_{|J|\leq r}\| D^{J,\nicefrac{1}{2}}_\varepsilon(t,\cdot)\|_{L^2(\Omega)},
\end{equation}
where $K^{J,\nicefrac{1}{2}}_\varepsilon\!
=\widetilde{\curl} \,V^{J,\nicefrac{1}{2}}_\varepsilon\!
+L^1[\widetilde{\pa} X^{J,\nicefrac{1}{2}}_\varepsilon]$ is given by
\eqref{eq:curlhalf}, $D^{J,\nicefrac{1}{2}}_\varepsilon\!
=e_1\fdh \sm T^J h+ \fdh \sm\,  \widetilde{\div} (T^J\! x)$ is given by
\eqref{eq:divhalf}.
Further, let
\begin{equation}
  \label{halfextranormcontrol}
X^{1,r,\nicefrac{1}{2}}_{\varepsilon,2}(t)
={\sum}_{|J|\leq r}\|\widetilde{\pa} X^{J,\nicefrac{1}{2}}_\varepsilon(t,\cdot)\|_{L^2(\Omega)},\quad\text{and}\quad
H^{r,\nicefrac{1}{2}}_{2}(t)\!=\!{\sum}_{|K|\leq r}\|\widetilde{\pa} \fdh T^K\widetilde{\pa} h(t,\cdot)\|_{L^2(\Omega)},
\end{equation}
and
\begin{equation}
X^{\boldsymbol{\times},r,\nicefrac{1}{2}}_{\varepsilon,2}(t)
={\sum}_{|J|\leq r}\|\widetilde{\curl} \, X^{J,\nicefrac{1}{2}}_\varepsilon(t,\cdot)\|_{L^2(\Omega)},\quad\text{and}\quad
X^{\bigfour\cdot,r,\nicefrac{1}{2}}_{\varepsilon,2}(t)
={\sum}_{|J|\leq r}\|\widetilde{\div} X^{J,\nicefrac{1}{2}}_\varepsilon(t,\cdot)\|_{L^2(\Omega)}.
\end{equation}
By \eqref{eq:curlhalfest}, \eqref{eq:curlequation2} and \eqref{eq:Divhalfder} we have
\begin{align}
 K^{r,\nicefrac{1}{2}\,\, \prime}_{\varepsilon,2}(t)
 &\lesssim C_r\big(H^{r-1,\nicefrac{1}{2}}_{2}(t)+X^{1,r,\nicefrac{1}{2}}_{\varepsilon,2}(t)\big),
 \label{eq:evolutionhalf1}\\
 X^{\boldsymbol{\times},r,\nicefrac{1}{2}\,\, \prime}_{\varepsilon,2}(t)
 &\lesssim C_r\big(K^{r,\nicefrac{1}{2}}_{\varepsilon,2}(t)+X^{1,r,\nicefrac{1}{2}}_{\varepsilon,2}(t)\big),
 \label{eq:evolutionhalf2}\\
 D^{r,\nicefrac{1}{2}\,\, \prime}_{\varepsilon,2}(t)
 &\lesssim C_r\big(X^{1,r,\nicefrac{1}{2}}_{\varepsilon,2}(t)+V^{1,r}_{2}(t)
 +X^{1,r}_{2}(t)\big).\label{eq:evolutionhalf3}
\end{align}
By \eqref{eq:divhalfest}, \eqref{eq:coordinatehalfest} and \eqref{eq:byenergyestimated} we have
\begin{align}
 X^{\bigfour\cdot,r,\nicefrac{1}{2}}_{\varepsilon,2}(t)
 &\lesssim C_r\big(D^{r,\nicefrac{1}{2}}_{\varepsilon,2}(t)+W^{r}_{2}(t)+V^{1,r}_{2}(t)
 +X^{1,r}_{2}(t)\big),\label{eq:boundshalf1}\\
X^{1,r,\nicefrac{1}{2}}_{\varepsilon,2}(t)+B^{\,r+1}_2(t)
&\lesssim C_r\big( X^{\boldsymbol{\times},r,\nicefrac{1}{2}}_{\varepsilon,2}(t)
+X^{\bigfour\cdot,r,\nicefrac{1}{2}}_{\varepsilon,2}(t)+B^{\,r+1}_{\mathcal{N},2}(t)
+X^{1,r}_{\varepsilon,2}(t)\big),\label{eq:boundshalf2}\\
  V^{\,r+1}_{2}(t)+B^{\,r+1}_{\mathcal{N},2}(t)&\lesssim C_0 E^{\,r+1}_2(t), \label{eq:boundshalf3}
\end{align}
where
\begin{equation}\label{eq:energyrdef}
E^{\,r}_2(t)={\sum}_{|I|\leq r}\sqrt{\mathcal{E}^I(t)},
\qquad
B^{\,r}_2(t)={\sum}_{|I|\leq r}\sqrt{\mathcal{B}^I(t)},
\qquad
B^{\,r}_{\mathcal{N},2}(t)={\sum}_{|I|\leq r}\sqrt{\mathcal{B}_{\mathcal{N}}^I(t)},
\end{equation}
and
 $\mathcal{E}^I(t)$ given by \eqref{eq:energydef} and $\mathcal{B}^I(t)$, $\mathcal{B}_{\mathcal{N}}^I(t)$
 are given by \eqref{eq:energyboundarydef}.
 The evolution equations \eqref{eq:evolutionhalf1},  \eqref{eq:evolutionhalf2}
 and  \eqref{eq:evolutionhalf3} with the bounds
  \eqref{eq:boundshalf1}, \eqref{eq:boundshalf2} and \eqref{eq:boundshalf3}
  together with the energy estimates for $E^{1+r}_2$, $W^{r}_2$ and $W^{r,1}_2$
  form a closed system.

\subsection{The $L^\infty$ estimates for lower derivatives}
\label{linftylowersec}
In the above we have assumed that we have control of the $L^\infty$ norms of lower derivatives that we will now prove assuming control of the $L^2$ norms for $0\leq t\leq T$.
 First by Sobolev's Lemma on the sphere and in the radial direction
\begin{align}
\|T^I V(t,\cdot)\|_{L^\infty}&\lesssim {\sum}_{|L|\leq 2}\| \widetilde{\pa} T^{I+L} V(t,\cdot)\|_{L^2}+\| T^{I+L} V(t,\cdot)\|_{L^2},\\
\|T^J \widetilde{\pa} h(t,\cdot)\|_{L^\infty}&\lesssim {\sum}_{|L|\leq 2}\| \widetilde{\pa} T^{J+L}  \widetilde{\pa} h(t,\cdot)\|_{L^2}+\| T^{J+L} \widetilde{\pa} h(t,\cdot)\|_{L^2},\\
\|T^J h(t,\cdot)\|_{L^\infty}&\lesssim {\sum}_{|L|\leq 2}\| \widetilde{\pa} T^{J+L}  h(t,\cdot)\|_{L^2}+\| T^{J+L}  h(t,\cdot)\|_{L^2}.
\end{align}
We now want to have bounds also for the $L^\infty$ norm of
$\widetilde{\pa} T^J V$.
The idea is now that in addition to the above bounds of the tangential derivatives, we have point wise equations for the divergence and the curl of $T^J V$ and $T^K\widetilde{\pa} h$ , so we can use the point wise elliptic estimate to get bounds for $\widetilde{\pa}T^J V$. These point wise bounds depends on point wise bounds on $K^J$, the modified curl of $T^J V$ and $\widetilde{\pa}T^J x$, for which we have point wise evolution equations,
and lower order terms that can be controlled inductively. More precisely,
by the estimates above we control $W^s_\infty(t)$ and $E^{1+s}_\infty(t)$,
for $s\leq r-3$. Moreover by \eqref{eq:lowestcurldivcoord}-\eqref{eq:lowestelliptich}
we see that there is a time $0<T_0\leq T$ depending only on $\|\curl{V}(0,\cdot)\|_{L^\infty(\Omega)}$, $\| \pa_y x(0,\cdot)\|_{L^\infty(\Omega)}$,
and a bound for $\|D_t h(t,\cdot)\|_{L^\infty(\Omega)}$
and $\sum_{T\in {\mathcal S}}\|T V(t,\cdot)\|_{L^\infty(\Omega)}$, for $0\leq t\leq T_0$, such that
\begin{equation}\label{lowbounds}
\|{}_{\!}\curl V(t,\cdot)\|_{L^\infty(\Omega)}\!\lesssim 2 \|{}_{\!}\curl V(0,\cdot)\|_{L^\infty(\Omega)},\qquad
\|\pa_y x(t,\cdot)\|_{L^\infty(\Omega)\!}\lesssim 2 \| \pa_y x(0,\cdot)\|_{L^\infty(\Omega)},\quad 0\!\leq \! t\!\leq \! T_0.
\end{equation}
Moreover for $t\leq T_0$ we have
\begin{equation}
\|\pa V(t,\cdot)\|_{L^\infty(\Omega)}\!
\lesssim \|\curl_{\!} V(0,\cdot)\|_{L^\infty(\Omega)}+\|\pa_y x(0,\cdot)\|_{L^\infty(\Omega)}+{\sum}_{S\in {\mathcal S}}\!\|S V(t,\cdot)\|_{L^\infty(\Omega)}
+\|D_t h(t,\cdot)\|_{L^\infty(\Omega)}.
\end{equation}
In other words we have a bound for $V_\infty^{1,0}$. Inductively, assuming that we have a bound for $V_\infty^{1,s-1}(t)$ and $E^{1+s}_\infty(t)$, $W^{s}_\infty(t)$ for $0\leq t\leq T_{s-1}$, we can therefore solve the system
\begin{align}
| K^{s\, \,\prime}_\infty(t)|&\lesssim c_s\big( X^{1,s}_\infty(t)+ V_\infty^{1,s-1}(t)+W^s_\infty(t)\big),
\label{eq:evolutionsystem.aa}\\
|{X}^{1,s\,\, \prime}_\infty(t)|&\lesssim c_s\big(K^s_\infty(t)+ X^{1,s}_\infty(t)+E^{1+s}_\infty(t)+W^{s}_\infty(t)\big),\label{eq:evolutionsystem.bb}
\end{align}
where $C_s$ depends on bounds of these quantities for smaller $s$, to get
that there is a $0<T_s\leq T_{s-1}$, depending only on a bound for $C_s$ and
for $V_\infty^{1,s-1}(t)$, $E^{1+s}_\infty(t)$, $W^{s}_\infty(t)$ for $0\!\leq\! t\!\leq \!T_{s-1}$,
 such that
\begin{equation}
K^s_\infty(t)\leq 2 K^s_\infty(0),\qquad X^{1,s}_\infty(t)\leq 2X^{1,s}_\infty(0),\qquad
0\leq t\leq T_s.
\end{equation}
Hence, we now get a bound also for
\begin{align}
V^{1,s}_\infty(t)&\lesssim C_s\big(K^s_\infty(0)+ X^{1,s}_\infty(0)+E^{1+s}_\infty(t)+W^{s}_\infty(t)\big),\label{eq:ellipticsystem.a2}\\
H^{s-1}_\infty(t)&\lesssim C_s\big({X}^{1,s-1}_\infty(0)+{V}_\infty^{1,s-1}(0)+W^{s}_\infty(t)\big),
\label{eq:ellipticsystem.b2}
\end{align}
which concludes the induction, and the $L^\infty$ bounds for lower derivatives.

\subsubsection{Lowest order $L^\infty$ estimates for a normal derivative of the divergence}
The lower order term introduced in $F_i^{\prime I}$ in \eqref{eq:symmetricequation} require a bound for $\widetilde{\pa}^2 h$, which we have above and the
lower order term introduced in $G^{\prime I}$ in \eqref{eq:highercontinuity} requires a bound for $\widetilde{\pa}\widetilde{\div} V.$
Since $\widetilde{\pa} \widetilde{\div} V=e_1 \widetilde{\pa} D_t h $ and we have
bounds for $\pave D_t h$, we have bounds for this quantity
as well.

\subsection{Control of the $L^2$ norms}\label{controlL2}
 In addition to the evolution equation for the
$L^2$ norms of the curl of the velocity and of the coordinate
\begin{align}
| K^{r\, \prime}_2(t)|&\lesssim C_r\big( X^{1,r}_2(t)+ V_2^{1,r-1}(t)+W^r_2(t)\big),
\label{eq:evolutionsystem2.a}\\
|{X}^{1,r\, \prime}_2(t)|&\lesssim C_r\big(K^r_2(t)+ X^{1,r}_2(t)+V^{1+r}_2(t)+W^{r}_2(t)\big),\label{eq:evolutionsystem2.b}
\end{align}
together with
\begin{align}
V^{1,r}_2(t)&\lesssim C_r\big(K^r_2(t)+ X^{1,r}_2(t)+V^{1+r}_2(t)+W^{r}_2(t)\big),\label{eq:ellipticsystem2.a}\\
H^{r-1}_2(t)&\lesssim C_r\big({X}^{1,r-1}_2(t)+{V}_2^{1,r-1}(t)+W^{r}_2(t)\big),
\label{eq:ellipticsystem2.b}
\end{align}
we also need evolution equations for $W_2^r$ and $V^{r+1}_2$.
Moreover by
\eqref{eq:waveequationenergy}
\begin{equation}\label{eq:energyestiamtewaver}
|W^{r\, \prime}_2(t)|\lesssim C_r \big(K^r_2(t)+ X^{1,r}_2(t)+V^{r+1}_2(t)+W^{r}_2(t) \big).
\end{equation}
With notation as in \eqref{eq:energyrdef} we have
\begin{equation}\label{eq:velocitybyenergy}
V_2^{\,r+1}(t)+B_{\mathcal{N},2}^{\,r+1}(t)\lesssim C_0 {E}_2^{\,r+1}(t),
\end{equation}
so it only remains to get an evolution equation for the energy $E_2^{\, r}(t)$.
This is much easier for Euler's equations than for the smoothed Euler's equation so we will start with the simple case:

\subsubsection{Control of the $L^2$ norms for Euler's equations}
By \eqref{eq:energyestiamteeuler} we have with notation as in \eqref{eq:energyrdef}
\begin{equation}\label{eq:energyestiamteeulerr}
|E_2^{\,r+1\,{}_{\,}\prime}(t)|
\lesssim  C_0 {E}_2^{\,1+r\!}(t)+C_r\bigtwo(K_2^{r}(t)+X_2^{1,r\!}(t)+W_2^r(t) \bigtwo),\quad\text{if}\quad \varepsilon=0,
\end{equation}
which provided the missing equation. Using the bounds \eqref{eq:ellipticsystem2.a}, \eqref{eq:ellipticsystem2.b} and \eqref{eq:velocitybyenergy}, the evolution equations
\eqref{eq:evolutionsystem2.a}, \eqref{eq:evolutionsystem2.b},
\eqref{eq:energyestiamtewaver} and \eqref{eq:energyestiamteeulerr} form a closed system so we conclude that there is a $T_r>0$ such that for $0\leq t\leq T_r$ we have
\begin{equation}
K_2^r(t)\leq 2K_2^r(0),\qquad X_2^{1,r}(t)\leq 2X_2^{1,r}(0),\qquad
W_2^r(t)\leq 2W_2^r(0),\qquad E_2^{r+1}(t)\leq 2E_2^{r+1}(0).
\end{equation}
Since a bound for $V_2^{1,r}$ and $H_2^{r-1}$ follow from these this concludes the proof of the apriori bound for the compressible Euler's equations.

\subsubsection{Control of the $L^2$ norms for the smoothed Euler's equations}
\label{controll2}
By \eqref{eq:energyestimate}
\begin{equation}
|{E}_2^{\,r+1\,{}_{\,}\prime}(t)|
\lesssim  C_0 {E}_2^{\,r+1}(t)+C_0{B}_2^{\,r+1}(t)+C_r\bigtwo(K_2^{r}(t)+X_2^{1,r\!}(t)+W_2^r(t) \bigtwo) .
\label{E2smooth}
\end{equation}
We are missing an estimate for ${B}_2^{\,r+1}(t)$ that we will get from the extra half derivative estimates for the coordinates using that the normal component $B_{\mathcal{N},2}^{\,r+1}(t)$ is bounded by the energy ${E}_2^{\,r+1}(t)$.
To get this to form a closed system we have to add the evolution equations \eqref{eq:evolutionhalf1},  \eqref{eq:evolutionhalf2}
 and  \eqref{eq:evolutionhalf3} with the bounds
  \eqref{eq:boundshalf1}, \eqref{eq:boundshalf2} and \eqref{eq:boundshalf3}
  together with the energy estimate for $E^{r+1}_2$ above and a bound for
 $H^{r-1,\nicefrac{1}{2}}_{2}$ that is needed in \eqref{eq:evolutionhalf1}.
That bound however requires a higher order energy time derivative estimate for the wave equation. With $\mathcal{W}^{J,s{}_{\!}}$ as in \eqref{eq:timederivativeswavenergydef} let
\begin{equation}
 W^{r,s}_2(t)={\sum}_{|J|\leq r}\mathcal{W}^{J,s{}_{\!}}(t).
\end{equation}
By \eqref{eq:waveequationenergytime} and \eqref{eq:waveequationenergytimeJ} we have
\begin{align}
|W_2^{r-1,2\,\prime}(t)|
&\lesssim C_0 W_2^{r-1,2}(t)+C_r\big(V_2^{1,r}+X_2^{1,r}+W_2^r\big),\\
|W_2^{r,1}(t)|
&\lesssim C_0 W_2^{r-1,2}(t)+C_r\big(V_2^{1,r}+X_2^{1,r}+W_2^r\big),
\end{align}
and by \eqref{eq:pahalfTKpa} we have
\begin{equation}
H^{r-1,\nicefrac{1}{2}}_{2}(t)
\lesssim C_r W^{r,1}_2(t)
+  \widetilde{X}^{1,r,\nicefrac{1}{2}}_2
+ C_r\big( V^{1,r}_2(t)+ X^{1,r}_2(t)\big).
\end{equation}
The evolution equations for the quantities $K_2^r$, $X_2^{1,r}\!\!$, $W_2^r$,
$E_2^{r+1}\!\!$, together with those for $ K^{r,\nicefrac{1}{2}}_{\varepsilon,2}\!\! $, $X^{\boldsymbol{\times},r,\nicefrac{1}{2}}_{\varepsilon,2}\!\!$, $D^{r,\nicefrac{1}{2}}_{\varepsilon,2}\!$ and
$W_2^{r\shortminus 1,2}$ form a closed system if we also use the bounds for $V_2^{1,r}\!\!$, $V_2^{r+1}\!\!$, $B_{\mathcal{N},2}^{r+1}$, $X^{1,r,\nicefrac{1}{2}}_{\varepsilon,2}\!\!$, $W_2^{r,1}$ and $H_2^{r\shortminus 1,\nicefrac{1}{2}}$ in terms of these quantities.
We conclude that there is a $T_r\!>\!0$ such that for $0\!\leq \!t\!\leq\! T_r$
\begin{equation}
K_2^r(t)\leq 2K_2^r(0),\qquad X_2^{1,r}(t)\leq 2X_2^{1,r}(0),\qquad
W_2^r(t)\leq 2W_2^r(0),\qquad E_2^{r+1}(t)\leq 2E_2^{r+1}(0),
\label{l2uniformbd}
\end{equation}
and
\begin{equation}
K^{r,\nicefrac{1}{2}}_{\varepsilon,2}(t)\!\leq \! 2K^{r,\nicefrac{1}{2}}_{\varepsilon,2}(0),\!\quad X^{\boldsymbol{\times},r,\nicefrac{1}{2}}_{\varepsilon,2}(t)\!\leq \! 2X^{\boldsymbol{\times},r,\nicefrac{1}{2}}_{\varepsilon,2}(0),\!\quad
D^{r,\nicefrac{1}{2}}_{\varepsilon,2}(t)\!\leq \! 2D^{r,\nicefrac{1}{2}}_{\varepsilon,2}(0),\!\quad W_2^{r\shortminus 1,2}(t)\!\leq\! 2W_2^{r\shortminus 1,2}(0),
\end{equation}
and the other quantities can be bound in terms of these.
This concludes the proof of the uniform apriori bounds for the smoothed Euler's equations.

\section{Uniform apriori bounds for the smoothed problem in
the relativistic case}
\label{rescaledreleulex}

We now return to the relativistic Euler equations \eqref{rescaledreleul}-
\eqref{rescaledrelcont}.
The proof of the energy estimates for this system uses the same
strategy as the proof of Theorem \ref{nonrelexist}.
The basic ingredients are energy estimates for an appropriate
smoothed-out version of the Euler equations which
control tangential derivatives, elliptic estimates which allow one to control all derivatives
in terms of the divergence, curl, and tangential derivatives, and
estimates for the wave equation satisfied by the enthalpy.

\subsubsection{Lagrangian coordinates}
\label{rellagrange}
Let $\D$ denote the closure of the set $\{ \rho(t,x) > 0\}$.
The Lagrangian coordinates are maps $x^\mu: [0, S]\times \Omega \to \D,
\mu = 0,1,2,3$ where $x^0=t$, defined by
\begin{equation}
 \frac{d}{ds} x^\mu(s, y) = v^\mu(x(s, y)), \qquad \mu = 0,1,2,3,
 \qquad
 x^0(0, y) = 0, \quad
 x^i(0, y) = y^i, \quad i = 1,2,3.
 \label{lagcoordsdef}
\end{equation}
We will write $\D_s = x(s, \Omega)$.
We also introduce the material derivative
$$
\hD_s = \frac{d}{ds}\big|_{y = const} =
 v^\mu \pa_\mu,
 $$
 and write $V(s,y)=v\big(x(s,y)\big)$. The relativistic Euler equations
\eqref{rescaledreleul} become
\begin{equation}
 \hD_s V^\mu  + \frac{1}{2}g^{\mu\nu}\pa_\nu \sigma =
 \Gamma_{\alpha\nu}^\mu V^{\alpha} V^{\nu}
 \qquad \text{ in } [0,s_1]\times \Omega, \quad
 \pa_\mu =\frac{\pa y^a}{\pa x^\mu} \frac{\pa }{\pa y^a},
\end{equation}
where we think of $\Gamma_{\mu\nu}^\alpha(x(s,y))$ as given functions of $y$.
Here we are summing over $a = 0,1,2,3$ and writing $y^0 = s$.
The continuity equation is
\begin{equation}
 \hD_s e(\sigma) + \nabla_\mu V^\mu = 0,
 \qquad \text{ where } e(\sigma) = \log(\rho(\sigma)/\sqrt{\sigma}).
\end{equation}

We are going to prove a local existence theorem in Lagrangian coordinates which
is analogous to Theorem \ref{lagthmnonrel}.
Let $\Omega\subset \M_0$ denote the unit ball.
We will assume that the metric $g$ satisfies the bound \eqref{metricbdsintro0}.

\begin{theorem}[Local existence for the relativistic problem
  in Lagrangian coordinates]
  \label{lwprellag}
  Fix $r \geq 10$ and a globally hyperbolic metric $g$ satisfying \eqref{metricbdsintro0}
  for some $G > 0$. Let
  $\mathring{V}, \sigma_0$ be initial data satisfying the compatibility conditions
  \eqref{relcompat} to order $r$, where $\mathring{V}$ is a
  timelike vector field satisfying $g(\mathring{V}, \mathring{V}) = - \mathring{\sigma}
  \leq -c_1 < 0$ for some constant $c_1$,
  and so that $E_0^r\! =\! \|\mathring{V}\|_{H^r(\Omega)}^2\!
  +\! \|\sigma\|_{H^r(\Omega)}^2 \!<\!\infty$. Suppose additionally that the
  Taylor sign condition $|\nabla \mathring{\sigma}|\! \geq c\! >\! 0$ holds
  on $\pa \Omega$ for some $c$ and that the sound speed
  \eqref{soundspeed} is such that \eqref{soundspeedbd}
  -\eqref{largesoundspeed} hold for $\delta$ sufficiently small.
  is sufficiently large.
  Then there is a continuous function $S \!= \! S(\mathcal{E}_0, G, 1/c) \!>\! 0$
 so that the following hold.

  For any $S' < S$,
  there are Lagrangian coordinates $x:[0,S']\times \Omega\to \mathcal{M}$
  and an enthalpy $\sigma : [0,S']\times \Omega \to \mathcal{M}$ so that with
  $\mathcal{D}_s \!=\! x(s, \Omega)$ and
  $V(s, y)\! =\! \frac{d}{ds} x(s, y)$, and $v(x(s,y))\! =\! V(s,y)$,
  the surfaces $\mathcal{D}_s$
  are spacelike and the equations \eqref{rescaledreleul}-\eqref{rescaledrelcont}
  hold on the domain $\mathcal{D} \!=\!
  \cup_{0 \leq s \leq S'} \{s\}\!\times\!\mathcal{D}_s$. Moreover,
  the following bounds hold
  \begin{equation}
    \sup_{0 \leq s \leq S'}
    \sum_{k\leq r}\!
    \int_\Omega\!
  |\pa^k  V(s)|^2\!+
    |\pa^k \!\sigma(s)|^2\!
    + |\fdh \pa T^J \! x(s)|^2 \kappa dy
    + \!\sum_{k \leq r}\!\int_{\pa \Omega}\!\!\!
    |\pa^k x(s)|^2dS
    \leq
    C(\mathcal{E}_0, S'\!\!, \sigma_1, c, \mathcal{G}_{r+2}).
  \end{equation}
\end{theorem}
In the above, the fractional tangential derivative $\fdh$ is defined
in section \ref{def T and FD}.
This does not quite imply our main result Theorem \ref{mainrelthm}, because this
result only gives a solution up to a surface of constant $s$
but the main theorem is stated in terms of a surface of
constant time $t$. This is because we construct our solution in Lagrangian
coordinates where it is more natural to work with the surfaces of constant
$s$. Turning this into a result which follows solutions up to a surface of
constant $t$ requires only minor modifications, see
section \ref{spacelikesec}.

\subsubsection{The set up for the proof in the relativistic case}
We proceed as in the previous section by first writing \eqref{rescaledreleul}
-\eqref{rescaledrelcont} as a wave equation for the enthalpy $\sigma$ coupled
to Euler's equations. We repeat the equations here for the convenience of the reader,
\begin{alignat}{2}
 V^\nu\nabla_\nu V^\mu + \frac{1}{2} \nabla^\mu \sigma &=
 0,
 &&\qquad \text{ in } \D_s, \label{rescaledreleul2}\\
 V^\nu\nabla_\nu e(\sigma) + \nabla_\mu V^\mu &= 0,
 &&\qquad
 \text{ in } \D_s.
 \label{rescaledrelcont2}
\end{alignat}

To get the wave equation for $\sigma$ we apply $\nabla^\mu =
g^{\mu\nu}\nabla_\nu$
to \eqref{rescaledreleul2} and use $\nabla g = 0$, which gives
\begin{equation}
 V^\nu \nabla_\nu \nabla_\mu V^\mu + \frac{1}{2}
 \nabla_\nu(g^{\mu\nu}\nabla_\mu
 \sigma)
 = -\nabla_\mu V^\nu \nabla_\nu V^\mu
 -R_{\mu\nu\alpha}^\mu V^\nu V^\alpha,
 \label{div1}
\end{equation}
where $R_{\mu\nu\alpha}^{\mu'}$ denotes the Riemann curvature
tensor, i.e.
\begin{equation}
 R_{\mu\nu\alpha}^{\mu'}V^{\mu}
 =\big[\na_\nu \na_\alpha-\na_\alpha\na_\nu\big] V^{\mu^\prime},\quad\text{where}\quad   R_{\mu\nu\alpha}^{\mu'}=\pa_\nu\Gamma_{\alpha \mu}^{\mu'}
 - \pa_\alpha\Gamma_{\mu \nu}^{\mu'}
 + \Gamma^{\mu'}_{\nu \beta}\Gamma^\beta_{\alpha \mu}
 - \Gamma^{\mu'}_{ \alpha \beta}\Gamma^\beta_{\nu\mu}.
 \label{}
\end{equation}

 Subtracting \eqref{div1} from $\hD_s = V^\nu\pa_\nu$ applied to \eqref{rescaledrelcont2} we find
\begin{equation}
 e'(\sigma) \hD_s^2 \sigma - \frac{1}{2} \nabla_\nu
 (g^{\mu\nu}\nabla_\mu
 \sigma)
 = \nabla_\mu V^\nu \nabla_\nu V^\mu + Q,
 \label{sigmawavesetup1sec}
\end{equation}
where
\begin{equation}
 Q = R_{\mu\nu\alpha}^\mu V^\nu V^\alpha
 - e''(\sigma) (\hD_s\sigma)^2.
 \label{Qdef}
\end{equation}
When $e'(\sigma) \equiv 0$ then this is just a wave equation
with respect to the metric $g$.

\subsubsection{The smoothed problem in the relativistic case}
\label{sec:smreldef}
Let $\!S_\varepsilon^* S_\varepsilon$ be
a regularization as in section \ref{smoothingsec}.
Given a velocity vector field $V(s,y)\!$, we
define the tangentially regularized velocity
\begin{equation}
\widetilde{V}^\mu = S_\ve^* S_\ve V^\mu,
\end{equation}
and coordinates $\xve$ by
\begin{equation}\label{eq:eulerlagrangiancoordxsmoothedrel}
\frac{d \widetilde{x}^\mu(s, y)}{ds}=
\widetilde{V}^\mu(s, y),\qquad \xve^0(0,y) = 0, \quad
 \xve^i(0,y)=x_0^i(y),\qquad y\in \Omega.
\end{equation}
We want $\sigma$ and $V_\mu$ to be
functions of $(s,y)\in [0,S]\times \Omega$, because we need to be in a fixed domain in order to construct a solution by iteration. However, we also like to be able to think of them as functions of $(t,\widetilde{x})$ because the formulation of the equations becomes simpler that way.
We define operators $\widetilde{\pa}, \hD_s$ on $[0,S]\times \Omega$ by
\begin{equation}
  \widetilde{\pa}_\mu
  = \frac{\pa y^\alpha}{\pa \widetilde{x}^\mu} \frac{\pa}{\pa y^\alpha},
  \qquad \hD_s = \frac{\pa}{\pa s}\Big|_{y = \text{const}}
  = \widetilde{V}^\mu \widetilde{\pa}_\mu.
\end{equation}
Note the operators $\widetilde{\pa}_\mu$ in the $y$ coordinates correspond to partial differentiation $\partial /\partial\widetilde{x}^\mu$ in the $\widetilde{x}$ coordinates.
 For a vector field $X$ we introduce the smoothed-out covariant
derivative
\begin{equation}
 \widetilde{\nabla}_\mu X^\nu = \pave_\mu X^\nu
 + \widetilde{\Gamma}_{\mu\gamma}^\nu X^\gamma, \qquad \text{where} \quad
 \widetilde{\Gamma}_{\mu\gamma}^\nu (s,y)=\Gamma_{\mu\gamma}^\nu\big(\widetilde{x}(s,y)\big)
 ,
 \label{smcov}
\end{equation}
whereas for functions $\widetilde{\nabla}_\mu f=\widetilde{\partial}_\mu f$. Note the operators $\widetilde{\na}_\mu$ in the $y$ coordinates correspond to covariant differentiation in the $\widetilde{x}$ coordinates with respect to the metric $ g(\xve)$. Hence
\begin{equation}
 \big[\widetilde{\nabla}_\mu \widetilde{\nabla}_\nu-\widetilde{\nabla}_\nu\widetilde{\nabla}_\mu\big] V^{\mu^\prime}= \widetilde{R}_{\mu\nu\alpha}^{\mu'}V_{\mu^\prime}
 ,\quad\text{where}\quad   \widetilde{R}_{\mu\nu\alpha}^{\mu'}(s,y)
 = {R}_{\mu\nu\alpha}^{\mu'}\big(\widetilde{x}(s,y)\big).
\end{equation}
With $\widetilde{g}(s,y)=g\big(\widetilde{x}(s,y)\big)$ we also
let $\widetilde{\nabla}^\mu =
\widetilde{g}^{\mu\nu}\widetilde{\nabla}_\nu$.

The smoothed-out equations that we consider are
\begin{alignat}{2}
 \widetilde{V}^\nu\widetilde{\na}_\nu V^\mu + \frac{1}{2}
 \widetilde{\nabla}^\mu\sigma &=
 0,
 &&\qquad \text{ in } \Omega, \label{rescaledreleul2smoothed}\\
 \widetilde{V}^\nu\widetilde{\nabla}_\nu e(\sigma) + \widetilde{\na}_\mu V^\mu &= 0,
 &&\qquad
 \text{ in } \Omega.
 \label{rescaledrelcont2smoothed}
\end{alignat}
As in the previous section if we apply $\widetilde{\nabla}^\mu $ to
\eqref{rescaledreleul2smoothed} and subtract the result
 from $\hD_s = \widetilde{V}^\nu\pa_\nu$ applied to \eqref{rescaledrelcont2smoothed} we find
\begin{equation}
 e'(\sigma) \hD_s^2 \sigma - \frac{1}{2} \widetilde{\nabla}_\nu
 (\widetilde{g}^{\mu\nu}\widetilde{\nabla}_\mu
 \sigma)
 = \widetilde{\nabla}_\mu \widetilde{V}^\nu \widetilde{\nabla}_\nu V^\mu + \widetilde{R}_{\mu\nu\alpha}^\mu \widetilde{V}^\nu V^\alpha
 - e''(\sigma) (\hD_s\sigma)^2.
\end{equation}
 depending linearly on $V$ and
subject to the boundary and initial conditions
 \begin{alignat}{2}
 \sigma &= \overline{\sigma}, &&\qquad \text{ on } [0,s_1]\times \pa\Omega,\label{smoothwavebc}\\
 \sigma|_{s = 0} &= \sigma_0,&&\qquad \text{ on }
 \Omega,\\
 \hD_s \sigma|_{s = 0} &= \sigma_1,&&\qquad \text{ on }
 \Omega.
 \label{smoothedwaveic}
\end{alignat}
Here, $\overline{\sigma} = \sigma|_{p = 0}$ is a constant.

\subsection{Norms and basic geometric constructions}
\label{spatialnorms}
We now introduce
some basic geometric quantities which we will use to
 control the solution.
 \subsubsection{Norms of spacetime quantities}
 \label{normsdefs}
It is convenient to introduce the following norms of spacetime
quantities. We let $H$ be the Riemannian metric
\begin{equation}
 H_{\mu\nu} = g_{\mu\nu} + 2\Tau_\mu\Tau_\nu,
 \label{Hdef}
\end{equation}
where $\Tau$ denotes the future-directed timelike covector
determining the time axis of the background metric. Explicitly if
$\tau$ denotes the time function of the background
metric then $\Tau_\mu
=\pa_\mu \tau/(-g(\nabla \tau, \nabla \tau))^{1/2}$.
The fact that \eqref{Hdef}
is positive-definite follows after decomposing into the directions
parallel to ${\Tau}$ and orthogonal to ${\Tau}$ and
noting that ${\Tau}$ is timelike so its orthogonal complement is
spacelike.

For a tensor field $\beta \! =\! \beta_{\mu_1\cdots \mu_k}
dx^{\mu_1}\!\cdots dx^{\mu_k}$\!\!,
 we write $|\beta|$ for the pointwise norm with respect to $H$\!,
 \begin{equation}
  |\beta|^2 = H^{\mu_1\nu_1}\cdots
  H^{\mu_k \nu_k}
  \beta_{\mu_1\cdots \mu_k}
  \beta_{\nu_1\cdots \nu_k}.
  \label{Hnormdef}
 \end{equation}
 For $1 \leq p < \infty$, thinking of the coefficients of $\beta$
 as depending on $(s,y)$, we write
 \begin{equation}
  \| \beta\|_{L^p(\Omega)}^p
  = \int_{\Omega} |\beta|^p\, \kappa dy,
  \qquad
  \|\beta\|_{L^\infty(\Omega)} =
  {\sup}_{y \in \Omega} |\beta(y)|.
  \label{HLpnormdef}
 \end{equation}

In later sections we will abuse this notation slightly and apply it to quantities
of the type $T^I V^\mu$
 or $T^I\wGamma_{\mu\nu}^\gamma$ which are not tensor fields
 since they do not transform the correct way under changes of
 coordinates. For terms of this
 type we will abuse notation and write e.g.
 \begin{equation}
  |T^I V|^2 = H_{\mu\nu} T^I V^\mu T^I V^\nu,
  \qquad
  |T^I\Gamma|^2=
  H^{\mu \mu'}H^{\nu\nu'} H_{\alpha\alpha'}
  T^I\Gamma^{\alpha'}_{\mu'\nu'}.
  \label{}
 \end{equation}
 Then these quantities are not invariant under coordinate changes but changing
 coordinates just generates lower-order terms.

\subsubsection{The Riemannian metric on $\Omega$}

We now introduce a Riemannian metric $G$ on the surfaces $\Omega_s =
\xve(s,\Omega)$ which
plays an important role in what follows. The idea is that we want to
write the wave operator $\wg^{\mu\nu}\pave_\mu \pave_\nu$ as
the sum of a second-order operator which is elliptic on $\Omega_s$ and two material
derivatives $D_s$, which by \eqref{freebdy} is tangent to the boundary.
The following
construction works on an arbitrary spacelike surface $\Sigma$ and has
nothing to do with Lagrangian coordinates so we will do it abstractly.
Let $n^\Sigma$ denote the timelike future-directed normal vector field
 to $\Sigma$.

Decompose the tangent space into a component along $n^\Sigma$
and a part orthogonal to $n^\Sigma$,
\begin{equation}
 g(X, Y) = -g(n^\Sigma, X) g(n^\Sigma, Y) + \overline{g}(X,Y),
 \label{ogdef}
\end{equation}
where $\overline{g}$ is the projection of the metric away from
$n^{\Sigma}$. It is non-negative and in fact positive-definite on
the tangent space of the spacelike surface $\Sigma$.
Decomposing $\widetilde{V}$ in the same way we find
\begin{equation}
 n^\Sigma = -\frac{1}{g(n^\Sigma, \widetilde{V})} (\widetilde{V} - \overline{\widetilde{V}}).
 \label{nSigma}
\end{equation}
Note that since $n^\Sigma\!\!, \widetilde{V}$ are both timelike,
$g(n^\Sigma\!\!, \widetilde{V})\! \not=\!0$\footnote{Otherwise we would have two orthogonal
timelike directions which is impossible since our spacetime is hyperbolic}.
Combining these formulas we have the following decomposition of $g$,
\begin{equation}
 g(X, Y) =G(X,Y) - \frac{1}{g(n^\Sigma, \widetilde{V})^2} g(\widetilde{V}, X) g(\widetilde{V}, Y)
 +\frac{1}{g(n^\Sigma, \widetilde{V})^2} \left(g(\widetilde{V}, X) g(\overline{\widetilde{V}}, Y)
 + g(\widetilde{V}, Y) g(\overline{\widetilde{V}}, X)\right),
 \label{Gdecomp}
\end{equation}
with
\begin{equation}
  G(X, Y) =  \overline{g}(X, Y) -\frac{1}{g(n^\Sigma, \widetilde{V})^2} g(\overline{\widetilde{V}}, X)
  g(\overline{\widetilde{V}}, Y),
  \label{Gdef}
\end{equation}
which is positive-definite when restricted to $T\Sigma$;
since $g(\overline{Y}, X) = g(\overline{Y}, \overline{X}) =
\overline{g}(Y, X)$,
\begin{equation}
 G(X, X) = \overline{g}(X, X) - \frac{g(\overline{\widetilde{V}}, X)}{g(n^\Sigma, \widetilde{V})^2}
 \geq \left(1 - \frac{g(\overline{\widetilde{V}}, \overline{\widetilde{V}})}{g(n^\Sigma, \widetilde{V})^2}\right)
 \overline{g}(X,X)
 = -\frac{g(\widetilde{V}, \widetilde{V})}{g(n^\Sigma, \widetilde{V})^2}
 \overline{g}(X,X).
 \label{Gpositivity0}
\end{equation}
Since $\widetilde{V}$ is timelike, the coefficient here is positive.

From the formula \eqref{Gdecomp} one sees that the
principal part of the wave operator $\wg^{\mu\nu}\pave_\mu\pave_\nu$ is
\begin{equation}
  \frac{1}{g(n^\Sigma, \widetilde{V})^2} D_s^2 + G^{\mu\nu}\pave_\mu \pave_\nu
  +\frac{2}{g(n^\Sigma, \widetilde{V})^2}  \overline{\widetilde{V}}^\nu \pave_\nu D_s,
  \label{waveexpression}
\end{equation}
where $G^{\mu\nu}\pave_\mu\pave_\nu$ is elliptic, thought of as an operator
on $\Sigma$. The important point in this
decomposition is that $D_s$ is tranverse to $\Sigma$ and will be tangent
to $\pa\Sigma$ in our applications.

The above decomposition also gives the following formula for the principal part of the divergence
\begin{equation}
 \wg^{\mu\nu}\pave_\mu X_\nu =
 G^{\mu\nu} \pave_\mu X_\nu
 +\frac{1}{g(n^\Sigma,\widetilde{V})^2}\Big(\sg^{\mu\nu} \widetilde{{V}}_\nu
 -\widetilde{{V}}^\mu\Big)
 \hD_s X_\mu
 + \Omega^{\mu\nu}\widetilde{\curl} X_{\mu\nu},
 \label{divexpression}
\end{equation}
where
 $ \Omega^{\mu\nu} =
 \frac{1}{g(n^\Sigma,\widetilde{V})^2}
\left(\widetilde{{V}}^\mu \sg^{\mu'\nu} \widetilde{{V}}_{\mu'}\!\!
+ \widetilde{{V}}^\nu \sg^{\nu'\mu} \widetilde{{V}}_{\nu'}\right)$

\subsubsection{The wave operator expressed in Lagragian coordinates}
\label{lagrangianwavecoords}

We record here an alternate expression for \eqref{waveexpression}
in Lagrangian coordinates.
With $n_\alpha=\pa_\alpha s$ the conormal to the surfaces $s=-const$ we can write
$\pa_\alpha=n_\alpha\,\pa_s +\overline{\pa}_\alpha $,
where $\overline{\pa}_\alpha$ differentiates along the surfaces $s=const$.
Since $\widetilde{V}^\alpha\na_\alpha =\widetilde{V}^\alpha\pa_\alpha s=1$ we have $\overline{\pa}_\alpha =\gamma_\alpha^{\alpha^\prime}\pa_{\alpha^\prime}$,
where $\gamma_\alpha^{\alpha^\prime}=\delta_\alpha^{\alpha^\prime}-n_\alpha \, \widetilde{V}^{\alpha^\prime}$. With $\xi_s=\widetilde{V}^\alpha \xi_\alpha$ and $\overline{\xi}_\alpha=\gamma_\alpha^{\alpha^\prime}\xi_{\alpha^\prime}$
the symbol for the wave operator can hence be decomposed
\begin{equation}
  \label{coordwavedecomp3}
\wg^{\alpha\beta}\xi_\alpha \xi_\beta =
\wg^{\alpha\beta}n_{\alpha} \, n_{\beta}\,\xi_s \xi_s +
2 \wg^{\alpha\beta}n_{\alpha}\, \xi_s\overline{\xi}_{\beta}
+\wg^{\alpha\beta}\overline{\xi}_{\alpha}
\overline{\xi}_{\beta}.
\end{equation}
The principal part that only differentiates along the surface $s=const$ is
\begin{equation}
  \label{coordwavedecomp2a}
\wg^{\alpha\beta}\overline{\xi}_{\alpha}
\overline{\xi}_{\beta}=G_1^{\alpha\beta}{\xi}_{\alpha}
{\xi}_{\beta},\qquad\text{where}\qquad
G_1^{\alpha\beta}=\wg^{\alpha'\beta'}\gamma_{\alpha'}^{\alpha}\gamma_{\beta'}^{\beta},
\end{equation}
We claim that this gives an elliptic operator restricted to the surfaces $s=const$.
i.e. $\wg^{\alpha\beta}\overline{\xi}_{\alpha}
\overline{\xi}_{\beta}>c\delta^{\alpha\beta}\overline{\xi}_{\alpha}
\overline{\xi}_{\beta}$, for some $c>0$.
To see this note that
 $\overline{\xi}^\alpha=\wg^{\alpha\beta}\overline{\xi}_\beta$ is in the orthogonal complement of ${\widetilde{V}}^\beta$, since $\wg_{\alpha\beta}\overline{\xi}^\alpha{\widetilde{V}}^\beta
=\overline{\xi}_\beta{\widetilde{V}}^\beta=0$, since $\widetilde{V}^\alpha\pa_\alpha s=1$.  Since
$\widetilde{V}$ is timelike
$\wg_{\alpha\beta}{\widetilde{V}}^\alpha{\widetilde{V}}^\beta<0$ it follows that
$\overline{\xi}$ is spacelike $\wg_{\alpha\beta}\overline{\xi}^\alpha\overline{\xi}^\beta>0$.

\subsubsection{The divergence theorem}
The following identities are straightforward consequences
of the usual divergence theorem (see Section \ref{divergencetheoremproofs}).
We record them explicitly here for the convenience of the reader.
\begin{lemma}
  \label{divergencelemma}
  Let $\D$ be a region bounded between two spacelike surfaces $\Sigma_{0},
  \Sigma_{1}$ with $\Sigma_{1}$
  lying to the future of $\Sigma_0$, and a timelike surface
  $\Lambda$.
  Let $dS^{\Sigma_j}$ denote the  measure induced by $\wg$
  on $\Sigma_{j}$
  and $d\Lambda$ the induced measure on $\Lambda$.
  Let $n^{\Sigma_j}$ denote the future-oriented normal
  to $\Sigma_{j}$. Then we have
  \begin{equation}
\int_{\D} \div X\, dV
 =
 \int_{\Sigma_{1}} \wg(n^{\Sigma_1}, X) dS^{\Sigma_1}
 -\int_{\Sigma_{0}}  \wg(n^{\Sigma_0}, X) dS^{\Sigma_0}
 + \int_{\Lambda} \wg(\widetilde{\N}, X) dS^\Lambda.
 \label{divtheorem1}
  \end{equation}

 If $\widetilde{V}$ is tangent to $\Lambda$ and
 $\Lambda_{\Sigma_0}^{\Sigma_1}$ denotes the portion of
 $\Lambda$ lying between $\Sigma_0, \Sigma_1$ then
 \begin{equation}
   \int_{\Lambda_{\Sigma_0}^{\Sigma_1}} \hD_s \phi \,dS^\Lambda
   =\int_{\Lambda \cap \Sigma_1}\phi \, \wg(n^{\Sigma_{1}}, \widetilde{V})
   dS'
   -
   \int_{\Lambda \cap \Sigma_0} \phi \,\wg(n^{\Sigma_{0}}, \widetilde{V})
   dS'
   +\int_{\Lambda} \phi\, \div_{\Lambda} \widetilde{V} dS,
  \label{transportbdy}
 \end{equation}
 where $\div_{\Lambda}$ denotes the divergence on $\Lambda$
 and $dS'$ is the measure on $\Lambda_{\Sigma_j}$
 induced by $dS^\Lambda$.
\end{lemma}
We note for later use that
$-\wg(n^{\Sigma_{j}}, \widetilde{V}) > g_0$ for a constant $g_0$
which follows since $\widetilde{V}$ and $n^{\Sigma_j}$ are timelike
and future-directed.

\subsection{Control of the norms from the energies, the divergence and the curl using elliptic estimates}
Let $\mathcal{T}$ denote the set of spacetime vector fields which
are tangential at the space boundary constructed as in
Section \ref{def T and FD}.
As in the non-relativistic case we will derive higher order energies for any combination of tangential vector fields $T^I$ and in
order to control the full gradient of the solution we will
need separate estimates for the antisymmetric part of the gradient
along with the trace.
Since $\sigma$ is constant on $\pa \Omega$ and hence $\widetilde{V}^\mu\widetilde{\pa}_\mu \sigma=0$ on the boundary we use the fact that $\sigma$ satisfies a wave equation
in the interior to get estimates.

In this section we let $c_J, c_r$ denote constants
depending on pointwise norms of lower-order terms.
With notation as in \eqref{Hnormdef}-\eqref{HLpnormdef},
\begin{equation}
 c_J = c_J\bigfour( {\sum}_{|K| \leq |J|/2} |\pave T^K \xve|+
 |\pave T^L V| + |\pave T^L \widetilde{V}|
 +|\pave T^L \pave \sigma|
 + |\pave T^L \wGamma| + |T^L \wGamma| \bigfour),
 \qquad
 c_r = {\sum}_{|J| \leq r} c_J,
 \label{lowercasecdefrel}
\end{equation}
and similarly $C_J$ depends on $L^\infty$ norms of lower-order
terms,
\begin{equation}
 C_{\!J\!} = C_{\!J}\bigfour( {\sum}_{|K| \leq |J|/2} \|\pave T^K \xve\|_{L^\infty}\!+
 \|\pave T^L V\|_{L^\infty}\!
 +\|\pave T^L \widetilde{V}\|_{L^\infty}\!
 +\|\pave T^L \pave \sigma\|_{L^\infty}\!
 + \|\pave T^L \wGamma\|_{L^\infty}\! + \|T^L \wGamma\|_{L^\infty}\!\! \bigfour),
 \label{capitalCdefrel}
\end{equation}
and $C_r\! =\! {\sum}_{|J| \leq r} C_J$, where $L^\infty = L^\infty(\Omega)$.

It is convenient to use slightly different notation for $c_1, C_1$ which is that
they denote constants depending on a fixed number of derivatives of these
quantities,
\begin{equation}
 c_1 = c_1\bigfour( {\sum}_{|K| \leq 4} |\pave T^K \xve|+
 |\pave T^L V| + |\pave T^L \widetilde{V}|
 +|\pave T^L \pave \sigma|
 + |\pave T^L \wGamma| + |T^L \wGamma| \bigfour),
 \label{}
\end{equation}
\begin{equation}
  C_{\!1\!} = C_{\!1}\bigfour( {\sum}_{|K| \leq 4} \|\pave T^K \xve\|_{L^\infty}\!+
  \|\pave T^L V\|_{L^\infty}\!
  +\|\pave T^L \widetilde{V}\|_{L^\infty}\!
  +\|\pave T^L \pave \sigma\|_{L^\infty}\!
  + \|\pave T^L \wGamma\|_{L^\infty}\! + \|T^L \wGamma\|_{L^\infty}\!\! \bigfour).
 \label{C1def}
\end{equation}

 \subsubsection{Control of the $L^2$ norms of the velocity and enthalpy}

By the pointwise estimate \eqref{pwH} we have a bound
of the form
\begin{equation}
{ \sum}_{|J| \leq r\!-\!1 }\! \|\pave T^J V\|_{L^2(\Omega)}^2
 \lesssim
 {\sum}_{|J| \leq r\!-\!1} \Big({\sum}_{T \in \T}
 \|T T^J V\|_{L^2(\Omega)}^2\!+
\!
\|\widetilde{\div}\, T^J \pave V\|_{L^2(\Omega)}\! +
 \!
 \|\widetilde{\curl}\, T^J V\|_{L^2(\Omega)}^2\!\Big),
 \label{usepwacou}
\end{equation}
with notation as \eqref{Hnormdef}-\eqref{HLpnormdef}, and
where we are writing
$$
 \widetilde{\div}\,T^J  V = \widetilde{\nabla}_\mu T^J V^\mu
= |\widetilde{g}|^{-1/2} \pave_\mu \big( | \widetilde{g}|^{1/2} T^J  V^\mu\big),\qquad\text{where}\quad  |\widetilde{g}|=-\operatorname{det}{\widetilde{g}},
 $$
 as well as
 $$
 \widetilde{\curl}\, T^J V_{\mu\nu}
  = \widetilde{\nabla}_\mu T^J V_\nu - \widetilde{\nabla}_\nu T^J V_\mu
   =\pave_\mu T^J V_\nu - \pave_\nu T^J V_\mu.
  $$
The first term on the right-hand side of \eqref{usepwacou}
will be controlled by the energy for the Euler equations, the second term
will be controlled from the continuity equation \eqref{rescaledrelcont2smoothed} and the third
will be controlled because we have an evolution equation for the curl.

We also have a pointwise estimate,
\begin{equation}
 |\pave T^J V| \lesssim
 | \widetilde{\div}\,T^J  V| +
 |\widetilde{\curl} T^J V| + {\sum}_{T \in\T} |T T^J V|,
\end{equation}
which is used to control various lower-order terms that arise in
the upcoming calculations.
The $L^\infty$ norms of the velocity and enthalpy can also be controlled
using the pointwise estimate \eqref{pwH} and this
strategy.

\subsubsection{The additional norm control of the smoothed coordinate $S_\varepsilon x$}
\label{additionalxve}
As in the non-relativistic case, the higher-order energies come
with an additional positive term on the boundary which is
is equivalent to
 $\sum_{|I|\leq r} \| \mathcal{N}\!\!\cdot\! T^I\! S_\varepsilon x\|_{L^2(\pa\Omega)}^2$ when the Taylor sign condition $|\nabla \sigma| > 0$
 holds on $\pa \Omega$. For the smoothed-out problem one also needs to control
 certain error terms and for this we need the following
 modification of the estimate
 in Section \ref{sec:divcurlL2}.

 With notation as in Section \ref{spatialnorms},
 In appendix \ref{relelliptic} we prove the following elliptic estimate,
 \begin{multline}
 {\sum}_{|J|\leq r-1}\vertiii{\pa T^J\! \fdh\sm x}_{L^2(\Omega)}^2\!
 +\!\!{\sum}_{|I|\leq r} \vertiii{ T^I\! S_\varepsilon x}_{L^2(\pa\Omega)}^2
 \leq C_{{}_{\!}1}{\sum}_{|I|\leq r}\vertiii{ n
 \!\cdot_{\oH} T^I\! S_\varepsilon x}_{L^2(\pa\Omega)}^2\\
 +C_{{}_{\!}1}{\sum}_{|J|\leq r-1}\!
 \|\div_{\oH} T^J\! \fdh\sm x\|_{L^2(\Omega)}^2\!
 + \|\curl T^J\! \fdh\sm x\|_{L^2(\Omega)}^2
 +\vertiii{\pa T^J \sm x}_{L^2(\Omega)}^2.
\end{multline}
 Here, the norm $\vertiii{\cdot}$ is taken over just the spatial components,
 see \eqref{appendixnormsdef}. We are also writing
$\n_{\oH} \cdot T^I S_\ve x =
\oH_{\mu\nu} \n_\mu \wg_{\mu\mu'} T^I S_\ve x^{\mu'}$ where
$\n$ denotes the unit conormal to $\pa\Omega$ at constant
$s$, normalized with respect
to the metric $\oH$, and $\div_{\oH}$ denotes the divergence
with respect to $\oH$ (see \eqref{divgdef}). The term involving the divergence
will be under control because it can be written in terms
of the divergence with respect to $\wg$ up to terms involving
the material derivative which are easier to deal with. We can
control the curl time since we have an evolution
equation for all components of the curl.
Using the boundary condition $\widetilde{V}^\mu
\widetilde{\N}_\mu = 0$ the boundary term
here will also be controlled by the energy, see
Section \ref{additionalrelxve}.

\subsection{Higher order equations for the velocity vector field}
\subsubsection{Higher order relativistic Euler's equations}

For any tangential field $T = T^a(y)\pa_{y^a}$ we have
\begin{equation}
 T\widetilde{\pa}_\mu \sigma = \widetilde{\pa}_\mu T\sigma -
 \widetilde{\pa}_\mu T \widetilde{x}^\nu\widetilde{\pa}_\nu \sigma,
 \qquad
 \mu,\nu = 0,1,2,3.
 \label{smcomm}
\end{equation}
If $T \in  \mathcal{S}$, the collection of spacetime tangential vector fields
given in \eqref{Sdef}, then $[T, \hD_s] = 0$ and
from \eqref{smcomm} and \eqref{rescaledreleul2smoothed}, we then have
\begin{equation*}
 \hD_s T V^\mu-\wg^{\mu\nu}\frac{1}{2}\widetilde{\pa}_\alpha \sigma \,\widetilde{\pa}_\nu T\widetilde{x}^\alpha +\frac{1}{2}\wg^{\mu\nu}\widetilde{\pa}_\nu T\sigma=
 -T(\Gamma_{\alpha\nu}^\mu V^{\alpha} \widetilde{V}^\nu).
\end{equation*}

Similarly applying $T^I=T^{I_1}\cdots T^{I_r}$  we get
 \begin{equation}
 \hD_s  T^I  V^\mu
-\frac{1}{2}\wg^{\mu\nu}\widetilde{\pa}_\alpha \sigma \,\widetilde{\pa}_\nu T^{I\!}
 \widetilde{x}^\alpha
+ \frac{1}{2}\wg^{\mu\nu}\widetilde{\pa}_\nu T^{I\!}\sigma
=F^{I \mu} ,
\end{equation}
where $F^I$ is a sum of terms of the form
\begin{itemize}
 \item $T^{I'}\wg \cdot \widetilde{\pa}  T^{I_1} \widetilde{x}\cdots \widetilde{\pa} T^{I_{k-1}} \widetilde{x}\cdot T^{I_{k}} \widetilde{\pa} \sigma$, for $I'+ I_1+\dots+I_k=I$
 with $|I_i| \leq |I|-1$
  and
 \item $T^{I_1} \Gamma \cdot T^{I_2} V \cdot T^{I_3} \widetilde{V}$ for
 $I_1 + I_2 + I_3 = I$,
\end{itemize}
and hence is lower order
\begin{equation}
|F^I|\lesssim c_I{\sum}_{|J|\leq |I|-1} |\widetilde{\pa} T^J \widetilde{x}|
+ |\widetilde{\pa} T^J \sigma|  + |T^J V| + |T^J  \widetilde{V}| + |T^J \wGamma|
+|T^JT \wg|.
\label{Frelest}
\end{equation}

We re-write this as
\begin{equation}\label{eq:symmetricequationrel}
 \hD_s T^I  V^\mu
-\frac{1}{2}\wg^{\mu\nu}
\widetilde{\pa}_\nu\big(\widetilde{\pa}_\alpha \sigma \, T^{I\!} \widetilde{x}^\alpha -  T^{I\!}\sigma\big)=F^{\prime I\mu}\!,
\end{equation}
where $F^{\prime I}=F^{I}-\widetilde{\na} \widetilde{\pa}_\alpha \sigma \,  \,T^{I\!} \widetilde{x}^\alpha $ is lower order.

\subsubsection{Higher order continuity equations}
Similarly, we have
\begin{equation}\label{eq:GIequationrel}
 e'(\sigma) \hD_s T^I \sigma + \widetilde{\pa}_\mu (T^I V^\mu)-\widetilde{\pa}_\mu T^{I\!} \widetilde{x}^\nu \,\,\widetilde{\pa}_\nu V^\mu =G^I,
\end{equation}
where $G^I$ is a sum of terms of the form
\begin{itemize}
 \item $\widetilde{\pa} T^{I_1} \widetilde{x}\cdots \widetilde{\pa} T^{I_{k-1}} \widetilde{x}\cdot \widetilde{\pa} T^{I_{k}}V$,
for $I_1+\dots+I_k=I$ and $|I_i|<|I|$, and
\item $T^{I_1} \Gamma \cdot T^{I_2} V$, for $I_1 + I_2 = I$,
and
\item $e^{(k+1)}(\sigma) T^{I_1} \sigma \cdots
T^{I_{k-1}} \sigma T^{I_{k}} \hD_s \sigma$, for $I_1 + \cdots I_k = I$,
\end{itemize}
and hence is lower order
\begin{equation}\label{eq:GIestimaterel}
 |G^I|\lesssim c_I {\sum}_{|J|\leq |I|-1}
 |\widetilde{\pa} T^J V|+|\widetilde{\pa} T^J\! x|
 + |T^J \wGamma| + |T^J V|,
 \end{equation}
 where $c_I$ is a constant that depends on
$|\widetilde{\pa} T^L \widetilde{x}|$, $|\widetilde{\pa} T^L V|$,
$|T^L \wGamma|$, $|T^L V|$, $|\widetilde{\pa} T^L V|$
 for
$|L|\leq |I|/2 $.
 Hence
\begin{equation}\label{eq:highercontinuityrel}
 e'(\sigma) \hD_s T^I \sigma+\widetilde{\pa}_\mu \bigtwo(T^I V^\mu-T^{I\!} \widetilde{x}^\nu \,\,\widetilde{\pa}_\nu V^\mu\bigtwo) =G^{{}^{\,}\prime I},
\end{equation}
where $G^{{}^{\,}\prime I}=G^I-T^{I\!} \widetilde{x}^\nu \,\,\widetilde{\pa}_\nu \,
 \widetilde{\pa}_\mu V^\mu$ is lower order.

\subsubsection{New unknowns}
We introduce
\begin{equation}
V^{I}_\mu = \wg_{\mu\nu} T^I V^\mu- \wg_{\mu\nu} \widetilde{\pa}_\alpha V^\nu\,T^{I\!} \widetilde{x}^\alpha,\qquad
\text{and}\qquad
\sigma^I=  T^{I\!}\sigma-\widetilde{\pa}_\nu \sigma \,\, T^{I\!} \widetilde{x}^\nu,
\end{equation}
and \eqref{eq:symmetricequationrel} and \eqref{eq:highercontinuityrel} take the form
\begin{align}\label{eq:highereulermodifiedrel}
\hD_s V^{I}_{\mu}+\frac{1}{2}  \widetilde{\pa}_\mu \sigma^I&=F_\mu^{\prime\prime I},\\
\label{eq:highercontinuitymodifiedrel}
e' \hD_s \sigma^{I}+\widetilde{\pa}_\mu(\wg^{\mu\nu} V^{I}_\nu)&=G^{\prime\prime I},
\end{align}
where $F_\mu^{\prime\prime I}\!\!= \!F_\mu^{\prime I}\!-\hD_s\big(\wg_{\mu\nu} \widetilde{\pa}_\alpha V^\nu\,T^{I\!} \widetilde{x}^\alpha\big) + (\hD_s \wg_{\mu\nu} )T^I V^\nu\!$ and
 $G^{\prime\prime I}\!\!=\!G^{\prime I}\!-\hD_s\big(\widetilde{\pa}_\nu \sigma \,\, T^{I\!} \widetilde{x}^\nu\big)$ are lower order.

 \subsubsection{The evolution equation for $\wg(V,V)$}
 \label{lengthsec}
 Recall that for the non-smoothed problem we have defined $V$ so that
 $V_\mu V^\mu = -\sigma$. Multiplying both sides of
 the Euler equations \eqref{rescaledreleul2} by $V^\mu$ we see
 that this condition is propagated for the non-smoothed problem
 and it is approximately propagated for the smoothed equation
 \eqref{rescaledreleul2smoothed} as well.
 We will need a higher-order version of this propagation equation.
 Multiplying both sides of \eqref{eq:highereulermodifiedrel}
 by $\widetilde{V}^\nu$ and using
 $\widetilde{V}^\mu \pave_\mu = \hD_s$ we find that
 \begin{equation}
  \hD_s ( \sigma^I + 2\widetilde{g}^{\mu\nu} V_\nu V^I_\mu) = (\widetilde{V}^\mu - V^\mu)\pave_\mu \sigma^I
  + F_{\mu}^{\prime\prime\prime I} V^\mu,
  \label{normerroreqn}
 \end{equation}
 where $F^{\prime\prime\prime I}_\mu = F^{\prime\prime I}_\mu +
 2\hD_s \widetilde{g}^{\mu\nu} V^I_\nu$.
 Writing $\pave_\mu \sigma^I = -2D_s V_\mu^I + 2F_{\mu}^{\prime\prime I}$
 and defining $L_1^I = \sigma^I + 2V^\mu V^I_\mu + 2(\widetilde{V}^\mu - V^\mu) V^I_\mu$, the
 above becomes
 \begin{equation}
  \hD_s L^I_1 = \hD_s(\widetilde{V}^\mu - V^\mu) V_\mu^I +
  2F_{\mu}^{I\prime\prime \prime} V^\mu + 2(\widetilde{V}^\mu - V^\mu)F_{\mu}^{\prime\prime I}.
 \end{equation}

 When $\ve = 0$, $L^I_1 = L^I = \sigma^I + 2 V^\mu V_\mu^I$
 and integrating \eqref{normerroreqn} we find
 the pointwise bound
 \begin{equation}
  |L^I(s)| \lesssim
  |L^I(s_0)| +
  s \sup_{s_0\leq s' \leq s}\left( |F^{\prime\prime I}(s)|
  |V(s)| + |\hD_s \wg(s)| |V(s)| |V^I(s)|\right),\qquad \text{ if } \ve = 0,
  \label{LIboundve0}
 \end{equation}
  For $\ve > 0$
 we bound
 $\widetilde{V}^\mu - V^\mu$ using \eqref{hnk}
 which gives
 \begin{equation}
  |L_1^I(s)| \lesssim |L_1^I(s_0)| +
  s\!\!\!\sup_{s_0\leq s' \leq s}\!\!\!\left( |F^{\prime\prime I\!}(s)|
  |V(s)|\! + |\hD_s \wg(s)| |V(s)| |V^{I\!}(s)|\!
  + C_0 \ve |V^{I\!}|\! + C_0 \ve  |F^{\prime\prime I\!}|\right),
  \label{LIboundve}
 \end{equation}
 with $C_1$ as in \eqref{C1def}.
  We also have
 \begin{equation}
   |\sigma^I + 2 V^\mu V_\mu^I - L_1^I|
   \lesssim C_0 \ve |V^I|,
  \label{LIboundve2}
 \end{equation}
 which gives a bound for $\sigma^I + 2V^\mu V_\mu^I$
 when $\ve > 0$.

\subsection{Higher-order energies for the velocity vector field}
\label{highorderVrel}
With $D_s \!= \!\widetilde{V}^\nu\pave_\nu$, in an arbitrary
coordinate system we have
\begin{align}
  \widetilde{g}^{\mu\nu}
  (D_s V_\mu^I \! + \frac{1}{2}  \pave_\mu \sigma^I)  V_\nu^I\!
  &= \frac{1}{2} \pave_\alpha \left(  \wg^{\mu\nu} V_\mu^I V_\nu^I \widetilde{V}^\alpha\!
   + \sigma^I \wg^{\alpha\beta} V_\beta^I
  \right)-\frac{1}{2}\sigma^I  \pave_\alpha
  (\wg^{\alpha\beta} V^I_\beta)
  -\frac{1}{2} \pave_\alpha ( \wg^{\mu\nu} \widetilde{V}^\alpha )
  V_\mu^I V_\nu^I\\
  &=
  \frac{1}{2} \pave_\alpha \left(  \wg^{\mu\nu} V_\mu^I V_\nu^I \widetilde{V}^\alpha\!
   + \sigma^I \wg^{\alpha\beta} V_\beta^I
  \right)+\frac{1}{4}\pave_\alpha
  \left( \widetilde{V}^\alpha e'\,(\sigma^I)^2\right)\\
  &\qquad\qquad-\frac{1}{2} \pave_\alpha ( \wg^{\mu\nu} \widetilde{V}^\alpha )
  V_\mu^I V_\nu^I
  -\frac{1}{2} \pave_\alpha(\widetilde{V}^\alpha e')
  (\sigma^I)^2
  -\frac{1}{2}\sigma^I  G^{\prime\prime I}.
  \label{multiplier}
\end{align}
Let $\wg$ denote the determinant of the matrix
$\wg_{\mu\nu}$. Then for a vector field $X = X^\alpha
\pave_\alpha$ we have the identity
\begin{equation}
 \pave_\alpha X^\alpha = \div X^\alpha - \frac{1}{2}X^\alpha \pave_\alpha
 \log |\wg|
 \label{divident}
\end{equation}

We now fix two spacelike surfaces $\Sigma_1, \Sigma_0
\subset \D$
with $\Sigma_1$ lying to the future of $\Sigma_0$ and which
are both bounded by the timelike surface $\Lambda$.
Let $n^{\Sigma_i}$ denote the future-directed normal
vector field to $\Sigma_i$.
Let $\Lambda_{\Sigma_i} = \Lambda \cap \Sigma_i$ and let
$\Lambda_{\Sigma_0}^{\Sigma_1}$ denote the portion of the timelike
surface $\Lambda$ lying between $\Sigma_0$ and $\Sigma_1$, and let
$D_{\Sigma_0}^{\Sigma_1}$ denote the region bounded between
$\Sigma_0, \Sigma_1$ and $\Lambda_{\Sigma_0}^{\Sigma_1}$.
Let $d\mu_{\widetilde{g}} = \sqrt{-\det\wg} dx dt$ be the measure
on $D_{\Sigma_0}^{\Sigma_1}$ induced
by the metric $\wg$. For a hypersurface $U$ let $dS^U$ denote the
corresponding surface measure.
Integrating the identity \eqref{multiplier}
over $D_{\Sigma_0}^{\Sigma_1}$ with respect
to $d\mu_{\widetilde{g}}$ and using
 the divergence theorem \ref{divergencelemma},
 the identity \eqref{divident} and the
 fact that $\wg(\widetilde{V}, \widetilde{N}) = 0$ on
 $\Lambda$, we find
\begin{multline}
 \int_{\Sigma_1}\left(
 \wg(V^I, V^I)\wg(\widetilde{V}, \widetilde{n}^{\Sigma_1})
 +\sigma^I \wg(V^I, \widetilde{n}^{\Sigma_1})
 +\frac{1}{2} e'(\sigma) (\sigma^I)^2 \wg(\widetilde{V}, \widetilde{n}^{\Sigma_1})
 \right)
 dS^{\Sigma_1}
 +\int_{\Lambda_{\Sigma_0}^{\Sigma_1}}\sigma^I \wg(V^I, \widetilde{\N})
 dS^{\Lambda} \\
 =
 \int_{\Sigma_0} \left(\wg(V^I, V^I)\wg(\widetilde{V}, \widetilde{n}^{\Sigma_0})
 +\sigma^I \wg(V^I, \widetilde{n}^{\Sigma_0})
 +\frac{1}{2} e' (\sigma^I)^2 \wg(\widetilde{V}, \widetilde{n}^{\Sigma_0})\right)
 dS^{\Sigma_0}\\
 + \int_{\D_{\Sigma_0}^{\Sigma_1}}\left(
 F^{\prime\prime I}_\mu V^{\mu I}
 +\sigma^I  G^{\prime\prime I}
 +
 \pave_\alpha ( \wg^{\mu\nu} \widetilde{V}^\alpha )
 V_\mu^I V_\nu^I + \pave_\alpha(\widetilde{V}^\alpha e'(\sigma))
   (\sigma^I)^2 \right)d\mu_{\widetilde{g}}
\\
+ \int_{\D_{\Sigma_0}^{\Sigma_1}}\left(
\wg^{\mu\nu} V_\mu^I V_\nu^I
+ \sigma^I \wg^{\alpha\beta} V_\beta^I+\frac{1}{2}
 e'(\sigma^I)^2\right)
 D_s \log |\wg|\,d\mu_{\widetilde{g}}.
 \label{intmult}
\end{multline}
Let $\Sigma$ denote either of $\Sigma_0, \Sigma_1$.
Recalling the definition $L^I_1 = \sigma^I + 2\widetilde{V}^\mu
V_\mu^I$ from section \ref{lengthsec},
 write
\begin{equation}
 -\sigma^I \wg(V^I, \widetilde{n}^{\Sigma})
 - \wg(V^I, V^I) \wg(\widetilde{V}, \widetilde{n}^{\Sigma})
 = Q[V^I](\widetilde{V}, \widetilde{n}^\Sigma)
 - L^I_1 \wg(V^I, \widetilde{n}^\Sigma),
 \label{}
\end{equation}
where $Q[V^I](\widetilde{V}, \widetilde{n}^\Sigma)$
denotes the energy-momentum tensor
\begin{equation}
 Q[V^I](X, Y)
 = 2 \wg(V^I, X)\wg(V^I, Y)
 -\wg(X, Y) \wg(V^I, V^I).
 \label{}
\end{equation}
If we set $Q^I= Q^{I}_1 + Q^I_2$ with
$Q^I_1 = Q[V^I]$ and $Q^I_2 = e' (\sigma^I)^2$, the identity
\eqref{intmult} reads
\begin{multline}
 \int_{\Sigma_1} \left(Q^I(\widetilde{V}, \widetilde{n}^{\Sigma_1})
 -L^I_1 \wg(V^I, \widetilde{n}^{\Sigma_1})\right) dS^{\Sigma_1}
 -\int_{\Lambda_{\Sigma_0}^{\Sigma_1}}\sigma^I \wg(V^I, \widetilde{\N}) ds^{\Lambda} \\
 =
 \int_{\Sigma_0} \left(Q^I(\widetilde{V}, \widetilde{n}^{\Sigma_0})
 -L^I_1 \wg(V^I, \widetilde{n}^{\Sigma_0}) \right)dS^{\Sigma_1}\\
 - \int_{\D_{\Sigma_0}^{\Sigma_1}}\left(
 F^{\prime\prime I}_\mu V^{\mu I}
 +\sigma^I  G^{\prime\prime I}
 +
 \pave_\alpha ( \wg^{\mu\nu} \widetilde{V}^\alpha )
 V_\mu^I V_\nu^I + \pave_\alpha(\widetilde{V}^\alpha e'(\sigma))
   (\sigma^I)^2 \right)d\mu_{\widetilde{g}}
\\
- \int_{\D_{\Sigma_0}^{\Sigma_1}}\left(
\wg^{\mu\nu} V_\mu^I V_\nu^I
+ \sigma^I \wg^{\alpha\beta} V_\beta^I+\frac{1}{2}
 e'(\sigma^I)^2\right)
 D_s \log |\wg|\,d\mu_{\widetilde{g}}.
 \label{enidentrel}
\end{multline}
This identity is the analogue of
the time-integral of the identity \eqref{dtEnewt}.
From the upcoming bound \eqref{emposlemint},
since $\widetilde{V}, n^{\Sigma_1}$
are timelike and future-directed, $Q^I(\widetilde{V}, n^{\Sigma_1})$
is positive-definite.

We now consider
 the integral over the timelike surface
$\Lambda_{\Sigma_0}^{\Sigma_1}$ .
Since $\sigma$ is constant on $\Lambda$
and by assumption $\widetilde{\N}^\mu \pave_\mu \sigma < 0$,
we have $T^I \sigma = 0$ and
$\widetilde{\pa}_\mu \sigma = -\widetilde{\N}_\mu
|\widetilde{\pa} \sigma|$. Therefore
\begin{equation}
  \widetilde{\mathcal{N}}^\mu V^{I}_\mu  =\widetilde{\mathcal{N}}_\mu
  T^I V^\mu - \widetilde{\mathcal{N}}_\mu
  \widetilde{\pa}_\alpha V^\mu T^I \widetilde{x}^\alpha, \qquad
  \sigma^I = \widetilde{\N}_\nu T^I \xve^\nu |\widetilde{\pa} \sigma|,
  \qquad \text{ on } \Lambda.
\end{equation}
Arguing as in section \ref{sec:nonrelbdy}, we find
\begin{equation}
  \widetilde{\N}_\mu T^I V^\mu = \hD_s \big(\widetilde{\N}_\mu T^I
  x^\mu\big)
  -\eta\,\widetilde{\mathcal{N}}_\mu \,  T^I {x}^{\mu}+
  \widetilde{\mathcal{N}}_\mu\widetilde{\pa}_\alpha \widetilde{V}^\mu\,  T^I {x}^{\alpha}
  -\widetilde{\mathcal{N}}_\mu \widetilde{\pa}_\alpha {V}^\mu\,  T^I \widetilde{x}^{\alpha},
\end{equation}
 with $\eta = \pave_\alpha \widetilde{V}^\beta \widetilde{\N}^\alpha
\widetilde{\N}_\beta$, so
\begin{multline}
 \int_{\Lambda_{\Sigma_0}^{\Sigma_1}} \sigma^I \wg(V^I, \widetilde{\N}) dS^\Lambda
 = \int_{\Lambda_{\Sigma_0}^{\Sigma_1}} \widetilde{\N}_\nu T^I \xve^\nu \hD_s \big(
 \widetilde{\N}_\mu T^I x^\mu\big) |\pave \sigma|   dS^\Lambda
 - \int_{\Lambda_{\Sigma_0}^{\Sigma_1}}
 \widetilde{\N}_{\nu} T^I \xve^\nu \widetilde{\N}_\mu T^I x^\mu
 \eta|\pave \sigma|  \,  dS^\Lambda \\
 + \int_{\Lambda_{\Sigma_0}^{\Sigma_1}}
 \widetilde{\N}_\nu T^I \xve^\nu
 \big( \widetilde{\N}_\mu \pave_\alpha \widetilde{V}^\mu T^I x^\alpha
 - \widetilde{\N}_\mu \pave_\alpha \widetilde{V}^\mu T^I \xve^\alpha
 \big)
 |\pave \sigma|  dS^\Lambda,
 \label{eq:boundaryproblemsrel2}
\end{multline}
compare with \eqref{eq:boundaryproblems}.

\subsubsection{Positivity of the energy-momentum tensor}
\label{empossec}
Recall that a vector field $Z$ is timelike if
$\wg(Z, Z) < 0$ and it is future-directed provided $\wg(Z,\tau) < 0$
with $\tau$ the generator of the time axis.
\begin{lemma}
  \label{emtensorpositivity}
  Let $X, N$ be future-directed timelike vector fields and
  let $Q[Z](X,N) = 2\wg(Z, X)\wg(Z, N) - \wg(X, N)
  \wg(Z, Z)$.
  There is a constant $d_0$ depending on $X, N$ so that
  with notation as in \eqref{Hnormdef},
  \begin{equation}
   Q[Z](X, N) \geq d_0 |Z|^2
 \end{equation}
\end{lemma}
\begin{proof}
  Replacing $N$ with $N/(-\wg(N, N))^{1/2}$ we can assume that
  $\wg(N, N) = -1$. We now write
  $
   X = X^N N + P_N X,$ $Z = Z^N N + P_NZ,
   $
  where $Y^N = -g(Y, N)$ and where $P_N$ denotes the orthogonal projection
  away from $N$. This decomposition gives
  \begin{equation}
   Q[Z](X, N) = X^N \left((Z^N)^2 + \wg(P_NZ, P_NZ)\right)
   - 2 Z^N \wg(P_NX, P_NZ).
   \label{}
  \end{equation}
  Since
  \begin{equation}
   |Z^N \wg(P_NX, P_NZ)|
   \leq \frac{1}{2} |P_NX| \left( (Z^N)^2 + \wg(P_NZ, P_NZ)\right),
   \label{}
  \end{equation}
  the above formula gives the lower bound
  \begin{equation}
   Q[Z](X, N) \geq \left(X^N - |P_NX|\right) \left((Z^N)^2 + \wg(P_NZ, P_NZ)\right).
   \label{}
  \end{equation}

  Since $X, N$ are future-directed we have $X^N = -\wg(X, N) > 0$.
  Abusing notation slightly and writing $|P_NX| = (\wg(P_NX, P_NX))^{1/2}$,
  we have
    \begin{equation}
     0 > \wg(X, X) = -(X^N)^2 + \wg(P_NX, P_NX ) =
    -\left( X^N + |P_NX|\right)\left(X^N - |P_NX|\right),
     \label{}
    \end{equation}
    so there is a constant $d_0' = d_0'(X, N)$ with
    \begin{equation}
     Q[Z](X, N) \geq d_0' \left((Z^N)^2 + \wg(P_NZ, P_NZ)\right).
     \label{QZXN}
    \end{equation}
    The result follows since the norm
    on the right-hand side of \eqref{QZXN} is equivalent to
    the norm \eqref{Hnormdef}.

\end{proof}

\subsubsection{The a priori bounds for the relativistic Euler equations}
When $\ve = 0$, we have $x = \xve$ so the last term in \eqref{eq:boundaryproblemsrel2}
vanishes. The first term is symmetric and since $D_s$ is tangent to $\Lambda$, by
\eqref{transportbdy}, we have
\begin{multline}
  \int_{\Lambda_{\Sigma_0}^{\Sigma_1}} \widetilde{\N}_\nu T^I \xve^\nu \hD_s \big(
  {\N}_\mu T^I x^\mu\big) |\pave \sigma|  dS^\Lambda
  =
  \int_{\Lambda_{\Sigma_0}^{\Sigma_1}} {\N}_\nu T^I x^\nu \hD_s \big(
  {\N}_\mu T^I x^\mu\big) |\pa \sigma| dS^\Lambda\\
  =
  \frac{1}{2}\int_{\Lambda_{\Sigma_1}} \left((T^I x^\mu){\N}_\mu\right)^2
  g({V}, n^{\Sigma_1})
  |\pa \sigma|\,dS^{\Lambda_{\Sigma_0}}- \frac{1}{2}\int_{\Lambda_{\Sigma_0}}\left((T^I x^\mu){\N}_\mu\right)^2g({V}, n^{\Sigma_0})
  |\pa \sigma|\, dS^{\Lambda_{\Sigma_0}}\\
  - \frac{1}{2}\int_{\Lambda_{\Sigma_0}^{\Sigma_1}}\left((T^I x^\mu){N}_\mu\right)^2
  \widetilde{D}_\mu({V}^\mu |\pa \sigma|)\, dS^\Lambda,
 \label{eq:boundaryproblemsrel3}
\end{multline}
where here $\Lambda_{\Sigma_j} = \Lambda \cap \Sigma_{j}$,
$dS^{\Lambda_{\Sigma_j}}$ denotes surface measure on $\Lambda_{\Sigma_j}$
and $\widetilde{D}_\mu$ denotes covariant differentiation on
$\Lambda$.

The above computations were done with respect to arbitrary spacelike
surfaces $\Sigma_1, \Sigma_0$ but for the sake of concreteness
we now foliate the domain $\D$ into spacelike surfaces
$\Sigma_s = x(s, \Omega)$ for $s \in [s_0, s_1]$ for some $s_0, s_1$
where the Lagrangian coordinates $x(s,\cdot)$ are defined in
\eqref{lagcoordsdef}.
Expressing the integrals in Lagrangian coordinates,
the energies are
\begin{equation}
 \mathcal{E}^I(s)
 = \int_{\Omega} |V^I(s)|^2\, {\kappa} dy
 + \int_{\Omega} e'(\sigma^I(s))^2\, {\kappa} dy
 + \int_{\Omega} |T^I x(s)|^2 \, {\kappa} dy
 + \int_{\pa \Omega}\left((T^I x^\mu(s)){\N}_\mu\right)^2 |\pa \sigma| {\nu} dS.
 \label{}
\end{equation}
Here, ${\kappa} dy$ is the surface measure on $\Sigma_s$
expressed in Lagrangian coordinates and ${\nu} dS$
is the surface measure on $\Lambda_{\Sigma_s}$ expressed
in Lagrangian coordinates.
In what follows we will drop the measures from our notation.

Since $V$ is timelike and future-directed, it follows
from Lemma \ref{emtensorpositivity} that
there are constants $D_1, D_2> 0$ depending on
$V$ and $\Sigma_s$
so that with notation as in
\eqref{Hnormdef} and $Q^I$ as in \eqref{enidentrel},
  \begin{equation}
   D_1 \int_{\Omega} (|V^I(s)|^2 +e'(\sigma)|\sigma^I(s)|^2)
   {\kappa} dy
   \leq \int_{\Sigma_s} Q^I({V}, n^{\Sigma_s})
   dS^{\Sigma_s}
   \leq D_2 \int_{\Omega} (|V^I(s)|^2 + e'(\sigma)|\sigma^I(s)|^2).
   \label{emposlemint}
  \end{equation}

Since ${V}, n^{\Sigma_s}$ are both timelike and future-directed
it follows that $-{g}({V}, n^{\Sigma_s}) > g_0 >0$
for a constant $g_0$
(see the comment below \eqref{nSigma}),
so combining the identity \eqref{enidentrel}
with $\Sigma_{j} = \Sigma_{s_j}$, \eqref{eq:boundaryproblemsrel3}
and using the lower bound \eqref{emposlemint} we have
\begin{equation}
 \mathcal{E}^I(s_1)- \mathcal{E}^I(s_0)
 \lesssim c_0 \int_{s_0}^{s_1} \mathcal{E}^I(s)
 + \int_{s_0}^{s_1} \int_{\Omega}
 (|F^{\prime \prime I}(s)|^2 + (G^{\prime\prime I}(s))^2)
 \,{\kappa} dy ds
 + \int_{\Omega} |L_1^I(s_1)|^2\, ds + \int_{\Omega} |L_1^I(s_0)|^2
 \, ds.
 \label{}
\end{equation}

Using the evolution equation \eqref{LIboundve0} to handle the terms involving
$L^I_1$ we have
\begin{multline}
  \mathcal{E}^I(s_1)- \mathcal{E}^I(s_0)
  \lesssim c_0 \int_{s_0}^{s_1} \mathcal{E}^I(s)\, ds
  + \int_{s_0}^{s_1} \int_{\Omega}
  (|F^{\prime \prime I}(s)|^2 + (G^{\prime\prime I}(s))^2)
  {\kappa} dy ds\\
  +
  (s_1-s_0)\sup_{s_0\leq s \leq s_1}\Big(
  C_I{\sum}_{|J| \leq |I|}  \mathcal{E}^J(s)
  + \int_{\Sigma_s} |F^{\prime\prime I}(s)| |V(s)|\Big) \kappa dy.
 \label{mainrelenbdnonsmooth2}
\end{multline}
The error terms on the right-hand side involving
$F^{\prime\prime I}, G^{\prime\prime I}$ can be bounded
in terms of lower order norms using
\eqref{Frelest}, \eqref{eq:GIestimaterel} and the elliptic
estimate \eqref{usepwacou}. If we take $s_1-s_0$ sufficiently
small and take the supremum over $s$ on both sides, the
highest-order term on the right-hand side $\sup_{s_0 \leq
s \leq s_1} \mathcal{E}^I(s)$ can be absorbed into the
left-hand side.
The energy also satisfies
\begin{equation}
 \int_{\Sigma_s} |T^I V|^2 + e'(\sigma) (T^I \sigma)^2
 + |T^I x|^2 \kappa dy\lesssim c_0 \mathcal{E}^I(s).
\end{equation}

\subsubsection{The a priori bounds for the smoothed relativistic Euler equations}

When $\ve > 0$, to handle the boundary term
from \eqref{eq:boundaryproblemsrel2} we instead argue as in the proof of
\eqref{eq:boundaryproblemssmoothedsymmetric} to move one of the smoothing operators
to the other factor which gives the following replacement for
\eqref{eq:boundaryproblemsrel2},
\begin{multline}\label{eq:boundaryproblemssmoothedsymmetricrel}
\int_{\Lambda_{\Sigma_0}^{\Sigma_1}} \sigma^I \wg(V^I, \widetilde{\N}) dS^\Lambda
=\!\!\int_{\Lambda_{\Sigma_0}^{\Sigma_1}}\widetilde{\mathcal{N}}_\nu  T^{I\!} {x}_\varepsilon^\nu D_s \big( \widetilde{\mathcal{N}}_\mu   T^{I\!} {x}_\varepsilon^{\mu} \big) |\widetilde{\pa}\sigma|
 d S^\Lambda
-\!\int_{\Lambda_{\Sigma_0}^{\Sigma_1}}
\widetilde{\mathcal{N}}_\nu T^{I\!} {x}_\varepsilon^{\,\nu}\,\widetilde{\mathcal{N}}_\mu  T^{I\!}  {x}_\varepsilon^{\mu}\, \eta|\widetilde{\pa} \sigma| d S^\Lambda\\
+\!\int_{\Lambda_{\Sigma_0}^{\Sigma_1}} \widetilde{\mathcal{N}}_\nu   T^{I\!}  {x}_\varepsilon^{\,\nu}\,
\bigtwo(\widetilde{\mathcal{N}}_\mu\widetilde{\pa}_\alpha
 \widetilde{V}^\mu\,  T^{I\!} {x}_\varepsilon^{\alpha}\!
-\widetilde{\mathcal{N}}_\mu\widetilde{\pa}_\alpha {V}^\mu\, \sm^2 T^{I\!} {x}_\varepsilon^{\alpha} \bigtwo) |\widetilde{\pa} \sigma|  d S^\Lambda
+\int_{\Lambda_{\Sigma_0}^{\Sigma_1}} T^{I\!} {x}_\varepsilon^{\,\nu} C_{\varepsilon \nu}^{\prime I} |\widetilde{\pa} \sigma|  d S^\Lambda,
\end{multline}
where $C_{\varepsilon j}^{\prime I}$ satisfy \eqref{eq:smoothingcommutator}.
The first term on the right-hand side is symmetric and so just
as in \eqref{eq:boundaryproblemsrel3}
\begin{multline}
  \int_{\Lambda_{\Sigma_0}^{\Sigma_1}} \widetilde{\N}_\nu T^I \xve^\nu \hD_s \big(
  \widetilde{\N}_\mu T^I x^\mu\big) |\pave \sigma|  dS^\Lambda
  \\
  =
  \frac{1}{2}\int_{\Lambda_{\Sigma_1}}
  \left((T^I x_\ve^\mu)\widetilde{\N}_\mu\right)^2
  \wg(\widetilde{V}, n^{\Sigma_1})
  |\pave \sigma|\,dS^{\Lambda_{\Sigma_0}}- \frac{1}{2}\int_{\Lambda_{\Sigma_0}}\left((T^I x_\ve^\mu)\widetilde{\N}_\mu\right)^2\wg(\widetilde{V}, n^{\Sigma_0})
  |\pave \sigma|\, dS^{\Lambda_{\Sigma_0}}\\
  - \frac{1}{2}\int_{\Lambda_{\Sigma_0}^{\Sigma_1}}\left((T^I x_\ve^\mu)\widetilde{\N}_\mu\right)^2
  \pave_\mu(\widetilde{V}^\mu |\pa \sigma| \log \widetilde{\kappa}_{\Lambda})\, dS^\Lambda.
 \label{eq:boundaryproblemsrel4}
\end{multline}
We now foliate by the spacelike surfaces
$\Sigma_s = \xve(s, \Omega)$ with $\xve$ as in
\eqref{eq:eulerlagrangiancoordxsmoothedrel}. With
\begin{equation}
 \mathcal{E}^I(s)
 = \int_{\Omega} |V^I|^2 \widetilde{\kappa} dy
 + \int_{\Omega} e'(\sigma^I(s))^2\, {\kappa} dy
 + \int_{\Omega} |T^I x(s)|^2 \, {\kappa} dy
 + \int_{\pa \Omega}\left((T^I x_\ve^\mu(s))\widetilde{\N}_\mu\right)^2 |\pa \sigma| {\nu} dS,
 \label{eq:energydefrel}
\end{equation}
and
\begin{equation}\label{eq:energyboundarydefrel}
\mathcal{B}^{I\!}(s)
=\!\!\int_{\pa\Omega}\!\!\! |T^I x_\ve(s)|^2 dS,\qquad
\mathcal{B}_{\mathcal{N}}^{\,I}(s)
=\!\!\int_{\pa\Omega}\!\!\! (\, \widetilde{\mathcal{N}}_\mu\!
T^{I\!} x_{\varepsilon }^\mu(s))^2  dS.
\end{equation}
Arguing as in section \ref{smoothennonrel}, we find
\begin{multline}
\mathcal{E}^I(s_1) \lesssim
\mathcal{E}^I(s_0) + \int_{s_0}^{s_1}\!\! \int_{\Omega}
|F^{\prime\prime I}(s) |^2 + (G^{\prime\prime I})^2\, \kappa
dy ds\\
+ c_0 \!\int_{s_0}^{s_1}\!\! \mathcal{B}^I(s)\, ds
+\! \int_{\Omega} |L^I(s_1)|^2\, \kappa dy
+\! \int_{\Omega} |L^I(s_2)|^2\, \kappa dy+
c_0 \!\!\int_{s_0}^{s_1}\!\! \mathcal{E}^I(s)\, ds.
\end{multline}
To deal with the contribution from $L^I$ we use
\eqref{LIboundve}-\eqref{LIboundve2} which gives
\begin{multline}
\mathcal{E}^I(s_1) \lesssim
\mathcal{E}^I(s_0) + \int_{s_0}^{s_1} \int_{\Omega}
|F^{\prime\prime I} |^2 + (G^{\prime\prime I})^2\, \widetilde{\kappa}
dy ds
+ c_0 \int_{s_0}^{s_1} \mathcal{B}^I(s)\, ds
\\
+  (s_1-s_0) \sup_{s_0 \leq s \leq s_1}
\Big( C_I {\sum}_{|J| \leq |I|}
\mathcal{E}^J(s') +
\int_{\Omega} |F^{\prime \prime I}(s)|
|V(s)|\widetilde{\kappa}\, dy\Big)\\
+ C_0\ve\int_{\Omega} |V^I(s_1)|^2\, \widetilde{\kappa}
dy
+
c_0 \int_{s_0}^{s_1} \mathcal{E}^I(s)\, ds.
\label{mainveposestimate}
\end{multline}
For $\ve$ and $s_1-s_0$ sufficiently small
the highest-order terms $\mathcal{E}^I$ on
the right-hand side can be handled by absorbing
as in the last section.

As in the Newtonian case
the energy only controls the normal component
of $x_\ve$ at the boundary,
\begin{equation}
  \mathcal{B}_{\N}^I(s)
 \lesssim
 \mathcal{E}^I(s),
\label{eq:byenergyestimatedrel}
\end{equation}
and so we need
additional argument to control all components of $x_\ve$
at the boundary. In the Newtonian
case the key ingredient was the elliptic estimate
\eqref{eq:coordinatehalfest} and in this case we will
instead use
the elliptic estimate from Lemma \ref{coordhalfestrel} which has
the same basic content but is in terms of the metric $\oH$ introduced
in \eqref{Gdef}, whereas in the Newtonian case all our estimates
were in terms of the Euclidean metric.

\subsubsection{The apriori energy bounds for the smoothed linear system} Given $U^\mu=V^\mu_{(k)}$ define $z^\mu=x^\mu_{(k)}$
for $\mu = 0,1,2,3$ by $dz/ds=
\overline{U}(z(s,y))$ and define
$\widetilde{V}=S_\varepsilon^* S_\varepsilon U$ and
$\widetilde{x}=S_\varepsilon^* S_\varepsilon z$.
Next, given $\widetilde{V}$ and $\widetilde{x}$ tangentially smooth define the new $V_{(k+1)}=V$ by solving the linear system \eqref{rescaledreleul2smoothed}- \eqref{smoothedwaveic}, and $x_{(k+1)}=x$ by $dx/ds=\sV(x(s,y))$. In Section \ref{existencerel}
we prove that the linear system \eqref{rescaledreleul2smoothed}
-\eqref{smoothedwaveic} is well-posed in an appropriate
energy space.
By the energy estimates from
the previous sections, after arguing as in
 Section \ref{sec:apriorilinear}, if we define
\begin{equation}
\mathcal{E}_k^{I\!}(s_1)=
 \int_{\Omega} |V_{(k)}^I(s_1)|^2\, \kappa dy +
 \int_{\Omega} e' (\sigma_{(k)}^I(s_1))^2 \kappa dy
 +\int_\Omega |T^I\! x_{(k)}(s_1)|^2\, \kappa dy,
\end{equation}
then
\begin{multline}
\mathcal{E}_{k+1}^{I\!}(s_1)\lesssim
\mathcal{E}_{k+1}^{I\!}(s_1)+  \frac{c_0}{\!\varepsilon\,}\!\!\!\int_{s_0}^{s_1}
\!\!\mathcal{E}_{k+1}^{I\!}(s) \, ds
+\frac{C_{\!J}^{(k)}\!\!}{\!\!\varepsilon} \!\!\!\!\!\!  \sum_{|J|\leq |I|-1}\!\int_{s_0}^{s_1}\!\!\!
|\widetilde{\pa}T^J\! x_{(k+1)}|_{L^2(\Omega)}^2\!
+|\widetilde{\pa}T^J V_{(k+1)}|_{L^2(\Omega)}^2\!
+ |T^J \wGamma|^2  ds\\
+\frac{C_J^{(k+1)}}{\ve}\!\!\sum_{|J| \leq |I|\shortminus 1}\!
\int_{s_0}^{s_1}
|\widetilde{\pa}T^J x_{(k)}|_{L^2(\Omega)}^2
+|\widetilde{\pa}T^J V_{(k)}|_{L^2(\Omega)}^2\,
ds,
\end{multline}
so again we only have an energy bound up to a time
$t = O(\ve)$.
\subsection{Higher-order wave and elliptic estimates for the enthalpy}
\label{enthalpyestrel}

From the wave equation \eqref{sigmawavesetup1}, we have
\begin{equation}
   e'(\sigma) \hD_s^2 T^{J\!} \sigma
   -\! \widetilde{\nabla}^\mu\big( T^J{\widetilde{\nabla}}_\mu \sigma\big)
   = P^J\!+Q^J,
   \label{higherorderwave}
   \end{equation}
   where $ P^J\!\!=\!
   \big[\widetilde{\na}_\mu,T^J]\widetilde{\pa}_\nu \sigma$ and
   $ Q^J\!\! = 2T^J Q
   +2 T^J (\pave_\mu \widetilde{V}^\nu \pave_\nu V^\mu)$
   with $Q$ as in \eqref{Qdef}. These
   are lower order:
 \begin{equation}
 |P^J|\lesssim c_J {\sum}_{|K|\leq |J|}
 |\widetilde{\pa} T^K \widetilde{x}|
 + |\widetilde{\pa} T^K \widetilde{V}| + |T^K \wg|
 +c_J {\sum}_{|K|\leq |J|-1}| \widetilde{\pa} T^{K}
 {\widetilde{\pa}} \sigma|,\label{QJbound}
 \,
 \end{equation}
 \begin{equation}
 |Q^J|\lesssim c_J {\sum}_{|K|\leq |J|}
 |\widetilde{\pa} T^K V|+|\widetilde{\pa} T^K \widetilde{V}|
 + |\widetilde{\pa}T^K \sigma|
 +
 |\widetilde{\pa} T^K \widetilde{x}|
 + |\pave T^K \wGamma|+ |T^K \wGamma|.
 \label{PJbound}
 \end{equation}

 \subsubsection{Higher order elliptic equations for the enthalpy}
 \label{higherorderellipticrel}

 To close our estimates we will use
 the pointwise elliptic estimate \eqref{pwH},
 \begin{equation}
 |\widetilde{\pa} T^K \widetilde{\pa} \sigma|
 \lesssim |\widetilde{\div} \,T^K \widetilde{\pa} \sigma|+|_{\,}\widetilde{\curl}\,
 T^K \widetilde{\pa} \sigma|
 + {\sum}_{S\in{\mathcal S}}|S T^K \widetilde{\pa} \sigma|.
 \label{usefulpwbd}
 \end{equation}
 To control the divergence we write \eqref{higherorderwave} in the form
 \begin{equation}
 \widetilde{\nabla}^\mu\big( T^K \widetilde{\nabla}_\mu \sigma\big) =
 e'(\sigma)\hD_s^2 T^K\sigma
  -P^K - Q^K -(\pave_\mu g^{\mu\nu})T^K\pave_\nu\sigma
 \end{equation}
 and so
 \begin{equation}
  |\widetilde{\div} \,T^K \pave \sigma|
  \lesssim
  |\hD_s^2 T^K \sigma|
  + c_0| \hD_s T^K \pave \sigma|
  + |T^K \pave \sigma|
  + |P^K| + |Q^K|
  \label{}
 \end{equation}

 Arguing as in section \ref{sec:elliptcenthalpy} to control
 $\curl T^K \widetilde{\pa} \sigma$
 from \eqref{usefulpwbd} we arrive at
 \begin{multline}\label{eq:paTpah2rel}
 | \widetilde{\pa} T^{K} {\widetilde{\pa}}  \sigma|
 \\
 \lesssim c_r\!\!\!\!\!\!\sum_{|K'|\leq |K|}\!\!\!\!\Big(
 |\widehat{D}_s^2 T^{K'} \!\sigma|
 + |\pave \hD_s T^{K'}\!\sigma|
 + \!\sum_{S\in\mathcal{S}} |S T^{K'\!} \widetilde{\pa}\sigma|
 +\!|\widetilde{\pa} T^{K'\!} V|+|\widetilde{\pa} T^{K'\!} \widetilde{V}|+|\widetilde{\pa} T^{K'\!} \widetilde{x}|
 + |\pave T^{K'} \wg| + |\pave T^{K'} \wGamma| \Big).
 \end{multline}

 \subsubsection{Higher order wave equation equations for the enthalpy}
 \label{enthalpyhigherorderrel}
 Multiply \eqref{higherorderwave} by $D_s T^J \sigma$ and write
 \begin{multline}
  \widetilde{\nabla}^\mu\big( T^J{\widetilde{\nabla}}_\mu \sigma\big)
  D_s T^J \sigma
  = \widetilde{\nabla}_\nu
  \big(\wg^{\mu\nu} T^J \widetilde{\nabla}_\mu \sigma
  D_s T^J\sigma - \tfrac{1}{2}
  \widetilde{V}^\nu \widetilde{g}^{\alpha\beta}T^J\widetilde{\nabla}_\alpha
  \sigma T^J \widetilde{\nabla}_\beta \sigma\big)
  \\
  - \widetilde{g}^{\mu\nu} T^J \widetilde{\nabla}_\mu \sigma R_\nu^J
  +\tfrac{1}{2} \widetilde{\nabla}_\nu \widetilde{V}^\nu
  \widetilde{g}^{\alpha\beta}
  T^J\widetilde{\nabla}_\alpha \sigma
  T^J\widetilde{\nabla}_\beta \sigma
  \label{enthem1}
\end{multline}
where here
\begin{equation}
 R_\nu^J =\pave_\nu T^J\Dt \sigma -  T^J\pave_\nu
 \Dt \sigma
 = {\sum}_{J_1 +\cdots J_k = J, |J_k| < J}
 r^J_{J_1\cdots J_k}
 \pave T^{J_1}\xve \cdots \pave T^{J_{k-1}} \xve
 \cdot T^{J_k}\pave \sigma
\end{equation}
for constants $r^{J}_{J_1\cdots J_k}$, which is lower order
\begin{equation}
 R^J \lesssim
 c_J{\sum}_{|J'| \leq |J|}
 |\pave T^{J'}\xve| + |T^{J'}\pave \sigma|,
 \label{commutatorsigmabound}
\end{equation}
where $c_J$ depends on the above quantities for $|K| \leq |J|/2$.

We also have
\begin{equation}
 e'(\sigma) D_s^2 T^I\sigma D_s T^I\sigma =
 \frac{1}{2} \widetilde{\nabla}_\mu ( \widetilde{V}^\mu
 e'(\sigma) (D_s T^I\sigma)^2 )
 - \frac{1}{2}\widetilde{\nabla}_\mu( \widetilde{V}^\mu e'(\sigma))
 (D_s T^I\sigma)^2.
 \label{enthem2}
\end{equation}
Define the modified energy-momentum tensor
$q^I_{\mu\nu}= q^{I,1}_{\mu\nu}
+ q^{I,2}_{\mu\nu} $ where
\begin{equation}
q_{\mu\nu}^{I,1} = \pave_\mu T^I\sigma
T^I \pave_\nu \sigma
 -\frac{1}{2}\wg_{\mu\nu} \wg^{\alpha\beta} T^I\pave_\alpha\sigma
 T^I\pave_\beta\sigma,
\qquad
q_{\mu\nu}^{I,2} = e'(\sigma)\wg_{\mu\nu}|\hD_s \sigma^I|^2.
\label{sigmaemtensor}
\end{equation}
Note the positions of the derivatives $\pave_\mu, \pave_\nu$ in
the first term.
Adding \eqref{enthem1}-\eqref{enthem2}, integrating over the
region $\D^{\Sigma_1}_{\Sigma_0}$ as in the previous section
and using the divergence theorem, we find
\begin{multline}
 \int_{\Sigma_1} q^I(\widetilde{V}, n^{\Sigma_1})
 dS^{\Sigma_1}
 =
 \int_{\Sigma_0} q^I(\widetilde{V}, n^{\Sigma_0})
 dS^{\Sigma_0}
 \\
 +
 \int_{\D^{\Sigma_1}_{\Sigma_0}}
 \left(\widetilde{g}^{\mu\nu} T^J \widetilde{\nabla}_\mu \sigma R_\nu^J
 + \frac{1}{2} \widetilde{\nabla}_\nu \widetilde{V}^\nu
  \wg^{\alpha\beta}T^I \pave_\alpha \sigma T^I\pave_\beta \sigma
 - \frac{1}{2}\widetilde{\nabla}_\mu ( \widetilde{V}^\mu e'(\sigma))
 (D_s T^I\sigma)^2\right)
 d\mu_{\widetilde{g}} \\
 +
 \int_{\D^{\Sigma_1}_{\Sigma_0}} (|P^J| + |Q^J|)|D_s T^J \sigma|d\mu_{\widetilde{g}}
 \label{sigmawaveidentity}
\end{multline}
Here we have used that the boundary term on $\Lambda$ drops out, which follows since $\sigma$ is constant there and
 $\wg(\widetilde{V},\widetilde{\mathcal{N}}) = 0$.

Considering just the incompressible case $e'(\sigma) = 0$ for the moment,
the standard energy-momentum tensor associated to the wave equation for
$\sigma$ is
\begin{equation}
 Q_{\alpha\beta}[T^I\pa\sigma]
 = T^I \pave_\alpha \sigma T^I \pave_\beta \sigma
 - \frac{1}{2} \wg_{\alpha\beta} \wg^{\mu\nu}
 T^I\pave_\mu\sigma T^I\pave_\nu \sigma,
\end{equation}
and the difference $q^I(X, Y) - Q[T^I\pa\sigma](X, Y)$ is lower-order,
\begin{equation}
 |q^{I,1}(X,Y) - Q[T^I\pa\sigma](X, Y)|
 \lesssim
 |[\pave, T^I]\sigma| |T^I \pave \sigma|
 \lesssim
 c_I
 {\sum}_{|J|+|K| \leq |I|} |\pave T^J \xve| |T^K\pave \sigma|,
\end{equation}
and so in particular, since $\widetilde{V}$ is timelike and future-directed,
 by the positivity of the energy-momentum tensor
$Q$ from Lemma \ref{emtensorpositivity}, there are constants
$D_3. D_4 > 0$ depending on $\widetilde{V}$ and the spacelike surface
$\Sigma$ so that
\begin{equation}
   D_3 \int_{\Sigma}(1 + e'(\sigma) ) |T^I \hD_s\sigma|^2 + |T^I \pave \sigma|^2
   \leq \int_{\Sigma} Q[T^I \pave \sigma](\widetilde{V}, n^\Sigma)
   \leq
   D_4 \int_{\Sigma}(1 + e'(\sigma) ) |T^I \hD_s\sigma|^2 + |T^I \pave \sigma|^2.
   \label{Qcoersigma}
  \end{equation}
  and so $q^I(\widetilde{V}, n^{\Sigma})$ is positive-definite to
  highest-order,
\begin{equation}
  \int_{\Sigma} |T^I\pave \sigma|^2 + |\hD_s T^I\sigma|^2
  \lesssim
  \int_\Sigma |q^I(\widetilde{V},n^{\Sigma}) |
  + c_I
  {\sum}_{|J| \leq |I|} |\pave T^J \xve|^2
  + {\sum}_{|J| \leq |I|-1} |\pave T^J \sigma|^2.
  \label{highestordercoer}
\end{equation}

As in the estimates for the Euler equations it is convenient
to foliate the domain $\D$ into the spacelike surfaces
$\Sigma_s = x(s, \Omega)$ determined by the Lagrangian coordinates.
Expressing the integrals in Lagrangian coordinates, the energies are
\begin{equation}
 \W^J(s)
 = \int_{\Omega} |D_s T^J \sigma|^2 + |T^J \pave \sigma|^2\, \widetilde{\kappa} dy
 + \int_\Omega e'(\sigma)(\hD_s T^J\sigma)^2\,  \widetilde{\kappa} dy.
\end{equation}
Using \eqref{sigmawaveidentity}, the bounds \eqref{Qcoersigma},
\eqref{highestordercoer} for $\W^J$ along with the bounds
\eqref{commutatorsigmabound}, \eqref{PJbound}, \eqref{QJbound} for the terms on
the right-hand side of \eqref{sigmawaveidentity}, we have
 \begin{equation}
   \label{mainrelwaveest}
  \W^{J\!}(s_1) - \W^{J\!}(s_0)
  \lesssim
  C_J\!\! \!\sum_{|J'| \leq |J|}\int_{s_0}^{s_1}\!\!\! \int_{\Omega}
  |\pave T^{J'}V|^2\! + |\pave T^{J'}\widetilde{V}|^2\!
  + |\pave T^{J'}\xve|^2\! + |\hD_s T^{J'} \!\sigma|^2\!
  + |\pave T^{J'} \wg|^2 +  |\pave T^{J'} \wGamma|^2\!,
 \end{equation}
where
\begin{equation}
 \W^J(s)- C_J {\sum}_{|J'| \leq |J|-1}
 \W^{J'}(s) \gtrsim c_0 \int_\Omega
\left( (\widetilde{\Tau}^\mu T^J \pave_\mu\sigma)^2
 + \overline{g}^{\mu\nu} T^J\pave_\mu \sigma T^J\pave_\nu \sigma
 \right)
 +
 e'(\sigma) (\hD_s T^J\sigma)^2 \, \kappa dy.
\end{equation}

\subsubsection{Estimates for the enthalpy with an additional
time derivative and an additional fractional derivative}
The arguments in sections \ref{sec:extradt}-\ref{sec:extrahalf}
go through with very minor modifications and the result is that
if we define
\begin{equation}
 \W^{K,j}(s) =
 \int_{\Omega}\left( (\widetilde{\Tau}^\mu \hD_s^j T^K \pave_\mu\sigma)^2
 + \overline{g}^{\mu\nu} \hD_s^j T^K\pave_\mu \sigma \hD_s^j T^K\pave_\nu \sigma
 \right) \widehat{\kappa} dy
 + \int_\Omega e'(\sigma)(\hD_s \hD_s^j T^K\sigma)^2\,
 \widetilde{V}^\alpha \widetilde{\Tau}_\alpha\kappa  dy,
\end{equation}
then for $|K| = r-2$ we have
\begin{multline}
\mathcal{W}^{K,2{}_{\!}}(s)
\lesssim
\mathcal{W}^{K,2{}_{\!}}(0) +
C_0{\sum}_{|K^\prime|\leq |K|} \mathcal{W}^{K^\prime\!{}_{\!},2{}_{\!}}(s)
\\
+C_K\!\!\! \sum_{|{J^{\prime}}|\leq |K|+1}\int_0^s\!\!\int_{\D_{s'}}
 |\widetilde{\pa} T^{J{}^{\prime}\!} V|^2\!
 +|\widetilde{\pa} T^{J^{\prime}\!} \widetilde{V}|^2\!
 +|\widetilde{\pa} T
 ^{J^{\prime}}\! \widetilde{x}|^2\!+|D_t T^{J^{\prime}}\! \sigma|^2\!+| T^{J^{\prime}} \widetilde{\pa} \sigma|^2\!
 + |\pave T^{J'} g|^2 \!+ |\pave T^{J'} \Gamma|^2\, \kappa dy
 ds',
\end{multline}
and for $|J| = r-1$,
\begin{multline}\label{eq:waveequationenergytimeJrel}
\mathcal{W}^{J,1{}_{\!}}(s)\lesssim C_0{\sum}_{|K^\prime|\leq |J|-1} \mathcal{W}^{K^\prime\!{}_{\!},2{}_{\!}}(s)
\\
+C_K\!\!\!\sum_{|{J^{\prime}}|\leq |J|}\int_0^s\!\!\int_{\D_{s'}}
 |\widetilde{\pa} T^{J{}^{\prime}\!} V|^2\!
 +|\widetilde{\pa} T^{J^{\prime}\!} \widetilde{V}|^2\!
 +|\widetilde{\pa} T
 ^{J^{\prime}}\! \widetilde{x}|^2\!+|D_t T^{J^{\prime}}\! \sigma|^2\!+| T^{J^{\prime}} \widetilde{\pa} \sigma|^2\!
 + |\pave T^{J'} \wg|^2 \!+ |\pave T^{J'} \wGamma|^2\, \kappa dy ds'.
\end{multline}
The estimate with an additional half-derivative follows from
Proposition \ref{eq:thewholecasedirichletrel} and reads
\begin{multline}\label{eq:pahalfTKparel0}
\vertiii{\widetilde{\pa} \fdh  T^K \widetilde{\pa} \sigma(s)}_{L^2(\Omega)}
\\
\lesssim C_K  {\sum}_{|J|\leq |K|+1}\|T^J \tr_{\oH}\pave^2 \sigma\|_{L^2(\Omega)}
+\|T^J \hD_s \pave \sigma\|_{L^2(\Omega)}
+ C_K{\sum}_{|J|\leq |K|+1}
\|\widetilde{\pa}\fdh  T^{J} \widetilde{x}(s)\|_{L^2(\Omega)}.
\end{multline}

We now want to control the term $T^J \tr_{\oH}\pave^2 \sigma$ on the
right-hand side and the idea is to relate $\tr_{\oH} \pave^2$
to $g^{\mu\nu}\pave_\mu\pave_\nu$ since we have an
equation for $g^{\mu\nu}\pave_\mu\pave_\nu\sigma$ involving lower-order
terms and material derivatives.
To get simpler notation we let $\wTau$ denote the unit future-directed
timelike normal to the surfaces of constant $s$ defined relative
to the metric $\wg$,
\begin{equation}
 \wTau^\mu = \wg^{\mu\nu} \pave_\nu s/ (-\wg(\widetilde{\nabla} s,\widetilde{\nabla} s))^{1/2}
 \qquad
 \widetilde{V}_{\wsubTau} = \wg(\widetilde{V}, \wsubTau).
 \label{}
\end{equation}

We recall the decomposition from \eqref{waveexpression}
which in this setting reads
\begin{equation}
  \wg^{\mu\nu}\pave_\mu\pave_\nu\sigma =
  -\frac{1}{(\widetilde{V}_{\wsubTau})^2}D_s^2\sigma + \oH^{\mu\nu}\pave_\mu\pave_\nu \sigma
  +2D_s\pave_{W} \sigma + L^\mu\pave_\mu \sigma, \qquad
  \text{ where }
  L^\mu =
 -\frac{1}{(\widetilde{V}_{\wsubTau})^2}
 +
 (\hD_s \widetilde{V}^\mu)
 \hD_s W^\mu,
\end{equation}
where here
\begin{equation}
 W^\mu = \frac{1}{\widetilde{V}_{\wsubTau}} \sg^{\mu\nu} \widetilde{V}_\nu,
 \qquad
 \pave_W = W^\mu \pave_\mu,
 \label{}
\end{equation}

It follows that
\begin{equation}
 T^J( \oH^{\mu\nu} \pave_\mu\pave_\nu \sigma)
 =T^J \widetilde{\nabla}_\mu \widetilde{\nabla}^\mu \sigma
 +
 \frac{1}{(\wTau_\mu \widetilde{V}^\mu)^2} T^J
 \hD_s^2 \sigma
 - 2D_s T^J\pave_{W} \sigma
 + B^J,
\end{equation}
where $B^J$ is lower order,
\begin{multline}
 B^J =
 \sum_{J_1 + J_2 = J, |J_2| \leq |J|}
 T^{J_1}\Gamma^\mu_{\mu\alpha} T^{J_2} \widetilde{\nabla}_\alpha \sigma\\
 -\sum_{J_1 + J_2 = J, |J_2| \leq |J|-1}
 T^{J_1} \left((\widetilde{V}_{\wsubTau})^{-2}\right)
T^{J_2}\hD_s \sigma
+\sum_{J_1 + J_2 = J, |J_2| < |J|}
(T^{J_1} L^\mu) T^{J_1}\pave_\mu \sigma
- L^\mu T^J\pave_\mu \sigma.
\end{multline}
Writing \eqref{sigmawavesetup1} in the form
\begin{equation}
 \widetilde{\nabla}_\mu\widetilde{\nabla}^\mu \sigma
 = 2 e'(\sigma) \hD_s^2\sigma - \widetilde{\nabla}_\mu \widetilde{V}^\nu
 \widetilde{\nabla}_\nu \widetilde{V}^\mu
 -
2\widetilde{R}_{\mu\nu\alpha}^\mu \widetilde{V}^\nu V^\alpha
 + e''(\sigma) (\hD_s\sigma)^2,
\end{equation}
we find the identity
\begin{equation}
 T^J (\oH^{\mu\nu} \pave_\mu\pave_\nu \sigma)
 = \! 2 T^J (e'(\sigma) \hD_s^2 \sigma)
\!-\!
 \frac{1}{(\wTau_\mu \widetilde{V}^\mu)^2} T^J
 \hD_s^2 \sigma
 - 2D_s T^J\pave_{W} \sigma
 + B^{\prime J},
\end{equation}
where $B^{\prime J}$ is lower order,
\begin{equation}
 B^{\prime J} = B^J
 - T^J\left(\widetilde{\nabla}_\mu \widetilde{V}^\nu
 \widetilde{\nabla}_\nu \widetilde{V}^\mu\right)
 -
2T^J\left(\widetilde{R}_{\mu\nu\alpha}^\mu \widetilde{V}^\nu V^\alpha\right)
 + T^J\left(e''(\sigma) (\hD_s\sigma)^2\right).
\end{equation}
and using that $\mathcal{W}^{J,1}$ controls material derivatives, we therefore
have
\begin{multline}\label{eq:pahalfTKparel}
\vertiii{\widetilde{\pa} \fdh  T^K \widetilde{\pa} \sigma(s)}_{L^2(\Omega)}
\lesssim C_K  {\sum}_{|J|\leq |K|+1}\mathcal{W}^{J,1{}_{\!}}(s)
+ C_K{\sum}_{|J|\leq |K|+1}
\|\widetilde{\pa}\fdh  T^{J} \widetilde{x}(s)\|_{L^2(\Omega)}\\
+ C_K {\sum}_{|J|\leq |K|+1}
 \|\widetilde{\pa} T^{J} V(s)\|_{L^2(\Omega)}
 +\|\widetilde{\pa} T^{J} \widetilde{V}(s)\|_{L^2(\Omega)}
 +\|\widetilde{\pa} T^{J} \widetilde{x}(s)\|_{L^2(\Omega)}\\
 + C_K {\sum}_{|J| \leq |K|} \|\fdh T^K \pave g(s)\|_{L^2(\Omega)} + \|\fdh T^K \pave \Gamma(s)\|_{L^2(\Omega)}.
\end{multline}

\subsection{The divergence estimates for the relativistic velocity and coordinate}
\label{reldivergeneestimatesection}

From \eqref{eq:GIequationrel} we have
\begin{equation}
 D^J = \pave_\mu T^J V^\mu
 +e'(\sigma) \hD_s T^J \sigma
 - \pave_\mu T^J \xve^\nu \pave_\nu V^\mu = G^J.
 \label{DJident}
\end{equation}
where $G^J$ is lower order, see \eqref{eq:GIestimaterel}.

\subsubsection{The improved half derivative divergence estimates used to estimate the coordinates}

In order to get the improved half-derivative estimate for the coordinate
the idea is to use the elliptic estimate from Lemma \ref{prop:divcurlL2rel}.
This is slightly different from what we encountered in the Newtonian case
since we need to write the spacetime divergence in terms of the metric
$\oH$ defined in \eqref{Gdef}. By the decomposition formula \eqref{Gdecomp}
there is a simple relationship between the spacetime divergence, the
divergence with respect to the Riemannian metric $\oH$ and the material
derivatives $\hD_s$ which are easier to control.

We first write \eqref{DJident} in the form
\begin{equation}
 \hD_s \big( e'(\sigma) T^J \sigma + \pave_\mu
 T^J x^\mu\big)
 = \pave_\mu T^J \xve^\nu \pave_\nu V^\mu
 - \pave_\mu T^J x^\nu \pave_\nu \widetilde{V}^\mu + G^J.
\end{equation}
In terms of the quantity
\begin{equation}
 X^J_\mu = \wg_{\mu\nu} T^J x^\nu,
\end{equation}
this says
\begin{equation}
  \hD_s \big( e'(\sigma) T^J \sigma + \wg^{\mu\nu}\pave_\mu X^J_\nu\big)
  = \pave_\mu T^J \xve^\nu \pave_\nu V^\mu
  - \pave_\mu T^J x^\nu \pave_\nu \widetilde{V}^\mu + G^{\prime J},
\end{equation}
where
\begin{equation}
 G^{\prime J} = G^J  + (\hD_s \wg^{\mu\nu}) \pave_\mu X^J_\nu
 + \hD_s\big( (\pave_\mu \wg^{\mu\nu}) X_\nu^J\big).
\end{equation}
Recall the decomposition of the divergence in terms of $\oH$ and components
parallel to $\widetilde{V}$ from \eqref{divexpression}
\begin{equation}
  \wg^{\mu\nu}\pave_\mu X^J_\nu =
  \oH^{\mu\nu}\pave_\mu X^J_\nu
  +\Big(
  W^\mu -\frac{\widetilde{V}^\mu}{{\widetilde{V}_{\wsubTau}\,}^{\!\!\!2}}\Big)
  \hD_s X_\mu^J
  + \Omega^{\mu\nu}\widetilde{\curl} X^J_{\mu\nu},
  \quad \text{ where } \Omega^{\mu\nu} =
 \frac{1}{2\widetilde{V}_{\wsubTau}}
  \left(\widetilde{V}^\mu W^\nu + \widetilde{V}^\nu W^\mu\right),
\end{equation}
and where $W^\mu = \frac{1}{\widetilde{V}_{\wsubTau}} \sg^{\mu\nu}\widetilde{V}_\nu$.
Using that $D_s X^J_\mu = \wg_{\mu\nu} T^J V^\nu + \hD_s \wg_{\mu\nu}T^J x^\nu$
we find
\begin{multline}
 \hD_s \left( e'(\sigma) T^J\! \sigma +
\oH^{\mu\nu}\pave_\mu X^J_\nu\! + A^\mu T^J V_\mu +(\hD_s \wg_{\mu\nu})A^\mu T^J\! x^\nu\!
+
\Omega^{\mu\nu} \widetilde{\curl} X^J_{\mu\nu}\right)\\
  = \pave_\mu T^J \xve^\nu \pave_\nu V^\mu\!
  - \pave_\mu T^J x^\nu \pave_\nu \widetilde{V}^\mu + G^{\prime J}\!\!,
\end{multline}
with $A^\mu =
W^\mu -{\widetilde{V}^\mu}/{{\widetilde{V}_{\wsubTau}\,}^{\!\!\!2}}$.

If we define
\begin{equation}
  D^{J, \nicefrac{1}{2}}_\ve\!\!
  = \fdh S_\ve (\oH^{\mu\nu}\pave_\mu X^J_\nu)+
  e'(\sigma) \fdh S_\ve T^J\! \sigma
   + \fdh S_\ve\!\big( A^\mu T^J V_\mu\!+\hD_s \wg_{\mu\nu}\,A^\mu T^J\! x^\nu\!\!
   +
   \Omega^{\mu\nu} \widetilde{\curl} X^J_{\mu\nu}\big),
  \label{DJepsdefrel}
\end{equation}
then arguing as in section \ref{sec:improvediv} to control the
error terms, we have
\begin{multline}
 \|\hD_s D^{J,\nicefrac{1}{2}}_\ve\|_{L^2}\\
 \lesssim
 C_2 \|\pave T^J \fdh S_\ve x\|_{L^2}
 + C_J {\sum}_{|J'| \leq |J|}
 \|\pave T^{J'}V\|_{L^2} +
 \|\pave T^{J'}x\|_{L^2}
 + C_J {\sum}_{|J'|\leq |J|}
 \|\fdh T^{J'}\wGamma\|_{L^2}
\end{multline}
with $L^2 = L^2(\Omega)$,
and
\begin{multline}
 \|\oH^{\mu\nu} \pave_\mu (\fdh S_\ve X^J_\nu)
 -\Omega^{\mu\nu} \fdh  \widetilde{\curl} S_\ve X^J_{\mu\nu}
 - D^{J,\nicefrac{1}{2}}_{\ve} \|_{L^2}\\
 \lesssim
 C_J\!\!\!\!\! \sum_{|J'| \leq |J|}\!\!\!
 \|T^{J'}\!\pave \sigma\|_{L^2}
 +
 \|T^{J'}\!\pave x\|_{L^2}
 +
 \|T^{J'}\!\pave V\|_{L^2}
 +
 C_J \!\!\!\!\!\sum_{|J'|\leq |J|}\!\!\!
 \|\fdh T^{J'}\wGamma\|_{L^2}
 + \|\fdh T^{J'}\! \wg\|_{L^2}.
 \label{Dvediffrelbd}
\end{multline}
The point of this estimate is that the quantity
$\oH^{\mu\nu} \pave_\mu (\fdh S_\ve X^J_\nu)$ appears on the
right-hand side of the elliptic estimate \eqref{ftang1G} applied to $X =
\fdh S_\ve X^J$. We will also separately control the term $\widetilde{\curl}
S_\ve X^J$ in the next section so the previous two bounds control
$\oH^{\mu\nu} \pave_\mu (\fdh S_\ve X^J_\nu)$.

\subsection{The curl estimates for the relativistic
velocity and coordinates}
\subsubsection{The curl estimates used to estimate $V$}

Multiplying both sides of \eqref{rescaledreleul2smoothed} by $\wg_{\mu\nu}$
and then applying $T^J$, we have
\begin{equation}
 \wg_{\mu\nu} \hD_s  T^J V^\nu  = -\frac{1}{2}T^J\pave_\mu \sigma
- \wg_{\mu\nu} (T^J \wGamma^\nu_{\alpha\beta}) \widetilde{V}^\alpha V^\beta - R_\mu^J,
\end{equation}
where $R_\mu^J$ is given by
\begin{equation}
 R_\mu^J =
\sum_{\substack{J' + J_1 + J_2 + J_3 = J,\\ |J'| < |J|}}
 (T^{J'} \wGamma^\nu_{\alpha\beta})
 (T^{J_1}\wg_{\mu\nu})
 (T^{J_1}\widetilde{V}^\alpha)
 (T^{J_2} V^\beta)
 + \sum_{\substack{J_1 + J_2 = J,\\ |J_2| < |J|}} (T^{J_1} \wg_{\mu\nu})( \hD_s T^{J_2}V^\nu).
\end{equation}
By the symmetry of the Christoffel symbols we have that
$\wg_{\mu\mu'} \pave_\nu T^J\wGamma^{\mu'}_{\alpha\beta}
 -
 \wg_{\nu\nu'} \pave_\mu T^J\wGamma^{\nu'}_{\alpha\beta}$ is lower-order
and
it follows that
 \begin{equation}
  \pave_{\mu} \hD_s( \wg_{\nu\nu'} T^J V^{\nu'}) -
  \pave_\nu \hD_s (\wg_{\mu\mu'}T^J V^{\mu'})
  = A_{\mu\nu}^J
 \end{equation}
 where $A_{\mu\nu}^J =
 \pave_\mu T^J \pave_\nu\sigma -
 \pave_\nu T^J \pave_\mu \sigma + \pave_\mu R^J_\nu - \pave_\nu R^J_\nu
 +\wg_{\mu\mu'} \pave_\nu T^J\wGamma^{\mu'}_{\alpha\beta}
  -
  \wg_{\nu\nu'} \pave_\mu T^J\wGamma^{\nu'}_{\alpha\beta}$ is lower order,
 \begin{equation}
  |A^J_{\mu\nu}|
  \lesssim
  c_0 |\pave T^J \xve| + c_J\!\!\!\!\!\!
  \sum_{|K| \leq |J|-1}\!\!\!\!\!
  |\pave T^K\pave \sigma| +
  \sum_{S \in \S}
  |ST^K \pave \sigma| + |\pave ST^K \sg|+ |\pave T^K V| + |\pave T^K \widetilde{V}| + |\pave T^K \xve|
  +|\pave T^K \wGamma|.
 \end{equation}
 Following the same steps as in section
 \ref{sec:curlestim} we find that there are
 linear forms $L^1_{\mu\nu}[\pave T^J \xve],
 L^2_{\mu\nu}[\pave T^J \xve]$ so that defining
 \begin{equation}
  K_{\mu\nu}^J = \pave_\mu (g_{\nu\nu'} T^J V^{\nu'}) - \pave_\nu
  (g_{\mu\mu'} T^J V^{\mu'})
  +L_{\mu\nu}^1[\pave T^J x],
  \label{eq:modifiedcurlrel}
 \end{equation}
 we have
 \begin{equation}
  \hD_s^2 K^J_{\mu\nu} = L_{\mu\nu}^2[\pave T^J x] -
  A^J_{\mu\nu}.
 \end{equation}
 Further, there is a linear form $L^3_{\mu\nu}[\pave T^J x]$
 so that
 \begin{equation}
  \hD_s (\widetilde{\curl} T^Jx)_{\mu\nu}
  = K_{\mu\nu}^J + L_{\mu\nu}^3[\pave T^J x].
 \end{equation}
 Here,
 \begin{equation}
  (\widetilde{\curl} T^J x)_{\mu\nu}
  = \pave_\mu (g_{\nu\nu'} T^J x^{\nu'})
  - \pave_\mu (g_{\mu\mu'} T^J x^{\mu'}).
  \label{curlTJx}
 \end{equation}

\subsection{The improved half derivative curl
estimates used to estimate the coordinates}
The argument in section \ref{sec:improvedcurl} also
goes through in the relativistic case with only superficial
changes. The result is that with
\begin{equation}
 K_{\mu\nu,\ve}^{J,\nicefrac{1}{2}}
 = \widetilde{\curl} (T^J \fdh S_\ve V)_{\mu\nu}
 + L_{\mu\nu}^1[\pave \fdh T^J S_\ve x],
\end{equation}
we have
\begin{equation}
 \hD_s K_{\mu\nu,\ve}^{J,\nicefrac{1}{2}}
 = L_{\mu\nu}^2[\pave \fdh T^J S_\ve x] - A_{\mu\nu,\ve}^{J, \nicefrac{1}{2}},
\end{equation}
where
\begin{equation}
 A_{\mu\nu,\ve}^{J,\nicefrac{1}{2}}
 = \pave_\mu \fdh S_\ve T^J \pave_\nu \sigma
 -\pave_\nu \fdh S_\ve T^J \pave_\mu \sigma,
\end{equation}
is lower-order,
\begin{equation}
 \|A_{\mu\nu,\ve}^{J, \nicefrac{1}{2}}\|_{L^2(\Omega)}
 \lesssim
 C_0\|\pave \fdh S_\ve T^J \xve\|_{L^2(\Omega)}
 + C_J\!\!\!\! \sum_{|K| \leq |J|-1}\!\!\!\!
 \|\pave \fdh T^K \pave \sigma\|_{L^2(\Omega)}
 + \|\pave \fdh S_\ve T^K \xve\|_{L^2(\Omega)}.
\end{equation}
Moreover
\begin{equation}\label{eq:curlequation2rel}
\hD_s{}_{\,}\widetilde{\curl}\big( {}_{\!} \fdh T^J \!\sm x\big)_{\!\mu\nu}\!
=K_{\mu\nu,\ve}^{J,\nicefrac{1}{2}}+L^3_{\mu\nu}
[\widetilde{\pa}\fdh T^J \!\sm x].
\end{equation}
with notation as in \eqref{curlTJx}.

Note also that by
Lemma \ref{lem:mutiplicationsmoothingcommute}, Lemma \ref{lem:gradientsmoothingcommute} and Lemma \ref{lem:halfderleibnitz}
\begin{equation}
\bigtwo\| \widetilde{\curl}\big(  T^J \!\fdh\sm V\big)-\fdh\sm{}_{\,}\widetilde{\curl}\big(  T^J  V\big)\bigtwo\|_{L^2(\Omega)}
\lesssim C_0\| \widetilde{\pa} T^{J}V\|_{L^2(\Omega)}.
\end{equation}

\subsubsection{The improved half derivative curl estimates used to estimate the coordinates}
We need to commute \eqref{rescaledreleul2smoothed} with $\sm$ and with $\fdh$. We have
\begin{equation}
\hD_{s\,}\fdh \sm  T^J  V^\mu =-\fdh \sm  T^J \widetilde{\nabla}^\mu \sigma,
\end{equation}
and hence
\begin{equation}
 \widetilde{\curl}\bigtwo(\!  \hD_s \fdh \sm  T^J  V{}_{\!}\bigtwo)_{\mu\nu}= -A_{\mu\nu,\ve}^{J, \nicefrac{1}{2}},
\end{equation}
where
\begin{equation}
2 A_{\mu\nu,\ve}^{J, \nicefrac{1}{2}}
=\widetilde{\pa}_\mu\left( \wg_{\nu\nu'} \fdh \sm  T^J \widetilde{\nabla}^{\nu'} \sigma\right)
-\widetilde{\pa}_\nu \left( \wg_{\mu\nu'} \fdh \sm  T^J \widetilde{\nabla}^{\mu'} \sigma\right).
\end{equation}
With
\begin{equation}\label{eq:curlhalfrel}
K_{\mu\nu,\ve}^{J,\nicefrac{1}{2}}=\widetilde{\curl}\big(T^J \fdh \sm \, V\big)_{\!\mu\nu}\!+L_{\mu\nu}^1\big[\widetilde{\pa}\fdh T^J\!\sm   x\big],
\end{equation}
we have
\begin{equation}\label{eq:curlhalfestrel}
\hD_s K_{\mu\nu,\ve}^{J,\nicefrac{1}{2}}
=L_{\mu\nu}^2\big[\widetilde{\pa} \fdh T^J\! \sm x\big]-A_{\mu\nu,\varepsilon}^{J,\nicefrac{1}{2}}.
\end{equation}
Here
$ A_{\mu\nu,\varepsilon}^{J,\nicefrac{1}{2}}$ is lower order,
\begin{equation}
 A_{\mu\nu,\varepsilon}^{J,\nicefrac{1}{2}}=\fdh \sm  A_{\mu\nu}^{J}+
 \frac{1}{2}\big[\widetilde{\pa}_\mu,\fdh \sm\big]T^J\wg_{\nu\nu'}\widetilde{\nabla}^{\nu'}\sigma
 -\frac{1}{2}\big[\widetilde{\pa}_\nu,\fdh \sm\big]T^J
 \wg_{\mu\mu'}\widetilde{\nabla}^{\mu'}\sigma .
\end{equation}
Arguing as in section \ref{sec:improvedcurl} we find
\begin{equation}
\|A_{\mu\nu,\varepsilon}^{J,\nicefrac{1}{2}}\|_{L^2(\Omega)}\lesssim C_0 \|\widetilde{\pa}\fdh \sm T^J \widetilde{x}\|_{L^2(\Omega)}+ C_J\!\!\!\!\sum_{|K|\leq |J|-1}\!\!\!\!
\|\widetilde{\pa}\fdh T^{K} {\widetilde{\pa}} \sigma\|_{L^2(\Omega)}+
\|\widetilde{\pa}\fdh \sm T^K\widetilde{x}\|_{L^2(\Omega)}.
\end{equation}

Moreover
\begin{equation}\label{eq:curlequation2rel2}
\hD_s{}_{\,}\widetilde{\curl}\big( {}_{\!} \fdh T^J \!\sm x\big)_{\!\mu\nu}\!
=K_{\mu\nu,\ve}^{J,\nicefrac{1}{2}}+L^3_{\mu\nu}[\widetilde{\pa}\fdh T^J \!\sm x].
\end{equation}
Here,
\begin{equation}
 \widetilde{\curl}\big( {}_{\!} \fdh T^J \!\sm x\big)_{\!\mu\nu}
= \pave_\mu (\wg_{\nu\nu'}\fdh T^J \!\sm x^{\nu'})-
 \pave_\nu (\wg_{\mu\mu'}\fdh T^J \!\sm x^{\mu'}).
\end{equation}

Note also that by
Lemma \ref{lem:mutiplicationsmoothingcommute}, Lemma \ref{lem:gradientsmoothingcommute} and Lemma \ref{lem:halfderleibnitz}
\begin{equation}
\bigtwo\| \widetilde{\curl}\big(  T^J \!\fdh\sm V\big)-\fdh\sm{}_{\,}\widetilde{\curl}\big(  T^J  V\big)\bigtwo\|_{L^2(\Omega)}
\lesssim C_0\| \widetilde{\pa} T^{J}V \|_{L^2(\Omega)}.
\end{equation}

\subsection{The elliptic estimates}
\label{ellipticestsecrel}
\subsubsection{The elliptic estimate for the velocity}
By Lemma \ref{pwHlemma},
\begin{equation}
 |\pave T^J V| \lesssim
 |\widetilde{\div} T^J V| + |\widetilde{\curl} T^J V|
 + {\sum}_{S \in \S} |S T^J V|,
 \label{}
\end{equation}
and so with $D^J$ defined as in
\eqref{DJident}
and $K^J$ as in \eqref{eq:modifiedcurlrel}, we
have
\begin{equation}
 |\pave T^J V| \lesssim |D^J| + |K^J|
 + c_0 \big( |\pave T^J \xve| + |\pave T^J x|\big)
 + |\hD_s T^J \sigma|
 + {\sum}_{S \in \S}
 |ST^J V|.
\end{equation}
From the formula \eqref{DJident} $D^J$ is lower order,
\begin{equation}
 |D^J| \lesssim
 c_J {\sum}_{|K| \leq |J|-1} |\pave T^K\xve|
 + |\pave T^K V|
 + |T^K \wGamma|
\end{equation}
where $c_J$ is a constant depending on
$|\pave T^L\xve| + |\pave T^L V| + |T^K\wGamma|$ for $|L| \leq |J|/2$.
Therefore
\begin{equation}
 |\pave T^J V| \lesssim |K^J|
 +c_0\big(|\pave T^J \xve| + |\pave T^J x|\big)
 + |\hD_s T^J \sigma|
 + \sum_{S \in \S}
 |ST^J V|
 + c_J\sum_{|K| \leq |J|-1} |\pave T^K\xve|
 + |\pave T^K V|
 +
 |T^K\wGamma|.
\end{equation}

\subsection{The elliptic estimate for the enthalpy}
From \eqref{eq:paTpah2rel} we have
\begin{equation}
 \sum_{|K| \leq r}
 |\pave T^K \pave \sigma| \lesssim c_r
 \sum_{|K| \leq r}
 \bigg(|\hD_s^2 T^K \sigma|
 + |\pave T^{K}\wg| + |\pave T^K \wGamma|
  + \sum_{S \in \S}
  |S T^K \pave \sigma|
  + |\pave T^K V| + |\pave T^K \widetilde{V}|
  + |\pave T^K \xve|\bigg).
\end{equation}

\subsection{The additional elliptic estimate for the
smoothed coordinate $S_\ve x$}
\label{additionalrelxve}
Applying the elliptic estimate from Proposition \ref{prop:divcurlL2rel}, to
$X^{J,\nicefrac{1}{2}}_{\ve,\nu} =
 \wg_{\mu \nu}(T^J \fdh S_\ve x^\mu)$
 and writing
 $X^{J}_{\ve,\nu} = \wg_{\mu\nu} T^J S_\ve x^\nu$,
we find
\begin{multline}
  \sum_{|J| \leq r-1}
  \vertiii{\pave X^{J,\nicefrac{1}{2}}_\ve }_{L^2(\Omega)}^2
  + \sum_{|I| \leq r}
  \vertiii{X^I}_{L^2(\pa \Omega)}^2
  \\
  \leq C_1\!\!
  \sum_{|I| \leq r}\!
   \vertiii{\widetilde{\n} \cdot_{\oH} X^I_\ve}_{L^2(\pa\Omega)}^2\!
   + C_1\!\!\!\!\!
   \sum_{|J| \leq r-1}\!\!\!
   \|\div_{\oH} X^{J,\nicefrac{1}{2}}_\ve\|_{L^2(\Omega)}^2\!
   +\|\curl X^{J,\nicefrac{1}{2}}_\ve\|_{L^2(\Omega)}^2\!
   + \|\pave X^J\|_{L^2(\Omega)}^2.
 \label{usefracestirel}
\end{multline}
Recall that $\widetilde{n}$ denotes the spacelike unit conormal to $\Omega$
at constant $s$ and $\widetilde{\n} \cdot_{\oH} X^I_\ve
= \oH^{\mu\nu} \widetilde{n}_\mu X^I_{\ve,\nu}$. We are also writing
$\div_{\oH}$ for the divergence with respect
to the Riemannian metric $\oH$ (see \eqref{divgdef}). In the upcoming sections we will
use evolution equations
to control the term involving the divergence and the curl on the right-hand side
of \eqref{usefracestirel}.
In the Newtonian case (see section \ref{additionalxve}) the boundary
term we encountered when using the corresponding estimate was already
directly controlled by the energy. However in this case the
two boundary terms are different and so we must show
that the boundary term in \eqref{usefracestirel}
is related to the boundary term in the definition of the energy.

Before doing this we note that \eqref{usefracestirel} in fact implies
a bound for all components of all derivatives of $X^{J,\nicefrac{1}{2}}$ in the interior
provided we also control $\fdh T^J V$,
\begin{multline}
\sum_{|J| \leq r-1}
  \|\pave X^{J,\nicefrac{1}{2}}_\ve\|_{L^2(\Omega)}^2
  \leq C_1
  \sum_{|I| \leq r}
   \vertiii{\widetilde{\n} \cdot_{\oH} X^I_\ve}_{L^2(\pa\Omega)}^2
   \\
   +C_1\!\!\!\!\!
   \sum_{|J| \leq r-1}\!\!\!
   \|\div_{\oH}X^{J,\nicefrac{1}{2}}_\ve\|_{L^2(\Omega)}^2
   +\|\curl X^{J,\nicefrac{1}{2}}_\ve\|_{L^2(\Omega)}^2
   + \|\pave X^J_\ve\|_{L^2(\Omega)}^2
   +\|\fdh T^JS_\ve V\|_{L^2(\Omega)}^2\\
   + C_{r-1}\!\! \sum_{|I| \leq r}
   \|\fdh T^I \wg \|_{L^2(\Omega)}^2.
 \label{usefracestirel2}
\end{multline}
This follows because the only terms missing from the left-hand side of
\eqref{usefracestirel} can be controlled if we control the components along
$\widetilde{V}$ and the curl. Using the decomposition \eqref{Gdecomp} to decompose
$H$ into components parallel to $\widetilde{V}$
and components in the image of
$\oH^{\mu}_\nu = \wg_{\mu\nu} \oH^{\nu'\nu}$,
we find
\begin{multline}
  \int_\Omega H^{\mu\nu} H^{\alpha\beta}
  \pave_\mu X^{J,\nicefrac{1}{2}}_{\ve, \alpha}
  \pave_\nu X^{J,\nicefrac{1}{2}}_{\ve, \beta}\, \kappa_G dy
  \\
 \lesssim
 \int_\Omega \widetilde{V}^\mu \widetilde{V}^\nu \widetilde{V}^\alpha
 \widetilde{V}^{\beta}
 \pave_\mu X^{J,\nicefrac{1}{2}}_{\ve, \alpha}
 \pave_\nu X^{J,\nicefrac{1}{2}}_{\ve,\beta}\, \kappa_G dy
 +\int_\Omega \oH^{\mu\nu} \oH^{\alpha\beta}
 \pave_\mu X^{J,\nicefrac{1}{2}}_{\ve,\alpha}
 \pave_\nu X^{J,\nicefrac{1}{2}}_{\ve,\beta}\, \kappa_G dy\\
 +\int_\Omega H^{\mu\nu}H^{\alpha\beta}
 \curl X^{J,\nicefrac{1}{2}}_{\ve, \mu \alpha}
 \curl X^{J,\nicefrac{1}{2}}_{\ve, \nu \beta}
 \, \kappa_G dy,
\end{multline}
and since $\widetilde{V}^\mu\pave_\mu X^{J,\nicefrac{1}{2}}_{\ve,\nu} =
\hD_s X^{J,\nicefrac{1}{2}}_{\ve,\nu} = \hD_s( \wg_{\mu\nu} T^JS_\ve x^\mu)$
and $\hD_S S_\ve x = S_\ve V$ we control the first term on the
right-hand side here by \eqref{usefracestirel2}.

To control the boundary term from \eqref{usefracestirel},
we start by writing it in terms of
boundary term appearing in the energy estimate \eqref{eq:energydefrel}
Let $q:[0,S]\times \Omega \to \R$ be any function satisfying
$q(s,y) < 0$ whenever $y$ is close to $\pa \Omega$ and with
$q(s,y) = 0$ whenever $y \in \pa\Omega$. Then
the conormal to the spacetime surface $[0,S]\times \pa\Omega$ is parallel
to $\pave_\mu q$. Also, for each
fixed value of $s = s'$, the conormal $\widetilde{\n}$ to the surface $
\{s = s'\}\cap \pa \Omega$ is parallel to $P^\nu_\mu \pave_\nu q$
where $P$ is the projection to the tangent space of $\{s = s'\}\cap
\pa \Omega$, given by $P_\mu^\nu = \delta_\mu^\nu +
\wTau^\nu \wTau_\mu$. In particular, we note that
\begin{equation}
 G^{\mu\nu} \pave_\mu q
  = \sg^{\mu\nu}\pa_\mu q -
  \frac{ \sg^{\mu'\nu} \widetilde{V}_{\mu'}}{(\widetilde{V}^\mu \wTau_\mu)^2}
  \sg^{\mu\nu'} \widetilde{V}_{\nu'}
  \pa_\mu q
  = G^{\mu\nu} P_{\mu}^{\mu'} \pave_{\mu'}q,
\end{equation}
so the component of $\pave q$ parallel to $\wTau$ drops out of
$G^{\mu\nu}\pave_\mu q$ and it follows that
\begin{equation}
  G^{\mu\nu} \widetilde{\n}_\mu X^I_{\ve,\nu} = \lambda G^{\mu\nu} \widetilde{\N}_\mu
  X^I_{\ve,\nu},
\end{equation}
for a function $\lambda$. Now we decompose $G$ in terms of $g$ and $V$ using
the formula \eqref{Gdecomp}, which gives
\begin{multline}
 \oH^{\mu\nu} \widetilde{\N}_\mu X^I_{\ve,\nu}\\
 = \widetilde{\N}_\mu T^I S_\ve x^\mu
 +\frac{1}{(\wTau_\mu \widetilde{V}^\mu)^2}\widetilde{V}^\mu \widetilde{V}_\nu
 \widetilde{\N}_\mu T^I S_\ve x^\nu
 - \frac{1}{\wTau_\mu \widetilde{V}^\mu}
 \left( (\sg_{\nu'\nu} \widetilde{V}^{\nu'})\widetilde{V}^\mu
 + (P_{\mu'}^\mu \widetilde{V}^{\mu'})\widetilde{V}_\nu \right)\widetilde{\N}_\mu T^I S_\ve x^\nu.
 \\
 =\widetilde{\N}_\mu T^I S_\ve x^\mu
 - \frac{P_{\mu'}^\mu \widetilde{V}^{\mu'}\widetilde{\N}_\mu}{\wTau_\mu \widetilde{V}^\mu}
 \left(\widetilde{V}_\nu T^I S_\ve x^\nu\right),
\end{multline}
where we used the boundary condition $\widetilde{\N}_\mu \widetilde{V}^\mu =0$.
The first term is what appears in the boundary term in the definition of the energy.
To control the second term the idea is to first control it by an interior term
and then to use that we control all components of the curl.
In what follows we can assume at least one of the vector fields $T^I$ appearing
in the definition of $X^J_{\ve,\nu}$ is spatial, $T^I = S T^{J}$ since
otherwise we can just use that $\hD_s x = V$ and we get a simpler estimate.
By Stokes' theorem and
$G^{\mu\nu}\widetilde{n}_\mu\widetilde{n}_\nu = 1$ on $\pa\Omega$,
\begin{multline}
 \int_{\pa \Omega}\left(S \widetilde{V}_\nu T^J S_\ve x^\nu\right)^2\, dS\\
 = 2\int_\Omega  (S \widetilde{V}_\nu  T^J S_\ve x^\nu)G^{\mu\alpha}\widetilde{n}^e_\alpha \pave_\mu
 (S \widetilde{V}_\nu  T^J S_\ve x^\nu)\, \kappa_G dy
 + \int_\Omega (S \widetilde{V}_\nu T^JS_\ve x^\nu)^2
 \pave_\mu (G^{\mu\alpha} \widetilde{\n}^e_\alpha \kappa_G)\, dy.
\end{multline}
Here, $\widetilde{n}^e$ denotes an arbitrary extension of $\widetilde{n}$
to a neighborhood of $\pa \Omega$. The second term is lower-order and for the
first we write
\begin{equation}
  G^{\mu\alpha}\widetilde{n}^e_\alpha \pave_\mu
  (S \widetilde{V}_\nu  T^J S_\ve x^\nu)
  =S \left( G^{\mu\alpha} \widetilde{n}^e_\alpha \pave_\mu (
  \widetilde{V}_\nu T^J S_\ve x^\nu)\right)
  + R^J.
\end{equation}
where $R^J = [G^{\mu\alpha}\widetilde{n}^e_\alpha \pave_\mu, S]
(\widetilde{V}_\nu T^J S_\ve x^\nu)$  is lower order,
\begin{equation}
 |R^J| \lesssim c_1{\sum}_{|K| \leq |J|}
 |\pave T^K S_\ve x|.
\end{equation}

Now we note that from \eqref{fracalg} and the Leibniz rule \eqref{leib3}
\begin{equation}
 \left| \int_\Omega f S g\, \kappa_G dy\right|
 \lesssim
 C_1
 \|\fdh f\|_{L^2(\Omega)}
 \|\fdh g\|_{L^2(\Omega)},
\end{equation}
for any functions $f,g$ so writing
$\pave_{\widetilde{\n}_G} = G^{\alpha\beta}\widetilde{\n}_\alpha \pave_\beta$
we have
\begin{multline}
 \left|\int_\Omega
 (T \widetilde{V}_\nu  T^J S_\ve x^\nu)
 T \left( \pave_{\widetilde{\n}_G}(
 \widetilde{V}_\nu T^J S_\ve x^\nu)\right) \, \kappa_G dy\right|\\
 \lesssim
 C_1
{ \sum}_{|K| \leq |J|+1} \|\fdh T^K S_\ve x\|_{L^2(\Omega)}
 \left\|\fdh \left(\pave_{\widetilde{\n}_G} (\widetilde{V}_\nu T^JS_\ve x^\nu)\right)\right\|_{L^2(\Omega)}.
\end{multline}
To deal with the second factor, we write
\begin{equation}
 \pave_\alpha (\widetilde{V}^\nu T^JS_\ve x^\nu)
 = \pave_\alpha (\widetilde{V}^\nu X^J_{\ve,\nu})
 =\widetilde{V}^\nu \pave_\alpha X^J_{\ve,\nu} + (\pave_\alpha \widetilde{V}^\nu)
 X^J_\nu
 = \hD_s X^J_\alpha + \widetilde{V}^\nu \curl X^J_{\ve, \alpha\nu}
 +(\pave_\alpha \widetilde{V}^\nu)
 X^J_{\ve,\nu}.
 \label{}
\end{equation}
We also have that $\fdh \hD_s X^J_\ve$ is lower-order
since $\hD_s x = V$,
\begin{equation}
 \|\fdh \hD_s X^J_\ve\|_{L^2(\Omega)}
 \lesssim C_1
 {\sum}_{|K| \leq |J|}
 \| \fdh T^K S_\ve V\|_{L^2(\Omega)}.
\end{equation}
Combining the above we have shown that
\begin{multline}
 \left| \int_{\pa\Omega} (\widetilde{V}_\nu T^I S_\ve x^\nu)^2\, dS_G\right|
 \lesssim
 \\
 \Big(
 \sum_{|K| \leq |J|+1}\|\fdh T^K S_\ve x\|_{L^2(\Omega)}
 \Big)
 \Big(\sum_{|J| \leq |I|-1}
 \|\fdh \curl T^J S_\ve x\|_{L^2(\Omega)}
 +
 \|\fdh T^J V\|_{L^2(\Omega)}\Big).
 \label{halfderivativebdy}
\end{multline}
Inserting this bound into \eqref{usefracestirel2} and absorbing the first factor in
\eqref{halfderivativebdy} into the left we find
\begin{multline}
 {\sum}_{|J| \leq r-1}\|\pave X^{J,\nicefrac{1}{2}}_\ve\|_{L^2(\Omega)}^2
 +
 {\sum}_{|J| \leq r-1}\|\pave X^{J}_\ve\|_{L^2(\pa \Omega)}^2
 \lesssim
 C_1 {\sum}_{|I| \leq r} \|\widetilde{\N}_\mu T^I S_\ve x^\mu\|_{L^2(\Omega)}\\
 + C_1
   {\sum}_{|J| \leq r\shortminus 1}
   \vertiii{\tr_{\oH}\pave  T^J \!\fdh S_\ve x}_{L^2(\Omega)}^2
   +\| \curl T^J\!\fdh S_\ve x\|_{L^2(\Omega)}^2
   + \|\pave T^J\! S_\ve x\|_{L^2(\Omega)}^2
   + \|T^J\! S_\ve V\|_{L^2(\Omega)}^2.
\end{multline}

\subsection{The combined div-curl evolution system}

The arguments from sections \ref{divcurlsection}-\ref{controll2} now go through
almost exactly as written.
With notation as in \eqref{Hnormdef}
we define
\begin{equation}
 X_{\alpha \mu}^{1,J} =
 \pa_{y^{\alpha}} (g_{\mu\nu} T^J x^\mu), \qquad
 \widetilde{X}_{\alpha \mu}^{1,J} =
 \pa_{y^{\alpha}} (g_{\mu\nu}T^J \xve^\mu).
\end{equation}
as well as the quantities
\begin{equation}
V^{1,r}\!=\!{\sum}_{|I|\leq r}|\widetilde{\pa} T^{I\!} V|,\qquad
X^{1,r}\!=\!{\sum}_{|I|\leq r}|\widetilde{\pa} T^I \!X|,\qquad
 K^r\!=\!{\sum}_{|I|\leq r}|K^{I\!}|,
\end{equation}
and
\begin{equation}
 V^{r}\!=\!{\sum}_{|I|\leq r}|T^{I\!} V|,\qquad
W^{r}\!=\!{\sum}_{|I|\leq r}|T^I\widetilde{\pa}\sigma |+
|\hD_s T^I \sigma|,\qquad \Sigma^r\!=\!{\sum}_{|I|\leq r}|\widetilde{\pa} T^I\widetilde{\pa} \sigma|.
\end{equation}
For the proof of our energy estimates it is natural
to prove bounds involving Lagrangian tangential derivatives
of the components of the metric and the Christoffel symbols,
but in the proof of existence we will have to consider
these quantities evaluated at different iterates
and for that purpose it is more natural to express
things in terms of Eulerian derivatives of the metric.
In other words for the energy estimates we will have
error terms which involve the quantities
\begin{equation}
 G^{1,r} = \sum_{|I| \leq r}
 \sum_{\mu,\nu,\gamma = 0}^3
 |\pave T^I \wGamma_{\mu\nu}^\gamma|
 + |\pave T^I \wg_{\mu\nu}|
 + G^r,
 \qquad
  G^{r} = \sum_{|I| \leq r}
  \sum_{\mu,\nu,\gamma = 0}^3
  |T^I \wGamma_{\mu\nu}^\gamma|
  + \sum_{\mu,\mu = 0}^3
  +|T^I \wg_{\mu\nu}|.
  \label{Grenergy}
\end{equation}
For the proof of existence it is better to define
\begin{equation}
 \wGamma_{\mu\nu}^{\gamma, I}(s,y)
 = (\pave^I \Gamma_{\mu\nu}^\gamma)(\xve(s,y)),
 \qquad
 \wg_{\mu\nu}^I(s,y) = (\pave^I g_{\mu\nu})(\xve(s,y)),
\end{equation}
and
\begin{equation}
 \widetilde{G}^{1,r} = \sum_{|I| \leq r}
 \sum_{\mu,\nu,\gamma = 0}^3
 |\pave \wGamma_{\mu\nu}^{\gamma, I}|
 + |\pave \wg_{\mu\nu}^I|
 + G^r,
 \qquad
  \widetilde{G}^{r} = \sum_{|I| \leq r}
  \sum_{\mu,\nu,\gamma = 0}^3
  | \wGamma_{\mu\nu}^{\gamma, I}|
  + \sum_{\mu,\mu = 0}^3
  +|\wg_{\mu\nu}^I|.
 \label{}
\end{equation}
By the chain rule we have the following bound, which is needed
for the proof of existence,
\begin{equation}
 G^{1, r}
 \lesssim c_r \widetilde{G}^{1,r} + c_rX^{1,r-1}.
 \label{conversion}
\end{equation}

With $L^p = L^p(\Omega)$, we introduce the quantities
\begin{align}
V^{1,r}_p(s)&=\|V^{1,r}(s,\cdot)\|_{L^p},\quad
&K^r_p(s)=\|K^r(s,\cdot)\|_{L^p},\qquad X^{1,r}_p(s)&=\|X^{1,r}(s,\cdot)\|_{L^p},\\
V^r_p(s)&=\|V^r(s,\cdot)\|_{L^p},\quad
&W^r_p(s)=\|W^r(s,\cdot)\|_{L^p},\qquad
\Sigma^r_p(s)&=\|\Sigma^r(s,\cdot)\|_{L^p}.
\label{normsdef1rel}
\end{align}
and
\begin{equation}
 G_p^{1,r}(s) = \|G^{1,r}(s,\cdot)\|_{L^p},
 \qquad
 G_p^{r}(s) = \|G^r(s,\cdot)\|_{L^p}.
 \label{gnormsbound}
\end{equation}

Following the steps in section \ref{divcurlsection}
and using the results of sections \ref{reldivergeneestimatesection}-\ref{ellipticestsecrel},
we arrive at
\begin{align}
| \hD_s K^{r}_p(s)|&\lesssim C_r\big( X^{1,r}_p(s)+ V_p^{1,r-1}(s)+W^r_p(s)\big),
\label{eq:evolutionsystem.arel}\\
|\hD_s {X}^{1,r}_p(s)|&\lesssim C_r\big(K^r_p(s)+ X^{1,r}_p(s)+V^{1+r}_p(s)+W^{r}_p(s)
+ G^{1+r}_p(s)\big),\label{eq:evolutionsystem.brel}
\end{align}
and
\begin{align}
V^{1,r}_p(s)&\lesssim C_r\big(K^r_p(s)+ X^{1,r}_p(s)+V^{1+r}_p(s)+W^{r}_p(s) + G^{1+r}_p(s)\big),\label{eq:ellipticsystem.arel}\\
\Sigma^{r-1}_p(s)&\lesssim C_r\big({X}^{1,r-1}_p(s)+{V}_p^{1,r-1}(t)+W^{r}_p(s)+ G^{1,r-1}_p(s)\big),
\label{eq:ellipticsystem.brel}
\end{align}
where $c_r$ depends on bounds for $X^{1,q}_\infty$, $V^{1,q}_\infty$, $\Sigma^q_\infty$ and $G^{1,q}_{\infty}$
for $q\leq r/2$.

Similarly, we introduce
\begin{equation}
X^{J,\nicefrac{1}{2}}_{\varepsilon,\nu}=\wg_{\mu\nu}\fdh T^J\!\sm x^\mu , \quad\text{and}\quad
V^{J,\nicefrac{1}{2}}_{\varepsilon,\nu}=\wg_{\mu\nu}\fdh T^J\!\sm  {}_{\!} V^\mu,
\end{equation}
and
\begin{equation}
K^{r,\nicefrac{1}{2}}_{\varepsilon,2}(s)
={\sum}_{|J|\leq r}\|K^{J,\nicefrac{1}{2}}_\varepsilon(s,\cdot)\|_{L^2(\Omega)},\quad\text{and}\quad
D^{r,\nicefrac{1}{2}}_{\varepsilon,2}(s)={\sum}_{|J|\leq r}
\| D^{J,\nicefrac{1}{2}}_\varepsilon(s,\cdot)\|_{L^2(\Omega)},
\end{equation}
where $K^{J,\nicefrac{1}{2}}_\varepsilon$ is given by
\eqref{eq:curlhalfrel} and $D^{J,\nicefrac{1}{2}}_\varepsilon$
is given by \eqref{DJepsdefrel},
as well as the quantities
\begin{equation}
X^{1,r,\nicefrac{1}{2}}_{\varepsilon,2}(s)
={\sum}_{|J|\leq r}\|\widetilde{\pa} X^{J,\nicefrac{1}{2}}_\varepsilon(s,\cdot)\|_{L^2(\Omega)},\quad\text{and}\quad
\Sigma^{r,\nicefrac{1}{2}}_{2}(s)\!=\!{\sum}_{|K|\leq r}\|\widetilde{\pa} \fdh T^K\widetilde{\pa} \sigma(s,\cdot)\|_{L^2(\Omega)},
\label{halfextranormcontrolrel}
\end{equation}
and
\begin{equation}
X^{\boldsymbol{\times},r,\nicefrac{1}{2}}_{\varepsilon,2}(s)
={\sum}_{|J|\leq r}\|\widetilde{\curl} \, X^{J,\nicefrac{1}{2}}_\varepsilon(s,\cdot)\|_{L^2(\Omega)},\quad\text{and}\quad
X^{\bigfour\cdot,r,\nicefrac{1}{2}}_{\varepsilon,2}(s)
={\sum}_{|J|\leq r}\|\widetilde{\div} X^{J,\nicefrac{1}{2}}_\varepsilon(s,\cdot)\|_{L^2(\Omega)},
\end{equation}
as well as the geometric quantities
\begin{equation}
  G^{r,\nicefrac{1}{2}}_p(s) =
  \| G^{r,\nicefrac{1}{2}}(s,\cdot)\|_{L^p(\Omega)},
  \quad \!\!\!\text{where}\!\!\!\quad
 G^{r,\nicefrac{1}{2}}\!\! =\!\!\! \sum_{|I| \leq r}
 \sum_{\mu,\nu,\gamma = 0}^3\!\!
 \|\fdh T^I \wGamma_{\mu\nu}^\gamma\|_{L^2(\Omega)}
 +\!\! \sum_{\mu,\mu = 0}^3 \!\!\|\fdh T^I \wg_{\mu\nu}\|_{L^2(\Omega)}.
 \label{geomnorms}
\end{equation}
From \eqref{eq:curlequation2rel}, \eqref{eq:curlequation2rel2}, and \eqref{Dvediffrelbd}
we have
\begin{align}
 \hD_s K^{r,\nicefrac{1}{2}}_{\varepsilon,2}(s)
 &\lesssim C_r\big(\Sigma^{r-1,\nicefrac{1}{2}}_{2}(s)+X^{1,r,\nicefrac{1}{2}}_{\varepsilon,2}(s)
 \big),
 \label{eq:evolutionhalf1rel}\\
 \hD_s
 X^{\boldsymbol{\times},r,\nicefrac{1}{2}}_{\varepsilon,2}(s)
 &\lesssim C_r\big(K^{r,\nicefrac{1}{2}}_{\varepsilon,2}(s)+X^{1,r,\nicefrac{1}{2}}_{\varepsilon,2}(s)\big),
 \label{eq:evolutionhalf2rel}\\
 \hD_s D^{r,\nicefrac{1}{2}}_{\varepsilon,2}(s)
 &\lesssim C_r\big(X^{1,r,\nicefrac{1}{2}}_{\varepsilon,2}(t)+V^{1,r}_{2}(s)
 +X^{1,r}_{2}(s)+ G^{r,\nicefrac{1}{2}}_2(s) \big).\label{eq:evolutionhalf3rel}
\end{align}
By \eqref{Dvediffrelbd}, \eqref{usefracestirel2} and \eqref{eq:byenergyestimatedrel} we have
\begin{align}
 X^{\bigfour\cdot,r,\nicefrac{1}{2}}_{\varepsilon,2}(s)
 &\lesssim C_r\big(D^{r,\nicefrac{1}{2}}_{\varepsilon,2}(s)+W^{r}_{2}(s)+V^{1,r}_{2}(s)
 +X^{1,r}_{2}(s) + G^{r,\nicefrac{1}{2}}(s)\big),\label{eq:boundshalf1rel}\\
X^{1,r,\nicefrac{1}{2}}_{\varepsilon,2}(s)+B^{\,r+1}_2(s)
&\lesssim C_r\big( X^{\boldsymbol{\times},r,\nicefrac{1}{2}}_{\varepsilon,2}(s)
+X^{\bigfour\cdot,r,\nicefrac{1}{2}}_{\varepsilon,2}(s)+B^{\,r+1}_{\mathcal{N},2}(s)
+X^{1,r}_{\varepsilon,2}(s)+ G^{r,\nicefrac{1}{2}}(s)\big),\label{eq:boundshalf2rel}\\
  V^{\,r+1}_{2}(s)+B^{\,r+1}_{\mathcal{N},2}(s)&\lesssim C_0 E^{\,r+1}_2(s), 
\end{align}
where
\begin{equation}\label{eq:energyrdefrel}
E^{\,r}_2(s)={\sum}_{|I|\leq r}\sqrt{\mathcal{E}^I(s)},
\qquad
B^{\,r}_2(s)={\sum}_{|I|\leq r}\sqrt{\mathcal{B}^I(s)},
\qquad
B^{\,r}_{\mathcal{N},2}(s)={\sum}_{|I|\leq r}\sqrt{\mathcal{B}_{\mathcal{N}}^I(s)},
\end{equation}
and
 $\mathcal{E}^I(s)$ given by \eqref{eq:energydefrel} and $\mathcal{B}^I(s)$, $\mathcal{B}_{\mathcal{N}}^I(s)$
 are given by \eqref{eq:energyboundarydefrel}.

\subsubsection{The $L^\infty$ estimates for lower derivatives}
\label{lowderivatives}
The arguments from section \ref{linftylowersec} go through without
change and give that there are $S_k >0$
depending on bounds for $G_{\infty}^{1,k}$ so that for
$0 \leq s \leq S_{k-1}$ we have
\begin{equation}
 K_{\infty}^k(s) \leq 2 K_{\infty}^k(0), \qquad
 X_{\infty}^{1,k}(s) \leq 2 X_{\infty}^{1,k}(0),
 \label{rellinftyfirst}
\end{equation}
as well as
\begin{align}
 V_{\infty}^{1,k}(s)
 &\lesssim C_k \big( K_{\infty}(0) + X^{1,k}_{\infty}(0)
 + E^{1+k}_{\infty}(s) + W_{\infty}^s(s)
 + G_{\infty}^{1+k}(s)\big),\\
 \Sigma_{\infty}^{k-1}(s)
 &\lesssim
 C_k \big( X^{1,k-1}_{\infty}(0) + V_{\infty}^{1,k-1}(0)
 + W_{\infty}^k(s) + G_{\infty}^{1+k}(s)\big).
 \label{rellinftylast}
\end{align}
We also need to know that $V$ is timelike and future-directed
in order to use Lemma \ref{emtensorpositivity}. Integrating
in time and taking $S_{k-1}$ smaller if needed,
from the above bounds we have that for $0 \leq s \leq
S_{k-1}$,
\begin{align}
 4|\wg(V(s), V(s)) - \wg(V(0), V(0))| &\leq -\wg(V(0), V(0)), \label{timelike}
 \\
4|\wg(V(s), \widetilde{\Tau}(s))
- \wg(V(0), \widetilde{\Tau}(0))|
&\leq -\wg(V(0), \widetilde{\Tau}(0)),
 \label{futuredirected}
\end{align}
and in particular this implies $V(s)$ is timelike
and future-directed for $s \leq S_{k-1}$.

\subsection{Control of the $L^2$ norms}
Just as in section \ref{controlL2},
we have an evolution equation for $L^2$ norms
of the curl of the velocity and of the coordinate
\begin{align}
 |\hD_s K^r_2(s)|
 &\lesssim C_r \big( X_2^{1,r}(s) + V_2^{1,r-1}(s)
 + W_2^r(s)\big),\label{K2rel}\\
 |\hD_s X_{2}^{1,r}(s)|
 &\lesssim
 C_r \big(K_2^r(s) + X_2^{1,r}(s) + V_2^{1+r}(s)
 + W_2^r(s) + G^{1+r}_2(s)\big),
 \label{X2bdrel}
\end{align}
and from the elliptic estimates we have
\begin{align}
  V_2^{1,r}(s)
  &\lesssim
  C_r \big( K_2^r(s) + X_2^{1,r}(s)
  + V_2^{1+r}(s) + W_2^r(s) + G^{1+r}_2(s)\big),\label{V2bdrel}\\
  \Sigma_2^{r-1}(s)
  &\lesssim
  C_r \big( X_2^{1,r-1}(s) + V_2^{1,r-1}(s)
  + W_2^r(s) + G^{1,r-1}_2(s)\big).
 \label{Sigma2bd}
\end{align}
From \eqref{mainrelwaveest} we have
\begin{equation}
 |W_2^r(s)|
 \lesssim
 C_r' \big( W_2^r(0) +
 {\sup}_{0 \leq s' \leq s } K_2^r(s') + X_2^{1,r}(s') + V_2^{r+1}(s') +
 W^r_2(s')\big),
 \label{W2relbd}
\end{equation}
where here $C_r'$ denotes a constant depending on the supremum over
$0 \leq s' \leq s$ of the above quantities with $r$ replaced by $r/2$,
and since
\begin{equation}
 V_2^{r+1}(s) + B_{\N,2}^{r+1}(s) \lesssim
 C_0 E_2^{r+1}(s),
\end{equation}
it just remains to get a bound for the
energy $E_2^{r+1}(s)$.
\subsubsection{Control of the $L^2$ norms for
Euler's equations}
\label{eulconclusionsec}
By the bounds \eqref{timelike}-\eqref{futuredirected}, $V$ is timelike
and future-directed provided we take
$s \leq S_{1}$ with $S_k$ defined
as in section \ref{lowderivatives}.
From \eqref{mainrelenbdnonsmooth2}, the bounds
\eqref{Frelest}, \eqref{eq:GIestimaterel} and \eqref{mainrelenbdnonsmooth2} for the quantities $F^I, G^I, H^I$,
and the results of the previous sections we have
\begin{equation}
 E^{r+1}_2(s)
 \lesssim
 E^{r+1}_2(0)
 +
 C_0's E^{1+r}_2(s) + C_r'
 s \sup_{0\leq s' \leq s}
 \big( K_2^r(s') +
 X^{1,r}_2(s')
 + W_2^r(s') + G_2^{r+1}(s')\big), \quad \ve = 0,
 \label{enstep1}
\end{equation}
and so combining this with the evolution equations \eqref{K2rel},\eqref{X2bdrel},
and the estimates \eqref{V2bdrel}-\eqref{W2relbd},
we see that there is $S_r > 0$ so that for $0 \leq s \leq S_r$,
\begin{equation}
 K_2^r(s) \leq 2K_2^r(0),
 \quad
 X_2^{1,r}(s) \leq 2 X_2^{1,r}(0),
 \quad
 W^r_2(s) \leq 2 W_2^r(0),
 \quad
 E_2^{r+1}(s) \leq 2 E_2^{r+1}(0),
 \label{relenestend}
\end{equation}
and this concludes the proof of the apriori bounds for the
relativistic Euler equations. The bound \eqref{highnormbd} follows directly
from \eqref{relenestend}, and the bound \eqref{lownormbd} follows after integrating
the bounds \eqref{rellinftyfirst}-\eqref{rellinftyfirst}
in time.

\subsubsection{Control of the $L^2$ norms for the smoothed
Euler's equations}
We argue almost exactly as in section \ref{controll2},
the only difference being that we use
 \eqref{mainveposestimate} in place of
 \eqref{mainrelenbdnonsmooth2} and so
\eqref{enstep1} needs to be replaced with the
bound
\begin{equation}
 E^{r+1}(s) \lesssim
 E_2^{r+1}(0)
 + C_0' (s + \ve) E_2^{1+r}(s) + C_r's
 \sup_{0\leq s' \leq s}\!\!
 \big(K_2^r(s') + X_2^{1,r}(s') + W_2^r(s')
 + G_2^{r+1}(s')\big), \!\quad
 \ve\! >\! 0,
\end{equation}
and taking $\ve$ sufficiently small
  we conclude that there is $S_r > 0$ such that
 for $0 \leq s \leq S_r$,
 \begin{equation}
   K_2^r(s) \leq 2K_2^r(0),
   \quad
   X_2^{1,r}(s) \leq 2 X_2^{1,r}(0),
   \quad
   W^r_2(s) \leq 2 W_2^r(0),
   \quad
   E_2^{r+1}(s) \leq 2 E_2^{r+1}(0),
 \end{equation}
 and
\begin{equation}
K^{r,\nicefrac{1}{2}}_{\varepsilon,2}(s)\!\leq \! 2K^{r,\nicefrac{1}{2}}_{\varepsilon,2}(0),\!\quad \! X^{\boldsymbol{\times},r,\nicefrac{1}{2}}_{\varepsilon,2}(s)\!\leq \! 2X^{\boldsymbol{\times},r,\nicefrac{1}{2}}_{\varepsilon,2}(0),\!\quad\!
D^{r,\nicefrac{1}{2}}_{\varepsilon,2}(s)\!\leq \! 2D^{r,\nicefrac{1}{2}}_{\varepsilon,2}(0),\!\quad\! W_2^{r\shortminus 1,2}(s)\!\leq \!2W_2^{r\shortminus 1,2}(0),
\end{equation}
which concludes the proof of the uniform apriori bounds
for the smoothed relativistic case.

\subsection{Estimates up to surfaces of constant $t$}
\label{spacelikesec}
The above argument relied on energy estimates up to surfaces of constant $s$, pointwise estimates up to some fixed $s$ and also that the wave operator expressed in the Lagragian coordinates and restricted to surfaces of constant $s$ was elliptic, which is also needed for the upcoming proof of existence.

The results of sections \ref{highorderVrel}-\ref{enthalpyestrel}
and the pointwise estimates hold for an arbitrary spacelike surface
so it only remains to check the ellipticity. Let
$x= \widehat{x}^\mu(t, y)$ be the Lagragian coordinate expressed with $t$ as a parameter,
\begin{equation}\label{eq:eulerlagrangiancoordxsmoothedrelhat}
\frac{d \widehat{x}^{\,\mu}(t, y)}{dt}=
\widehat{V}^\mu(t, y),\qquad\widehat{x}^{\,0}(0,y) = 0, \quad
\widehat{x}^{\,i}(0,y)=x_0^i(y),\qquad y\in \Omega,
\end{equation}
where
\begin{equation}
\widehat{V}^\mu(t, y)=\widetilde{V}^\mu(s, y)/\widetilde{V}^0(s, y).
\end{equation}

We can write
$\pa_\alpha=\pa_\alpha t\,D_t +\widehat{\pa}_\alpha $,
where $\widehat{\pa}_\alpha$ differentiate along the surfaces $t=const$
and $\pa_\alpha t=\delta_{\alpha 0}$.
We have $\widehat{\pa}_\alpha =\widehat{\gamma}_\alpha^{\alpha^\prime}\pa_{\alpha^\prime}$,
where $\widehat{\gamma}_\alpha^{\alpha^\prime}=\delta_\alpha^{\alpha^\prime}-\delta_{\alpha 0}\, \widehat{V}^{\alpha^\prime}$. We have $\widehat{\gamma}_i^{\alpha^\prime}=\delta_i^{\alpha^\prime}$ and
$\widehat{\gamma}_0^{0}=0$, $\widehat{\gamma}_0^{j}=-\widehat{V}^j$.
With $\xi_t=\widehat{V}^\alpha \xi_\alpha$ and $\widehat{\xi}_\alpha=\widehat{\gamma}_\alpha^{\alpha^\prime}\xi_{\alpha^\prime}$
we have $\widehat{\xi}_0=-\widehat{V}^j \xi_j$ and  $\widehat{\xi}_i= \xi_i$.
The symbol for the wave operator can hence be decomposed
\begin{equation}
g^{\alpha\beta}\xi_\alpha \xi_\beta =g^{00}\xi_t^2+2 g^{0\beta}\xi_t\,\widehat{\xi}_{\beta}
+g^{\alpha\beta}\widehat{\xi}_{\alpha}
\widehat{\xi}_{\beta}.
\end{equation}
The principal part that only differentiates along the surface $t=const$ is
\begin{equation}
g^{\alpha\beta}\widehat{\xi}_{\alpha}
\widehat{\xi}_{\beta}=\widehat{G}^{\alpha\beta}{\xi}_{\alpha}
{\xi}_{\beta},\qquad\text{where}\qquad
\widehat{G}^{\alpha\beta}=g^{\alpha'\beta'}
\widehat{\gamma}_{\alpha'}^{\alpha}\widehat{\gamma}_{\beta'}^\beta
=\big(g^{ij}+g^{00}\widehat{V}^i \widehat{V}^j-2g^{i0} \widehat{V}^j\big)\xi_i\xi_j.
\end{equation}
We claim that this gives an elliptic operator restricted to the surfaces $t=const$.
i.e. $\widehat{G}^{ij}{\xi}_{i}
{\xi}_{j}>c\delta^{ij}{\xi}_{i}
{\xi}_{j}$, for some $c>0$.
In fact with $\widehat{\xi}^\alpha=g^{\alpha\beta}\widehat{\xi}_\beta$ is in the orthogonal complement of $\widehat{V}^\beta$, since $g_{\alpha\beta}\widehat{\xi}^\alpha\widehat{V}^\beta
=\widehat{\xi}_\beta\widehat{V}^\beta=0$, and since $\widehat{V}$ is timelike
$g_{\alpha\beta}\widehat{V}^\alpha\widehat{V}^\beta<0$ it follows that
$\widehat{\xi}$ is spacelike $g_{\alpha\beta}\widehat{\xi}^\alpha\widehat{\xi}^\beta>0$.
Therefore the results from the previous section hold up to arbtrary spacelike
surface.

\section{Existence for the smoothed and nonsmoothed problems}
\label{nonlinexistence}

We now use the bounds
from the previous two sections
 to prove existence for both
 the Newtonian and the relativistic
 problems. As in the earlier
sections, the argument in the relativistic and Newtonian cases are nearly identical so we start with the
simpler Newtonian case.

\subsection{Existence for the compressible problem}

In section \ref{existence} we prove that the linear problem \eqref{eq:eulerlagrangiancoordxsmoothed}-\eqref{eq:waveequationsmoothed} has a solution
in an appropriate function space, and the next step is to use this in an iteration scheme to find a solution for the smoothed problem. This is however greatly simplified because the continuity equation holds for the linear system,
which means that the estimates given above for the smoothed problem also will hold for the
iterates, with one exception, which is that we don't have the symmetry of the boundary term
in the basic Euler energy estimate for the iterates.
Because of the smoothing we can still estimate it, but at the cost of introducing a power of $1/\varepsilon$.
This just means that we have to choose the time interval of existence small depending on $\varepsilon$, which as we shall see in the next section is not a problem because one can repeat this local existence result to prove existence for as long as we have apriori bounds.

Let us now write up the iteration scheme to solve the nonlinear smoothed problem. Let $V^{(0)}$ and $x^{(0)}$ be given by the approximate solution satisfying the compatibility conditions for initial data in section \ref{compatsec}. Now given $U=V^{(k)}$ define $z=x^{(k)}$ by
\begin{equation}
\frac{d z}{dt}=U(t,z),\qquad z(0,y)=x_0(y),\qquad y\in \Omega,
\end{equation}
and define $\widetilde{V}$ and $\widetilde{x}$ by
\begin{equation}
\widetilde{V}=S_\varepsilon^* S_\varepsilon U,\qquad
\widetilde{x}=S_\varepsilon^* S_\varepsilon z .
\label{iteratesm}
\end{equation}

Next, given $\widetilde{V}$ and $\widetilde{x}$ tangentially smooth define the new $V^{(k+1)}=V$ by solving the linear system
\begin{equation}\label{eq:eulerlagrangiancoordsmoothed2}
D_t V^i=-\delta^{ij}\widetilde{\pa}_j h,\qquad \text{where}\quad
D_t=\partial_t\big|_{y =const},\qquad \widetilde{\pa}_i =\frac{\pa y^a}{\pa \widetilde{x}^i} \frac{\pa}{\pa y^a},
\end{equation}
where $h$ is given by
\begin{equation}\label{eq:waveequationsmoothed2}
  D_t\big( e_1 D_t h\big) - \widetilde{\Delta} h = \widetilde{\pa}_i \widetilde{V}^j\, \widetilde{\pa}_j V^i, \quad \quad\text{with}\quad
 h\big|_{\pa \Omega}=0,\quad\text{where}\quad \widetilde{\Delta}\!=\delta^{ij}\nave_i\nave_j.\!
\end{equation}
(If $e_1$ in \eqref{eq:e1cond} is not constant then we
 evaluate it at the previous iterate of $h$ to get a linear system.)
In Section \ref{existence} we prove that this linear system has a solution $V$ in the energy space i.e. so that
the quantities $E_2^{r+1}(t)$
defined in \eqref{eq:energyrdef} are finite for $t \leq T'(\ve)$.
Taking the divergence of \eqref{eq:eulerlagrangiancoordsmoothed} and subtracting
it from \eqref{eq:waveequationsmoothed} shows that
\begin{equation}
  \label{eq:continuityequationsmoothed}
 e_1D_t h=-\widetilde{\div} V,
\end{equation}
if this holds initially. Then define the new $x=x^{(k+1)}$ by
\begin{equation}
\frac{d x}{dt}=V(t,x),\qquad x(0,y)=x_0(y),\qquad y\in \Omega.
\end{equation}

All the apriori estimates for the smoothed problem given in the previous section hold,
except that the boundary term in the energy estimate for Euler's equation needs to
be handled differently, as explained in Section \ref{sec:apriorilinear}.
This gives a $\varepsilon$ dependent bound of the right hand sides of the energy
estimates but we obtain a uniform constant by integrating over a small time.
This gives uniform energy bounds for the sequence of iterates independent of $\varepsilon$ up to a time dependent on $\varepsilon>0$.

Note that even though the existence for the linear system is in norms with
integer numbers of derivatives and does not give any extra half tangential regularity,
 all the estimates for an extra half derivative for the coordinate has a smoothing in them so there is no problem with regularity in the above iteration scheme.

\begin{prop}
  \label{nlsmoothexist}
  Fix $r \geq 9$, $\ve > 0$ sufficiently small and initial data $(V_0, h_0)$ satisfying the compatibility conditions
  \eqref{eq:compatibilityconditiondef} to order $r$ as well as
  the Taylor sign condition \eqref{tsc}.
  Let $E_0 =  \|V_0\|_{H^{r+1}(\Omega)}^2
  + \|\pa h_{0}\|_{H^{r}(\Omega)}^2$.
  Then there is a continuous function
$T_\ve= T_\ve(E_0, c) > 0$ so that the nonlinear smoothed
problem \eqref{eq:eulerlagrangiancoordsmoothed}-\eqref{eq:waveequationsmoothed}
 has a solution $(V, h)$ defined for $[0, T_\ve]$
 so that with
$W_2^{r}$,
 $\!\!E_2^{1+r}$,
  $\!\!  V^{1,r}_2$,
$\!\! X^{1,r,\nicefrac{1}{2}}_{\ve,2}$ and
$H_2^{r-1}$ defined as in
 \eqref{normsdef1},
 \eqref{eq:energyrdef},
 \eqref{higherorderwave} and \eqref{halfextranormcontrol},
 for $0 \leq t \leq T_\ve$
\begin{equation}
 {\sup}_{0 \leq t \leq T_{\ve}} E_2^{r+1}(t)
 + W_2^{r,1}(t) + W_2^{r-1,2}(t) + V_2^{1,r}(t)
 + X_{\ve,2}^{1,r,\nicefrac{1}{2}}(t) + H_2^{r-1}(t)
 < C(E_0, c).
 \label{nonlinbd}
\end{equation}
In fact we control normal derivatives of $V$ and $h$ to highest order,
\begin{equation}
 {\sup}_{0 \leq t \leq T_{\ve}}
 {\sum}_{k+\ell \leq r}\|D_t D_t^k V(t)\|_{H^{\ell}(\Omega)}
 + {\sum}_{k+\ell \leq r}\|\pave D_t^k V(t)\|_{H^{\ell}(\Omega)}
 < C(E_0, c).
 \label{fullnonlinbd}
\end{equation}
\end{prop}
\begin{proof}
  We construct $V$ using the iteration described above. Specifically,
  with $V_0, h_0$ as in the statement of the theorem, let $V^{(0)}(t,y) = \overline{V}(t,y)$
  denote a power series which solves the equation to order $r$ at $t = 0$
  as described in appendix \ref{compatsec}, defined on an arbitrary
  time interval $[0, T_0]$. It is only for this step that we need
  the initial data to be more regular than the solution we expect to get back.
   Now, given
  $U=V_{(k)}$ define $z=x_{(k)}$ such that $dz/dt=U(t,z)$ and then define
$\widetilde{V}=S_\varepsilon^* S_\varepsilon U$ and
$\widetilde{x}=S_\varepsilon^* S_\varepsilon z$.
We are going to prove that this sequence is bounded with respect to the norms
\begin{equation}
 \|u\|_{r+1, T_0} = {\sup}_{0 \leq t \leq T_0}
 {\sum}_{k+\ell \leq r+1}\|D_t D_t^k u(t)\|_{H^{\ell}(\Omega)}
 + {\sum}_{k+\ell \leq r+1}\|D_t^k u(t)\|_{H^{\ell}(\Omega)}
\end{equation}
The reason we control an additional time derivative of the solution compared
to the number of space derivatives is explained
in section \ref{compatsec}.

If
$\|\widetilde{V}\|_{r+1, T_0} + \|\S \xve\|_{r+1, T_0} < \infty$
for some $T_0 > 0$
then by Proposition \ref{linexist}, the linear problem
 \eqref{eq:eulerlagrangiancoordsmoothed2}-\eqref{eq:waveequationsmoothed2} has a solution $V = V_{(k+1)}$
on the time interval $[0, T_0]$
satisfying $\|V\|_{r+1, T_0} < \infty$. Let us note at this point
that the reason we need a bound for $\|\S \xve\|_{r+1, T_0}$ is that
in the proof of Proposition \ref{linexist} we need to use the elliptic
estimate from Proposition \ref{prop:dirichlet}.

To construct
the next iterate, we need to know that
$\|S_\ve^* S_\ve V\|_{r+1, T'}
+ \|\S S_\ve^* S_\ve x\|_{r+1, T'} < \infty$.
This follows from the bounds
\eqref{existencebd}-\eqref{existencebd2}
and the smoothing estimate $\|\S S_\ve f\|_{L^2(\Omega)}
\lesssim \ve^{-1} \|f\|_{L^2(\Omega)}$.
Having constructed the sequence $V_{(k+1)}$,
 we now prove a uniform bound for the iterates. As mentioned these
uniform bounds
follow in nearly the same way that we proved
the apriori bounds for the nonlinear problem,
except that the evolution equation \eqref{E2smooth} needs to be replaced
by
\begin{multline}
|{E}_{2,(k+1)}^{\,r+1\,{}_{\,}\prime}(t)|
\lesssim  \frac{C_0^{(k)}}{\ve} {E}_{2,(k+1)}^{\,r+1}(t)+
\frac{C_0^{(k)}}{\ve}{B}_{2,(k+1)}^{\,r+1}(t)+
\frac{C_r^{(k)}}{\ve}\bigtwo(K_{2,(k+1)}^{r}(t)+X_{2,(k+1)}^{1,r\!}(t)+
W_{2,(k+1)}^r(t)
\bigtwo)
\\
\frac{C_0^{(k+1)}}{\ve} {E}_{2,(k)}^{\,r+1}(t)+
\frac{C_0^{(k+1)}}{\ve}{B}_{2,(k)}^{\,r+1}(t)+
\frac{C_r^{(k+1)}}{\ve}\bigtwo(K_{2,(k)}^{r}(t)+X_{2,(k)}^{1,r\!}(t)+
W_{2,(k)}^r(t)
\bigtwo),
 \label{}
\end{multline}
which follows from \eqref{dtEmainnonreliter}.
Here the constants $C_0^{(k)}, C_r^{(k)}$ are as in
\eqref{capitalCdef} but with $V$ replaced by the previous
iterate $V_{(k)}$ and with $\xve$ replaced by $\xve_{(k)}$
and similarly the quantities $E_{2,(\ell)}$, $B_{2,(\ell)}$,
$K_{2,(\ell)}$ and $X_{2,(\ell)}$ are defined as in
\eqref{normsdef1}-\eqref{normsdef2}.

By induction we find that there is
$T_r^\ve>0$ depending only on the initial data and on $\ve>0$
so that the bounds \eqref{l2uniformbd} hold for $V_{(k+1)}$
for $0 \leq t \leq T_r^\ve$.
 Arguing in almost exactly the same way one can
 prove that $V_{(k)}$ is a Cauchy sequence in a lower norm, i.e. that
 $|V_{2,(k_1)}^{r}(t) - V_{2,(k_2)}(t)| \to 0$ as $k_1, k_2 \to \infty$.
 From the uniform bounds we see that the sequence
 $V_{(k)}$ converges weakly to a limit $V$ satisfying
 the bound \eqref{nonlinbd} and from
 the Cauchy estimates it follows that this convergence
 is strong and so $V$ solves the nonlinear smoothed problem.

 It remains to prove the bound \eqref{fullnonlinbd} for the
 full derivatives of the solution which will be needed
 to extend the solution to a uniform time interval
 in the next result. This bound follows from our energy estimate using
  elliptic estimates, estimates for the wave equation, and estimates
  for the transport equation for the curl using a minor
  modification of the argument we used to prove the energy estimates.
  We just give a sketch of how
 to control $\|V(t)\|_{H^r(\Omega)}$ and $\|\pave V(t)\|_{H^{r-1}(\Omega)}$,
 since bounds for the other terms appearing in
 \eqref{fullnonlinbd} follow in a similar and simpler way. To control $V$,
 the strategy
 is to use the pointwise bound \eqref{app:pwdiff}
 to control full derivatives of $V$ in terms of full derivatives of the curl, divergence,
 and full derivatives of $x$. These can be bounded in nearly the same way as we bounded
 the tangential derivatives of these terms since that part of the argument only relied on
 differentiating transport equations and using pointwise inequalities.
 It was only when we commuted these equations with additional fractional tangential
 derivatives that it was important to only commute with tangential derivatives.
 Once we have
 bounds for $V$ the bound for $h$ follows from the pointwise elliptic estimate
 as in \eqref{eq:paTpah2} and estimates for the wave equation which we already encountered
 in section \ref{enthalpyestimates}. We therefore just discuss how to control $V$.

 As in the proof of \eqref{eq:ellipticV1}, if we define
$
 K^{\prime J}_{ij} = \widetilde{\curl} \pa^J V_{ij}
 +L_{ij}^1[\pave T^J x],
 $ after applying the pointwise estimate \eqref{app:pwdiff}
 and taking the $L^2$ norm we find that
 \begin{equation}
  \sum_{|J| \leq r }\|\pave \pa^J V\|_{L^2}
  \lesssim C_r\!\!
  \sum_{|J| \leq r} \|K^{\prime, J}\|_{L^2} +
  \| \pave \pa^{J} \xve\|_{L^2} +
  \|\pave \pa^J x\|_{L^2}
  +\|D_t \pa^J h\|_{L^2} +
  \sum_{S \in \mathcal{S}}
  \|S T^JV\|_{L^2} + \|\pa^J V\|_{L^2},
  \label{topordernormal}
\end{equation}
where here $C_r$ is defined as in \eqref{capitalCdef} but where the norms now depend
on full derivatives, and where $L^2 = L^2(\Omega)$. By induction, bounds for
the energies, and the estimates for the wave equation
 it remains to prove bounds for $\|\pave \pa^J\xve\|_{L^2},
\|\pave \pa^J x\|_{L^2}$ and for the curl term $\|K^{\prime, J}\|_{L^2}$.
To control the norms of $x$ and $\xve$ we note that it is enough to control
$\|\pave \pa^J x\|_{L^2}$ since  the smoothing is a bounded
operator and that the commutator with the derivatives
$\pave \pa^J$ is lower-order by
Lemma \ref{lem:gradientsmoothingcommute}. A bound for this term and for
$K^{\prime, J}$ follows
easily since in the same way we proved \eqref{divxeqn} and \eqref{eq:modifiedcurl}
we have evolution equations
\begin{equation}
  D_t \bigtwo( e_1  \pa^J h+  \widetilde{\div} (\pa^J \!x)\bigtwo)
  =\widetilde{\pa}_i \pa^{J\!} \widetilde{x}^k \,\,\widetilde{\pa}_k V^i -\widetilde{\pa}_i \pa^{J\!} {x}^k \,\,\widetilde{\pa}_k  \widetilde{V}^i+G^{\prime \prime J},
  \qquad
  D_t K^{\prime J}_{ij}
= L_{ij}^2[\pave \pa^J x] - A_{ij}^{\prime J},
\end{equation}
where $G^{\prime \prime J}\!$ is lower order and
 $A_{ij}^{\prime J}\!$ is the antisymmetric part of $\pave_i \pa^J \pave_j h$
which is lower-order.  \qedhere
\end{proof}

\begin{prop}\label{uniformtime} Fix $r \geq 10$ and initial data $(V_0, h_0)$
  satisfying the compatibility conditions \eqref{eq:compatibilityconditiondef}
  to order $r$ and define $E_0$ as in the previous
  Proposition. Then there is
   $T_1 = T_1(E_0) > 0$ so that
   for any $\ve > 0$ the nonlinear smoothed problem
   \eqref{eq:eulerlagrangiancoordsmoothed}-\eqref{eq:waveequationsmoothed} has a solution defined for
   $[0, T_1]$ so that the bounds
   \eqref{nonlinbd}- \eqref{fullnonlinbd} hold for $0 \leq t \leq T_1$.
\end{prop}
\begin{proof}
  Let $T_0$ denote the largest time so that
  the nonlinear smoothed problem has a solution $V$ with $\sup_{0 \leq t \leq T'}
  V_2^{1,r}(t) + H^{r-1}_2(t) < \infty$ whenever $T' < T_0$.
  By Proposition \ref{nlsmoothexist}, $T_0 > 0$.

  By the energy estimates
  in section \ref{sec:nonrel} there is $T_E > 0$ depending only on $E_0$
  and a lower bound for $\nabla h$ at the boundary
   so that for any $\ve > 0$, any
  solution defined on $[0, T_E]$ with finite energy
  satisfies the energy estimate
  \eqref{nonlinbd} for $0 \leq t \leq T_E$.

  The result now follows since $T_0 \geq T_E$.
  Indeed, if $T_0 < T_E$ then we note that
  by Proposition \ref{linexist}, the compatibility conditions hold at $t = T_0$ and
  so replacing $t$ with $t - T_0$ and replacing
  the initial data $(V_0, h_0)$ with $(V_{T_0}, h_{T_0}) = (V, h)|_{t = T_0}$,
  by Proposition \ref{nlsmoothexist} we could extend the solution to a slightly
  larger time interval $[0, T_0 + \delta)$ for $\delta > 0$, which contradicts
  maximality of $T_0$.
\end{proof}

We can now provide the existence result in the Newtonian case.
\begin{theorem}
  \label{nonrelexist} Fix $r \geq 10$ and initial data
  $(V_0, h_0)$ with $
  E_0 = \| V_0\|_{H^r(\Omega)}^2 + \| h_0\|_{H^r(\Omega)}^2 < \infty$
  satisfying the compatibility conditions
  \eqref{eq:compatibilityconditiondef} as well as the Taylor sign condition
  \eqref{tsc}, with sufficiently large sound speed \eqref{soundspeed}. Then there is a
  continuous function $\mathscr{T} = \mathscr{T}(E_0, c_0, c_s) > 0$ so
  that Newtonian Euler equations
  \eqref{nonrelcont}-\eqref{nonrelcont} with $f = g = 0$
  have a solution $(V, h)$ defined on a time interval
  $0 \leq t \leq T' \leq \mathscr{T}$ and so that
  the bounds \eqref{fullnonlinbd}-\eqref{nonlinbd} hold
  with $\ve = 0$ for $t \leq T'$.
\end{theorem}
\begin{proof}
  From the results in Section \ref{compatconstruction}, given initial
  data $(V_0, h_0)$ satisfying the compatibility conditions
  \eqref{eq:compatibilityconditiondef} to
  order $r$ when $\ve = 0$,
  one can construct data $(V_0^\ve, h_0^\ve)$ satisfying the corresponding conditions
  to the same order for $\ve > 0$ sufficiently small and so
   that $V_0^\ve \to V_0, h_0^\ve \to h_0$ as $\ve \to 0$.
  For $\ve > 0$ sufficiently small, let $V_{\ve}$ denote the
   solution to the nonlinear smoothed problem
   constructed in Proposition \ref{nlsmoothexist}
   with initial data $(V_0^\ve, h_0^\ve)$,
   let $h_\ve$ denote the corresponding enthalpy
   and $x_\ve$ the corresponding Lagrangian coordinate.
   By
   Proposition \ref{uniformtime} this solution
   can be extended to a time interval $[0, T_1]$
   with $T_1$ independent of $\ve$.

  Writing $\pave_\ve = \pa/\pa \xve_{\ve}$ with $\xve_\ve = S_\ve^* S_\ve x_\ve$,
  define
  \begin{equation}
   V_{p,\ve}^{1,s, *} = {\sum}_{|I| \leq s}
   \|\pa_y T^I V_\ve\|_{L^p}
   +
   \| T^I V_\ve\|_{L^p},
   \quad
   H_{p,\ve}^{1,s, *}= {\sum}_{|I| \leq s}\|\pa_y T^I \pave_\ve h_\ve\|_{L^p}
   +\| T^I \pave_\ve h_\ve\|_{L^p},
   \label{starnorm1}
 \end{equation}
 \begin{equation}
   X_{p,\ve}^{1,s, *}={\sum}_{|I| \leq s} \|\pa_y T^I x_\ve\|_{L^p}
   +\| T^I x_\ve\|_{L^p}.
   \label{starnorm2}
  \end{equation}
  From \eqref{l2uniformbd} and \eqref{lowbounds} we have a uniform bound for
  $V_{2, \ve}^{1,r,*}\!\!, H_{2,\ve}^{1,r,*}\!\!, X_{2,\ve}^{1,r,*}\!\!$
  on the time interval $[0, T_1]$ as well as for
  $V_{\infty, \ve}^{1,r/2,*}\!\!,
  H_{\infty, \ve}^{1,r/2,*}\!\!, X_{\infty, \ve}^{1,r/2,*}\!\!$.
  Therefore there are
  $V, h, x$ with $V_{2}^{1,r,*}\!\!, H_{2}^{1,r,*}\!\!,
  X_{2}^{1,r,*}\!\!,V_{\infty, \ve}^{1,r/2,*} \!< \!\infty$ so that after passing
  to a subsequence $(V_\ve, \pave_\ve h_\ve, x_\ve)
  \to (V, \pa_x h, x)$ weakly. Here the quantities
  $A_p^{1,r,*}$ are defined as in \eqref{starnorm1},\eqref{starnorm2}
  but with $(V_\ve, x_\ve, h_\ve)$ replaced with $(V,x, h)$.
  At this point one can use that we also have uniform bounds for the full
  norms \eqref{fullnonlinbd} to conclude that the limit satisfies the nonlinear equation but
  in fact one just needs bounds for tangential and time derivatives, as follows.

  From the above bounds and
  the compactness of $H^1$ in $L^2$,
  $D_t V_\ve \to D_t V$ and $D_t h_\ve \to D_t h$ strongly
  since the $T^I$ involve
  time derivatives. It remains to prove
   that $\div_\ve V_\ve \to \div V$ which is not immediate because it is nonlinear and we only have weak convergence
   (the bound for $H_{\infty,\ve}^{1,s,*}$ gives a uniform
   bound for the other nonlinear term $\pave h_\ve$). To get this convergence we claim that
   we have a uniform bound for $|\pa^2_y V_\ve|$. Assuming the claim,
   by the Arzela-Ascoli theorem, passing to another
  subsequence we then find that $\pa_y V_\ve \to \pa_y V$ pointwise and by the dominated convergence
  theorem it then converges strongly in $L^2$. Therefore the product
  $\pave_i V^i_\ve = \pa y^a/\pa \xve_\ve^i \pa_a V^i_\ve$ converges weakly,
  as required.

  To prove the bound for $\pa_y^2 V_\ve$, we start by
  using the pointwise inequality \eqref{pwnonrel},
  \begin{equation}
   |\pa^2 V_\ve| \lesssim
   |\pa \div V_\ve| + |\pa \curl V_\ve| + |\pa \S V_\ve| + |\pa V_\ve|.
   \label{2derivpw}
  \end{equation}
  The last two terms are uniformly bounded by \eqref{eq:ellipticsystem.a2}.
  Differentiating $\widetilde{\div}_\ve V_\ve =
  e_1D_t h_\ve$ and using the uniform bounds for
  $\pa D_t h_\ve$ from $H_{\infty,\ve}^{1,r/2,*}$ we get a uniform bound for
  $|\pa \widetilde{\div}_\ve V_\ve|$ as well. It remains to get a
  bound for the derivative of the curl but we have the evolution equation
  $
   |D_t \pa \curl V_\ve(t)|
   \lesssim
   |\pa^2 \widetilde{V}_\ve| |\pa V_\ve|
   +
   |\pa \widetilde{V}_\ve| |\pa^2 V_\ve|
   $
  so shrinking the time interval if needed and combining this with
  \eqref{2derivpw} we get the uniform bound for $\pa_y^2 V_\ve$.
\end{proof}

\subsection{Existence in the relativistic case}
\label{nonrelsmexist}
  Existence for the relativistic problem now follows by following exactly
  the same strategy. First, we solve the nonlinear smoothed problem
  \eqref{rescaledreleul2smoothed}
  -\eqref{rescaledrelcont2smoothed} for $\ve > 0$ on a time interval
  which depends on $\ve$. By the energy
estimates from Section \ref{rescaledreleulex} we can extend this solution to a
uniform time interval and then take $\ve \to 0$.

We recall here the assumptions we are making about
the background metric quantities. We define, at any
point $x \in \mathcal{M}$ in the Eulerian frame
\begin{equation}
 \mathcal{G}^{r} ={\sum}_{|I| \leq r}
 {\sum}_{\mu,\nu,\gamma = 0}^3
 |\pa_x^{I}\wGamma_{\mu\nu}^\gamma|
 + |\pa_x^I \wg_{\mu\nu}|.
 \qquad
 \mathcal{G}^r_p = \|\mathcal{G}^r\|_{L^p(\mathcal{M})}.
\end{equation}
Then we will assume that we have
\begin{equation}
 \mathcal{G}^r_2 \leq G,
 \label{metricbds}
\end{equation}
for some $G < \infty$. We also need to assume
that the initial rescaled velocity field is timelike
and that the enthalpy does not degenerate in the domain,
\begin{equation}
  g(\mathring{V}, \mathring{V}) = -\mathring{\sigma}
  \leq -c_1 < 0.
 \label{databounds}
\end{equation}

The existence result for the nonlinear smoothed problem is
the following, which follows in the same way that
Theorem \eqref{nlsmoothexist}
did but using the linear existence theory from section \ref{existencerel}
and the estimates from section \ref{rescaledreleulex}, using the following iteration.
Given $U = V^{(k)}$ define $z = x^{(k)}$ by
\begin{equation}
 \frac{dz}{ds} = U(z), \qquad
 z_0(0,y) = 0, z_i(0, y) = y_i, y \in \Omega,
 \label{}
\end{equation}
and define the smoothing of $z$ as in \eqref{iteratesm}. Define also
$\widetilde{g}(s,y) = g(\xve(s,y))$ and
$\widetilde{\Gamma}^\nu_{\mu\gamma}(s,y) = \Gamma^\nu_{\mu\gamma}(\widetilde{z}(s,y))$
and for a vector field $X$ define the smoothed-out covariant
derivative $\widetilde{\nabla}X$ as in \eqref{smcov}.
Now define $V^{(k+1)} = V$ by solving
\begin{equation}
  \widetilde{V}^\nu \widetilde{\nabla}_\nu V^\mu + \frac{1}{2} \widetilde{\nabla}^\mu
  \sigma = 0,
\end{equation}
where $\sigma$ is given by solving
\begin{equation}
 e'(\sigma) \hD_s^2 \sigma - \frac{1}{2} \widetilde{\nabla}_\nu
 (\widetilde{g}^{\mu\nu}\widetilde{\nabla}_\mu
 \sigma)
 = \widetilde{\nabla}_\mu \widetilde{V}^\nu \widetilde{\nabla}_\nu V^\mu + \widetilde{R}_{\mu\nu\alpha}^\mu \widetilde{V}^\nu V^\alpha
 - e''(\sigma) (\hD_s\sigma)^2,
\end{equation}
with $\sigma = \overline{\sigma}$ on the boundary, where recall $\overline{\sigma} = \sigma|_{p = 0}$ is constant. Given the above $V$, define the new $x = x^{(k+1)}$ by solving
\begin{equation}
 \frac{d x^\mu }{ds} = V^\mu(x(s,y)), \qquad
 x^0(0, y) = 0, x^i(0, y) = y^i.
\end{equation}
Then the a priori estimates from the previous section hold and we arrive at
the basic existence result for the nonlinear smoothed problem.

\begin{prop}
  \label{nlsmoothexistrel}
  Fix $r \geq 9$, $\ve > 0$ sufficiently small and initial data $(\mathring{V}, \mathring{\sigma})$ satisfying the compatibility conditions
  \eqref{relcompat} to order $r$ as well as
  the Taylor sign condition \eqref{tsc} and the
  condition \eqref{databounds}.
  Suppose that for some $T > 0$, there is a coordinate
  system $x^\mu$ so that the coefficients of the metric
  $g = g_{\mu\nu} dx^\mu dx^\nu$ satisfy \eqref{metricbds}.
  Let $E_0 = \|\mathring{V}\|_{H^{r+1}(\Omega)}^2
  + \|\pa \mathring{\sigma}\|_{H^{r}(\Omega)}^2$
  Then there is a continuous function
$S_\ve= S_\ve(E_0, c , c_1, G) > 0$
 so that the nonlinear smoothed
problem \eqref{eq:eulerlagrangiancoordsmoothed}-\eqref{eq:waveequationsmoothed}
 has a solution $(V, h)$ defined for $[0, S_\ve]$
 so that with
$W_2^{r}$,
 $\!\!E_2^{1+r}$,
  $\!\!  V^{1,r}_2$,
$\!\! X^{1,r,\nicefrac{1}{2}}_{\ve,2}$ and
$\Sigma_2^{r-1}$ defined as in
 \eqref{normsdef1rel},
 \eqref{eq:energyrdefrel},
and \eqref{halfextranormcontrolrel},
 for $0 \leq s \leq S_\ve$
\begin{equation}
 {\sup}_{0 \leq s \leq S_{\ve}} E_2^{r+1}(s)
 + W_2^{r,1}(s) + W_2^{r-1,2}(s) + V_2^{1,r}(s)
 + X_{\ve,2}^{1,r,\nicefrac{1}{2}}(t) + H_2^{r-1}(s)
 < C(E_0, c, c_1, G),
 \label{nonlinbdtangnormrel}
\end{equation}
In fact we control normal derivatives of $V$ and $\sigma$ to highest order,
\begin{equation}
 {\sup}_{0 \leq s \leq S_{\ve}}
 {\sum}_{k+\ell \leq r}\|D_s D_s^k V(s)\|_{H^{\ell}(\Omega)}
 + {\sum}_{k+\ell \leq r}\|\pave D_s^k V(s)\|_{H^{\ell}(\Omega)}
 < C(E_0, c, c_1, G).
 \label{nonlinbdfullnormrel}
\end{equation}
\end{prop}
\begin{proof}
 The argument proceeds in the same way as the proof of
 Proposition \ref{nlsmoothexist}, using the bounds from
 Section \ref{rescaledreleulex} in place of the bounds from Section \ref{sec:nonrel}.
 There is one additional detail which is that the a priori
 bounds from Section \ref{rescaledreleulex} were written in terms
 of the norms of the geometric data $G^{1,r}$ defined in
 \eqref{Grenergy}. In the iteration,
  $\widetilde{g}$,$\widetilde{\Gamma}$ need to be interpreted
  as being evaluated at the previous iterate and we therefore
 need a uniform bound for the terms involving $G^{1,r}$.
 This follows directly from \eqref{conversion} and \eqref{metricbds}.
\end{proof}

Next, we show that the solution constructed
in the previous proposition can be extended to a time
interval whose length is independent of $\ve$. This
follows in the same way as Proposition \ref{uniformtime}
after
using the energy estimate \eqref{relenestend}.
\begin{prop}\label{uniformtimerel} Fix $r \geq 10$ and initial data $(\mathring{V}, \mathring{h})$
  satisfying the compatibility conditions \eqref{eq:compatibilityconditiondef}
  to order $r$ and define $E_0$ as in the previous
  Proposition. Then there is
   $S_1 = S_1(E_0, c, c_1, G) > 0$ so that
   for any $\ve > 0$ the nonlinear smoothed problem
   \eqref{eq:eulerlagrangiancoordsmoothed}-\eqref{eq:waveequationsmoothed} has a solution defined for
   $[0, S_1]$ so that the bounds
   \eqref{nonlinbdtangnormrel}-\eqref{nonlinbdfullnormrel} hold for
   $0 \leq s \leq S_1$.
\end{prop}

In the same way that Propositions \ref{nlsmoothexist} and \ref{uniformtime}
gave Theorem \ref{nonrelexist}, Propositions
\ref{nlsmoothexistrel} and \ref{uniformtimerel} imply
\begin{theorem}
  \label{relexist} Fix $r\!  \geq\!  10$ and initial data
  $(\mathring{V}, \mathring{\sigma})$ with $
  E_0 =\| V_0\|_{H^{r+1}(\Omega)}^2 + \|\pa h_0\|_{H^r(\Omega)}^2 = \!< \! \infty$ satisfying the compatibility conditions
  \eqref{eq:compatibilityconditiondef}. Then there is a
  continuous function $\mathscr{S}\!  = \! \mathscr{S}(N_0^{r+1}\!\!  ,
  c, c_1)\!  > \! 0$ so
  that relativistic Euler equations
  \eqref{rescaledrelcont}-\eqref{rescaledreleul}
  have a solution $(V, h)$ for
  $0 \leq s \leq S' \leq \mathscr{S}$ and so that
  the bounds \eqref{nonlinbdtangnormrel}-\eqref{nonlinbdfullnormrel} hold
  with $\ve = 0$ for $s \leq S'$.
\end{theorem}

\appendix

\section{Tangential smoothing, fractional derivatives, vector fields and norms}

\subsubsection{The tangential derivatives and tangential norms}
\label{def T and FD}
Since $\Omega$ is the unit ball, the vector fields
\begin{equation}\label{eq:rotations}
\Omega_{ab} =  y^a \pa_{y^b} - y^b \pa_{y^a} , \qquad a,b = 1,2,3,
\end{equation}
are tangent to $\pa\Omega$ and span the tangent space there.
With $\eta$ the cutoff function defined above, let:
\begin{equation}
  \label{Sdef}
 \S = \cup_{a,b \,= 1,2,3}\{ \eta\, \Omega_{ab}, (1-\eta) \pa_{y^a}\}.
\end{equation}
In analogy with the two dimensional case, when $\S$ is just the derivative with respect to the angle in polar coordinates, we will now introduce some simplified notation for the norms.
Suppose that $f:\Omega\to \mathbf{R}$ is a function and $\S\!=\!\{S_1,\dots,S_{N}\}$ is a family of vector fields that are tangential to the boundary at the boundary that span the tangent space there.
Let $\S f$ stand for the map
$\S f\!\!:\!\Omega\to \!\mathbf{R}^{N}\!\!\!$, whose components are $S_j f$, for $j\!{}_{\!}={}_{\!}\!1,{}_{\!}...,{}_{\!}N\!$.
For $r$ an integer,  let $\S^r\!\!{}_{\!}={}_{\!}\!\S\!\!\times\!{}_{\!}\cdots\!\times\! \S$($r$ times) and let $S^I\!\!\in\! \S^r\!$ stand
for a product of $r$ vector fields in $\S\!$, where $I\!=\!(i_1,{}_{\!}...,{}_{\!}i_{r\!})\!\in\! [1,N]\!\times\!\cdots\!\times\! [1,N]$ is a multiindex of length $|I|\!=\!r$. Let $\S^r f\!$ stand for the map $\S^r\! f\!\!:\!\Omega\!\to \!\mathbf{R}^{N r}\!\!\!$, whose components are $S^I \!f$, for $1\!\leq\! i_j\!\leq\! N\!$, $j\!=\!1,{}_{\!}...{}_{{}_{\!}},r$. The norm of $\S^r\! f\!$ is
\begin{equation}\label{eq:simplifiedtangentialnotation}
|\S^r \!f|^2= \S^r \!f\cdot \S^r \! f,\quad\text{where}\quad
\S^r f\cdot \S^r g= {\sum}_{|I|=r,\,\,S^I\in\S^r} S^I\! f \,\, S^I g.
\end{equation}

Moreover, let
\begin{equation}
 ||W||_{H^{k{}_{\!},r}} ={\sum}_{\ell \leq r}
 ||\S^\ell W||_{H^{k}(\Omega)}.
\end{equation}

We will use similar notation for space time vector fields tangential to the boundary.
Let $\mathcal{T}\!=\!\S\!\cup\! D_t$, and
$\mathcal{T}^r\!\! =\!\mathcal{T}\!{}_{\!}\times \!\!\cdots\!\!\times \!\mathcal{T}$($r{\!}$ times), $\mathcal{T}^{r\!,k}\!{}_{\!}=\!\S^r {}_{\!}\!\!\times \!{}_{\!} D_{{}_{\!}t}^k$\!.
For $\!K\!{}_{\!}=\!(_{\!}I{}_{\!},{}_{\!}k{}_{{}_{\!}})$ a multiindex with $|I|\!{}_{\!}=\!r$, we write
$T^K\!\!=\!S^I \!D_{{}_{\!}t}^k\!$, $S^I\!\!\in\!\S^r$\!\!.

\subsubsection{Global operators defined in terms of local coordinates}\label{sec:localcoord}
There is a family of open sets $V_{{}_{\!}\mu} $, $\mu\!=\!1,\dots,N$ that cover $\pa \Omega$ and  onto diffeomorphisms
$\Phi_\mu\!:\! (\shortminus 1,1)^{2 \!}\!\to\! V_{{}_{\!}\mu} $. \!We fix a collection of cutoff functions $\chi_{\mu}{}_{\!}\!:\!\pa \Omega
\!\to \!\R$ so that $\chi_\mu^2$ form a partition of unity subordinate to the cover $\{V_\mu\}_{\mu = 1}^N$,
 as well as another family of
``fattened'' cutoff functions $\tC_\mu$ so that the support of
$\tC_\mu$ is contained in $V_\mu$ and so that
$\tC_\mu \!\equiv\! 1$ on the support of $\chi_\mu$. Recalling that
$\Omega$ is the unit ball, we set $W_{\!\mu} \!=\! \{r\omega, r\! \in\! (1/2, 1],\, \omega\!
\in \!V_{\!\mu}\}$
for $\mu\! = \! 1,\dots, N$
and let $W_0$ be the ball of radius $3/4$ so that the collection
$\{W_{\!\mu}\}_{\mu = 0}^N$ covers $\Omega$.
Then $y\!=\!\Psi_{\!\mu}(\widehat{z}) \!= z^{3\,}\omega$, where  $\omega\!=\!\Phi_\mu(z)$ and
$\widehat{z}\!=\!(z,z^3)$, is a diffeomorphism $\Psi\!:\!(\shortminus 1,1)^2\! \times\! (1/2,1]\!\to \!W_{\!\mu}$.
Let $\eta\!:\![0,1] \!\to\! \R$ be
a bump function so that $\eta(r)\! =\! 1$ when $1/2\! \leq \!r \!\leq\! 1$ and
$\eta(r) \!= \!0$ when $r\! < \!1/4$.
We define cutoff functions on $\Omega$ by setting
$
\chi_\mu (\widehat{z})\!= \!\chi_\mu(z) \eta(z^3) ,
$ for
$\mu\!\geq\! 1$, and $\chi_0$ so $\sum \chi_\mu^2\!=\!1$.
Let $\Psi^\prime_{\!\mu}\!=\! \pa y/\pa \widehat{z}$ and $\Psi^\prime_{\!\mu}\!=\!\pa \omega/\pa z$.
Then $\det{\Psi^\prime_{\!\mu}}\!=\! r^2 \det{\Phi^\prime_\mu}$.

In the local coordinates the tangential vector fields \eqref{eq:rotations} takes the form
\begin{equation}
S =S^a(z)\, \pa/\pa z^a, \qquad \text{with}\quad S^3(z)=0.
\end{equation}
Moreover we can write
\begin{equation}\label{eq:flatcoord}
\widetilde{\pa}_i = \widehat{J}_i^d \widehat{\pa}_d, \quad \text{where}\quad \widehat{J}_i^d
=\pa \widehat{z}^d\!/\pa\widetilde{x}^i,\quad\text{and}\quad \widehat{\pa}_d=\pa/\pa \widehat{z}^d=(
\Psi^\prime_\mu)_d^a\pa_a,\quad \pa_a=\pa/\pa  y^a.
\end{equation}
For a linear operator ${A}$ defined in local coordinates on the sphere we define a global operator $A$ by
\begin{equation}\label{eq:localtoglobal}
Af\!=\!\tsum A_\mu f,\quad\text{where}\quad
A_\mu f\! =
\chi_\mu m_\mu^{-1} {A}\big[m_\mu f_{\!\mu}\big]_{\!}\circ_{\!}\Psi_{\!\mu}^{-1}\!\!,\qquad f_{\!\mu}(z)\!=\!(\chi_\mu f)_{\!}\circ_{\!} \Psi_{\!\mu}(z,z^3).
\end{equation}
Here $m_\mu=|\det{\Psi_{\!\mu}^\prime}|^{1/2}$ is inserted so that
$A$ is symmetric with the measure $dy$ if it is with the measure $dz$ for fixed $z_3$ since $dS(\omega)=m_\mu^2 dz$. For the smoothing the symmetry in spherical coordinates makes things simpler since it will mean that the global operator defined by \eqref{eq:localtoglobal} is symmetric on the sphere.

However for the fractional derivative in only defined locally in each coordinate system so in that case we will pick $m_\mu=1$. Then we have
\begin{equation}
\widehat{\pa}_d\big( A[f_\mu]\circ \Psi^{-1}\big)\!=\!(\widehat{\pa}_d A[f_\mu])\circ \Psi^{-1} \!
=\![\widehat{\pa}_d, A][f_\mu]\circ \Psi^{-1}\!+A[\widehat{\pa}_d f_\mu]\circ \Psi^{-1\!\!},\qquad
\widehat{\pa}_d f_\mu\!=\!( \widehat{\pa}_d f)_\mu\!+\big( \widehat{\pa}_d\chi_\mu\, f\big)_{\!}\circ_{\!} \Psi_{\!\mu}
\end{equation}
and
\begin{equation}
S\big( A[f_\mu]\circ \Psi^{-1}\big)=(S A[f_\mu])\circ \Psi^{-1}
=[S, A][f_\mu]\circ \Psi^{-1}+A[S f_\mu]\circ \Psi^{-1}\!,\qquad
S f_\mu=( S f)_\mu+ \big( S\chi_\mu\, f\big)_{\!}\circ_{\!} \Psi_{\!\mu}
\end{equation}
Hence the commutators between the global operator $A$ and $\widehat{\pa}_i$ or $S$ consist of the commutators between these in the local coordinates plus terms when the derivatives fall on the cutoffs or measures which are lower order.

\subsubsection{Tangential smoothing}
\label{smoothingsec}
 Let
$\varphi \!:\! \R^2 \!\!\to\! \R$ be even, supported in $R = (-1,1)^2$ with $\int_{\R^2}
 \!\varphi =\! 1$ and
\begin{equation}
 S_\ve f(z) = \int_{\R^2} \varphi_{\ve}(z-w) f(w) dw,\qquad\text{where}\qquad
 \varphi_{\ve}(z)\! =\! {\ve^{-2}} \varphi\big({z}/{\ve}\big).
\end{equation}
be a smoothing operator.
Because $\varphi$ is even, $S_\ve$ is symmetric; for any functions
$f, g: \R^2 \to \R$:
\begin{equation}
 \int_{\R^2}  S_\ve f(z)\,\, g(z)\, dz =
 \int_{\R^2}  f(z) \,\, S_\ve g(z)\, dz.
\end{equation}
We now define global symmetric operators on $\Omega$ or $\pa \Omega$ by \eqref{eq:localtoglobal}:
\begin{align}
 \sm f = {\sum}_{\mu = 0}^N
  S_{\ve,\mu} f.
 \label{smoothing}
\end{align}

\subsubsection{Commutators with smoothing}
We have
\begin{lemma}
  \label{smoothinglemma}
 With $\sm$ defined by \eqref{smoothing}, if $k \geq m$ then:
 \begin{equation}
  ||\sm f||_{H^k(\pa \Omega)} \lesssim \ve^{m-k} ||f||_{H^{m}(\pa \Omega)},
  \quad\text{and}\quad
  ||\sm f - f||_{H^k(\pa \Omega)} \lesssim \ve ||f||_{H^{k+1}(\pa \Omega)},
 \end{equation}
 and \begin{equation}
 \|S_\varepsilon f\|_{L^\infty(\pa \Omega)}\leq \|f\|_{L^\infty(\pa \Omega)}.
 \end{equation}
 Moreover, for $k=0,1$:
 \begin{equation}
   ||\sm(fg) -f \sm g||_{H^k(\pa \Omega
   )}\lesssim \ve^{1-k} ||f||_{C^{1+k}(\pa \Omega)} ||g||_{L^2(\pa \Omega)},
 \end{equation}
 and for $n=0,1$
 \begin{equation}
   \label{hnk}
   ||\sm(fg) -f \sm g||_{H^{n,k}(\Omega
   )}\lesssim \ve^{1-k} ||f||_{C^{n,1+k}(\Omega)} ||g||_{H^n(\Omega)},
 \end{equation}
 where
 \begin{equation*}
 \|f\|_{C^{n,k}}={\sum}_{|I|\leq k,\, S\in\mathcal{S}}\|S^I\! f\|_{C^n},
 \quad\text{and}\quad
 \|f\|_{H^{n,k}}={\sum}_{|I|\leq k,\, S\in\mathcal{S}}\|S^I\! f\|_{H^n}.
 \end{equation*}
\end{lemma}
\begin{proof}
The proof for $k=0$ follows from the local expression
and the fact that $|w| \leq \ve$ in the support of $\varphi_\ve$,
\begin{equation}\label{eq:theaboveintegral}
 S_\ve(fg)(z) - f(z)S_\ve(g)(z)
 = \int_{\R^2} \varphi_\ve(w)g(z-w)\big( f(z-w) - f(z)\big)
 dw.
\end{equation}
The proof for $k=1$ follows from differentiating this and integrating by parts if the derivative falls on $g$,
see the proof of Lemma \ref{lem:mutiplicationsmoothingcommute}.
\end{proof}
There is an improvement
in the commutators with smoothing for tangential derivatives:

\begin{lemma}\label{lem:mutiplicationsmoothingcommute} We have   $[\sm,D_t]\!=\!0$. If $S\!=\!S^a(y)\pa_a$ is a tangential vector field then for
$k\!=\!0,1$:
\begin{align}
\|[\sm,S]g\,\|_{H^k(\pa\Omega)}+\|[\sm,\pa_r]\,g\|_{H^k(\pa\Omega)}&\lesssim \|g\|_{H^k(\pa\Omega)} ,\label{eq:smoothingvectorfieldcommute}\\
  ||\sm(f S g) - f \sm S g||_{H^k(\pa\Omega)}
  &\lesssim || f||_{C^{k}(\pa\Omega)} ||g||_{H^{k}(\pa\Omega)}.
  \label{eq:smoothingmultiplicationcommute}
 \end{align}
 Moreover for $n=0,1$
 \begin{align}
\|[\sm,S]g\|_{H^{n,k}(\Omega)}+\|[\sm,\pa_r]g\|_{H^{n,k}(\Omega)}&\lesssim \|g\|_{H^{n,k}(\Omega)} ,\label{eq:smoothingvectorfieldcommuteinterior}\\
  ||\sm(f S g) - f \sm S g||_{H^{n,k}(\Omega)}
  &\lesssim || f||_{C^{n,k}(\Omega)} ||g||_{H^{n,k}(\Omega)}.
  \label{eq:smoothingmultiplicationcommuteinterior}
 \end{align}
\end{lemma}

\begin{proof}[Proof of Lemma \ref{lem:mutiplicationsmoothingcommute}] In local coordinates such that $S=S^d(z)\pa/\pa z^d$, with $S^3=0$, we have,
 neglecting that the measure depends on the coordinates,
 \begin{equation*}
  \big( \sm( S g) -  S\sm \,g \big)(z)
  = \int_{\R^2}  \big( S^d(z - \ve w) - S^d(z) \big) \frac{\pa g(z - \ve w)}{\pa z^{d}} \varphi(w) \, dw.
  \end{equation*}
  Writing $(S g)(z - \ve w)= S^d(z - \ve w)\varepsilon^{-1} \pa  g(z - \ve w)/\pa w^{d}$ and integrating by parts this becomes:
  \begin{equation*}
   \big( \sm( S g) -  S\sm \,g \big)(z)=
      \int_{\R^2} \!\! \frac{\pa S^d(z \! -\! \ve w) }{\pa z^d} g(z\!-\!\ve w) \varphi(w) \, dw
    + \int_{\R^2}\!\! \frac{ S^d(z\! -\! \ve w)\! -\! S^d(z)}{\varepsilon}
    g(z\!-\!\ve w) \frac{\pa \varphi(w)}{\pa w^{ d}}  \, dw.
  \end{equation*}
   Both terms are bounded by the right-hand side of \eqref{eq:smoothingvectorfieldcommute}, for $k=0$ and the case $k=1$
   follows from differentiating this.
   In a similar way we have
  \begin{equation*}
  \big( \sm( f S g) - f \sm S g \big)(z)
  = \int_{\R^2}  \big( f(z - \ve w) - f(z) \big) (S g)(z - \ve w) \varphi(w) \, dw,
  \end{equation*}
    and integrating by parts as above we get
  \begin{multline*}
    \big( \sm( f S g) -{}_{\!} f \sm S g \big)(z)\\
    =\!  \int_{\R^2} \!\!\! (Sf)(z \! - \! \ve w) g(z \!- \!\ve w) \varphi(w) \, dw
   +\int_{\R^2}\!\!\!\frac{ f(z\! - \! \ve w)\! -{}_{\!}\! f(z)}{\varepsilon}
    g(z \!- \!\ve w) \frac{\pa \big( S^d(z\!-  \!\ve w)\varphi(w)\big)\!\!}{\pa w^{ d}}  \, dw.
  \end{multline*}
  \eqref{eq:smoothingmultiplicationcommute} follows from this.
\end{proof}

In order to control the commutators $[\nave, \sm]$, we need the following two lemmas:
\begin{lemma}\label{lem:gradientsmoothingcommute}
Suppose that
\begin{equation*}
|\pa \widetilde{x}/\pa y|+|\pa y/\pa \widetilde{x}|\leq M_0.
\end{equation*}
Then if $S=S^a(y)\pa_a$ is a tangential vector field we have
  \begin{equation}
   ||[\nave_i, \sm] S g||_{L^2(\Omega)} +
   ||S[\nave_i, \sm] g||_{L^2(\Omega)}
   + ||[\nave_i,S] g||_{L^2(\Omega)}
   \lesssim C(M_0) {\sum} _{|I|\leq 1} ||\pa S^I \widetilde{x} ||_{C^0} ||g||_{H^{1}(\Omega)}.
   \label{jpavecomm}
  \end{equation}
\end{lemma}
\begin{proof}[Proof of Lemma \ref{lem:gradientsmoothingcommute}]
In the local coordinates such that $S\!=S^d(z)\pa/\pa z^d$, with $S^3\!=0$,
we write $\widetilde{\pa}_i\!=\widehat{J}_i^d\widehat{\pa}_d$,
where $\widehat{J}_i^d=\pa \widehat{z}^d\!/\pa\widetilde{x}^i$,
and $\widehat{\pa}_d=\pa/\pa \widehat{z}^d$. We have
$[\widehat{J}_i^d\widehat{\pa}_d ,\sm]=[\widehat{J}_i^d ,\sm]\widehat{\pa}_d+\widehat{J}_i^d[\widehat{\pa}_d ,\sm]$
and
\begin{equation}
[\widehat{J}_i^d\widehat{\pa}_d ,\sm]S g=[\widehat{J}_i^d ,\sm]S\widehat{\pa}_d g+[\widehat{J}_i^d ,\sm][\widehat{\pa}_d,S]g+\widehat{J}_i^d[\widehat{\pa}_d ,\sm]S g.
\end{equation}
Here the first and main term on the right is dealt with using \eqref{eq:smoothingmultiplicationcommute}. The second one is lower order.
The last one is dealt with using \eqref{eq:smoothingvectorfieldcommute}
for $d=1,2$ and the fact that $[\widehat{\pa}_d,\sm]=0$.
\end{proof}

\subsubsection{The tangential fractional derivatives and norms}
We will need to use fractional tangential derivatives to control our solution
and we will define these operators
in coordinates.
If $F: \R^2 \to \R$, we define:
\begin{equation}
 \fd^s F(z) = \int_{\R^2} e^{iz\cdot \xi} \langle \xi \rangle^s \hat{F}(\xi)\, d\xi,
 \quad \text{where}\quad
 \hat{F}(\xi) = \int_{\R^2} e^{-iz\cdot \xi} F(z)\, dz,
\end{equation}
and we define fractional tangential derivatives on $\Omega$ by:
\begin{equation}
 \fd_\mu^{s} f = \widetilde{\chi}_\mu (\fd^{s} f_\mu)\circ\Psi_{\mu}^{-1},\quad f_\mu=(\chi_\mu f)\circ\Psi, \qquad \mu = 1,..., N.
 \label{fdmudef}
\end{equation}
We also set $\fd_0^s f \!=\!\chi_0( \langle \pa \rangle^{s} f_0)\!\circ\!\Psi_{0}^{-1}\!\!$,
where $\langle \pa \rangle^{s}$ is defined by taking the Fourier transform
in all directions.

For $s \in \R$, $k \in \mathbb{N}$, we define:
\begin{equation}
 || f ||_{H^s(\pa \Omega)} = {\sum}_{\mu = 1}^N ||\fdm^s f ||_{L^2(\pa \Omega)},
  \quad\text{and}\quad
 || f ||_{H^{(n, s)}(\Omega)}
  = {\sum}_{\mu = 0}^N ||\fdm^s f||_{H^n(\Omega)}.
 \label{sobspacedef}
\end{equation}

For $0<s<1$ let  $\S^s f:\Omega\to \mathbf{R}^N$, or $\langle\pa_\theta\rangle^s$ be the map
whose components are $\fd^{s}_\mu f$, for $\mu=0,\dots,N$,
and define the inner product
\begin{equation}\label{eq:globalhalfderdef}
\big(\fd^s f\big)
 \cdot \big(\fd^s g\big)={\sum}_{\mu=1,\dots,N} \big(\fd^s_\mu f\big)
\big( \fd^s_\mu g\big).
\end{equation}
Moreover let
$\S^{r+s} f:\Omega\to \mathbf{R}^{N+1}$ be the map whose components are
$\fd^{s}_\mu S^I  f$.  The norm of $\S^r f\!$ is
\begin{equation}\label{eq:simplifiedtangentialnotationfractional}
|\S^{r+s}\! f|^2= \S^{r+s}\! f\cdot \S^{r+s}\! f,\quad\text{where}\quad
\S^{r+s}\! f\cdot \S^{r+s} g= {\sum}_{\mu=1,\dots,N}{\sum}_{|I|=r,\,\,S^I\in\S^r}\fd^{s}_\mu S^I \!f \, \,\fd^{s}_\mu S^I g.
\end{equation}

\begin{lemma}
  \label{fracalg}
If $S \in \S$, then:
\begin{equation}
 \Big| \int_{\pa\Omega} f S g\, dS(y)\Big|
 \leq C||f||_{H^{1/2}(\pa \Omega)} ||g||_{H^{1/2}(\pa \Omega)},
 \qquad
 \Big| \int_{\Omega} f S g\, dy \Big|
 \leq C||f||_{H^{(0,1/2)}(\Omega)} ||g||_{H^{(0,1/2)}(\Omega)}.
\end{equation}
\end{lemma}

\subsubsection{Commutators with the fractional derivative}
In local coordinates we have ``Leibniz rule'':
\begin{lemma}\label{lem:LeibnitzBasic}
  If $F, G: \R^2 \to \R$ have compact support, then:
  \begin{align*}
   ||\fdh(FG) - F \fdh G||_{L^2(\R^2)}
   &\lesssim ||F||_{H^2(\R^2)} ||G||_{L^2(\R^2)},\\
   ||\fdh(FG) - F \fdh G||_{H^s(\R^2)}
   &\lesssim ||F||_{H^{3}(\R^2)} ||G||_{H^{s-1/2}(\R^2)},\qquad 0\leq s\leq 1,
  \end{align*}
\end{lemma}
\begin{proof} The Fourier transform of $\fdh(FG) - F \fdh G$ is
\begin{equation*}
\langle \xi\rangle^{\!1{}_{\!}/2} \widehat{ FG} (\xi)
  - \widehat{ (F \fdh G)}(\xi)=\int \big(\langle \xi\rangle^{\!1{}_{\!}/2}
  -\langle \xi-\eta\rangle^{\!1{}_{\!}/2}\big) \widehat{F}(\eta)  \widehat{G} (\xi-\eta) \, d\eta.
\end{equation*}
 Using the elementary estimate
 $|\langle\xi\rangle^{1/2} - \langle \xi-\eta\rangle^{1/2}| \lesssim
 \langle \eta\rangle\langle \xi\rangle^{-1/2}$ and Cauchy-Schwarz we have:
 \begin{equation*}
  \bigtwo|\langle \xi\rangle^{\!1{}_{\!}/2} \widehat{ FG} (\xi)
  - \widehat{ (F \fdh G)}(\xi)\bigtwo|^2
  \lesssim  \int_{\R^2}  \!\!\langle \eta\rangle^{4}
  |\widehat{F}(\eta)|^2 d\eta\!
  \int_{\R^2} \!\! \langle\eta\rangle^{\!-3} |\widehat{G}(\xi\!-\!\eta)|^2 d\eta.
 \end{equation*}
 Integrating in $\xi$, changing variables, and using the fact that
 $\int_{\R^2} \langle\xi-\eta\rangle^{-3} \, d\xi\! \leq \!C$,
 we have:
 \begin{equation*}
 ||\langle\xi\rangle^{\! 1{}_{\!}/2}\widehat{FG} - \widehat{(F \fdh G)}||_{L^2(\R^2)}^2 \lesssim
 ||F||_{H^2(\R^2)}^2\!
   \int_{\R^2}\!
  \int_{\R^2} \!\!\langle\xi\!-\!\eta\rangle^{\!-3}  |\widehat{G}(\eta)|^2 d\eta
  \, d\xi
  \lesssim ||F||_{H^2(\R^2)} ||G||_{L^2(\R^2)}.
  \end{equation*}
  The first estimate now follows from Plancherel's theorem.

  If $ s \!\leq \!1/2$ we can further estimate
   $|\langle\xi\rangle^{1/2} \!-  \!\langle \xi \!- \!\eta\rangle^{1/2}|  \langle \xi\rangle^s \!\lesssim \!
 \langle \eta\rangle\langle \xi\rangle^{s-1/2} \!\lesssim \!
 \langle \eta\rangle^{3/2-s}\langle \xi \!- \!\eta\rangle^{s-1/2}$ and if $s\geq 1/2$ we can estimate $|\langle\xi\rangle^{1/2}  \!-\langle \xi \!- \!\eta\rangle^{1/2}| \langle \xi\rangle^s\!\lesssim\!
 \big(\langle \eta\rangle^{s-1/2} \!+  \langle \xi \!- \!\eta\rangle^{s-1/2}\big)\langle \eta\rangle$,
 and this leads to the second estimate.
\end{proof}

We note that our Sobolev norms are independent of change of coordinates:
\begin{lemma}\label{lem:changeofvariables} Let $F:\R^2 \to \R$ has compact support and let $G=F\circ \Psi$ where
$\Psi$ be a $C^1$ diffeomorphism.
Then $\| \fd^s F\|_{L^2(\R^2)}\lesssim \| \fd^s G\|_{L^2(\R^2)}\lesssim \| \fd^s F\|_{L^2(\R^2)}$.
\end{lemma}
\begin{proof} This is directly by changing variables on the space side seen to be true for the
$L^2$ part of the norms so it suffices to prove the inequalities for homogeneous Sobolev spaces, i.e. with $\fd^s$ replaced by
$|\pa_\theta|^s$. The proof will use the alternative characterization of the fractional Sobolev norms (see Proposition 3.4 in \cite{DPV}):
\begin{equation}
\int \int \frac{|F(x)-F(y)|^2}{|x-y|^{2+2s}} dx dy= C_s \int |\xi|^{2s} |\widehat{F}(\xi)|^2 \, d\xi.
\end{equation}
With this alternative characterization the proof of the lemma just follows from changing variables, since $|x-y|\lesssim |\Psi(x)-\Psi(y)|\lesssim |x-y|$.
\end{proof}

  \begin{lemma} We have
  \begin{equation}
   ||(1-\tC_\mu) \fdh f_\mu||_{L^2(\R^2)}\lesssim ||f_\mu||_{L^2(\R^2)}.
  \end{equation}
   The same estimate holds with $\pa \Omega$ replaced by
   $\Omega$ and $H^{1/2}(\pa \Omega)$ replaced with
   $H^{(0,1/2)}(\Omega)$.
  \end{lemma}
  \begin{proof}
     Since $\tC_\mu =1$ on the support of $\chi_\mu$ and hence on the support of
    $f_\mu$ it follows from Lemma \ref{lem:LeibnitzBasic} that
    \begin{equation}
     ||(1-\tC_\mu) \fdh f_\mu||_{L^2(R)}
     =|| \fdh (\tC_\mu f_\mu)-\tC_\mu \fdh f_\mu||_{L^2(R)}
     \leq C ||f_\mu||_{L^2(R)}\leq C ||f||_{L^2(\pa \Omega)}.
      \tag*{\qedhere}
    \end{equation}
  \end{proof}

\begin{lemma}\label{lem:halfderivativepartialcommute} For $k=0,1$ we have
\begin{align}
\|[\fdhm\!,\widehat{\pa}_d\, ]\, g\|_{H^k(\pa\Omega)}&\lesssim \| g\|_{H^{k+1/2}(\pa\Omega)},\\
\|[\fdhm\!,S\, ]\, g\|_{H^k(\pa\Omega)}&\lesssim \| g\|_{H^{k+1/2}(\pa\Omega)},
\end{align}
and
\begin{align}
\|[\fdhm\!,\widehat{\pa}_d\, ]\, g\|_{H^{0,k}(\Omega)}&\lesssim \|g\|_{H^{0,k+1/2}(\Omega)},\\
\|[\fdhm\!,S\, ]\, g\|_{H^{0,k}(\Omega)}&\lesssim \|g\|_{H^{0,k+1/2}(\Omega)}.
\end{align}
\end{lemma}
\begin{proof} Since $ \fd^{1/2}=\langle(\widehat{\pa}_1,\widehat{\pa}_2)\rangle^{1/2}$ commutes with $\widehat{\pa}_d$ it is just a matter of $\widehat{\pa}_d$ falling on the cutoffs or changes of variables in the definition of $\fdhm$ which produces a lower order term of the form
 \begin{equation}
 \widetilde{\chi}_\mu (\fd^{1/2} g_\mu)\circ\Psi_{\mu}^{-1}\!\!,\qquad g_\mu=\big((\widehat{\pa}_d\chi_\mu) f\big)\circ\Psi_\mu=\sum_\nu\big((\widehat{\pa}_d\chi_\mu)\chi_\nu^2  f \big)\circ\Psi_\nu
 \circ \Psi_{\nu\mu}=\sum_\nu\big((\widehat{\pa}_d\chi_\mu)\chi_\nu  f_\nu \big)
 \circ \Psi_{\nu\mu},
\end{equation}
where $\Psi_{\nu\mu}\!=\!\Psi_\nu^{-1}\!\circ\Psi_\nu$. The inequalities for $\widehat{\pa}_d$   follows directly from Lemma
\ref{lem:LeibnitzBasic} and Lemma \ref{lem:changeofvariables} applied to these. For the case of $S=S^d(z)\widehat{\pa}_d$ there is an additional commutator in the local coordinates of $S$ and
$\fd^{1/2}$ which is also controlled by Lemma \ref{lem:LeibnitzBasic}.
\end{proof}

As a consequence of the above lemmas we have:
\begin{lemma} \label{lem:halfderleibnitz}
We have
 \begin{align*}
  ||\fdhm (fg) - f \fdhm g||_{L^2(\pa \Omega)}
 & \lesssim ||f||_{C^2(\pa \Omega)} ||g||_{L^2(\pa \Omega)},\\
  ||\fdhm (fg) - f \fdhm g||_{H^1(\pa \Omega)}
  &\lesssim ||f||_{C^3(\pa \Omega)} || g||_{H^{1/2}(\pa \Omega)},\\
  ||\fdhm (f S g) - f \fdhm S g||_{L^2(\pa \Omega)}
  &\lesssim ||f||_{C^3(\pa \Omega)} || g||_{H^{1/2}(\pa \Omega)}.
 \end{align*}
Moreover, for $n=0,1$
\begin{align}
 ||\fdhm (fg) - f \fdhm g||_{H^n(\Omega)}
   &\lesssim ||f||_{C^{n,2}( \Omega)} || g||_{H^n( \Omega)},\\
  ||\fdhm (fg) - f \fdhm g||_{H^{n,1}(\Omega)}
   &\lesssim ||f||_{C^{n,3}( \Omega)} || g||_{H^{n,1/2}( \Omega)},\\
  ||\fdhm (f  S g) - f \fdhm S g||_{H^{n,0}(\Omega)}
   &\lesssim ||f||_{C^{n,3}( \Omega)} || g||_{H^{n,1/2}( \Omega)}.
  \label{leib3}
 \end{align}
\end{lemma}

Moreover
\begin{lemma}\label{lem:gradientfractionalcommute}
Suppose that
\begin{equation*}
|\pa \widetilde{x}/\pa y|+|\pa y/\pa \widetilde{x}|\leq M_0.
\end{equation*}
We have
  \begin{equation}
   ||\,[\nave_i, \fdhm] f||_{L^2(\Omega)}
   \lesssim C(M_0) {\sum} _{|I|\leq 2} ||\pa S^I \widetilde{x} ||_{C^0} || f||_{H^{1}(\Omega)}.
  \end{equation}
  and
    \begin{equation}
   ||\,S [\nave_i, \fdhm] f||_{L^2(\Omega)}+||\, [\nave_i, \fdhm] S f||_{L^2(\Omega)}
   \lesssim C(M_0) {\sum} _{|I|\leq 3} ||\pa S^I \widetilde{x} ||_{C^0} || f||_{H^{1,1/2}(\Omega)}.
  \end{equation}
\end{lemma}
\begin{proof} Writing $\widetilde{\pa}_i\!=\widehat{J}_i^d\widehat{\pa}_d$ we have
\begin{equation}
[\widehat{J}_i^d\widehat{\pa}_d ,\fdhm]=\widehat{J}_i^d[\widehat{\pa}_d ,\fdhm]+[\widehat{J}_i^d ,\fdhm]\widehat{\pa}_d,
\end{equation}
 where the first term is estimated by Lemma \ref{lem:halfderivativepartialcommute} and the second by Lemma \ref{lem:halfderleibnitz}.
\end{proof}

\subsubsection{Commutators with smoothing and the fractional derivative}
Since both smoothing and fractional derivatives are multiplication operators on the Fourier side it follows that they commute in local coordinates and hence
\begin{equation}
\| \fd^s \sm f\|_{H^{k}(\R^2)}\lesssim  \| \fd^s  f\|_{H^{k}(\R^2)}.
\end{equation}
Similarly in local coordinates $[\widehat{\pa}_d,\sm]$ is either $0$ or lower order. Therefore as in the proof of Lemma \ref{lem:halfderivativepartialcommute} we have
\begin{lemma}\label{lem:fractionalsmoothing} For $k=0,1$, $0\leq s\leq 1$
\begin{align}
\| \fd^s_\mu \sm f\|_{H^{k}(\pa\Omega)}&\lesssim  \| f\|_{H^{k+s}(\pa\Omega)},\\
\| [\fd^s_\mu, \sm ] f\|_{H^{k}(\pa\Omega)}&\lesssim  \| f\|_{H^{k+s-1}(\pa\Omega)},\\
\| [\widehat{\pa}_d , \sm ] f\|_{H^{k}(\pa\Omega)}&\lesssim  \| f\|_{H^{k-1}(\pa\Omega)},
\end{align}
and for $n=0,1$
 \begin{align}
\| \fd^s_\mu \sm f\|_{H^{n,k}(\Omega)}&\lesssim  \|  f\|_{H^{n,k+s}(\Omega)},\\
\|[\fd^s_\mu ,\sm ] f\|_{H^{n,k}(\Omega)}&\lesssim  \|  f\|_{H^{n,k+s-1}(\Omega)},\\
\|[\widehat{\pa}_d ,\sm ] f\|_{H^{n,k}(\Omega)}&\lesssim  \|  f\|_{H^{n,k-1}(\Omega)}.
\end{align}
\end{lemma}
We can now generalize Lemma \ref{lem:mutiplicationsmoothingcommute} to estimate in the fractional norm:
\begin{lemma}\label{lem:mutiplicationsmoothingcommutefractional} We have   $[\sm,D_t]\!=\!0$. If $S\!=\!S^a(y)\pa_a$ is a tangential vector field then for
$0\leq s\leq 1$:
\begin{align}
\|[\sm,S]g\,\|_{H^s(\pa\Omega)}+\|[\sm,\pa_r]\,g\|_{H^s(\pa\Omega)}&\lesssim \|g\|_{H^s(\pa\Omega)} ,\label{eq:smoothingvectorfieldcommutefractional}\\
  ||\sm(f S g) - f \sm S g||_{H^s\pa\Omega)}
  &\lesssim || f||_{C^{1}(\pa\Omega)} ||g||_{H^{s}(\pa\Omega)}.
  \label{eq:smoothingmultiplicationcommutefractional}
 \end{align}
 Moreover for $n=0,1$
 \begin{align}
\|[\sm,S]g\|_{H^{n,s}(\Omega)}+\|[\sm,\pa_r]g\|_{H^{n,s}(\Omega)}&\lesssim \|g\|_{H^{n,s}(\Omega)} ,\label{eq:smoothingvectorfieldcommuteinteriorfractional}\\
  ||\sm(f S g) - f \sm S g||_{H^{n,s}(\Omega)}
  &\lesssim || f||_{C^{n,1}(\Omega)} ||g||_{H^{n,s}(\Omega)}.
  \label{eq:smoothingmultiplicationcommuteinteriorfractional}
 \end{align}
\end{lemma}
\begin{proof} By the proof of Lemma \ref{lem:mutiplicationsmoothingcommute} in local coordinates such that $S=S^d(z)\pa/\pa z^d$, with $S^3=0$, we have,
 neglecting that the measure depends on the coordinates,
  \begin{equation*}
   \big( \sm( S g) -  S\sm \,g \big)(z)=
      \int_{\R^2} \!\! \frac{\pa S^d(z \! -\! \ve w) }{\pa z^d} g(z\!-\!\ve w) \varphi(w) \, dw
    + \int_{\R^2}\!\! \frac{ S^d(z\! -\! \ve w)\! -\! S^d(z)}{\varepsilon}
    g(z\!-\!\ve w) \frac{\pa \varphi(w)\!}{\pa w^{ d}}  \, dw.
  \end{equation*}
Since by Lemma \ref{lem:changeofvariables} the fractional Sobolev norm is invariant under changes of coordinates and the same coordinate system works in the overlap of the cutoffs we can apply Lemma \ref{lem:LeibnitzBasic} in the same coordinate system as the smoothing to the expression above below the integral signs and that gives
\eqref{eq:smoothingvectorfieldcommutefractional} and \eqref{eq:smoothingvectorfieldcommuteinteriorfractional}.

Also by the proof of Lemma  \ref{lem:mutiplicationsmoothingcommute}
\begin{multline*}
    \big( \sm( f S g) - {}_{\!}f \sm S g \big)(z)\\
    =\!  \int_{\R^2} \!\!\! (Sf)(z \! - \! \ve w) g(z \!- \!\ve w) \varphi(w) \, dw
   +\int_{\R^2}\!\!\!\frac{ f(z\! - \! \ve w)\! -{}_{\!}\! f(z)}{\varepsilon}
    g(z \!- \!\ve w) \frac{\pa \big( S^d(z\!-  \!\ve w)\varphi(w)\big)\!\!}{\pa w^{ d}}  \, dw,
  \end{multline*}
  and similarly applying Lemma \ref{lem:LeibnitzBasic} below the integral signs give
  \eqref{eq:smoothingmultiplicationcommutefractional} and
  \eqref{eq:smoothingmultiplicationcommuteinteriorfractional}.
\end{proof}

Combining the above lemmas we get:
\begin{lemma} \label{lem:halfderleibnitzandsmoothing}
We have
 \begin{align*}
  ||\fdhm\sm (fg) - f \fdhm\sm g||_{L^2(\pa \Omega)}
 & \lesssim ||f||_{C^2(\pa \Omega)} ||g||_{L^2(\pa \Omega)},\\
  ||\fdhm \sm(fg) - f \fdhm \sm g||_{H^1(\pa \Omega)}
  &\lesssim ||f||_{C^3(\pa \Omega)} || g||_{H^{1/2}(\pa \Omega)},\\
  ||\fdhm \sm(f S g) - f \fdhm \sm S g||_{L^2(\pa \Omega)}
  &\lesssim ||f||_{C^3(\pa \Omega)} || g||_{H^{1/2}(\pa \Omega)}.
 \end{align*}
Moreover, for $n=0,1$
\begin{align*}
 ||\fdhm\sm (fg) - f \fdhm\sm g||_{H^n(\Omega)}
   &\lesssim ||f||_{C^{n,2}( \Omega)} || g||_{H^n( \Omega)},\\
  ||\fdhm\sm (fg) - f \fdhm\sm g||_{H^{n,1}(\Omega)}
   &\lesssim ||f||_{C^{n,3}( \Omega)} || g||_{H^{n,1/2}( \Omega)},\\
  ||\fdhm\sm (f  S g) - f \fdhm\sm S g||_{H^{n,0}(\Omega)}
   &\lesssim ||f||_{C^{n,3}( \Omega)} || g||_{H^{n,1/2}( \Omega)}.
 \end{align*}
\end{lemma}

\begin{lemma}\label{lem:gradientfractionalsmoothingandfractionalderivativecommute}
Suppose that
\begin{equation*}
|\pa \widetilde{x}/\pa y|+|\pa y/\pa \widetilde{x}|\leq M_0.
\end{equation*}
We have
  \begin{equation}
   ||\, [\nave_i, \fdhm\sm ] f||_{L^2(\Omega)}
   \lesssim C(M_0) {\sum} _{|I|\leq 2} ||\pa S^I \widetilde{x} ||_{C^0} || f||_{H^{1}(\Omega)},
   \label{fractionalandsmoothingcomm}
  \end{equation}
  and
   \begin{equation}
   ||\,S [\nave_i,{}_{\!} \fdhm\sm ] f||_{L^2(\Omega)}
   +||\, [\nave_i,{}_{\!} \fdhm\sm ] S f||_{L^2(\Omega)}
   \lesssim C(M_0) {\sum} _{|I|\leq 3} ||\pa S^I \widetilde{x} ||_{C^0} || f||_{H^{1,1/2}(\Omega)}.
   \label{fractionalandsmoothingcommimproved}
  \end{equation}
\end{lemma}
\begin{proof} We have $[\nave_i, \! \fdhm \!\sm ] f\!=\![\nave_i,  \!\fdhm]\sm  f\!+\!\fdhm[\nave_i, \sm ] f$. The first term is estimated by Lemma \ref{lem:gradientfractionalcommute}  and  the second term can be estimated by Lemma \ref{lem:gradientsmoothingcommute}. This proves \eqref{fractionalandsmoothingcomm} and \eqref{fractionalandsmoothingcommimproved} for the first term with $S$ to the left of the commutator.

It remains to prove \eqref{fractionalandsmoothingcommimproved} for the second term with $S$ to the right of the commutator. We have
\begin{equation}
[\nave_i, \! \fdhm \!\sm ]S f\!=\![\nave_i,  \!\fdhm]\sm  S f\!+\!\fdhm[\nave_i, \sm ]S f.
\end{equation}
Here
\begin{equation}
[\nave_i,  \!\fdhm]\sm  S f=[\nave_i,  \!\fdhm]S\sm  f\!+\![\nave_i,  \!\fdhm][\sm,S]   f,
\end{equation}
where the first term is estimated by Lemma \ref{lem:gradientfractionalcommute} and
Lemma \ref{lem:fractionalsmoothing}, and the second by
Lemma \ref{lem:mutiplicationsmoothingcommute} and Lemma \ref{lem:gradientfractionalcommute}.
Hence it remains to estimate
\begin{equation}
\fdhm[\nave_i, \sm ]S f=\fdhm[\widehat{J}_i^d\widehat{\pa}_d, \sm ]S f
\!=\!\fdhm[\widehat{J}_i^d \!,\sm]\widehat{\pa}_d S f\!+\!\fdhm\!\widehat{J}_i^d[\widehat{\pa}_d ,\sm]S f.
\end{equation}
Since $[\widehat{\pa}_d,S]$ is either $0$ or a tangential vector field it follows that  $\fdhm[\widehat{J}_i^d \!,\sm][\widehat{\pa}_d, S]f$ can be estimated by  Lemma \ref{lem:mutiplicationsmoothingcommute}. Moreover
$[\fdhm\!,\widehat{J}_i^d][\widehat{\pa}_d ,\sm]S f$ can be estimated by Lemma
 \ref{lem:halfderleibnitz} and Lemma \ref{lem:mutiplicationsmoothingcommute}.
 Hence it remains to estimate
 \begin{equation}
 \fdhm[\widehat{J}_i^d \!,\sm]S\widehat{\pa}_d f
 +\widehat{J}_i^d\fdhm[\widehat{\pa}_d ,\sm]S f.
 \end{equation}
 Here the $L^2$ norm of the second term can be estimated by the $L^\infty$ norm of
 $\widehat{J}_i^d$ time the $L^2$ norm of the other factor which is a tangential pseudo differential operator of order $3/2$ so it is under control by the right hand side of \eqref{fractionalandsmoothingcommimproved}.
Hence it remains to estimate
\begin{equation}
\fdhm[\widehat{J}_i^d \!,\sm]S\widehat{\pa}_d f
=\fdhm\widehat{J}_i^d \sm S\widehat{\pa}_d f
- \fdhm\sm\bigtwo(\widehat{J}_i^d  S\widehat{\pa}_d f \bigtwo).
\end{equation}
Here
\begin{equation}
\fdhm\widehat{J}_i^d \sm S\widehat{\pa}_d f
=\widehat{J}_i^d \fdhm\sm S\widehat{\pa}_d f
+[\fdhm\!,\widehat{J}_i^d ] S\sm \widehat{\pa}_d f
+[\fdhm\!,\widehat{J}_i^d ][\sm, S]\widehat{\pa}_d f,
\end{equation}
where the second term in the right can be estimated by Lemma \ref{lem:halfderleibnitz} and Lemma \ref{lem:fractionalsmoothing} and the third term by Lemma \ref{lem:halfderleibnitz} and Lemma \ref{lem:mutiplicationsmoothingcommute}.
Hence it remains to estimate
\begin{equation}
\widehat{J}_i^d \fdhm\sm S\widehat{\pa}_d f
-\fdhm\sm\bigtwo(\widehat{J}_i^d  S\widehat{\pa}_d f \bigtwo),
\end{equation}
which follows from Lemma \ref{lem:halfderleibnitzandsmoothing}.
\end{proof}

\section{Basic elliptic estimates}\label{sec:ellipticestimatesproofs}
We collect here some elliptic estimates which will be used in the course of
the proof. These estimates all appear in \cite{GLL19} as well as in
some of the earlier references \cite{L05a}, \cite{CHS13}.
\subsubsection{The estimates used to estimate $V$}
The proof of the following lemma can be found in
\cite{L05a}, \cite{GLL19}.
\begin{lemma}
  \label{app:pwdiff}
  There is a constant $c_0 = c_0(|\pa \xve|)$ so that
  if $\alpha$ is a $(0,1)$-tensor
  on $\Omega$ then
  \begin{equation}
   |\widetilde{\pa} \alpha |
   \leq c_0\big(
   |\widetilde{\div} \alpha |+
   |\widetilde{\curl} \alpha |
   +  |\S \alpha |\big), \qquad \text{ on } \Omega.
   \label{pwnonrel}
  \end{equation}
\end{lemma}

\subsubsection{The improved half derivative estimates used to estimate the coordinates}\label{sec:divcurlL2}

\begin{prop}\label{prop:divcurlL2} There is a constant
  $C_{\!0}$ depending on $\|\pa \xve\|_{L^\infty(\Omega)}$ so that
  if $\alpha$ is a vector field on $\Omega$ then
\begin{equation}
 \! ||\alpha||_{H^{1{}_{\!}}(\Omega)}^2\!
\leq\! C_{{}_{\!}0\!}\Big(\! ||\!\div \alpha||_{L^2(\Omega)}^2\!
 +{}_{\!} ||\!\curl \alpha||_{L^2(\Omega)}^2\!
 + \! \!\int_{\pa \Omega}\!\!\!\! \mathcal{N}_{{}_{\!}i} \mathcal{N}_{\!j}\fdh\! \alpha^i
 \!\cdot_{\!} \fdh\! \alpha^j dS
 + ||\alpha||_{L^2(\pa \Omega)}^2\!
 + ||\alpha||_{L^2(\Omega)}^2\!\Big).
 \label{ftang1}
\end{equation}
Here $\fdh$ is a half angular derivative defined locally in coordinates in \eqref{fdmudef},
and the inner product is the sum over coordinate charts $\fdh \alpha^i
 \cdot \fdh \alpha^j=\sum_{\mu} \big(\fdhm \alpha^i\big)
\big(\fdhm \alpha^j\big)$  in \eqref{eq:globalhalfderdef}. Moreover
\begin{equation}
 \! ||\alpha||_{H^{1{}_{\!}}(\Omega)}^2\!
\leq\! C_{{}_{\!}1\!}\Big(\! ||\!\div \alpha||_{L^2(\Omega)}^2\!
 +{}_{\!} ||\!\curl \alpha||_{L^2(\Omega)}^2\!
 + \! \!\int_{\pa \Omega}\!\!\!\!\gamma_{ij}\fdh\! \alpha^i
 \!\cdot_{\!} \fdh\! \alpha^j dS
 + ||\alpha||_{L^2(\pa \Omega)}^2\!
 + ||\alpha||_{L^2(\Omega)}^2\!\Big).
\end{equation}
where $\gamma_{ij}$ is the tangential metric.
\end{prop}
\begin{prop}\label{prop:divcurlL22} There is a constant
  $C_{\!0}$ depending on $\|\pa \xve\|_{L^\infty(\Omega)}$ so that
  if $\beta$ is a vector field on $\Omega$ then
\begin{multline}
 \! ||\fdh \beta||_{H^{1{}_{\!}}(\Omega)}^2\!\\
\leq\! C_{{}_{\!}1\!}\Big(\! ||\!\div\fdh \beta||_{L^2(\Omega)}^2\!
 +{}_{\!} ||\!\curl \fdh\beta||_{L^2(\Omega)}^2\!
 + \! \!\int_{\pa \Omega}\!\!\!\! \mathcal{N}_{{}_{\!}i} \mathcal{N}_{\!j\,} \S \beta^i
 \!\cdot \S\beta^j dS
 + ||\beta||_{L^2(\pa \Omega)}^2\!
 + || \S\beta||_{L^2(\Omega)}^2\!\Big).
\end{multline}
Here $\fdh$ is a half angular derivative defined locally in coordinates in \eqref{fdmudef},
and $S \beta^i\!\cdot \S\beta^j$ is the inner product of all tangential derivatives defined in \eqref{eq:simplifiedtangentialnotation}. Moreover
\begin{equation}
 \! ||\fdh\beta||_{H^{1{}_{\!}}(\Omega)}^2\!
\leq\! C_{{}_{\!}1\!}\Big(\! ||\!\div \fdh\beta||_{L^2(\Omega)}^2\!
 +{}_{\!} ||\!\curl \fdh\beta||_{L^2(\Omega)}^2\!
 + \! \!\int_{\pa \Omega}\!\!\!\!\gamma_{ij\,} \S \beta^i
 \!\cdot  \S\beta^j dS
 + ||\beta||_{L^2(\pa \Omega)}^2\!
 + || \S \beta||_{L^2(\Omega)}^2\!\Big).
\end{equation}
where $\gamma_{ij}$ is the tangential metric.
\end{prop}
These propositions are a consequence of the following two lemmas proven below:

\begin{lemma}\label{lem:ellipticboundary}
 If $\alpha$ is a vector field then:
 \begin{equation}
   ||\widetilde{\pa} \alpha||_{L^2(\tD_t)}^2   =
   ||\widetilde{\div} \alpha||_{L^2(\tD_t)}^2 + \frac{1}{2} ||\widetilde{\curl} \alpha||_{L^2(\tD_t)}^2
  +\int_{\pa \tD_t} \Big(\alpha^j (\gamma_j^k \widetilde{\pa}_k \alpha_i) \mathcal{N}^i
  - \alpha_i (\gamma_j^k \widetilde{\pa}_k \alpha^j )\mathcal{N}^i\Big).
 \end{equation}
\end{lemma}

\begin{lemma}\label{lem:normaltangential}
  If $\alpha$ is a (0,1)-tensor
  on $\Omega$ and $\gamma$ denotes the metric on $\pa \Omega_t$, then:
 \begin{equation}
  \Big|\int_{\pa \Omega} \big(\gamma^{ij} - \mathcal{N}^i\mathcal{N}^j\big)\alpha_i
  \alpha_j \kappa dy_\gamma\Big|
  \leq 2\Big|\int_{\Omega} \div ( \alpha)\, \alpha_{j}\mathcal{N}^j+ \curl \alpha_{ij}\,
  \alpha^i \mathcal{N}^j dx \Big|
  + K ||\alpha||_{L^2(\Omega)}^2.
 \end{equation}
\end{lemma}
We also need estimates for the Dirichlet problem that keep track of the regularity of the boundary and that uses the minimal amount of regularity of the boundary:
\begin{prop}\label{prop:dirichlet} Suppose that $q=0$ on $\pa \Omega$. Then
\begin{equation}
\|\widetilde{\pa}  T^K \widetilde{\pa} q\|_{L^2(\Omega)}
\lesssim c_0{\sum}_{S\in\mathcal{S}}\| \widetilde{\pa} {S} T^{K} \widetilde{x}\|_{L^2(\Omega)}
+c_K\!{\sum}_{|K'|\leq |K|} \bigtwo( \| T^{K'} \widetilde{\triangle} q\|_{L^2(\Omega)}
+\| \widetilde{\pa}T^{K'} \widetilde{x}\|_{L^2(\Omega)}\bigtwo) ,
\end{equation}
and
\begin{equation}\label{eq:thehalfcasedirichletrel}
\|\widetilde{\pa} \fdh  T^K \widetilde{\pa} q\|_{L^2(\Omega)}\lesssim c_K\!{\sum}_{|K'|\leq |K|,\, k=0,1}\!\bigtwo( \| \fd^{\!\nicefrac{k}{2}} T^{K'\!}\widetilde{\triangle} q \|_{L^2(\Omega)}+
\| \widetilde{\pa}\fd^{\!\nicefrac{k}{2}} \mathcal{S}^1 T^{K'\!} \widetilde{x}\|_{L^2(\Omega)} \bigtwo),
\end{equation}
where $c_K$ depends on $ \widetilde{\pa} T^{N}\widetilde{x}$ and  $\widetilde{\pa} T^{N}\widetilde{\pa}q$, for $|N|\leq |K|/2+3$.
\end{prop}
The proof follows from Lemma \ref{lem:dirichlet} below
\begin{lemma}\label{lem:dirichlet} Suppose that $q=0$ on $\pa \Omega$. Then
\begin{multline}\label{eq:thewholecase}
\|\widetilde{\pa}  T^K \widetilde{\pa} q\|_{L^2(\Omega)}\lesssim
c_0{\sum}_{S\in\mathcal{S}}\| \widetilde{\pa} {S} T^{K} \widetilde{x}\|_{L^2(\Omega)}
+c_0\| \widetilde{\div} \big( T^{K} \widetilde{\pa} q\big)\|_{L^2(\Omega)}\\
 +c_K{\sum}_{|L|\leq |K|-1}\| \widetilde{\div} \big( T^{L} \widetilde{\pa} q\big)\|_{L^2(\Omega)}
+c_K{\sum}_{|K'|\leq |K|}\| \widetilde{\pa} T^{K'} \widetilde{x}\|_{L^2(\Omega)} ,
\end{multline}
and
\begin{equation}\label{eq:thehalfcase}
\|\widetilde{\pa} \fdh  T^K \widetilde{\pa} q\|_{L^2(\Omega)}\lesssim c_K\!{\sum}_{|K'|\leq |K|,\, k=0,1}\!\bigtwo( \| \widetilde{\div} \big(\fd^{\!\nicefrac{k}{2}} T^{K'\!} \widetilde{\pa} q\big)\|_{L^2(\Omega)}
+\| \widetilde{\pa}\fd^{\!\nicefrac{k}{2}} \mathcal{S}^1 T^{K'\!} \widetilde{x}\|_{L^2(\Omega)} \bigtwo),
\end{equation}
where $c_K$ depends on $ \widetilde{\pa} T^{N}\widetilde{x}$ and  $\widetilde{\pa} T^{N}\widetilde{\pa}q$,  for $|N|\leq |K|/2+3$.
\end{lemma}
Lemma \ref{lem:dirichlet} is a consequence of
the following two lemmas, Lemma \ref{app:pwdiff} applied $\alpha=T^K \widetilde{\pa} q$, and induction.
\begin{lemma}\label{lem:weakdirichlet} Suppose that $q=0$ on $\pa \Omega$. Then
\begin{multline}
\|S T^K \widetilde{\pa} q\|_{L^2(\Omega)}\lesssim
c_0{\sum}_{S\in\mathcal{S}}\| \widetilde{\pa} {S} T^{K} \widetilde{x}\|_{L^2(\Omega)}
+c_0\| \widetilde{\div} \big( T^{K} \widetilde{\pa} q\big)\|_{L^2(\Omega)}\\
 +c_K{\sum}_{|L|\leq |K|-1}\| \widetilde{\pa}  T^{L} \widetilde{\pa} q\|_{L^2(\Omega)}
+c_K{\sum}_{|K'|\leq |K|}\| \widetilde{\pa} T^{K'} \widetilde{x}\|_{L^2(\Omega)} ,
\end{multline}
and
\begin{equation}
\|S \fdh   T^K \widetilde{\pa} q\|_{L^2(\Omega)}\!
\lesssim c_K\!{\sum}_{|K'|\leq |K|}\!\bigthree(\| \widetilde{\div} \big(\fdh  T^{K'} \widetilde{\pa} q\big)\|_{L^2(\Omega)}
+{\sum}_{k=0,1}\!\| \widetilde{\pa}\fdh S^k T^{K'} \widetilde{x}\|_{L^2(\Omega)}\!\bigthree) ,
\end{equation}
where $c_K$ depends on $ \widetilde{\pa} T^{N}\widetilde{x}$ and  $\widetilde{\pa} T^{N}\widetilde{\pa}q$,  for $|N|\leq |K|/2+3$.
\end{lemma}

\begin{lemma}\label{lem:curl} Let
$A_{ij}^J=\widetilde{\pa}_i T^J \widetilde{\pa}_j q-\widetilde{\pa}_j T^J \widetilde{\pa}_i q$. We have
\begin{equation}\label{eq:anitsymmetricestimateappb}
\|A^K\|_{L^2(\Omega)}\lesssim
c_0\| \widetilde{\pa}  T^{K} \widetilde{x}\|_{L^2(\Omega)}
+c_K\!{\sum}_{|L|\leq |K|-1}\!\bigthree(\| \widetilde{\pa}  T^{L} \widetilde{\pa} q\|_{L^2(\Omega)}
+ \| \widetilde{\pa}  T^{L} \widetilde{x}\|_{L^2(\Omega)}\bigthree),
\end{equation}
where $c_K$ stands for a constant that depends on
$\widetilde{\pa} T^N \widetilde{x}$ and  $\widetilde{\pa} T^{N}\widetilde{\pa}q$,  for
$|N|\leq |K|/2 $.

Moreover let
$A_{ij}^{J,\nicefrac{1}{2}}
=\widetilde{\pa}_i \fdh T^J \widetilde{\pa}_j q-\widetilde{\pa}_j \fdh T^J \widetilde{\pa}_i q$. Then
\begin{multline}\label{eq:anitsymmetricestimatehalfappb}
\|A^{K,\nicefrac{1}{2}}\|_{L^2(\Omega)}
\lesssim c_0\|\widetilde{\pa} T^K \widetilde{\pa} q\|_{L^2(\Omega)}\\
+c_K\!\!\! \sum_{k=0,1}\!\!\! \bigtwo( \sum_{|L|\leq |K|-1}\!\!\!\!\!\!\| \widetilde{\pa} \big( \fd^{\!\nicefrac{k}{2}}T^{L} \widetilde{\pa} q\big)\|_{L^2(\Omega)}
+ \!\!\!\!\!\!\!\! \!\!\!\sum_{|N|\leq |K|/2+3}\!\!\!\!\!\!\!\!\!\| \widetilde{\pa} T^N \widetilde{\pa} q\|_{L^\infty(\Omega)}
\!\!\!\!\!\!\sum_{|K'|\leq |K|}\!\!\!\!\! \| \widetilde{\pa} \fd^{\!\nicefrac{k}{2}} T^{K'} \widetilde{x}\|_{L^2(\Omega)}\bigtwo).
\end{multline}
\end{lemma}

\subsubsection{The proofs of the basic elliptic estimates}

\begin{proof}[Proof of Lemma \ref{lem:ellipticboundary}]
  Integrating by parts:
  \begin{equation}
   ||\widetilde{\pa} \alpha||_{L^2(\tD_t)}^2
   = -\int_{\tD_t} \delta^{ij}
   \alpha_i \Dve \alpha_j + \int_{\pa \tD_t}
   \delta^{ij} \alpha_i \mathcal{N}^k \widetilde{\pa}_k \alpha_j.
   \label{ibpapp1}
  \end{equation}
  We insert the identity:
  \begin{equation}
   \Delta \alpha_j = \delta^{k\ell}\widetilde{\pa}_k(\widetilde{\pa}_\ell \alpha_j)
   = \delta^{k\ell} \widetilde{\pa}_k \big( \widetilde{\pa}_j \alpha_\ell + \curl \alpha_{\ell j}\big)
   = \widetilde{\pa}_j \div \alpha + \delta^{k\ell}\widetilde{\pa}_k \curl \alpha_{\ell j},
  \end{equation}
into the first term in \eqref{ibpapp1}
  and integrate by parts again:
  \begin{equation}
   \int_{\tD_t}\delta^{ij} \alpha_i \Dve \alpha_j
   = \int_{\pa \tD_t} \mathcal{N}^i \alpha_i \div \alpha
   + \delta^{ij} \mathcal{N}^\ell \alpha_i \curl \alpha_{\ell j} dS
   - \int_{\tD_t} (\div \alpha)^2
   + \delta^{k\ell}\delta^{ij} \widetilde{\pa}_k \alpha_i
   \curl \alpha_{\ell j}.
   \label{green1}
  \end{equation}
  Note that by the antisymmetry of curl:
  \begin{multline}
   \delta^{k\ell}\delta^{ij} \widetilde{\pa}_k \alpha_i
   \curl \alpha_{\ell j}\!\\
   =\frac{1}{2} \delta^{k\ell}\delta^{ij} (\widetilde{\pa}_k \alpha_i+\widetilde{\pa}_i \alpha_k)
   \curl \alpha_{\ell j}+\frac{1}{2} \delta^{k\ell}\delta^{ij} (\widetilde{\pa}_k \alpha_i-\widetilde{\pa}_i \alpha_k)
   \curl \alpha_{\ell j}\!=\frac{1}{2} \delta^{k\ell}\delta^{ij}  \curl \alpha_{k i}
   \curl \alpha_{\ell j},
  \end{multline}
  so \eqref{ibpapp1} becomes:
  \begin{equation}
   ||\widetilde{\pa} \alpha||_{L^2(\tD_t)}^2
   = ||\div \alpha||_{L^2(\tD_t)}^2 + \frac{1}{2} ||\curl \alpha||_{L^2(\tD_t)}^2
   +\int_{\pa \tD_t}
   \mathcal{N}^k \alpha^j \widetilde{\pa}_k \alpha_j -
   \mathcal{N}^i \alpha_i \div \alpha
    - \mathcal{N}^\ell \alpha^j
   \curl \alpha_{\ell j} .
   \end{equation}
   Here:
   \begin{multline}
    \mathcal{N}^k \alpha^j \widetilde{\pa}_k \alpha_j -
   \mathcal{N}^i \alpha_i \div \alpha
    - \mathcal{N}^\ell \alpha^j
   \curl \alpha_{\ell j}= \mathcal{N}^k \alpha^j \widetilde{\pa}_j \alpha_k -
   \mathcal{N}^i \alpha_i \div \alpha\\
   =\mathcal{N}^k \alpha_\ell \mathcal{N}^\ell  \mathcal{N}^j\widetilde{\pa}_j \alpha_k +\mathcal{N}^k \alpha_\ell \gamma^{\ell j}\widetilde{\pa}_j \alpha_k -
   \mathcal{N}^i \alpha_i( \mathcal{N}^k \mathcal{N}^\ell + \gamma^{\ell k})\widetilde{\pa}_k \alpha_\ell
   =\mathcal{N}^k \alpha_\ell \gamma^{\ell j}\widetilde{\pa}_j \alpha_k -
   \mathcal{N}^i \alpha_i \gamma^{\ell k}\widetilde{\pa}_k \alpha_\ell.\tag*{\qedhere}
   \end{multline}
\end{proof}

\begin{proof}[Proof of Lemma \ref{lem:normaltangential}]
 We have the following identity
\begin{equation*}
\pa_i\big(   \alpha^i \alpha_j \mathcal{N}^j \big)
 - \pa_j\big( \alpha^i  \alpha_i\, \mathcal{N}^j \big)/2
 =\div (\alpha)\, \alpha_{j}\mathcal{N}^j+ \curl \alpha_{ij}\,
  \alpha^i \mathcal{N}^j+\alpha^i \alpha^j \pa_i \mathcal{N}_j -|\alpha|^2  \pa_j \mathcal{N}^j /2.
\end{equation*}
Integrating this over the domain gives the lemma.
\end{proof}

\begin{proof}[Proof of Lemma \ref{lem:weakdirichlet}] Integrating by parts we get
\begin{multline}
\int_\Omega S   S T^K\! q\, \, \widetilde{\div} \big( T^K \widetilde{\pa} q\big) d\widetilde{x}
=- \int_\Omega \widetilde{\pa}_i S   S T^K\! q\, \,  T^K \widetilde{\pa}^i q d\widetilde{x}\\
=-\int_\Omega S\widetilde{\pa}_i S T^K \!q\, \, T^K \widetilde{\pa}^i q\, d\widetilde{x}
+\int_\Omega \widetilde{\pa}_i S\widetilde{x}^k\, \, \widetilde{\pa}_k S  T^K \! q\, \,  T^K \widetilde{\pa}^i q\, d\widetilde{x}\\
=\int_\Omega \widetilde{\pa}_i S T^K\!q\, \,(S+\widetilde{\div} S) T^K \widetilde{\pa}^i q\,  d\widetilde{x}
+\int_\Omega \widetilde{\pa}_i S\widetilde{x}^k\, \, \widetilde{\pa}_k S T^K \! q\, \,  T^K \widetilde{\pa}^i q\, d\widetilde{x} .
\end{multline}
The proof of the first inequality follows from this and
\begin{equation}
\widetilde{\pa}_i T^J \! q -\! T^J\widetilde{\pa}_i   q
=R_i^J\!,\quad\text{where}\quad R_i^J\!=\widetilde{\pa}_i T^J\widetilde{x}^k\,\,\widetilde{\pa}_k q+\!\!\!\!\!\!\sum_{J_1+\dots+J_k=J,\,|J_i|<|J|}\!\!\!\!\! r_{\!J_1\dots J_k}^J\!\widetilde{\pa}_i T^{J_1} \widetilde{x}\cdots  \! \widetilde{\pa}T^{J_{k-1}} \widetilde{x}\cdot  T^{J_k} {\widetilde{\pa}}^{i}  q.
\end{equation}

To prove the second inequality we integrate by parts again
\begin{multline}
\int_\Omega S  \fdh S T^K q\, \, \widetilde{\div} \big(\fdh  T^K \widetilde{\pa} q\big) d\widetilde{x}
=- \int_\Omega \widetilde{\pa}_i\big( S  \fdh S T^K q\big)\, \,\big(\fdh  T^K \widetilde{\pa}^i q\big) d\widetilde{x}\\
=-\int_\Omega S\widetilde{\pa}_i\big(\fdh  S T^K q\big)\, \,\fdh  T^K \widetilde{\pa}^i q\, d\widetilde{x}
+\int_\Omega \widetilde{\pa}_i S\widetilde{x}^k\, \, \widetilde{\pa}_k\big(\fdh S  T^K q\big)\, \,\fdh  T^K \widetilde{\pa}^i q\, d\widetilde{x}\\
=\int_\Omega \widetilde{\pa}_i\big( \fdh S T^K q\big)\, \,(S+\widetilde{\div} S)\fdh  T^K \widetilde{\pa}^i q\,  d\widetilde{x}
+\int_\Omega \widetilde{\pa}_i S\widetilde{x}^k\, \, \widetilde{\pa}_k\big(\fdh  S T^K q\big)\, \,\fdh  T^K \widetilde{\pa}^i q\, d\widetilde{x} .
\end{multline}
Here by Lemma \ref{lem:gradientfractionalcommute}
\begin{equation}
\bigtwo\|  \widetilde{\pa}_i\fdh S T^K q - \fdh \widetilde{\pa}_i S T^K  q\bigtwo\|_{L^2(\Omega)}\lesssim c_0\| \widetilde{\pa} T^K q\|_{H^{\, 0,1/2}(\Omega)}.
\end{equation}
Using Lemma \ref{lem:halfderleibnitz} we get
\begin{equation}
\| \fdh R^J\|_{L^2(\Omega)}
\lesssim c_J \!\!\!\!\sum_{|K|\leq |J|-1} \| \fdh T^K \widetilde{\pa} q \|_{L^2(\Omega)}
+c_J \!\!\!\! \sum_{|N|\leq |J|/2+3}\!\!\!\!\!\!\| T^N \widetilde{\pa} q\|_{L^\infty(\Omega)}
\!\! \!\!\sum_{|J'|\leq |J|}\!\!\!\!\| \widetilde{\pa}\fdh T^{J'} \widetilde{x}\|_{L^2(\Omega)},
\end{equation}
where $c_J$ depends on $ \widetilde{\pa} T^{N}\widetilde{x}$ for $|N|\leq |J|/2+3$.
The lemma follows from these estimates and induction.
\end{proof}

\begin{proof}[Proof of Lemma \ref{lem:curl}]
Recall that for some constants $a^K_{K_1\dots K_k}$
\begin{equation}
A^K_{ij}=\widetilde{\pa}_i T^K\widetilde{x}^k\,\,\widetilde{\pa}_k \widetilde{\pa}_j q
-\widetilde{\pa}_j T^K\widetilde{x}^k\,\,\widetilde{\pa}_k \widetilde{\pa}_i q
+{\sum}_{K_1+\dots+K_k=K,\, |K_i|<|K|}\,\, a^K_{K_1\dots K_k}
\widetilde{\pa}_i T^{K_1} \widetilde{x}\cdots  \! \widetilde{\pa}T^{K_{k-1}} \widetilde{x}\cdot \!\widetilde{\pa} T^{K_k} {\widetilde{\pa}}^{i}  q,
\end{equation}
from which \eqref{eq:anitsymmetricestimateappb} follows.
By Lemma \ref{lem:gradientfractionalcommute}
\begin{equation}
\| A^{K,\nicefrac{1}{2}}\|_{L^2(\Omega)}\lesssim \| \fdh A^K\|_{L^2(\Omega)}
+c_0\|\widetilde{\pa} T^K \widetilde{\pa} q\|_{L^2(\Omega)},
\end{equation}
and by Lemma \ref{lem:halfderleibnitz} and Lemma \ref{lem:gradientfractionalcommute}
\begin{equation*}
\| \fdh \! A^{K\!}\|_{L^2(\Omega)}\lesssim c_K\!\!\! \sum_{k=0,1}\!\!\! \bigtwo( \!\!\!\! \sum_{\,\,\,\,|L|\leq |K|-1\!\!\!\!}\!\!\!\!\!\!\!\| \widetilde{\pa} \big( \! \fd^{\!\nicefrac{k}{2}}T^{L} \widetilde{\pa} q\big)\|_{L^2(\Omega)}
+ \!\!\!\!\!\!\!\! \!\!\sum_{|N|\leq |K|/2+3}\!\!\!\!\!\!\!\!\!\!\| T^{N\!} \widetilde{\pa} q\|_{L^\infty(\Omega)}
\!\!\!\!\!\!\sum_{|K'|\leq |K|}\!\!\!\!\! \| \widetilde{\pa} \fd^{\!\nicefrac{k}{2}} T^{K'\!} \widetilde{x}\|_{L^2(\Omega)}\!\bigtwo),
\end{equation*}
which proves \eqref{eq:anitsymmetricestimatehalfappb}.
\end{proof}

\begin{proof}[Proof of Lemma \ref{prop:dirichlet}]
 By \eqref{pwnonrel} we have for $k=0,1$:
\begin{equation}
|\widetilde{\pa}\fd^{\!\nicefrac{k}{2}}  T^K \widetilde{\pa} h|
\lesssim |_{\,}\widetilde{\div}\,\fd^{\!\nicefrac{k}{2}}  T^K \widetilde{\pa} h|+|_{\,}\widetilde{\curl}\,\fd^{\!\nicefrac{k}{2}} T^K \widetilde{\pa} h|
+ {\sum}_{S\in{\mathcal S}}|S\fd^{\!\nicefrac{k}{2}} T^K \widetilde{\pa} h|.
\end{equation}
\eqref{eq:thewholecase} follows from this for $k=0$, Lemma \ref{lem:curl}, Lemma \ref{lem:weakdirichlet} and induction to deal with the lower order term.
The proof of \eqref{eq:thehalfcase} follows in the same way apart from that we also have to use
\eqref{eq:thewholecase}.
\end{proof}

\section{Basic elliptic estimates with respect to the Lorentz metric $g$}
\label{relelliptic}
In this section we prove some generalizations
of the estimates from Section \ref{sec:ellipticestimatesproofs}. The proofs
appear at the end of this section.

For a one-form $\beta = \beta_\mu d\widetilde{x}^\mu$ we write
\begin{equation}
 \widetilde{\div} \beta = \widetilde{\nabla}^\mu \beta_\mu,
 \qquad
 \widetilde{\curl}\beta_{\mu\nu}
 = \widetilde{\nabla}_\mu \beta_\nu -
  \widetilde{\nabla}_\nu \beta_\mu
  = \widetilde{\pa}_\mu \beta_\nu -
   \widetilde{\pa}_\nu \beta_\mu,
\end{equation}
where in the last step we used the symmetry of the Christoffel symbols.
Let $\widehat{\Tau}$ denote the future-directed timelike vector defining
the time axis of the background spacetime $(g, M)$,
\begin{equation}
 \widehat{\Tau}^\mu = \nabla^\mu t /(-g(\nabla t, \nabla t))^{1/2}.
 \label{}
\end{equation}
We will work in terms of the following Riemannian metric on the spacetime
$M$,
\begin{equation}
 H^{\mu\nu} = g^{\mu\nu} +2 \widehat{\Tau}^\mu \widehat{\Tau}^\nu,
\end{equation}

For one-forms $\alpha$ and two-forms
$\omega$ we will use the pointwise norms
\begin{equation}
  |\alpha|^2 = H^{\mu\nu} \alpha_\mu \alpha_\nu,
  \qquad
 |\omega|^2 = H^{\mu\nu}H^{\alpha\beta} \omega_{\mu\alpha}\omega_{\nu\beta},
\end{equation}

We have the following pointwise estimate.
\begin{lemma}
  \label{pwHlemma}
  There is a constant $c_0 = c_0(|\pa \xve|)$ so that
  for any one-form $\beta$ on $\D$ we have
 \begin{equation}
  |\pave \beta| \leq
  c_0\big(
  |\widetilde{\div} \beta| + |\widetilde{\curl} \beta|
  + |\S\beta| + |\beta|\big).
  \label{pwH}
 \end{equation}
\end{lemma}
Recall that $\mathcal{S}$ runs over the family of spacetime vector fields which
are tangent to $\pa \Omega$.

We will also need some elliptic estimates on the surfaces $\Omega_{s'}
=\Omega \times \{s = s'\}$ of constant $s'$.
For this we work in terms of the Riemannian metric $G$ defined in
\eqref{Gdef} which we recall here.
\begin{equation}
 G_{\mu\nu} = \sg_{\mu\nu} -
 W_\mu W_\nu, \qquad
 \text{ where }\quad
 W_\mu = \frac{ \sg_{\mu\nu} \widetilde{V}^\nu}{\widetilde{V}^\nu \wTau_\nu},
\end{equation}
which satisfies
\begin{equation}
  G(\xi,\xi) \geq c \sg(\xi,\xi),
 \label{Gpositivityapp}
\end{equation}
for a constant $c$ (see \eqref{Gpositivity0}), for any vector $\xi \in
T\Omega_{s'}$.

For one-forms $X$ and two-tensors $\omega$ we write
\begin{equation}
  \vertiii{X}_{L^2(\Omega)}^2
   = \int_{\Omega} G^{\mu\nu} X_\mu X_\mu \, \kappa_{\oH} dy,
  \qquad
 \vertiii{\omega}_{L^2(\Omega)}^2
 =
 \int_{\Omega} G^{\alpha\beta} G^{\mu\nu} \omega_{\alpha \mu}\omega_{\beta \nu}
 \,\kappa_{\oH} dy.
 \label{appendixnormsdef}
\end{equation}
Here, $\kappa_{\oH} = \det \oH^{1/2}$. Then $\vertiii{X}_{L^2(\Omega)}$
is positive definite when restricted to one-forms $X$ which are
cotangent to $\Omega$ at fixed $s$.

We will also work in terms of covariant differentiation $\overline{\nabla}$
with respect to the metric $G$ which satisfies $\overline{\nabla}G  =0$.
 If $X$ is a one-form then it is given by
\begin{equation}
 \overline{\nabla}_\mu X_\nu = G^{\mu'}_{\mu} G_{\nu}^{\nu'} \nabla_{\mu'}X_{\nu'}.
 \label{}
\end{equation}
Here, $G^{\mu}_{\nu}$ denotes orthogonal projection to the tangent space
of $\Omega$ with respect to the metric $G$,
\begin{equation}
 G^{\mu}_{\nu} =
 g_{\nu\nu'} G^{\mu\nu'}.
 \label{}
\end{equation}
We also write
\begin{equation}
 \div_{G} X = G^{\mu\nu} \overline{\nabla}_\mu X_\nu.
 \label{divgdef}
\end{equation}
and we have
 \begin{lemma}
   \label{prop:divcurlL2rel}
There is a constant
  $C_{\!1}$ depending on $\vertiii{\pa \xve}_{L^\infty(\Omega)}$
  as well as $c_L$ from \eqref{Gpositivityapp}
  so that with
  notation as in \eqref{appendixnormsdef},
  if $X$ is a one-form on $\Omega$
\begin{multline}
 \! \vertiii{\pave X}_{L^{2{}_{\!}}(\Omega)}^2\!
\leq\! C_{{}_{\!}1\!}\Big(\! \vertiii{\div_{\oH} X}_{L^2(\Omega)}^2\!
 +{}_{\!} \vertiii{\wcurl X}_{L^2(\Omega)}^2\!
 \\
 + \! \!\int_{\pa \Omega}
 (\oH^{\mu\nu} \n_{{}_{\!}\mu} \fdh\! X_\nu)
 \cdot
 (\oH^{\alpha\beta} \n_{{}_{\!}\alpha} \fdh\! X_\beta) dS
 + \vertiii{X}_{L^2(\pa \Omega)}^2\!
 + \vertiii{X}_{L^2(\Omega)}^2\!\Big).
 \label{ftang1G}
\end{multline}
Here $\fdh$ is a half angular derivative defined locally in coordinates in \eqref{fdmudef},
and the inner product is the sum over coordinate charts $\fdh \xi^a
 \cdot \fdh \xi^b=\sum_{\mu} \big(\fdhm \xi^a\big)
\big(\fdhm \xi^b\big)$ as in \eqref{eq:globalhalfderdef}.  \end{lemma}

The next result is similar to Proposition \ref{prop:dirichlet}
and follows in almost exactly the same way. We omit
the proof.
\begin{prop}\label{prop:dirichletrel} Suppose that $q=0$ on $\pa \Omega$. Then with notation as in \eqref{appendixnormsdef},
\begin{multline}\label{eq:thewholecasedirichletrel}
\vertiii{\widetilde{\pa}  T^K \widetilde{\pa} q}_{L^2(\Omega)}\\
\lesssim C_0{\sum}_{S\in\mathcal{S}}\| \widetilde{\pa} {S} T^{K} \widetilde{x}\|_{L^2(\Omega)}
+C_K\!{\sum}_{|K'|\leq |K|} \bigtwo( \|T^{K'} \tr_{\oH} \pave^2 q\|_{L^2(\Omega)}
+ \|T^{K'} \hD_s \pave q\|_{L^2(\Omega)}
+\| \widetilde{\pa}T^{K'} \widetilde{x}\|_{L^2(\Omega)}\bigtwo) ,
\end{multline}
and
\begin{multline}\label{eq:thehalfcasedirichletrel2}
\vertiii{\widetilde{\pa} \fdh  T^K \widetilde{\pa} q}_{L^2(\Omega)} \\
\lesssim
C_K\!{\sum}_{|K'| \leq |K|,\, k=0,1}\!\bigtwo( \|\fd^{\!\nicefrac{k}{2}} T^{K'\!}\tr_{\oH} \pave^2 q \|_{L^2(\Omega)}
+ \| \fd^{\!\nicefrac{k}{2}} T^{K'\!} \hD_s \pave q\|_{L^2(\Omega)} +
\| \widetilde{\pa}\fd^{\!\nicefrac{k}{2}} \mathcal{S}^1 T^{K'\!} \widetilde{x}\|_{L^2(\Omega)} \bigtwo),
\end{multline}
where $C_K$ depends on $ \|\widetilde{\pa} T^{L}\widetilde{x}\|_{L^\infty(\Omega)},$
$\|\widetilde{\pa} T^{L}\widetilde{\pa}q\|_{L^\infty(\Omega)}$,
 $\|\widetilde{\pa} T^L \tH\|_{L^\infty}$
 for $|L|\!\leq\! |K|/2\!+\!3$, and on $c_L$ from \eqref{Gpositivityapp}.
\end{prop}

\subsection{Proofs of the basic elliptic estimates used
in the relativistic case}

\begin{proof}[Proof of Lemma \ref{pwHlemma}]
  This is similar to the proof of the pointwise lemma
  from \cite{GLL19}. First, we note that the metric
  $H$ is equivalent to the metric $h^{\mu\nu}
  = \wg^{\mu\nu} + 2\widetilde{u}^\mu \widetilde{u}^\nu$ and so it is enough
  to prove \eqref{pwHlemma} with all pointwise norms replaced by
  the norms with respect to $h$. It is more convenient to state the
  result in terms of $H$ since that does not depend on the fluid variables
  but for the proof it is better to work with $h$ since it is more clear
  how the material derivatives enter.
  Let $\widetilde{\N}$ denote the unit normal with respect to $h$
  to $\pa \Omega$ (at constant $s$) and extend it to a tubular neighborhood
  of the boundary.
  Since the right-hand side of \eqref{pwHlemma} controls the material derivative
  $\hD_s \beta$ and since $\hD_s = \widetilde{V}^\mu\pave_\mu$ is
  parallel to $\widetilde{u}^\mu\pave_\mu$, it is enough to prove that if $\omega =
 \omega_{\alpha\beta}d\widetilde{x}^\alpha d\widetilde{x}^\beta$ is a symmetric two-tensor
satisfying $\wg^{\alpha\beta}\omega_{\alpha\beta} = 0$ and
$\widetilde{u}^\alpha\widetilde{u}^\beta\omega_{\alpha\beta} = 0$ then
\begin{equation}
\wg^{\alpha\beta}\wg^{\mu\nu} \omega_{\alpha\mu}\omega_{\beta\nu}
 \leq C q^{\alpha\beta} \wg^{\mu\nu}\omega_{\alpha\mu}\omega_{\beta\nu},
\end{equation}
where $q^{\alpha\beta} = h^{\alpha\beta} -
\widetilde{\N}^\alpha\widetilde{\N}^\beta$ is the projection onto
the orthogonal complement to $\widetilde{\N}$.

Writing $h^{\alpha\beta} = q^{\alpha\beta}+\widetilde{\N}^a \widetilde{\N}^b$
and using the symmetry of $\omega$ as well as the fact that the component of
$h$ along $u$ annihilates $\omega$, we have
\begin{equation}
\wg^{\alpha\beta}\wg^{\mu\nu}\omega_{\alpha\mu}
\omega_{\beta\nu}
 =
 q^{\alpha\beta} q^{\mu\nu}\omega_{\alpha\mu}\omega_{\beta\nu}
 + \widetilde{N}^\alpha \widetilde{N}^\beta
 \widetilde{N}^\mu \widetilde{N}^\nu
 \omega_{\alpha\mu}\omega_{\beta\nu}
 + 2 q^{\alpha\beta} \widetilde{N}^\mu \widetilde{N}^\nu
 \omega_{\alpha\mu}\omega_{\beta \nu}
 \label{ident1}
\end{equation}
If $\omega$ additionally satisfies $\wg^{\alpha\mu}
\omega_{\alpha\mu} = 0$ then
the second term on the first line is
\begin{equation}
  \widetilde{N}^\alpha\! \widetilde{N}^\beta\!
  \widetilde{N}^\mu\! \widetilde{N}^\nu\!
  \omega_{\alpha\mu}\omega_{\beta\nu}
  = \!\big( \widetilde{N}^\alpha \!\widetilde{N}^\mu \omega_{\alpha\mu}\big)^2\!\!
   = \big( \wg^{\alpha\mu} \omega_{\alpha \mu} - q^{\alpha\mu} \omega_{\alpha\mu}\big)^2\!\!
   =\!
   \big( \wg^{\alpha\mu} \omega_{\alpha \mu}\big)^2\! +
   \big( q^{\alpha\mu} \omega_{\alpha\mu}\big)^2\!
   - 2 \wg^{\alpha\mu} \omega_{\alpha \mu}q^{\beta\nu} \omega_{\beta\nu}.
\end{equation}
Inserting this identity into \eqref{ident1} we have
\begin{equation}
  \wg^{\alpha\beta}\wg^{\mu\nu}\omega_{\alpha\mu}
  \omega_{\beta\nu}
  =q^{\alpha\beta} q^{\mu\nu}\omega_{\alpha\mu}\omega_{\beta\nu}
  +2 q^{\alpha\beta} \widetilde{N}^\mu \widetilde{N}^\nu
  \omega_{\alpha\mu}\omega_{\beta \nu}
  +
  \big( \wg^{\alpha\mu} \omega_{\alpha \mu}\big)^2 +
  \big( q^{\alpha\mu} \omega_{\alpha\mu}\big)^2
  - 2 \wg^{\alpha\mu} \omega_{\alpha \mu}q^{\beta\nu} \omega_{\beta\nu}.
\end{equation}
Using the symmetry of $\omega$ we have
$(q^{\alpha\mu}\omega_{\alpha\mu})^2 \leq C q^{\alpha\beta}
q^{\mu\nu} \omega_{\alpha\mu}\omega_{\beta\nu}$ and this
gives the result.
\end{proof}

\begin{proof}[Proof of Lemma \ref{prop:divcurlL2rel}]

In the same way that Lemma \ref{lem:ellipticboundary}
implied \eqref{ftang1}, Lemma \ref{prop:divcurlL2rel} is a consequence
of the following identity after noting that the boundary term only involves
derivatives which are tangent to $\pa \Omega$.
We recall the definition of the norms $\vertiii{\beta}_{L^2}$ from
\eqref{Hnormdef}.
  \begin{lemma}
    \label{coordhalfestrel}
    Let $\n_\mu$ denote the spacelike unit conormal to $\pa \Omega$ normalized
    with respect to
    the metric $\oH$, defined in \eqref{Gdef}. If $X$ is a one-form, then
   \begin{multline}
     \int_\Omega \oH^{\mu\nu}\oH^{\alpha\beta}
     \overline{\nabla}_\mu X_\alpha
     \overline{\nabla}_\nu X_\beta\, \kappa_G dy\\
     =
     \int_\Omega (\overline{\nabla}^\mu X_\mu)^2\, \kappa_G dy
     + \frac{1}{2} \int_{\Omega} \oH^{\mu\nu}\oH^{\alpha\beta}
     \curl X_{\mu \alpha}
     \curl X_{\nu \beta}\, \kappa_{\oH} dy
     -\int_{\Omega}
     R_{\oH}^{\mu\nu} X_{\mu} X_\nu\, \kappa_{\oH} dy\\
    +\int_{\pa \Omega}
    \left(
    \oH^{\mu\nu}\oH^{\alpha\beta}
    X_\beta  \n_\mu
    \sonabla_\nu X_\alpha
    -
    \oH^{\mu\nu} \oH^{\alpha\beta} \n_\alpha X_\beta  \sonabla_\mu X_\nu
    -
    \oH^{\mu\nu}\oH^{\alpha\beta} \n_\mu X_\beta (\sonabla_\nu X_{\alpha}
    -\sonabla_\alpha X_\nu)\right)
    dS_{\oH}.
   \end{multline}
  Here, $\sonabla_\mu X_\nu$ denotes covariant differentiation tangent to
  $\pa \Omega$ with $s$ held constant, given by
  \begin{equation}
   \sonabla_\mu X_\nu = (\delta_\mu^\alpha - G^{\alpha\nu'}\n_\mu\n_{\nu'} )\onabla_\alpha
   X_\nu.
   \label{}
  \end{equation}
  We are also writing $R_G$ for the Ricci curvature tensor of $G$ and
  \begin{equation}
   \curl X_{\mu\nu} = \nabla_\mu X_\nu - \nabla_\nu X_\mu
   =\pave_\mu X_\nu - \pave_\nu X_\mu.
   \label{}
  \end{equation}
  \end{lemma}

Lemma \ref{coordhalfestrel} is proven in essentially the same way
that we proved Lemma \ref{lem:ellipticboundary}.
We start by recording the divergence theorem in terms of $\div_G$,
\begin{equation}
 \int_{\Omega} \div_G X \, \kappa_G dy = \int_{\pa \Omega} G^{\mu\nu} n_\mu
 X_\nu\, dS_G,
 \label{stokesG}
\end{equation}
where $\kappa_G dy$ is the Riemannian volume element with respect to $G$
and $dS_G$ denotes the corresponding surface measure, and
$n_\mu$ is the unit conormal to $\pa \Omega$ normalized with respect to
$G$.

Using \eqref{stokesG} along with the fact that $\onabla G = 0$,
  \begin{equation}
    \int_{\Omega} \oH^{\mu\nu} \oH^{\alpha\beta}
    \onabla_{\!\mu} X_\alpha \onabla_{\!\nu} X_\beta \, \kappa_G dy
   = -\!\!\int_{\Omega}\!\! \big(\oH^{\mu\nu}\onabla_{\!\nu}\onabla_{\!\mu} X_\alpha\big)
   \oH^{\alpha\beta}
   X_\beta \,\kappa_G dy
    +\!\! \int_{\pa \Omega}\!\!\!
   \oH^{\mu\nu}\oH^{\alpha\beta}\n_\nu X_\beta \overline{\nabla}_\mu X_\alpha\,dS_G.
   \label{ibpapp1acou}
  \end{equation}

  We have the identity
  \begin{multline}
   \oH^{\mu\nu}\onabla_\nu\onabla_\mu X_\alpha =
   \oH^{\mu\nu} \onabla_\alpha \onabla_\nu X_\mu
   + \oH^{\mu\nu}
   \onabla_\mu(\onabla_\nu X_\alpha - \onabla_\alpha X_\nu)
   + \oH^{\mu\nu}R_{G \alpha \mu \nu}^\beta X_{\beta}
   \\
   = \onabla_\alpha (\div_G X) + \oH^{\mu\nu}\onabla_\mu \overline{\curl} X_{\nu\alpha}
   + \oH^{\mu\nu}R_{G \alpha \mu \nu}^\beta X_{\beta} ,
   \label{divcurlident}
  \end{multline}
  where $R_G$ denotes the curvature tensor of $G$,
  \begin{equation}
   R_{G\alpha \mu \nu}^\beta X_\beta = [\onabla_\alpha \onabla_\mu - \onabla_\mu
   \onabla_\alpha] X_\nu.
   \label{}
  \end{equation}

  Inserting \eqref{divcurlident}
into the first term on the right of \eqref{ibpapp1acou}
and integrate by parts again:
\begin{multline}
 \int_{\Omega}\int_{\Omega} \oH^{\alpha\beta}
 X_\beta  \big(\oH^{\mu\nu}\onabla_\mu\onabla_\nu X_\alpha\big)
 = -\int_{\Omega} (\div_{\oH} X)^2
 + \oH^{\mu\nu} \oH^{\alpha\beta} \onabla_\mu X_\beta  \overline{\curl} X_{\nu \alpha}
 \\
 +\int_{\pa \Omega} \oH^{\alpha\beta} \n_\alpha  X_\beta
\div_\oH X +\oH^{\mu\nu} \oH^{\alpha\beta} \n_\mu \overline{\curl} X_{\nu \alpha} X_\beta \, dS_{\oH}
 -\int_{\Omega}\oH^{\mu\nu}\oH^{\alpha\beta}
 R_{G \alpha \mu \nu}^\gamma X_{\gamma} X_\beta
 \label{acougreen1}
\end{multline}
By the antisymmetry of curl, $
 \oH^{\mu\nu}\oH^{\alpha\beta} \onabla_{\!\mu} X_\beta
 \overline{\curl} X_{\nu\alpha}
 =
 \frac{1}{2} \oH^{\mu\nu}\oH^{\alpha\beta}  \overline{\curl} X_{\mu \beta}
 \overline{\curl} X_{\nu \alpha},
 $
so \eqref{ibpapp1acou} becomes
\begin{multline}
  \int_{\Omega} \oH^{\mu\nu} \oH^{\alpha\beta}
  \onabla_\mu X_\alpha \onabla_\nu X_\beta
  =\int_{\Omega}(\div_{\oH} X)^2 + \frac{1}{2}
  \oH^{\mu\nu} \oH^{\alpha\beta} \overline{\curl} X_{\mu \beta} \overline{\curl}
  X_{\nu \alpha}\\
  + \int_{\pave \Omega} \oH^{\mu\nu}\oH^{\alpha\beta}
  X_\beta  \n_\mu
  \onabla_\nu X_\alpha -
  \oH^{\alpha\beta} \n_\alpha X_\beta  \div_{\oH} X
  - \oH^{\mu\nu}\oH^{\alpha\beta} \n_\mu X_\beta \overline{\curl} X_{\nu \alpha}
  -\int_{\Omega}\oH^{\mu\nu}\oH^{\alpha\beta}
  R_{G \alpha \mu \nu}^\gamma X_{\gamma} X_\beta.
\end{multline}
We now use that $G^{\mu\nu}\n_\mu \n_\nu = 1$ to write
\begin{equation}
 \onabla_\mu X_\nu =
 \n_\mu \onabla_\n X_\nu +
 \sonabla_\mu X_\nu,
 \qquad
 \text{ where }
 \sonabla_\mu = (\delta_\mu^\nu - G^{\nu\nu'}\n_\mu\n_{\nu'} )\onabla_\nu,
 \quad
 \onabla_\n = G^{\mu\nu} \n_\mu \onabla_\nu.
 \label{}
\end{equation}
Using this expression, the boundary term is the integral of
\begin{multline}
  \oH^{\mu\nu}\oH^{\alpha\beta}
  X_\beta  \n_\mu
  \onabla_\nu X_\alpha -
  \oH^{\alpha\beta} \n_\alpha X_\beta  \div_{\oH} X
  - \oH^{\mu\nu}\oH^{\alpha\beta} \n_\mu X_\beta \overline{\curl} X_{\nu \alpha}
 \\
 =
 \oH^{\alpha\beta} X_\beta \onabla_\n X_\alpha
 -\oH^{\mu\nu}\oH^{\alpha\beta} \n_\alpha \n_\mu X_\beta \onabla_\n X_\nu
 -\oH^{\alpha\beta} X_\beta \onabla_\n X_\alpha
 + \oH^{\mu\nu}\oH^{\alpha\beta}\n_\alpha \n_\mu X_\beta\onabla_\n X_\nu
 \\
 +\oH^{\mu\nu}\oH^{\alpha\beta}
 X_\beta  \n_\mu
 \sonabla_\nu X_\alpha
 -
 \oH^{\mu\nu} \oH^{\alpha\beta} \n_\alpha X_\beta  \sonabla_\mu X_\nu
 -
 \oH^{\mu\nu}\oH^{\alpha\beta} \n_\mu X_\beta (\sonabla_\nu X_{\alpha}
 -\sonabla_\alpha X_\nu).
 \label{}
\end{multline}
Noting that the terms on the second line cancel, we get the result.
\end{proof}

\section{The divergence theorem }
\label{divergencetheoremproofs}

The identity \eqref{divtheorem1} is nothing but
the usual divergence theorem, see e.g. \cite{CB08}.
If $\widetilde{D}$ denotes intrinsic covariant differentiation on $\Lambda$,
\begin{equation}
 \div_{\Lambda} T = \widetilde{D}_\mu T^\mu.
\end{equation}
and the divergence theorem on $\Lambda$ says
\begin{equation}
 \int_{\Lambda_{\Sigma_0}^{\Sigma_1}} \div_{\Lambda} T\, dS^{\Lambda} =
 \int_{\Lambda_{\Sigma_1}} \wg(n^{\Sigma_1}, T) \,dS^{\Lambda_{\Sigma_1}}
 +
 \int_{\Lambda_{\Sigma_0}} \wg(n^{\Sigma_0}, T) \,dS^{\Lambda_{\Sigma_0}}.
 \label{divbdy}
\end{equation}
where $n^{\Sigma_i}$ denotes the future-directed normal vector field
to $\Sigma_i$ defined relative to $\wg$ and
$\Lambda_{\Sigma_i}= \Lambda \cap \Sigma_i$.
If $\widetilde{V}^\mu$ is tangent to $\Lambda$
then with $D_s = \widetilde{V}^\mu\pave_\mu$,
\begin{equation}
  \hD_s\phi =
 \widetilde{V}^\mu\pave_\mu \phi = \widetilde{V}^\mu \widetilde{D}_\mu
 \phi
 = \div_{\Lambda}(\widetilde{V} \phi) - \phi \div_{\Lambda} \widetilde{V},
\end{equation}
and integrating this expression and using \eqref{divbdy} gives
\eqref{transportbdy}.

\section{Existence for the linear and smoothed problem}
\label{existence}

In this section we give a sketch of the proof of existence for the linear problems
we use in our iteration scheme. Since this is a linear problem
with tangentially smoothed coefficients, existence on a time interval
depending on the smoothing parameter is nearly an immediate consequence of the a priori estimates
we proved in the earlier sections.
We first discuss the Newtonian case.

\subsection{Existence for the linear and smoothed Newtonian problem}

Fix a tangentially smooth vector field $\widetilde{V}$
and define $\xve$ by
\begin{equation}
 \frac{d\xve(t,y)}{dt}
 =\widetilde{V}(t,y), \qquad
 \xve(0,y) = y.
 \label{}
\end{equation}
The linear problem we consider is
\begin{equation}
 D_t V_i + \pave_i h[V] = 0, \qquad \text{ in }
 [0,T_1]\times \Omega, \qquad V|_{t = 0} = V_0,
 \label{ode}
\end{equation}
with $\pave_i = \frac{\pa}{\pa \xve_i} = \frac{\pa y^a}{\pa \xve^i} \frac{\pa }{\pa y^a}$,
and where $h = h[V]$ is determined by solving the wave equation
\begin{equation}
 e_1D_t^2 h - \widetilde{\Delta} h =
 (\pave_i \widetilde{V}^j)(\pave_j V^i),
 \qquad
 h|_{\pa \Omega} = 0,
 \qquad h|_{t = 0} = h_0, \quad D_t h|_{t = 0} = h_1.
 \label{hwaveappendix}
\end{equation}
Here $\widetilde{\Delta} = \delta^{ij}\pave_i \pave_j$.

To solve \eqref{ode} we are going to show that it is
an ODE in a certain function space (for $\ve > 0)$, and existence
then follows from a standard Picard iteration.
Fix $r\geq 10$ and for $T_0 > 0$, define the norms
\begin{equation}
 \| u\|_{X_{T_0}} ={\sup}_{0 \leq t \leq T_0}
 \|u(t)\|_{r+1},
\end{equation}
where
\begin{equation}
  \|u(t)\|_{r+1} = {\sum}_{k+\ell \leq r}\|D_t D_t^k u(t)\|_{H^{\ell}(\Omega)}
  + {\sum}_{k+\ell \leq r}\|D_t^k u(t)\|_{H^{\ell}(\Omega)}.
   \label{simplenormdef}
\end{equation}
The reason we work with norms that control one
additional time derivative will be explained in
section \ref{compatsec}.
In that section we show that
the map $V\mapsto h[V]$ is well-defined if
the compatibility conditions hold and
$\|\widetilde{V}\|_{X_{T_1}} \!+\!
\|\S\widetilde{x}\|_{X_{T_1}} < \infty$,
where $\S \widetilde{x}$ is defined in
\eqref{eq:simplifiedtangentialnotation} and involves
tangential derivatives of $\xve$.
With
\begin{equation}
 H[V](t,y) = -\int_0^t \pave h[V](t',y)\, dt'.
 \label{}
\end{equation}
in Section \ref{compatsec},
we show that $H$ is bounded and Lipschitz on $X_{T_1}$,
\begin{align}
 \|H[V]\|_{X_{T_1}}
 &\leq C( \|\widetilde{V}\|_{X_{T_1}}, \|\S\widetilde{x}\|_{X_{T_1}})
 \left(\|\overline{V}\|_{X_{T_1}} + T_1
 \|V\|_{X_{T_1}}\right),
 \label{odebound}
 \\
 \|H[V_1] - H[V_2]\|_{X_{T_1}}
 &\leq T_1C( \|\widetilde{V}\|_{X_{T_1}}, \|\S\widetilde{x}\|_{X_{T_1}},
 \|V_1\|_{X_{T_1}}, \|V_2\|_{X_{T_1}})
 \|V_1 - V_2\|_{X_{T_1}}.
 \label{odebound2}
\end{align}
In \eqref{odebound}, $\overline{V}$ is a power series in time which
solves the equation at $t = 0$ to order $r$ (see \eqref{approxsln})
and is determined from the initial data $V_0, h_0$ and satisfies
$
 \|\overline{V}\|_{X_{T_1}}
 \lesssim \|V_0\|_{H^r(\Omega)} + \|\pave h_0\|_{H^{r-1}(\Omega)}.
$

Assuming these bounds hold, existence follows from a
straightforward Picard iteration.
\begin{prop}[Existence for the linear and smoothed problem]
  \label{linexist}
  Let $r \geq 10$ and suppose that the initial data $(V_0, h_0)$
  satisfies the compatibility conditions \eqref{eq:compatibilityconditiondef}
  to order $r$. Let $\widetilde{V} \in X_{T_1}$ for some $T_1 > 0$.
  Then there is a time $T \leq T_1$ so that the linear smoothed problem \eqref{ode} has a unique solution
  $V \in X_{T}$  and if
   $\overline{V}$ denotes a formal power series
  solution at $t = 0$ defined as in \eqref{approxsln}, $V$ satisfies
  the bound
  \begin{equation}
   \|V\|_{X_{T}} \leq C\left(\|\widetilde{V}\|_{X_{T_1}}   \|\S \widetilde{x}\|_{X_{T_1}}\right)
   \|\overline{V}\|_{X_{T}},
   \label{existencebd}
  \end{equation}
  and the enthalpy satisfies
  \begin{equation}
   \sup_{0 \leq t \leq T}
   \sum_{k +\ell \leq r-1}\!\!\!\!
   \|D_t^k \pave h(t)\|_{H^\ell(\Omega)}
   +\|D_t^k D_t h(t)\|_{H^\ell(\Omega)}
   \leq C
   \big(\!\!\!\!\sum_{k +\ell \leq r-1}\!\!\!\!
   \|D_t^k \pave h(0)\|_{H^\ell(\Omega)}
   +\|D_t^k D_t h(0)\|_{H^\ell(\Omega)}\big),
   \label{existencebd2}
  \end{equation}
  with $C \!= \!C\big(\|\widetilde{V}\|_{X_{T_1}},  \|\S \widetilde{x}\|_{X_{T_1}}\!\big)$.
  \!Moreover, the compatibility conditions hold at time
  $t \!=\! T\!$ to order $r$\!.
\end{prop}

\begin{proof}
  We are going to solve \eqref{ode} by iteration and so we need to ensure that the map
  $V\mapsto h[V]$ is well-defined at each step. In particular
  we need to ensure that if $V$ satisfies the compatibility conditions from the upcoming section
  then so does the resulting $W$.
  We therefore work in the space
  \begin{equation}
   X_{T_1,c} = \{V : \|V\|_{X_{T_1}} < \infty,
   D_t^k V|_{t = 0} = V_k, k = 0,..., r+1\},
  \end{equation}
  where the $V_k$ are given by \eqref{hkdef}.
  We claim that if $V \in X_{T, c}$ and
  $W$ satisfies $D_t W = -\pave h[V]$ then
  $W \in X_{T_1, c}$ as well.
  First, by the results of the upcoming section \ref{compatsec}
  given $V \in X_{T_1, c}$, $\pave h[V]$ is well-defined
  and by \eqref{odebound} the resulting $W$ with $D_t W
  = -\pave h$
  satisfies the bound \eqref{odebound}.
  It remains
  to check the time derivatives at $t = 0$. For these we compute
  \begin{equation}
   D_t^{k}V'|_{t = 0} = D_t^{k-1} \pave h[V]|_{t = 0} = V_k,
   \label{}
  \end{equation}
  which is just the definition of the $V_k$.
  Using the bounds \eqref{odebound}-\eqref{odebound2}, the existence result and
  the bounds follow by a standard
  iteration argument. The fact that the compatibility conditions
  hold at later times as well follows directly from the construction
  of the enthalpy, see section \ref{galerkinsec}.
\end{proof}

It remains to prove that under the hypotheses of the above
Proposition, the map $V \mapsto h[V]$ is well-defined
and that \eqref{odebound}-\eqref{odebound2} hold.
This is done in the next section.

\subsubsection{The compatibility conditions and existence for
the wave equation for the enthalpy}
\label{compatsec}
 Because of the continuity equation and that $h = 0$ on $\pa \Omega$,
the initial data $V_0$ must satisfy $ \widetilde{\div} V_0 = 0$
on $\pa \Omega$.
 Taking more time derivatives we see that we must also have
 $D_t^k (\widetilde{\div} V)|_{t = 0} = 0$ on the boundary
 which places additional restrictions on the initial data
 that we now write out explicitly.

 Fix a diffeomorphism
 $x_0:\Omega \to \Omega$. Let $\overline{V}
 = \sum_{k\geq 0}  V_k{t^k}/{k!}$, $\overline{h} =  \sum_{k \geq 0}
  h_k{t^k}/{k!}$, and $\overline{x} = x_0 + t\overline{V}$
  be a formal power series solution to
 \eqref{eq:eulerlagrangiancoordsmoothed2}-\eqref{eq:continuityequationsmoothed} at $t= 0$,
 \begin{equation}
  D_t^k\big(D_t \overline{V} + \widetilde{\pa} \overline{h}\big)|_{t = 0} = 0,
  \qquad
  D_t^k \big( e_1 D_t \overline{h} + \widetilde{\div} \overline{V}\big) = 0, \qquad
  k = 0,1,...r.
  \label{approxsln}
 \end{equation}
 Here, we are writing $\widetilde{\pa} = \widetilde{\pa}_{\overline{x}}$
 for the derivatives with respect to the smoothed version of $\overline{x}$
 and similarly for $\widetilde{\div}$.
 From these equations we see that for $k \geq 1$, there are functions $G_k,G_k,
 $ so that
 \begin{align}
  h_k &= G_k(h_0, x_0 , V_0, ..., V_{k-1}),\qquad
  V_k = F_k(h_0, x_0, V_0, ..., V_{k-1}),
  \label{hkdef}
 \end{align}
 using the second equation in \eqref{approxsln} to replace time derivatives
 of $\overline{h}$ at $t = 0$ with a function of $V_0, V_1,..., V_{k-1}$.

 We say that intial data $(V_0, h_0)$ \emph{satisfy the compatibility
 conditions to order r} if, with the sequence
 $V_1, V_2,..., V_{r}$ and the functions $G_k$ defined as in \eqref{hkdef},
 we have
 \begin{equation}
   G_k(h_0, x_0, V_0,...., V_{k-1}) \in H^1_0(\Omega),
   \quad \text{ for } k = 0,..., r.
  \label{eq:compatibilityconditiondef}
 \end{equation}
The significance of \eqref{eq:compatibilityconditiondef} is that
$G_k$ must vanish on $\pa \Omega$.

Provided the compatibility conditions \eqref{eq:compatibilityconditiondef} hold,
using e.g. a Galerkin method (see \cite{GLL19} for a detailed proof) or duality (see \cite{Hormander2007}),
one can prove that the wave equation \eqref{hwaveappendix} has a solution $h$ with
   \begin{equation}
    D_t^k h,\,D_t^{k-1} \pave h \in L^\infty([0,T_0]; H^{r+1-k}(\Omega)),\quad k = 0,...,r+1,
   \end{equation}
   provided $\|V\|_{r+1,T_0} + \|\widetilde{V}\|_{r+1,T_0}
   +\|\S \xve\|_{r+1, T_0} < \infty$.

The hypothesis in Theorem \ref{mainthmnonrel} is that our
initial data satisfy the compatibility conditions
\eqref{eq:compatibilityconditiondef} to order $r$ when $\ve = 0$ but in order to construct
a solution for the smoothed problem we will also need
initial data which satisfies the compatibility conditions
to the same order with $\ve > 0$.
In Appendix E of \cite{GLL19}
it was shown that this can be done under
our hypotheses and we indicate
the main points in the upcoming section \ref{compatconstruction}.

It just remains to prove the bounds \eqref{odebound}-\eqref{odebound2}. In fact we have already proved essentially
the same bounds in section \ref{sec:enthalpywave}.
The only substantial difference is that here we need to control normal
derivatives to top order whereas in Section \ref{sec:enthalpywave} we closed
estimates for tangential derivatives
to top order. This does not cause any
serious difficulties and we sketch how
to prove the needed bounds. See also \cite{GLL19} for
a detailed proof of almost the same result.

We will just discuss how to control the highest-order
part of the norm $\|H\|_{X_{T}}$ coming from the first
term in the definition of the norm in \eqref{simplenormdef}.
The second term in the definition of the norm is simpler to deal with.
After taking one time derivative we need bounds
for $\|\pa_y^\ell D_t^k \pave h\|_{L^2(\Omega)}$
where $\ell + k = r$. If $\ell > 0$,
we start by commuting $D_t^k$ with $\pave h$. The commutator
will be harmless at this point because it involves time
derivatives of $\widetilde{V}$ which we control to higher order,
and so it is enough to control $\pa_y^\ell \pave D_t^k h$.
To control this term we first use
the pointwise estimate \eqref{app:pwdiff} and the elliptic
estimate from Proposition \ref{prop:dirichlet} for the Dirichlet problem,
and so it suffices to control $\pa_y^{\ell-1} D_t^k \Dve h$.
We note that when $k = 0$ this estimate
requires a bound for $\|\S \xve\|_{X_{T_1}}$ which is
why this quantity appears in our estimates.
Writing \eqref{hwaveappendix} as
\begin{equation}
 \Dve h = -(\pave_i \widetilde{V}^j)(\pave_j V^i)
 + e_1 D_t^2 h, \qquad
 h|_{\pa \Omega } =0.
\end{equation}
and applying $\pa_y^{\ell-1}D_t^k$,
we see that the term
$\pa_y^{\ell-1} D_t^k \left((\pave_i \widetilde{V}^j)(\pave_j V^i)\right)$ is
 lower-order and so it is enough
to control $\pa_y^{\ell-1} D_t^{k+2} h$. Now we note
that the number of space derivatives falling on $h$ has
been reduced by two while the number of time derivatives
falling on $h$ has been increased by two. Repeating this
argument as many times as needed, it remains
to prove bounds for $\|D_t^r \pave h\|_{L^2(\Omega)}
+ \|D_t^{r+1} h\|_{L^2(\Omega)}$. For this we use
the estimates for the wave equation
as in section \ref{enthalpyestimates}, which requires
applying $D_t^r$ to both sides of \eqref{hwaveappendix}. We therefore need a bound for the term
$\|D_t^r (\pave_i \widetilde{V}^j \pave_j V^i)\|_{L^2(\Omega)}$.
When we encountered this term in earlier
proof of the a priori bounds, we used that
$D_t V = -\pave h$ to close the estimates (see section
\ref{sec:extradt}) but we do not have an equation for
$V$ here. Instead we just note that this term involves time
derivatives to top order and so we can
control it by the first
term in the definiton of the norm
$\|\cdot\|_{X_{T_1}}$. This is the
reason our norm involves an additional time derivative.
Integrating the lower-order terms in time we get
\eqref{odebound}.
The Lipschitz estimate \eqref{odebound2} is proven in
the same way.

\subsection{Existence for the linear and smoothed relativistic problem}
\label{existencerel}
We now prove the same result for the linear relativistic
problem. Fix a tangentially smooth vector field
$\widetilde{V}$ and define $\xve= \xve(s,y)$ by
\begin{equation}
 \frac{d}{ds} \xve^\mu(s,y) = \widetilde{V}^\mu(s,y),
 \qquad \xve^0(0,y) = 0, \quad
 \xve^i(0,y) = y^i, \quad i = 1,2,3.
\end{equation}
The linear problem we consider is
\begin{equation}
 D_s V_\mu + \frac{1}{2} \pave_\mu \sigma =
 \wGamma_{\mu\nu}^\alpha V_\alpha \widetilde{V}^\nu,
 \qquad \text{ in } [0,S_1]\times \Omega,
 \qquad V_\mu \big|_{s = 0} = \mathring{V}_\mu,
 \label{linVappendixrel}
\end{equation}
where $\sigma = \sigma[V]$ is determined by solving
the wave equation

\begin{equation}
 e'(\sigma) D_s^2 \sigma - \frac{1}{2}
 \widetilde{\na}_\nu ( \widetilde{g}^{\mu\nu}
 \widetilde{\na}_\mu \sigma)
 = \widetilde{\na}_\mu \widetilde{V}^\nu \widetilde{\na}_\nu
 V^\mu + \widetilde{R}_{\mu\nu\alpha}^\mu - e''(\sigma)
 (D_s\sigma)^2\!,
 \quad
 \sigma|_{\pa \Omega} \!=\! 0,
 \quad\!\!
 \sigma|_{s = 0} \!= \!\sigma_0, \quad\!\!
 D_s \sigma|_{s = 0}\! = \!\sigma_1.
 \label{linwaverelappendix}
\end{equation}

As in the previous section, we will show that
\eqref{linVappendixrel} is an ODE in a function space.
The norms we work with are
\begin{equation}
 \| u \|_{X_{S_0}} = {\sup}_{0 \leq s \leq S_0}
 \|u(s)\|_{r+1},
\end{equation}
where
\begin{equation}
 \|u(s)\|_{r+1} = {\sum}_{k+\ell \leq r}
  \|D_sD_s^{k} u(s)\|_{H^{\ell}(\Omega)}
 + {\sum}_{k+\ell \leq r}
\|D_s^k \pave u(s)\|_{H^{\ell}(\Omega)}
  + {\sum}_{k+\ell \leq r/2+2}
  \|\pa^\ell D_s u(s)\|_{L^\infty(\Omega)},
 \label{simplenormdefrel}
\end{equation}
where here
$\|\beta\|_{H^\ell(\Omega)} = \sum_{\ell' \leq \ell}
\|\pa_y^{\ell'} \beta\|_{L^2(\Omega)}$
where
$\|\cdot\|_{L^2(\Omega)}$ is defined as in \eqref{HLpnormdef} and
controls both space and time components.
In section \ref{compatsecrel} we prove that the map
$V\mapsto \sigma[V]$ is well-defined if the
compatibility conditions hold and $\|\widetilde{V}\|_{X_{S_1}}
+\|\S \xve\|_{X_{S_1}} < \infty$.
With
\begin{equation}
 \Sigma[V](s,y) = -\frac{1}{2}\int_{0}^s \pave \sigma[V](s',y)\, ds',
 \label{}
\end{equation}
in section \ref{compatsecrel} we prove the
bounds
\begin{align}
 \|\Sigma[V]\|_{X_{S_1}}
 &\leq C \left( \|\widetilde{V}\|_{X_{S_1}},
 \|\S \xve\|_{X_{S_1}},
 \|\wg\|_r\right)
 \left( \|\overline{V}\|_{X_{S_1}} + S_1 \|V\|_{X_{S_1}}\right),\label{oderel}\\
 \|\Sigma[V_1]-\Sigma[V_2]\|_{X_{S_1}}
 &\leq S_1 C \left( \|\widetilde{V}\|_{X_{S_1}},
 \|\S \xve\|_{X_{S_1}},
 \|V_1\|_{X_{S_1}},
 \|V_2\|_{X_{S_1}},\|\wg\|_{r+2}\right)
 \|V_1-V_2\|_{X_{S_1}}.
 \label{oderel2}
\end{align}
Here, $\|\wg\|_{r+2}$ is defined as in \eqref{simplenormdefrel}.
As in the previous section, this gives
existence for \eqref{linVappendixrel}.
\begin{prop}[Existence for the linear relativistic problem]
 Fix $r \geq 10$ and suppose that the initial data
 $\mathring{V}, \mathring{\sigma}$ satisfies the compatibility
 conditions \eqref{relcompat} to order $r+1$
 and so that $\mathring{\rho} \geq \rho_1 > 0$ with
 $\mathring{\rho} = \rho|_{s = 0}$, for some constant $\rho_1> 0$. Let
 $\widetilde{V}\in X_{S_1}$ for some $S_1 > 0$.
 Then there is $S > 0$ so that the linear smoothed problem \eqref{linVappendixrel}
 has a unique solution $V \in X_{S}$ with
 and moreover with $\overline{V}$ the formal power series
 solution at $s = 0$ defined as in \eqref{powerseriesrel}, $V$
 satisfies the bound
 \begin{equation}
  \|V\|_{X_{S}} \leq
  C \left( \|\widetilde{V}\|_{X_{S_1}},
  \|\xve\|_{X_{S_1}},
  \|\S \xve\|_{X_{S_1}},
  \|\wg\|_{r+2}\right)
  \|\overline{V}\|_{X_{S}},
 \end{equation}
 and the enthalpy satisfies
 \begin{equation}
  \sup_{0 \leq s \leq S}
  {\sum}_{k +\ell \leq r}
  \|D_s^k \pave \sigma(s)\|_{H^{\ell}(\Omega)}
  + \|D_s^k D_s \sigma(s)\|_{H^{\ell}(\Omega)}
  \leq
  C \Big(\sum_{k+\ell \leq r}\!\!\!
  \|D_s^k \pave \sigma(0)\|_{H^{\ell}(\Omega)}
  + \|D_s^{k} D_s \sigma(0)\|_{H^{\ell}(\Omega)} \Big),
  \label{}
 \end{equation}
 with $C = C\left( \|\widetilde{V}\|_{X_{S_1}},
 \|\S\xve\|_{X_{S_1}},\|\wg\|_{r+2}\right)$. Moreover the resulting density
 $\rho = \rho(\sigma)$ defined by solving \eqref{sigmadef} satisfies $\rho(s,y) > \rho_1/2$ for
 $s \leq S, y \in \Omega$.
\end{prop}

It just remains to prove that $V\mapsto \sigma[V]$ is
well-defined and that the bounds
\eqref{oderel}-\eqref{oderel2} hold.

\subsubsection{The compatibility conditions and existence for
the wave equation for the relativistic enthalpy}
\label{compatsecrel}

Let $\overline{V} = \sum_{k \geq 0} \frac{s^k}{k!}
V_k, \overline{\sigma} = \sum_{k \geq 0} \frac{s^k}{k!}
\sigma_k$ and $\overline{x}^i = x_0^i + s\overline{V}^i,
\overline{x}^0 = s \overline{V}^0$ be a formal
power series solution to \eqref{rescaledreleul2smoothed}-\eqref{rescaledrelcont2smoothed}
 in the sense that
\begin{equation}
 \hD_s^{k} \big( \hD_s \overline{V} + (1/2)\pave \overline{\sigma}\big)|_{s = 0} = 0,
 \qquad
 \hD_s^{k} \big( e(\overline{\sigma}) \hD_s \overline{\sigma} +
 \widetilde{\div} \overline{V}\big)|_{s = 0} = 0,
 \qquad
 k = 0,..., r,
 \label{powerseriesrel}
\end{equation}
where $\pave = \pave_{\overline{x}}$ denotes differentiation
with respect to the smoothed version of $\overline{x}$.
Here, to get more uniform notation we are writing
$V_0 = \mathring{V}$ for the initial velocity instead
of for the time component of $V$.
From these equations we see that there are functions
$G_k, F_k$ with
\begin{equation}
 \sigma_k = G_k(\sigma_0, x_0, V_0,...,
 V_{k-1}),
 \qquad
 V_k = F_k(\sigma_0, x_0, V_0, ..., V_{k-1}),
 \label{}
\end{equation}
and we say that initial data
$(V_0, \sigma_0)$ satisfy the compatibility conditions
to order $r$ if we have
\begin{equation}
 G_k(\sigma_0, x_0, V_0, ..., V_{k-1}) \in H^1_0(\Omega),
 \qquad \text{ for } k = 0,..., r.
 \label{relcompat}
\end{equation}
Here, for simplicity of notation we are ignoring the dependence on
the metric and the Christoffel symbols.
If the compatibility conditions hold to order $r$ then
as in the Newtonian case
one can use a Galerkin method to construct
a solution $\sigma$ to the wave equation \eqref{linwaverelappendix}
which satisfies
\begin{equation}
  D_s^k \sigma, \,
  D_s^{k-1} \pave \sigma \in L^\infty([0,S_0]; H^{r+1-k}(\Omega)),
  \quad k = 0,..., r+1,
\end{equation}
provided $\|V\|_{r,S_0}+\|\widetilde{V}\|_{r,S_0} + \|\S \xve\|_{r,S_0} < \infty$.
The only difference with the Newtonian case is that
the
structure of the wave operator on the left-hand side
of \eqref{linwaverelappendix} is a bit less obvious.
The observation which one needs is that using
the formula \eqref{waveexpression}
or equivalently the identities
\eqref{coordwavedecomp1}-
\eqref{coordwavedecomp2}, the operator on the left-hand side
of \eqref{linwaverelappendix} can be decomposed into the sum of $s$ derivatives
$D_s^2$ and an operator which is elliptic when restricted
to surfaces of constant $s$. Then the estimates which are
needed to construct a solution by a Galerkin approximation
follow in essentially the same way as the estimates
we proved in sections
 \ref{higherorderellipticrel}-\ref{enthalpyhigherorderrel}.
To prove the bounds \eqref{oderel}-\eqref{oderel2} one
argues exactly as in section \ref{compatsecrel} but using the energy
estimates from section \ref{enthalpyhigherorderrel} and the elliptic
estimate from Proposition \ref{prop:dirichletrel}.

\subsection{Construction of initial data satisfying the compatibility conditions
for the smoothed problem}
\label{compatconstruction}

In our main theorem we assumed that we were giving initial data which
satisfies compatibility conditions for the non-smoothed problem
 but in our construction we need to find initial data which satisfies compatibility
 conditions for the smoothed-out problem which are different.
  In this section
 we sketch how to construct such data. See Proposition
 E.2 of \cite{GLL19} for a detailed proof.

We suppose that we are given vector fields $V, \widetilde{V}$
which are sufficiently smooth and consider the wave equation
\begin{equation}
  D_t\big( e_1 D_t h\big) - \widetilde{\Delta} h = \widetilde{\pa}_i \widetilde{V}^j\, \widetilde{\pa}_j V^i, \quad \textrm{ in }
 [0, t_1]\! \times\!\Omega, \quad\text{with}\quad
 h\big|_{ [0, t_1] \times\pa \Omega}=0,\quad\text{where}\quad \widetilde{\Delta}\!=\delta^{ij}\nave_i\nave_j.
 \label{modelwavecc}
\end{equation}
As in earlier sections we will just discuss the case that
$e_1 > 0$ is a constant, the general case is  similar.

We now fix $\ve \geq 0$ and suppose that there are power series
$\overline{h}(t, y) = \sum_{k\geq 0} t^k h_k^\ve(y)/k!$,
 $\overline{\widetilde{V}}(t,y)=
 \overline{V}(t,y) = \sum_{k \geq 0} t^k V_k^\ve(y)/k!$,
 $\overline{\xve}(t,y) = \overline{x}(t,y) =
 \sum_{k \geq 0} t^k x^\ve_k(y)/k!$
 which satisfy the equation \eqref{modelwavecc},
 the Euler equations \eqref{eq:eulerlagrangiancoordxsmoothed}
 and the equations $D_t x = V$, $D_t \xve = \widetilde{V}$
 to order $r$ at $t = 0$.
 With $h_1^\ve$ defined by $e_1 h_1^\ve = \div V_0^\ve$,
we say that the initial data $(h_0^\ve, h_1^\ve)$
satisfies the compatibility
conditions to order $r$ if $h_k^\ve \in H^1_0(\Omega), k = 0,..., r$. The important part of this
definition is the vanishing at the boundary. The statement
about the power series just means that the higher-order
 coefficients $h_2^\ve,..., h_r^\ve$ are determined from
the given data $h_0^\ve, h_1^\ve$ by taking time derivatives of
\eqref{modelwavecc} at $t = 0$,
\begin{equation}
 e_1 h_k^\ve = \widetilde{\Delta} h_{k-2}^\ve
 + F_k^\ve[h_{(k-1)}^\ve],
 \label{determinephik}
\end{equation}
where we are evaluating the coefficients of $\widetilde{\Delta}$
at $t = 0$ and where we have introduced the notation
$h_{(j)}^\ve = (h_{-2}^\ve, h_{-1}^\ve, h_0^\ve,..., h_j^\ve)$
with $h_{-2}^\ve = x_0^\ve, h_{-1}^\ve = V_0^\ve$,
and where $F_k^\ve$ depends on up
to two derivatives of its arguments and is given by
\begin{equation}
 F_k^\ve[h_{(k-1)}^\ve] = D_t^{k-2}
 \Big(\big( \pave_i \overline{\widetilde{V}}^j \pave_j \overline{V}^i\big)
 +[D_t^{k-2}, \widetilde{\Delta}] \overline{h}
 \Big)\big|_{t = 0}.
\end{equation}
The parameter $\ve$ enters through the definition of
$\pave$ as well as $\widetilde{\Delta}$.
In this expression, $\pave, \widetilde{\Delta}$ are defined
as in \eqref{eq:eulerlagrangiancoordsmoothed} but with $x$ replaced by $\overline{x}$.
Using the fact that \eqref{eq:eulerlagrangiancoordxsmoothed}
 holds at $t = 0$ one can write
 time derivatives of $\overline{V},
 \overline{\widetilde{V}}$ and $t = 0$ in terms of the
 higher-order coefficients $h_0^\ve,..., h_r^\ve$ and similarly one can
 write the time derivatives of $\overline{\xve}^\ve, \overline{x}^\ve$
 at $t = 0$ in terms of $V_0^\ve, h_0^\ve,..., h_r^\ve$.

The result we need is then the following.
\begin{prop}
  Suppose that the initial data $(h_0, h_1)$ is such that when
  $\ve = 0$ and with $h_k^0$ defined by \eqref{determinephik},
  we have $h_k^0 \in  H^1_0(\Omega)$ for $k = 0,..., r$.
  Suppose additionally that $e_1$ is sufficiently small.
  For $\ve > 0$ sufficiently small, there is initial data
  $(h_0^\ve, h_1^\ve)$ so that with $h_k^\ve$ defined by
  \eqref{determinephik} we have $h_k^\ve \in H^1_0(\Omega)$
  for $k = 0,..., r$.
\end{prop}
To prove this result we look for data of the form
$(h_0^\ve, h_1^\ve) = (h_0 + u_0^\ve,
h_1 + u_1^\ve)$. Inserting this into
 \eqref{determinephik} we see that if
we define $u_k^\ve$ by solving
\begin{equation}
 \widetilde{\Delta} u_{k-2}^\ve + G_k[u_{(k-1)}^\ve] = \kappa u_k^\ve,
 \qquad \text{ in } \Omega,
 \qquad
 u_k^\ve =0, \qquad \text{ on } \pa \Omega,
\end{equation}
where
 $
 u_{r-1}^\ve = u_{r}^\ve = 0
 $
and where $G_k$ is given by
\begin{equation}
  G_k[u_{(k-1)}^\ve]
  =
  \left( F_k^\ve[h_{(k-1)} + u^\ve_{(k-1)}]
  - F_k^\ve[h_{(k-1)}]\right)
  + \left(F_k^\ve[h_{(k-1)}]
  - F_k[h_{(k-1)}]\right)
  + \big(\widetilde{\Delta} - \Delta\big)h_{k-2},
 \label{}
\end{equation}
then the resulting $h_0^\ve, h_1^\ve$ satisfy the compatibility
conditions to order $r$. To get back the data for
$V_0^\ve$ for $\ve > 0$ one just takes $V_0^\ve = V_0 +
\nabla u_{-1}^\ve$ where $\Delta u_{-1}^\ve = e_1 h_1^\ve$,
$u_{-1}^\ve = 0$ on $\pa\Omega$.
The above gives a system of nonlinear elliptic equations which can
be solved by iteration. Given $(u_{0}^{\ve,\nu-1},
\dots u_{r}^{\nu-1})$, construct $(u_{0}^{\ve,\nu},...,  u_{r}^{\ve,\nu})$
by solving the system
\begin{align}
  \widetilde{\Delta} u_{k-2}^{\ve,\nu}
  + G_k[u_{(k-1)}^{\ve,\nu-1}] = e_1 u_k^{\ve, \nu},
  \qquad \text{ in } \Omega,
  \qquad
  u_k^{\ve,\nu} =0 \qquad \text{ on } \pa \Omega,
 \label{}
\end{align}
and
\begin{equation}
 u_{r-1}^{\ve,\nu} = u_{r}^{\ve,\nu} = 0, \qquad
 \text{ in } \Omega.
 \label{}
\end{equation}
Provided $e_1$ is taken sufficiently small, one can use
the elliptic estimates from Proposition \ref{prop:dirichlet} to prove that the above sequence $(u^{\ve,\nu}_0,..., u^{\ve,\nu}_r)$ is
uniformly bounded and Cauchy with respect to the norms
$\sum_{k \leq r} \|u^{\nu,\ve}_k\|_{H^{r-k}(\Omega)}$.
See Proposition E.2 of \cite{GLL19} for a detailed proof.

\subsection{Construction of compatible data for the relativistic problem}
Data for the relativistic problem is constructed using the same steps as in the previous section.
The wave equation is
\begin{equation}
 e'({}_{\!}\sigma{}_{\!}) \hD_s^2 \sigma - \frac{1}{2} \widetilde{\nabla}_{\!\nu}
 (\widetilde{g}^{\mu\nu}\widetilde{\nabla}_{\!\mu}
 \sigma\!)
 \!= \!\widetilde{\nabla}_{\!\mu} \widetilde{V}^\nu \widetilde{\nabla}_{\!\nu} V^\mu + \widetilde{R}_{\mu\nu\alpha}^\mu \widetilde{V}^\nu V^\alpha\!
 - e''({}_{\!}\sigma{}_{\!}) (\hD_s\sigma)^2\!\!,
 \quad \textrm{in }
[0, s_1]{}_{\!} \times{}_{\!}\Omega \text{ with }
\sigma\big|_{ [0, s_1] \times\pa \Omega}\!=\!0
 \label{modelwaveccrel}
\end{equation}
The compatibility conditions for this equation are defined
as in the previous section. We suppose that we are given
formal power series in $s$,
$\overline{\sigma}(s, y) = \sum_{k\geq 0} s^k \sigma_k^\ve(y)/k!$,
 $\overline{\widetilde{V}}(s,y)=
 \overline{V}(s,y) = \sum_{k \geq 0} s^k V_k^\ve(y)/k!$,
 $\overline{\xve}(s,y) = \overline{x}(s,y) =
 \sum_{k \geq 0} s^k x^\ve_k(y)/k!$
 which satisfy the equation \eqref{modelwaveccrel} to order
 $r$ at $s = 0$. We can then solve for the higher-order
 coefficients $\sigma_2^\ve,..., \sigma_r^\ve$ in terms of
 $\sigma_0^\ve, \sigma_1^\ve$ and the compatibility conditions
 are that the $\sigma_k^\ve$ satisfy $\sigma_k^\ve \in H_0^1(\Omega), k = 0,... r$.

Simple modifications of the arguments used to prove
Proposition E.2 from \cite{GLL19}, using the elliptic
estimates from Proposition \ref{prop:dirichletrel}
in place of the elliptic estimate (5.8) from \cite{GLL19},
can be used to prove:
\begin{prop}
  Suppose that the initial data $(\sigma_0, \sigma_1)$ is such that when
  $\ve = 0$,
  we have $\sigma_k^0 \in  H^1_0(\Omega)$ for $k = 0,..., r$.
  Suppose additionally that $e_1 = e'(0)$ is sufficiently small.
  For $\ve > 0$ sufficiently small, there is initial data
  $(\sigma_0^\ve, \sigma_1^\ve)$ so that $\sigma_k^\ve \in H^1_0(\Omega)$
  for $k = 0,..., r$.
\end{prop}

\section{The Galerkin method}
\label{galerkinsec}

In this section, for the sake of completeness we include a sketch of a
Galerkin method which can be used to prove existence for the wave equation
\eqref{sigmawavesetup1} for the enthalpy.
 We just discuss the Newtonian case, the relativistic case being similar.

Let $P_\lambda $ denote the orthogonal projection onto the space spanned by eigenfunctions
\begin{equation}
P_\lambda f={\sum}_{\lambda_k\leq\lambda} \langle f,\psi_k\rangle \psi_k,
\end{equation}
with eigenvalues $\leq \!\lambda$.
We now want to find the solution $h^\lambda$ to the equation
\begin{equation}\label{eq:waveequationsmoothedlambda}
  D_t\big( e_1 D_t h^\lambda\big) - \widetilde{\Delta}_\lambda  h^\lambda =P_\lambda F,\quad \textrm{ in }
 [0, t_1]\! \times\!\Omega, \quad\text{with}\quad
 h^\lambda\big|_{ \pa \Omega}=0,
\end{equation}
where
$
 \widetilde{\Delta}_\lambda =P_\lambda  \widetilde{\Delta}P_\lambda
 $, with initial data
\begin{equation}
h^\lambda \big|_{t=0}=P_\lambda h_0,\qquad D_t h^\lambda\big|_{t=0}=P_\lambda h_1,
\end{equation}
Here as before we have for simplicity assumed that $e_1$ is constant.
This equation means that $h^\lambda$ is in the span of the eigenfunctions with eigenvalues $\lambda_k \leq \lambda$:
\begin{equation}
h^\lambda(t,y)={\sum}_{\lambda_k\leq \lambda} d^\lambda_k(t)\psi_k(y)
\end{equation}
and \eqref{eq:waveequationsmoothedlambda} is nothing but a system of second order ordinary differential
equations for $d^\lambda_k$ in disguise, obtained by taking the inner product with the eigenfunction of eigenvalues $\leq \lambda$. Since the number of equations are the same as the number of eigenvalues this system and hence the equation has a unique solution.

Multiplying the equation by $D_t h^\lambda$ and integrating with respect to the measure $dy$ we can remove the projections since one factor is already in the span of the eigenfunctions with eigenvalues $\lambda_k\leq \lambda$:
\begin{equation}
\int_{\Omega} D_t h^\lambda\, D_t (e_1 D_t h^\lambda) dy
- \int_{\Omega}D_t h^\lambda \, \widetilde{\Delta} h^\lambda  dy
=\int_{\Omega}D_t h^\lambda \,F  dy.
\end{equation}
 Hence $h^\lambda$ satisfy exactly the same energy estimate as $h$ with the exception that initial data are projected, but since the projection is bounded
 on the spaces we are considering it leads to the same energy bound as for $h$.
Now, in the previous sections we mostly integrated with respect to the measure
$d\widetilde{x}=\kappa dy$ in order that $\widetilde{\triangle}$ would be symmetric, however the difference just introduces a lower order term that can be controlled by the energy. Using this uniform energy bound obtained for
\begin{equation}
\int_{\Omega} e_1 (D_t h^\lambda)^2 dy
+ \int_{\Omega}\delta^{ij}  \widetilde{\pa}_i h^\lambda\, \widetilde{\pa}_j h^\lambda  dy,
\end{equation}
one obtains weak solutions as in \cite{Evans2010}. The proof there is for
time independent operator but can easily be modified as in \cite{GLL19}. Moreover by differentiating the equation with respect to $t$ one obtains the same energy bounds for $h^\lambda $ replaced by $D_t h^\lambda$ and this  gives a solution in $H^2$ using the equation and the elliptic estimate for $\widetilde{\triangle} h^\lambda$. Since we have constructed our solution
as a limit of eigenfunctions which vanish at the boundary and since we have uniform
estimates, it follows that the compatibility conditions hold at later times.

\subsection*{Acknowledgements}
Research of DG was partially supported by
the Simons Center for Hidden Symmetries and Fusion Energy.
Research of HL was supported in part by Simons Foundation Collaboration Grant 638955.

\bibliographystyle{abbrv}

\end{document}